\documentclass[a4paper, 11pt]{report}
\usepackage[T1]{fontenc}
\usepackage[utf8]{inputenc}
\usepackage[all]{xy}
\usepackage{amsfonts}

\usepackage{cite}
\usepackage{enumerate}		
\usepackage{hyperref}		
\usepackage{amssymb} 		
\usepackage{mathrsfs} 		
\usepackage{bbm} 		
\usepackage{extarrows} 		
\usepackage{stmaryrd} 		
\usepackage{mathtools}		

\usepackage{amsmath}	
\usepackage{amsthm} 		
\usepackage[nottoc]{tocbibind} 

\usepackage{makeidx} 
\makeindex

\newtheoremstyle{citedthm}%
{3pt}
{3pt}
{\itshape}
{}
{\bfseries}
{.}
{.5em}
{\thmname{#1} \thmnumber{#2}\thmnote{\normalfont#3}}
\newtheoremstyle{citeddefn}%
{3pt}
{3pt}
{}
{}
{\bfseries}
{.}
{.5em}
{\thmname{#1} \thmnumber{#2}\thmnote{\normalfont#3}}

\theoremstyle{citedthm}
\newtheorem{thm}{Theorem}
\numberwithin{thm}{section}
\newtheorem{prop}[thm]{Proposition}
\newtheorem{lemma}[thm]{Lemma}

\newtheorem{cor}[thm]{Corollary}

\newtheorem{metathm}[thm]{Metatheorem}

\theoremstyle{citeddefn}
\newtheorem{defn}{Definition}
\numberwithin{defn}{section}
\newtheorem{ex}{Example}
\numberwithin{ex}{section}

\numberwithin{nota}{section}
\newtheorem{remark}{Remark}
\newtheorem{remarks}[remark]{Remarks}
\numberwithin{remark}{section}

\usepackage[pdftex,dvipsnames]{xcolor}  
\usepackage[colorinlistoftodos,prependcaption,textsize=tiny]{todonotes}
\setuptodonotes{color=blue!10!white}
\usepackage{tikz}
\usepackage{tikz-cd} 
\usetikzlibrary{positioning,intersections} 
\usetikzlibrary{decorations.pathmorphing} 
\usetikzlibrary{arrows}

\usepackage{lipsum}                     
\usepackage{xargs}                      

\newcommandx{\unsure}[2][1=]{\todo[linecolor=red,backgroundcolor=red!25,bordercolor=red,#1]{#2}}
\newcommandx{\change}[2][1=]{\todo[linecolor=blue,backgroundcolor=blue!25,bordercolor=blue,#1]{#2}}
\newcommandx{\info}[2][1=]{\todo[linecolor=OliveGreen,backgroundcolor=OliveGreen!25,bordercolor=OliveGreen,#1]{#2}}
\newcommandx{\improvement}[2][1=]{\todo[linecolor=Plum,backgroundcolor=Plum!25,bordercolor=Plum,#1]{#2}}
\newcommandx{\thiswillnotshow}[2][1=]{\todo[disable,#1]{#2}}

\newcommand{\yo}{\text{\usefont{U}{min}{m}{n}\symbol{'110}}}
\DeclareFontFamily{U}{min}{}
\DeclareFontShape{U}{min}{m}{n}{<-> dmjhira}{}

\usepackage{xspace} 
 
\newcommand{\ie}{i.e.\@\xspace}
\newcommand{\cfr}{cf.\@\xspace}
\newcommand{\st}{s.t.\@\xspace}
\newcommand{\ac}{`}

\renewcommand{\epsilon}{\varepsilon}
\renewcommand{\phi}{\varphi}
\renewcommand{\Im}{\operatorname{Im}}

\newcommand{\inv}{^{-1}}
\newcommand{\op}{^{\textup{op}}}
\newcommand{\Vop}{^{\textup{V}}}
\newcommand{\co}{^{\textup{co}}}
\newcommand{\can}{^{\textup{can}}}

\newcommand{\fp}{_{\textup{fp}}}
\newcommand{\cont}{_{\textup{cont}}}
\newcommand{\Gr}{^{\textup{Gr}}}
\newcommand{\Gir}{\operatorname{\textup{Gir}}}
\newcommand{\Jeqv}{^{J}}
\newcommand{\etale}{^{\textup{étale}}}
\newcommand{\JPeqv}{^{J_P}}
\newcommand{\Teqv}{^{T}}
\newcommand{\name}[1]{^\ulcorner#1^\urcorner}

\newcommand{\fib}{\mathchoice
	{{\textstyle\int}}
	{{\int}}
	{{\int}}
	{{\int}}
}

\newcommand{\Hom}{\operatorname{Hom}}
\newcommand{\Id}{\operatorname{Id}}
\renewcommand{\lim}{\operatorname{lim}}
\newcommand{\colim}{\operatorname{colim}}
\newcommand{\dom}{{\operatorname{dom}}}
\newcommand{\cod}{{\operatorname{cod}}}
\newcommand{\id}{\operatorname{id}}

\newcommand{\Ev}{\operatorname{Ev}}

\newcommand{\Spec}{\operatorname{Spec}}
\newcommand{\ran}{\mathrm{ran}}
\newcommand{\lan}{\mathrm{lan}}
\newcommand{\Lan}{\mathrm{Lan}}
\newcommand{\Ran}{\mathrm{Ran}}
\newcommand{\comma}[2]			
{\mbox{$(#1\!\downarrow\!#2)$}}

\newcommand{\Ocal}{\mathcal{O}}		
\newcommand{\For}{\operatorname{For}}	
\newcommand{\sheafify}{\mathrm{a}}	
\newcommand{\Giraud}{\mathfrak{G}}	
\newcommand{\powerset}{\mathscr{P}}	
\newcommand{\canst}{\mathcal{S}}	
\newcommand{\trunc}{t}

\newcommand{\CAT}{\mathbf{CAT}}
\newcommand{\Cat}{\mathbf{Cat}}
\newcommand{\Site}{\mathbf{Site}}
\newcommand{\Cosite}{\mathbf{Com}}

\newcommand{\Com}{\mathbf{Com}}
\newcommand{\Topos}{\mathbf{Topos}}
\newcommand{\EssTopos}{\mathbf{EssTopos}}
\newcommand{\LocTopos}{\mathbf{LocTopos}}
\newcommand{\Set}{\mathbf{Set}}
\newcommand{\Top}{\mathbf{Top}}
\newcommand{\Loc}{\mathbf{Loc}}
\newcommand{\Preord}{\mathbf{Preord}}
\newcommand{\Locale}{\mathbf{Locale}}
\newcommand{\Frame}{\mathbf{Frame}}
\newcommand{\onecat}{\mathbbm{1}}
\newcommand{\twocat}{\mathbbm{2}}
\newcommand{\Ring}{\mathbf{Ring}}

\newcommand{\Ind}{\mathbf{Ind}}
\newcommand{\Psh}{\mathbf{Psh}}
\newcommand{\Sh}{\mathbf{Sh}}
\newcommand{\St}{\mathbf{St}}

\newcommand{\Mor}{{\textup{Mor}}}
\newcommand{\Lax}{\mathbf{Lax}}

\newcommand{\Fib}{{\mathbf{Fib}}}
\newcommand{\Cl}{{\mathbf{c}}}
\newcommand{\Spl}{{\mathbf{s}}}
\newcommand{\Etale}{{\mathbf{Etale}}}

\newcommand{\Sub}{\textup{Sub}}

\newcommand{\acat}{\mathbb{A}}

\newcommand{\ccat}{\mathbb{C}}
\newcommand{\dcat}{\mathbb{D}}
\newcommand{\ecat}{\mathbb{E}}
\newcommand{\fcat}{\mathbb{F}}
\newcommand{\icat}{\mathbb{I}}

\newcommand{\xcat}{\mathbb{X}}

\newcommand{\abicat}{\mathcal{A}}
\newcommand{\bbicat}{\mathcal{B}}
\newcommand{\cbicat}{\mathcal{C}}
\newcommand{\dbicat}{\mathcal{D}}
\newcommand{\ebicat}{\mathcal{E}}
\newcommand{\fbicat}{\mathcal{F}}
\newcommand{\gbicat}{\mathcal{G}}
\newcommand{\hbicat}{\mathcal{H}}
\newcommand{\ibicat}{\mathcal{I}}
\newcommand{\jbicat}{\mathcal{J}}
\newcommand{\kbicat}{\mathcal{K}}

\newcommand{\mbicat}{\mathcal{M}}

\newcommand{\pbicat}{\mathcal{P}}
\newcommand{\qbicat}{\mathcal{Q}}
\newcommand{\rbicat}{\mathcal{R}}
\newcommand{\sbicat}{\mathcal{S}}

\newcommand{\ubicat}{\mathcal{U}}
\newcommand{\vbicat}{\mathcal{V}}

\newcommand{\xbicat}{\mathcal{X}}

\newcommand{\zbicat}{\mathcal{Z}}
\newcommand{\Ifrak}{\mathfrak{I}}

\newcommand{\Lfrak}{\mathfrak{L}}

\newcommand{\pfrak}{\mathfrak{p}}
\newcommand{\qfrak}{\mathfrak{q}}

\newcommand{\Etopos}{\mathscr{E}}
\newcommand{\Ftopos}{\mathscr{F}}

\newcommand{\imp}{\!\Rightarrow\!}
\newcommand{\mono}{\rightarrowtail}
\newcommand{\isorightarrow}{\xrightarrow{\sim}}
\newcommand{\Isorightarrow}{\xRightarrow{\sim}}

\makeatletter
\newbox\xrat@below
\newbox\xrat@above
\newcommand{\xrightarrowtail}[2][]{%
	\setbox\xrat@below=\hbox{\ensuremath{\scriptstyle #1}}%
	\setbox\xrat@above=\hbox{\ensuremath{\scriptstyle #2}}%
	\pgfmathsetlengthmacro{\xrat@len}{max(\wd\xrat@below,\wd\xrat@above)+.6em}%
	\mathrel{\tikz [>->,baseline=-.75ex]
		\draw (0,0) -- node[below=-2pt] {\box\xrat@below}
		node[above=-2pt] {\box\xrat@above}
		(\xrat@len,0) ;}}
\makeatother

\tikzset{Rightarrow/.style={double equal sign distance,>={Implies},->},
	triple/.style={-,preaction={draw,Rightarrow}}}

\begin{document}

\title{Relative topos theory via stacks}

\author{Olivia Caramello \hspace{0.2cm} and \hspace{0.2cm} Riccardo Zanfa}

\date{July 9, 2021}

\maketitle

\tableofcontents

\chapter*{Conventions and notation}
\addcontentsline{toc}{chapter}{\protect\numberline{}Conventions and notation}

By `site' we will mean `small-generated site' (see \cite[Example 3.24]{shulman.exact}): the pair $(\cbicat,J)$ is a \emph{small-generated site}\index{site!small-generated} if $\cbicat$ is locally small and has a small $J$-dense subcategory. We will do so because most results about small sites still hold for small-generated sites; moreover, this wider approach encompasses into the theory of sites every Grothendieck topos $\Etopos$, seen as the small-generated site $(\Etopos,J\can_\Etopos)$, and every geometric morphism, whose inverse image is then a morphism of sites. This makes the language much more pliable. 

Given a site $(\cbicat,J)$, the topos of sheaves is denoted by $\Sh(\cbicat,J)$, and the sheafification-inclusion embedding is denoted by $\sheafify_J\dashv \iota_J:\Sh(\cbicat,J)\hookrightarrow[\cbicat\op,\Set]$\index{$\iota_J$}\index{$\sheafify_J$}. The symbol $\yo_\cbicat:\cbicat\hookrightarrow[\cbicat\op,\Set]$\index{$\yo_\cbicat$} denotes the Yoneda embedding, while $\ell_J:\cbicat\rightarrow \Sh(\cbicat,J)$\index{$\ell_J$} the composite $\sheafify_J\yo_\cbicat$. In some cases, especially in proofs, we will use the French school notations $\widehat{\cbicat}$\index{$\widehat{\cbicat}$} and $\widetilde{\cbicat}$\index{$\widetilde{\cbicat}$} as a shorthand for respectively $[\cbicat\op,\Set]$ and $\Sh(\cbicat,J)$.

Two adjoints forming a geometric morphism will as always be denoted by $F^*\dashv F_*$; we will denote by $F_!$ a further left adjoint of $F^*$ (when it exists $F$ is said to be \emph{essential}\index{geometric morphism!essential}). 

Natural transformations, and more in general 2-cells in 2-categories, will be denoted with a double arrows $\Rightarrow$. We set as convention that a natural transformation between geometric morphisms $\alpha:F\Rightarrow G$ is a natural transformation between their inverse images: that is, $\alpha: F^*\Rightarrow G^*$. 

In the context of a fixed model of set theory, we denote by $\CAT$ the $2$-category of locally small categories and by $\Cat$ the $2$-category of small categories.

$\Topos$ is the $2$-category of toposes and geometric morphisms, while $\EssTopos$ the $2$-category of toposes and \emph{essential} geometric morphisms. The category of functors from $\abicat$ to $\bbicat$ is indicated with $[\abicat,\bbicat]$; in particular, $[\cbicat\op,\Set]$ denotes the topos of presheaves over $\cbicat$. Generic categories are denoted with the cursive font ($\cbicat$, $\dbicat$...); in particular we will use the calligraphic font $\Etopos$, $\Ftopos$... for toposes. Pseudofunctors will usually be denoted by blackboard letters, such as $\dcat:\cbicat\op\rightarrow\Cat$. 

When speaking about an arrow $y:Y\rightarrow X$ of $\cbicat$ as an object of the slice category $\cbicat/X$, we will refer to it as $[y]$ in square brackets; arrows of $\cbicat/X$ on the other hand will not have a specific notation (so for instance we will write $z:[yz]\rightarrow [y]$).

We shall adopt the standard notation for opposite categories: the opposite of a $1$-category $\cbicat$ will be denoted with $\cbicat\op$, and the same notation will be used for the category obtained by reversing the $1$-cells of a $2$-category; on the other hand, if $\cbicat$ is a $2$-category, we shall denote by $\cbicat\co$\index{$(-)\co$} the $2$-category obtained by reversing its $2$-cells.

\chapter{Introduction}

In this work we develop \emph{relative topos theory}\index{relative topos theory}, that is the theory of toposes over an arbitrary base topos, by using the language of stacks.

Having an efficient formalism for doing topos theory over an arbitrary base topos is essential for several reasons. Relativity techniques for schemes have played a key role in Grothendieck's refoundation of algebraic geometry. We aim for an analogous formalism for toposes, allowing for even more general and powerful change-of-base techniques. In fact, a standard methodology dating back to the old days of topos theory consists in establishing results over a given topos by regarding it as the basic universe relative to which such results are formulated (for instance, by using the internal language); these results are then \emph{externalized} to yield formulations of them in terms of a different universe, most often the classical set-theoretic one. The non-trivial nature of the externalization process makes this an effective technique for proving results in an intuitive and synthetic way, leaving to the topos-theoretic \ac machine' the canonical task of translating such results in the context of alternative universes. Of course, the more efficient is the formalism for doing relative topos theory, the more effective this methodology will be for the desired applications. 

We should mention that an approach to relative topos theory by using the notions of internal category and internal site has been developed since the seventies by a number of category theorists, most notably Lawvere, Diaconescu and Johnstone (see, for instance, \cite{diaconescu.pullback} and \cite{elephant}). This approach is not satisfactory for us, as the notion of internal category is quite rigid and not suitable for expressing natural higher-categorical phenomena such as those arising in geometric situations, for instance when considering the canonical stack of a geometric morphism (which is \emph{not} an internal category). For these reasons, we have resorted to the original approach introduced by Giraud in \cite{giraud.classifying}, developing our theory starting from it. 

This document is a first draft of a much longer text developing the foundations of our theory of relative toposes via stacks; we shall progressively release new versions containing additions, revisions and updates. 

In particular, by using a suitable generalization of the stack semantics of \cite{Shulman2} and \cite{Shulman1}, we shall introduce appropriate relative versions of all the main notions of classical topos theory, such as the notions of flat functor, morphism and comorphism of sites, separating sets, denseness conditions, and prove relative generalizations of the key theorems concerning them, notably including Diaconescu's theorem and Giraud's theorem.     

We shall also investigate the logical counterpart of our geometric approach to relative toposes through relative sites, by introducing \emph{relative geometric logic}; this will be a higher-order parametric logic allowing one to quantify over \ac parameters' coming from the base topos, whilst maintaining the fundamental geometric features allowing for the existence of associated classifying toposes.

\section{Stacks for relative topos theory}

In the literature, stacks have often been called `the right notion of sheaf of categories', or `the right notion of $2$-sheaf'. 
It is superfluous to discuss the relevance of stacks to geometry: their application (which started in the early sixties in Grothendieck's \textit{entourage}) provided a way to effectively manipulate geometric objects such as schemes, and shed new light on several fundamental themes in algebraic geometry. From that point of view, a stack is a sheaf for which \textit{local equality} of data is substituted with \textit{local equality up to canonical isomorphism}: a simple but fundamental example of this is the gluing of schemes (or even topological spaces), that may be locally isomorphic but not locally equal. Standard references on stacks in algebraic geometry are \cite{LaumonBailly} and \cite{vistoli.stack}. In this text we shall be concerned with stacks of categories, but most of our constructions and techniques can be easily modified to obtain analogous results for stacks of groupoids. 

Stacks also provide a more flexible alternative to the notion of internal category. An internal category in a sheaf topos is really just a sheaf whose values happen to live not just in $\Set$, but in $\CAT$. Therefore, this notion is too rigid to encompass most phenomena that occur in nature: consider for instance Example \ref{ex:indexedcategories}(ii), which one would be eager to consider a sheaf of categories, if it were not for the fact that the composition of pullbacks is the pullback along the composite only up to isomorphism. Using the Yoneda lemma for fibrations, any $\cbicat$-indexed category $\dcat$ \textit{can} be strictified and treated as an internal category, but that appears to be an artificious rather than a simplifying strategy: instead of manipulating stacks in their right environment, strictification forces them into a context unnaturally rigid for them. For a treatment of the relationship between stacks and internal categories, we refer the reader to \cite{bunge.stacksandinternalcat} and \cite{BungePare}.

In the context of this work, the role of stacks will be \emph{two-fold}:

\begin{enumerate}[(1)]
	\item On the one hand, the notion of stack represents a higher-order categorical generalization of the notion of sheaf. Accordingly, categories of stacks on a site represent higher-categorical analogues of Grothendieck toposes, which are categories of sheaves on a site (up to equivalence). One can thus expect to be able to lift a number of notions and constructions pertaining to sheaves (resp. Grothendieck toposes) to stacks (resp. categories of stacks on a site). Examples of such notions, which we shall treat in this work, are the stackification construction (which generalizes sheafification), direct and inverse images of stacks, the theory of morphisms between categories of stacks and their relations with functors between the sites. The relationship between sheaves and stacks is indeed very natural; in fact, for any site, the category of sheaves on it embeds as a reflective subcategory of the category of stacks on it (see section \ref{sec:bigpicture}). The categories of stacks on a site inherit some of the pleasant categorical properties of Grothendieck toposes. Our main tool for studying them will be the fundamental adjunction of Chapter \ref{chap:fundadj} between indexed categories and relative geometric morphisms. Indeed, another class of categories which is naturally related to the category of stacks on a site is that of geometric morphisms to a given base topos. Like categories of stacks, these categories satisfy a number of good categorical properties. 
	
	\item On the other hand, as observed above, stacks on a site $({\cal C}, J)$ generalize internal categories in the topos $\Sh({\cal C}, J)$. Since (usual) categories can be endowed with Grothendieck topologies, so stacks on a site can also be endowed with suitable analogues of Grothendieck topologies. This leads to the notion of \emph{site relative to a base topos}, which is instrumental for developing the theory of \emph{relative toposes} (with respect to an arbitrary base topos). The development of this theory parallels that of the classical theory; indeed, by using a general stack semantics, we will be able to introduce, in a canonical, not \emph{ad hoc} way, natural generalizations to relative sites of the classical notions of morphism and comorphism of sites, flat functors, separating sets for a topos, denseness conditions etc.
\end{enumerate} 

As in classical topos theory the two perspectives are unified by the well-known result that every Grothendieck topos is equivalent to the category of sheaves on itself, regarded as a site with the canonical topology, so in relative topos theory we have an analogue of this result, obtained by replacing that canonical (external) site with a relative site whose underlying stack is the canonical stack on the topos.

\section{The big picture}\label{sec:bigpicture}

Our general framework for investigating Grothendieck toposes from a geometric, fibrational point of view is based on a network of 2-adjunctions, as follows:

\begin{equation*}
\begin{tikzcd} 
\Ind_\cbicat \ar[d, "s_{J}", xshift=1ex]  \arrow[r, "\Lambda", yshift=1ex] \ar[r, phantom, "\bot"{yshift=0ex}] &   \ar[l, "\Gamma", yshift=-1ex]  \Topos\slash \Sh({\cal C}, J)\co \\
\St({\cal C}, J) \ar[u, right hook->, xshift=-1ex] \ar[u, phantom, "\vdash"{xshift=0ex}] \ar[d, "E\circ \Lambda'", xshift=1ex]  \arrow[r, "\Lambda'", yshift=1ex] \ar[r, phantom, "\bot"{yshift=0ex}]     &  \EssTopos_{\Sh({\cal C}, J)} \ar[dl, "\quad\:\:\: L"]  \ar[l, "\Gamma'", yshift=-1ex]  \ar[u, right hook->, xshift=-1ex]             \\
\Sh({\cal C}, J) \ar[u, phantom, "\vdash"{xshift=0ex}] \ar[u, right hook->, xshift=-1ex] \ar[ur, "E\quad\:\:", xshift=5ex] \ar[ur, phantom, "\dashv"{sloped, rotate=-90, xshift=1ex, yshift=2ex}]  &  
\end{tikzcd}
\end{equation*}	
In this diagram $\Ind_\cbicat$ denotes the $2$-category of $\cal C$-indexed categories, $\St({\cal C}, J)$ the $2$-category of $J$-stacks on $\cal C$ (where $({\cal C}, J)$ is a small-generated site), $s_{J}$ the stackification functor (see Theorem \ref{thm:stackification_esiste}), $\Topos$ the $2$-category of Grothendieck toposes and geometric morphisms and $\EssTopos_{\Sh({\cal C}, J)}$\index{$\EssTopos_{\Sh({\cal C}, J)}$} the full subcategory of $\Topos\slash \Sh({\cal C}, J)\co$ on the essential geometric morphisms. 

The functor $L$ sends an object $P$ of $\Sh({\cal C}, J)$ to the canonical local homeomorphism $\Sh({\cal C}, J)\slash P\to \Sh({\cal C}, J)$. 

The functor $E$ sends an essential geometric morphism $f:\Etopos\to \Sh({\cal C}, J)$ to the object $f_{!}(1_{\Etopos})$ (where $f_{!}$ is the left adjoint to the inverse image $f^{\ast}$ of $f$), and it acts on morphisms as follows. Given a geometric morphism $f:[p:{\cal F} \to \Sh({\cal C}, J)] \to [q:{\cal E} \to \Sh({\cal C}, J)]$, we have $f^{\ast}\circ q^{\ast}\cong p^{\ast}$; we define the arrow $p_{!}(1_{\cal F})\to q_{!}(1_{\cal E})$ in $\Sh({\cal C}, J)$ associated with $f$ as the transpose of the arrow $1_{\cal F}\to p^{\ast}(q_{!}(1_{\cal E}))\cong f^{\ast}(q^{\ast}(q_{!}(1_{\cal E})))$ obtained by applying $f^{\ast}$ to the unit $1_{\cal E}\to q^{\ast}(q_{!}(1_{\cal E})$ at $1_{\cal E}$ of the adjunction $(q_{!}\dashv q^{\ast})$. 

The functor $\Lambda$ sends a $\cal C$-indexed category $P$ to the geometric morphism $\Sh({\cal G}(P), J_{P})\to \Sh({\cal C}, J)$ induced by the fibration $\pi_{P}:{\cal G}(P)\to {\cal C}$ obtained by applying the Grothendieck construction to $P$, regarded as a comorphism of sites $({\cal G}(P), J_{P})\to ({\cal C}, J)$, where $J_{P}$ is the smallest Grothendieck topology making $\pi_{P}$ a comorphism of sites to $({\cal C}, J)$. In the converse direction, the functor $\Gamma$ associates with a geometric morphism $f:{\Etopos}\to \Sh({\cal C}, J)$ the indexed category sending any object $c$ of $\cal C$ to the category of geometric morphisms over $\Sh({\cal C}, J)$ from $\Sh({\cal C}, J)\slash \ell_{\cal C}(c)$ (where $\ell_{\cal C}$ is the canonical functor ${\cal C}\to \Sh({\cal C}, J)$) to ${\Etopos}$, regarded as a topos over $\Sh({\cal C}, J)$ via $f$. The functor $\Lambda$ takes values in the subcategory $\EssTopos_{\Sh({\cal C}, J)}\hookrightarrow \Topos\slash \Sh({\cal C}, J)\co$ (in fact, more specifically, by Corollary 4.58 \cite{denseness}, it takes values in the full subcategory on the locally connected geometric morphisms). The functor $\Lambda'$ is the restriction of $\Lambda$ to the full subcategories $\St({\cal C}, J)$ and $\EssTopos_{\Sh({\cal C}, J)}$. The functor $\Gamma$ takes values in the subcategory $\St({\cal C}, J)\hookrightarrow \textbf{Ind}_\cbicat$; we denote by $\Gamma'$ the restriction of $\Gamma$ to the subcategories $\EssTopos_{\Sh({\cal C}, J)}$ and $\St({\cal C}, J)$. 

We refer to the adjunction $(\Lambda\dashv \Gamma)$ as to the \emph{fundamental adjunction}. This adjunction is a pointfree generalization of the well-known adjunction between presheaves on a topological space $X$ and bundles to $X$. It allows to interpret many topos-theoretic constructions in a geometric way; for instance, the associated sheaf functor is obtained as the restriction to presheaves of the composition $\Gamma \circ\Lambda$ of the two functors involved in the fundamental adjunction. In particular, the values of this functor can be described in terms of locally defined collections of comorphisms of sites. This adjunction also allows to describe the operations of direct and inverse images of sheaves (and more generally stacks) in a geometric way.

The geometric morphisms in the essential image of the functor $L$ are called the \emph{étale morphisms} or \emph{local homeomorphisms} to $\Sh({\cal C}, J)$; indeed, such morphisms satisfy abstract analogues of the properties of étale maps (or local homeomorphisms) of topological spaces.

The embedding of $\EssTopos_{\Sh({\cal C}, J)}$ into $\Topos\slash \Sh({\cal C}, J)\co$ does \emph{not} admit a left adjoint. Indeed, if such an adjoint existed, its composite with $E$ would yield a left adjoint to the canonical functor $\Sh({\cal C}, J) \to \Topos\slash \Sh({\cal C}, J)\co$, and hence this latter functor would preserve arbitrary limits, which is easily seen not to be the case (this is actually due to the fact that inverse image functors of geometric morphisms do not preserve in general arbitrary limits).  

Of particular interest is the functor $E\circ \Lambda'$, which is left adjoint to the inclusion $\Sh({\cal C}, J)$ into $\St({\cal C}, J)$. As such, this functor can be thought of as a truncation functor which associates with a stack the \ac best sheaf' approximating it: we will study it in detail in Section \ref{sec:truncation_functor}.

The $2$-adjunction $(\Lambda \dashv \Gamma)$ also yields a $2$-adjunction 

	\[
	\begin{tikzcd}
		\phantom{\slash}\Cl\Fib^J_\cbicat \arrow[r, "\Lambda_{\EssTopos\co/\Sh(\cbicat,J)}", bend left, start anchor={north east}, end anchor={north west}] \ar[r,phantom, "\vdash"{rotate=90}]&   \EssTopos\co/\Sh(\cbicat,J)\phantom{{}^J_\cbicat} \arrow[l, "\Gamma_{\EssTopos\co/\Sh(\cbicat,J)}", bend left,  start anchor={south west}, end anchor={south east}]
	\end{tikzcd}
	\]
between $\Ind_\cbicat$ and the $2$-category $\EssTopos\co/\Sh(\cbicat,J)$, where the $2$-functor  $\Lambda_{\EssTopos\co/\Sh(\cbicat,J)}$ acts as $\Lambda$ and the $2$-functor $\Gamma_{\EssTopos\co/\Sh(\cbicat,J)}$ is the subfunctor of $\Gamma$ obtained by taking the essential geometric morphisms (that is, $$\Gamma_{\EssTopos\co/\Sh(\cbicat,J)}(\Etopos):=\EssTopos\co/\Sh(\cbicat,J)(\widetilde{\cbicat/-},\Etopos):\cbicat\op\rightarrow\CAT$$ for any $\Sh({\cal C}, J)$-topos $\Etopos$).

\chapter{Preliminaries}\label{chap:preliminaries}

As already mentioned in the introduction, our point of view on relative topos theory stems from \cite{giraud.classifying}, and it rests on the following two assumptions: that relative toposes over a Grothendieck topos $\Sh(\cbicat,J)$ are to be investigated by using $J$-stacks $\cbicat\op\rightarrow\Cat$, and in turn that stacks can be interpreted as continuous comorphisms over $(\cbicat,J)$. All the relevant tools and techniques, which we shall exploit in the following, are collected in this initial chapter. In order to make our work as self-contained as possible, along with some new results there will be many well-known propositions, of which we tried to give just the necessary details to help a less experienced reader. 

In the first sections, we will recall the notion of $\cbicat$-indexed categories and fibrations over $\cbicat$: in particular, we will recall the equivalence-stable notion of \emph{Street fibration}, and show how one can translate the classical arguments about Grothendieck fibrations to the equivalence-stable setting.  As fibrations over a base category $\cbicat$ generalize presheaves, \emph{stacks} over a site $(\cbicat,J)$ generalize sheaves: their definition, along with the generalization to Street fibrations of some classical results, is provided in Section \ref{sec:stack}. In particular, the \emph{canonical stack} for a site (Subsection \ref{sec:canonical_stack}) will prove to be one of the main protagonists of our later results. We will also observe that stacks and sheaves are connected by a truncation functor, related with factorization systems for geometric morphisms and comorphisms of sites (Section \ref{sec:truncation_functor}).

The interplay between the fibrational and the indexed standpoint will be the \textit{leitmotiv} of many of the following results, and in particular it will be at the core of the fundamental adjunction of Chapter \ref{chap:fundadj}:
\[
\begin{tikzcd}
	\phantom{\slash}\Ind_\cbicat \arrow[r,bend left,"\Lambda", ""{below, name=A}, start anchor={north east}, end anchor={north west}] &   \Topos/\Sh(\cbicat,J)\co \arrow[l, bend left, "\Gamma", ""{above, name=B}, start anchor={south west}, end anchor={south east}] \ar[from=A, to=B, "\top"{rotate=180}, phantom]
\end{tikzcd},
\]
where the left-hand side is the $\cbicat$-indexed standpoint while the right-hand side is the topos-theoretic (fibrational) standpoint. The passage $\Lambda$ from fibrations to toposes will be justified in light of the theory of \emph{continuous comorphisms} of sites, whose basic results are gathered in Sections \ref{section:comorphisms} and \ref{sec:giraudtpl}; on the other hand the right adjoint $\Gamma$ behaves as a hom-functor, meaning that its left adjoint $\Lambda$ is a colimit functor: we have thus collected some results about weighted colimits in Section \ref{sec:colimiti}. In particular, we have focused our attention on the problem of connecting various notions of weighted colimits (pseudocolimits, lax and oplax colimits) by using localizations of fibrations (Section \ref{sec:localization_fibrations}).

\section{Indexed categories and Street fibrations}
Let us start from recalling the basic notions of indexed category and Street fibration: the latter is an equivalence-stable generalization of the more notorious notion of Grothendieck fibration (see \cite{nlab:street_fibration} and its references). From now on, every time we will mention fibrations or fibred categories we will implicitly be talking about fibrations in the sense of Street, unless specified otherwise. It is known that in $\Cat$ Street fibrations are actually \textit{equivalent} to Grothendieck fibrations, therefore all results for Grothendieck fibrations that are equivalence-stable hold also in the Street context; nonetheless we will provide a sketch of proof for some of them. 

Since here and in the following we will be using various notions of slice category, let us begin by setting the notation once and for all:
\begin{defn}\label{def:slice}
	Consider a 1-category $\cbicat$ and an object $X$: by $\cbicat/X$ we shall denote the usual \emph{slice category}\index{category!slice}, such that
	\begin{itemize}
		\item objects are arrows $y:Y\rightarrow X$ of $\cbicat$, and
		\item arrows $z:[w:W\rightarrow X]\rightarrow [y:Y\rightarrow X]$ correspond to arrows $z:W\rightarrow Y$ of $\cbicat$ such that $yz=w$.
	\end{itemize}
	If $\abicat$ is a strict 2-category we can still perform its 1-categorical slice over an object $A$, which we will denote by $\abicat/_1A$\index{$\cbicat/_1X$}. The right 2-categorical notion for an arbitrary lax 2-category $\abicat$ is that of \emph{slice 2-category} $\abicat/A$\index{$\cbicat/X$}, thus defined:
	\begin{itemize}
		\item 0-cells are arrows $b:B\rightarrow A$;
		\item 1-cells from $[p:C\rightarrow A]$ to $[b:B\rightarrow A]$ correspond to pairs $(c,\gamma)$ where $c: C\rightarrow B$ and $\gamma:bc\Isorightarrow p$;
		\item 2-cells from $(c,\gamma)$ to $(d,\delta):[p]\rightarrow[b]$ are 2-cells $\omega:c\Rightarrow d$ such that $\phi=\gamma(b\circ \omega)$.
	\end{itemize}
	If we do not require that $\gamma$ in the definition of 1-cell is invertible, we obtain instead the notion of \emph{lax-slice-2-category} $\abicat\sslash A$\index{$\abicat\sslash A$} (see \cite{nlab:slice_2-category}).
\end{defn}

Categories indexed over a base category $\cbicat$ are nothing but presheaves taking values in $\CAT$ instead of $\Set$; however, it proves necessary to replace the rigid functoriality of presheaves with pseudofunctoriality as follows: 
\begin{defn}\label{def:indexed_category}
	Consider a category $\cbicat$: a \textit{pseudofunctor}\index{pseudofunctor} $\dcat:\cbicat\op\rightarrow \CAT$ is the datum of 
	\begin{itemize}
		\item a category $\dcat(X)$ for each object $X$ in $\cbicat$,
		\item a functor $\dcat(y):\dcat(X)\rightarrow\dcat(Y)$ for each arrow $y:Y\rightarrow X$ in $\cbicat$,
		\item a natural isomorphism $\phi^\dcat_X:1_{\dcat(X)}\Isorightarrow \dcat(1_X)$ for each $X$ in $\cbicat$, and
		\item a natural isomorphism $\phi^\dcat_{y,z}:\dcat(z)\dcat(y)\Isorightarrow \dcat(yz)$ for each composable pair $y,z$ of arrows in $\cbicat$,
	\end{itemize}
	
	satisfying the following compatibility conditions: for any $y:Y\rightarrow X$,
	$$\phi^\dcat_{1_X,y}(\dcat(y)\circ \phi^\dcat_X)=\phi^\dcat_{y,1_Y}(\phi^\dcat_{Y}\circ \dcat(y))=\id_{\dcat(y)}:\dcat(y)\Rightarrow \dcat(y)$$
	and for any $w:W\rightarrow Z$, $z:Z\rightarrow Y$, $y:Y\rightarrow X$,
	$$\phi^\dcat_{y,zw}(\phi^\dcat_{z,w}\circ\dcat(y))=\phi^\dcat_{yz,w}(\dcat(w)\circ \phi^\dcat_{y,z}):\dcat(w)\dcat(z)\dcat(y)\Rightarrow \dcat(yzw).$$
	In short, a $\cbicat$-indexed category is functorial up to canonical 2-isomorphisms; a \emph{strict $\cbicat$-indexed category} is a functor $\dcat:\cbicat\op\rightarrow\CAT$ in the usual sense.
	
	A \textit{pseudonatural transformation}\index{transformation!pseudonatural} $F:\dcat\Rightarrow \ecat$ consists of the following data: for every $X$ in $\cbicat$ of  a functor $F_X:\dcat(X)\rightarrow \ecat(X)$ and for every $y:Y\rightarrow X$ of $\cbicat$ a natural isomorphism $F_y:\ecat(y)F_X\Isorightarrow F_Y\dcat(y)$ 
	satisfying the following compatibility conditions: for every $X$ in $\cbicat$
	$$F_X\circ \phi^\dcat_X=F_{1_X}(\phi^\ecat_X\circ F_X):F_X\Rightarrow F_X\dcat(1_X)$$
	and for every $z:Z\rightarrow Y$, $y:Y\rightarrow X$, 
	$$F_{yz}(\phi^\ecat_{y,z}\circ F_X)=(F_Z\circ \phi^\dcat_{y,z})(F_z\circ \dcat(y))(\ecat(z)\circ F_y): \ecat(z)\ecat(y)F_X\Rightarrow F_Z\dcat(yz).$$
	If we do not ask the components $F_y$ to be invertible, we have what is called a \emph{lax natural transformation}; if moreover we ask that the $F_y$'s go in the opposite direction, we can suitably adapt the two axioms above to obtain the definition of \emph{oplax natural transformation}\index{transformation!lax, oplax}.
	
	A \emph{modification of pseudo-/oplax/lax natural transformations}\index{modification}, $\xi: F\Rrightarrow G$, consists for every $X$ in $\cbicat$ of a natural transformation $\xi_X:F_X\Rightarrow G_X$ such that for every $y:Y\rightarrow X$ the identity
	$$G_y (\ecat(y)\circ \xi_X)=(\xi_Y\circ \dcat(y))F_y: \ecat(y)F_X\Rightarrow G_Y\dcat(y)$$
	is satisfied.
	
	Pseudofunctors $\cbicat\op\rightarrow\CAT$, their pseudonatural transformations and modifications are respectively the 0-cells, 1-cells and 2-cells of a 2-category which we shall denote by $[\cbicat\op,\CAT]_{ps}$\index{$[\cbicat\op,\CAT]_{\bullet}$}. If instead of pseudonatural transformations we consider lax natural transformations as 1-cells we still have a 2-category, denoted by $[\cbicat\op,\CAT]_{lax}$; similarly, if the 1-cells are oplax natural transformations we shall use the notation $[\cbicat\op,\CAT]_{oplax}$.
	
	Pseudofunctors $\cbicat\op\rightarrow\CAT$ can also be thought as \emph{$\cbicat$-indexed categories}\index{indexed category}; their pseudonatural transformations are thus called \emph{$\cbicat$-indexed functors} and the modifications \emph{$\cbicat$-indexed natural   transformations}. This justifies the more compact notation $\Ind_\cbicat$\index{$\Ind_\cbicat$} for the 2-category $[\cbicat\op,\CAT]_{ps}$, which we will adopt whenever the focus is not really on the kind of transformation considered.
\end{defn}

If we consider a set $I$ and a $I$-indexed family of sets $\{X_i\ |\ i\in I \}$, there is an obvious way of gluing them all together: one considers their disjoint union $\coprod_{i\in I} X_i$. The original sets can now be retrieved as fibres of the canonical projection map $\coprod_i X_i \rightarrow I$. In a similar fashion, the whole information contained in a $\cbicat$-indexed category $\dcat$ can be glued together in one single category over $\cbicat$, exploiting the well-known \emph{Grothendieck construction}:
\begin{defn}\index{Grothendieck construction}
	Any $\cbicat$-indexed category $\dcat:\cbicat\op\rightarrow\CAT$ is canonically associated to a functor $p_\dcat:\gbicat(\dcat)\rightarrow\cbicat$, defined as follows: the objects of $\gbicat(\dcat)$ are pairs $(X,U)$ with $X$ in $\cbicat$ and $U$ in $\dcat(X)$, and arrows of $\gbicat(\dcat)$ are pairs $(y,a):(X,V)\rightarrow (X,U)$ where $y:Y\rightarrow X$ and $a:V\rightarrow \dcat(y)(U)$ in $\dcat(Y)$. The identity arrow of $(X,U)$ is the pair $(1_X, \phi^\dcat_X(U))$, while composition of arrows is defined by the equation $(y,a)\circ(z,b)=(yz, \phi_{y,z}(U) \dcat(z)(a) b)$. The functor $p_\dcat$ acts by forgetting the second component.
	
	From a $\cbicat$-indexed functor $F:\dcat\rightarrow\ecat$ we can obtain a functor $\gbicat(F):\gbicat(\dcat)\rightarrow\gbicat(\ecat)$ such that $p_\ecat \gbicat(F)=p_\dcat$. We set $\gbicat(F)(X,U):=(X,F_X(U))$, and for any arrow $(y,a):(Y,V)\rightarrow (X,U)$ in $\gbicat(\dcat)$ we set $\gbicat(F)(y,a)$ to be the arrow $(t, F_y(U)\inv F_Y(a)):(Y, F_Y(V))\rightarrow (X, F_Y(U))$. Moreover, a $\cbicat$-indexed natural transformation $\xi:F\Rightarrow G$ is sent to a natural trasformation $\gbicat(\xi):\gbicat(F)\Rightarrow \gbicat(G)$ satisfying the identity $p_\ecat\circ \gbicat(\xi)=\id_{p^\dcat}$. This means that $\gbicat$ provides a strict 2-functor $\gbicat:\Ind_\cbicat\rightarrow \CAT/\cbicat$\index{$\gbicat$} (which moreover factors through the sub 2-category of strictly commutative triangles over $\cbicat$).
	\[
	\begin{tikzcd}[column sep=1.2cm]
		\dcat\ar[r, bend left, "F", ""{below, name=A}] \ar[r, bend right, "G"', ""{above, name=B}] \ar[from=A, to=B, Rightarrow, "\xi"] &\ecat&\hspace{-1cm}\longmapsto\hspace{-1cm}&\gbicat(\dcat) \ar[dr, "p_\dcat"', bend right] \ar[r, bend left, "\gbicat(F)", ""{name=A, below}] \ar[r, bend right, "\gbicat(G)"', ""{name=B, above}] \ar[ Rightarrow, from=A, to=B, "\gbicat(\xi)"] & \gbicat(\ecat) \ar[d, "p_\ecat"]\\
		&&&& \cbicat
	\end{tikzcd}\]
\end{defn}	
Applying $(-)\op$ to the fibres and transition morphisms of a $\cbicat$-indexed category we end up with another $\cbicat$-indexed category. Since this process will be useful in the following, let us set the following notation:
\begin{defn}\label{def:Vop}\index{$(-)\Vop$}
	Consider a pseudofunctor $\dcat:\cbicat\op\rightarrow \CAT$: we denote by $\dcat\Vop:\cbicat\op\rightarrow \CAT$ the $\cbicat$-indexed category defined as follows:\begin{itemize}
		\item for every $X$ in $\cbicat$ we set $\dcat\Vop(X):=\dcat(X)\op$;
		\item for every $y:Y\rightarrow X$ we set
		\[\dcat\Vop(y):=\dcat(y)\op:\dcat(X)\op\rightarrow \dcat(Y)\op.\]	
	\end{itemize} 
\end{defn}
\begin{remark}
 The Grothendieck construction can be performed on a much wider class of functors, namely that of lax functors $\dcat:\cbicat\op\rightarrow\CAT$ and their oplax natural transformations, and one can again obtain a 2-functor $\Lax(\cbicat\op,\CAT)_{oplax}\rightarrow \CAT/\cbicat$: this is implied, though not explicitly stated, in \cite{2dimcategories}.
	For the later constructions though (in particular the equivalence in Corollary \ref{cor:equivalenza_strfib_indcat}), it is necessary to work with $\cbicat$-indexed categories.
\end{remark}
As it happened in our example with indexed families of sets, the $\cbicat$-indexed category $\dcat$ can be recovered from $p:\gbicat(\dcat)\rightarrow\cbicat$: indeed, once chosen $X$ in $\cbicat$ then the fibre of $p$ at $X$ is the collection of objects $(X,U)$ of $\gbicat(\dcat)$ and of morphisms of the form $(1,a)$ between them, which is isomorphic to the category $\dcat(X)$. This justifies calling $\dcat(X)$ the \emph{fibre of $\dcat$ over $X$}; we can similarly recover each functor $\dcat(y)$, which is called a \emph{transition morphism} between the fibres (the name \emph{pullback functor} can also be found in the literature, though it can be misleading and we will avoid it).

Let us see some examples of $\cbicat$-indexed categories:
\begin{ex}\label{ex:indexedcategories}\noindent 
	\begin{enumerate}[(i)]
		\item Any presheaf $P:\cbicat\op\rightarrow \Set$ can be seen as particular kind of $\cbicat$-indexed category which is strict and \emph{discrete}: that is, it is a strict functor and its fibres are all discrete categories. In this case one usually writes $\fib P$ for $\gbicat(P)$ and calls it the \emph{category of elements}\index{category!of elements}\index{$\fib(-)$} of $P$: objects are still pairs $(X,U)$ with $U\in P(X)$, and arrows are of the form $y:(Y,V)\rightarrow (X,U)$ where $y:Y\rightarrow X$ and $P(y)(U)=V$. The functor $p_P:\fib P\rightarrow \cbicat$ simply forgets the second component. A morphism of presheaves $f:P\Rightarrow Q$ is mapped to a functor $\fib f:\fib P\rightarrow\fib Q$ operating by sending $y:(Y, Py(U))\rightarrow (X,U)$ to $y:(Y, f_YPy(U)) \rightarrow (X, f_X(U))$. In particular, we remark that $\fib \yo(X)\simeq \cbicat/X$, and that for $y:Y\rightarrow X$, $\fib \yo(y)$ acts as the postcomposition functor $y\circ -:\cbicat/Y\rightarrow \cbicat/X$ (in the following we will use the shorthand $\fib y$ for $\fib \yo(y)$).
		
		\item For any geometric morphism $F:\Ftopos\to \Etopos$, we have an $\Etopos$-indexed category ${\mathbb I}_{F}$\index{${\mathbb I}_{F}$} sending a object $E$ of $\Etopos$ to the category ${\Ftopos}\slash F^{\ast}(E)$ and any arrow $g:E\to E'$ in $\Etopos$ to the pullback functor ${\Ftopos}\slash F^{\ast}(E')\to {\Ftopos}\slash F^{\ast}(E)$ along the arrow $F^{\ast}(g):F^{\ast}(E)\to F^{\ast}(E')$. Applying the Grothendieck construction to $\mathbb{I}_F$ we obtain that $\gbicat(\mathbb{I}_F)$ is the comma category $\comma{\Ftopos}{F^*}$, which is endowed with a canonical projection to $\Etopos$.
		
		\item If $\cbicat$ is a finitely complete category with a canonical choice of pullbacks, then it admits the canonical $\cbicat$-indexed category $\cbicat\op\rightarrow\CAT$ mapping every object $X$ to the slice category $\cbicat/X$, and every morphism $y:Y\rightarrow X$ to the pullback functor $y^*:\cbicat/X \rightarrow \cbicat/Y$. Applying the Grothendieck construction, one sees that the corresponding category over $\cbicat$ is the category $\Mor(\cbicat)$ of morphisms of $\cbicat$ and their commutative squares, with the codomain functor $\cod:\Mor(\cbicat)\rightarrow\cbicat$ as structural functor. 
		
		\item If we consider a Grothendieck topos $\Etopos$ and its identity geometric morphism $1_\Etopos$, the $\Etopos$-indexed category $\mathbb{I}_{1_\Etopos}$ of item (ii) coincides with the one in the previous item: this is what we will call the \emph{canonical stack over $\Etopos$} (see Subsection \ref{sec:canonical_stack}).
	\end{enumerate}
\end{ex}
It is now time to introduce the other side of the coin, namely fibrations over $\cbicat$: their axiomatization captures the basic features of functors of the kind $\gbicat(\dcat)\rightarrow \cbicat$.
\begin{defn}
	Consider a functor $p:\dbicat\rightarrow \cbicat$: an arrow $f:A\rightarrow B$ in $\dbicat$ is said to be \emph{cartesian}\index{cartesian arrow} if for every other $h:C\rightarrow B$ and any $g:p(C)\rightarrow p(A)$ such that $p(h)=p(f)g$ there is a unique $k:C\rightarrow A$ such that $p(k)=g$ and $h=fk$.
	\[
	\begin{tikzcd}
		C \ar[dr, dashed, "!k"] \ar[d, "h"']& \\
		B & A\ar[l, "f"]\end{tikzcd}\xmapsto{p}\hspace{1ex}\begin{tikzcd}
		p(C) \ar[d, "p(h)"'] \ar[dr, "g"]&\\
		p(B) & p(A) \ar[l, "p(f)"]
	\end{tikzcd}\]
\end{defn}

\begin{defn}
	Consider a functor $p:\dbicat\rightarrow \cbicat$: it is a \emph{(Street) fibration}\index{fibration!Street} if for every $A$ in $\dbicat$ and any $x:X\rightarrow p(A)$ in $\cbicat$ there are a cartesian arrow $f:B\rightarrow A$ and an isomorphism $\theta:X\isorightarrow p(B)$ such that $p(f)\theta=x$:
	\[
	\begin{tikzcd}
		p(B) \ar[dr, "p(f)"]&\\
		X \ar[u, "\theta"', "\sim"{ sloped}] \ar[r, "x"'] & p(A)	
	\end{tikzcd}\]
	The arrow $f$ is a \emph{cartesian lift} for $x$. The fibration $p$ is said to be \emph{cloven}\index{fibration!cloven}, or \emph{with cleavage}, if for every $x$ and $A$ there is a choice of a cartesian lift and an isomorphism as above: we will denote them respectively by $\hat{x}_A:\dom(\hat{x}_A)\rightarrow A$\index{$\widehat{x}_A$} and $\theta_{f,A}:X\rightarrow \dom(p(\hat{x}_A))$\index{$\theta_{f,A}$}.
	In particular, $p$ is a \emph{Grothendieck fibration}\index{fibration! Grothendieck} if every $x$ has a cartesian lift $f$ such that $p(f)=x$; a \emph{Grothendieck cleavage} is a cleavage where all the isomorphisms $\theta_{f,A}$ are identities.
\end{defn}	
\begin{prop}
	Consider a $\cbicat$-indexed category $\dcat:\cbicat\op\rightarrow\CAT$: then $p:\gbicat(\dcat)\rightarrow\cbicat$ is a cloven Grothendieck fibration.
\end{prop}
\begin{proof}
	Consider an arrow $y:Y\rightarrow X$ of $\cbicat$ and an object $(X,U)$ in the fibre of $X$. Then the arrow $(y,1_{\dcat(y)(U)}):(Y, \dcat(y)(U))\rightarrow (X,U)$ is a cartesian lift of $y$: it obviously projects to $y$, and for any other $(h,b):(Z,W)\rightarrow (X,U)$ of $\gbicat(\dcat)$ such that $h=yz$ for some $z$, then $(h,b)=(y,1)(z,\phi_{y,z}\inv b)$, and $(z, \phi_{y,z}\inv b)$ is a unique lift for $z$. 
\end{proof}

\begin{remarks}
\begin{enumerate}[(i)]\label{rmk:covGroth.i}
    \item The `evilness' of the definition of Grothendieck fibrations stems from the fact that the condition $p(f)=x$ above forces the equality $p(B)=X$ on objects. We will provide an explicit proof later that Street fibrations are precisely Grothendieck fibrations up to equivalence (see Corollary \ref{cor:strfib_grfib_equivalenti}).
    \item Given a Grothendieck fibration $p:\dbicat\rightarrow\cbicat$, an arrow is said to be \emph{vertical}\index{vertical arrow} if its image via $p$ is an identity arrow. In the case of Street fibrations, we define the `non-evil' version of this as follows: an arrow is \emph{vertical} if its image via $p$ is invertible.
    \item The Grothendieck construction, and the symbol $\gbicat$, will always mean for us the construction just defined, \emph{even when applied to covariant pseudofunctors}. Indeed, it will happen later that we will consider covariant pseudofunctors $R:\cbicat\rightarrow\CAT$: by seeing them as $\cbicat\op$-indexed categories, we will perform the Grothendieck construction to obtain a category over $\cbicat\op$,
    \[
    p_R:\gbicat(R)\rightarrow \cbicat\op.
    \]
    We stress this, because there is also a notion of \emph{covariant} Grothendieck construction for covariant pseudofunctors (\cfr Paragraph 2 of \cite{nlab:grothendieck_construction}), but it is \emph{not} the same as applying the contravariant Grothendieck construction to $R$, seen as a $\cbicat\op$-indexed category: instead, the covariant Grothendieck construction associates $R$ with the category over $\cbicat$ \[p_{R\Vop}\op:\gbicat(R\Vop)\op\rightarrow\cbicat.\]
    In general, an \emph{opfibration}\index{opfibration} is any functor $p$ such that its opposite $p\op$ is a fibration: therefore, the functor $p_{R\Vop}\op$ is usually called the Grothendieck opfibration associated to $R$.
\end{enumerate}
    
\end{remarks}

\begin{remarks}\label{remark:proprietà_frecce_cartesiane} Let us list here some useful properties of cartesian arrows:
	\begin{enumerate}[(i)]
		\item Given a cartesian arrow $f:A\rightarrow B$, any other arrow $k:C\rightarrow A$ is determined uniquely by the composite $fk$ and the projection $p(k)$ in $\cbicat$. We will often use this in the following to show that two arrows in $\dbicat$ are equal. 
		\item If $f$ is cartesian, then $fg$ is cartesian if and only if $g$ is cartesian.
		\item Consider an arrow $x:X\rightarrow p(U)$ and two distinct cartesian lifts of $x$, \ie two cartesian arrows $a:A\rightarrow U$ and $b:B\rightarrow U$ with isomorphisms $\alpha:X\isorightarrow p(A)$ and $\beta:X\isorightarrow p(B)$ such that $x=p(a)\alpha=p(b)\beta$: then it is immediate, exploiting the cartesianity of both $a$ and $b$, to prove that there is a unique $\gamma:A\isorightarrow B$ such that $b\gamma=a$ and $p(\gamma)=\beta\alpha\inv$. This has the immediate corollary that whenever $x:X\rightarrow p(U)$ is an isomorphism, then all its cartesian lifts are isomorphisms: this happens because $x$ admits among its lifts the identity arrow $1_U$, and all lifts of $x$ are isomorphic to it. 
	\end{enumerate}
\end{remarks}

\begin{remark}\label{remark:frecce_chi_lambda}
	Let us now introduce two families of canonical arrows that exist for every cloven fibration, relating lifts of compositions of morphisms: we will need them later in order to define the pseudofunctor associated with a cloven Street fibration.
	
	Consider $y:Y\rightarrow X$ and $x:X\rightarrow p(U)$ in $\cbicat$: we want to compare the two lifts $\widehat{x}_U$ and $\widehat{xy}_U$. To do so, consider the arrow $\theta_{x,U}y:Y\rightarrow p(\dom(\widehat{x}_U))$ and its cartesian lift $\widehat{\theta_{x,U}y}_{\dom(\widehat{x}_U)}$: then the composite $\widehat{x}_U\widehat{\theta_{x,U}y}_{\dom(\widehat{x}_U)}$ is still cartesian, and moreover it lifts $xy$, as the following commutative diagram shows:
	\[
	\begin{tikzcd}
		Y \ar[r, "y"] \ar[d, "\theta_{\theta_{x,U}y, \dom(\widehat{x}_U)}"',"\sim"{sloped}]&  X\ar[dr, "x"] \ar[d, "\theta_{x,U}", "\sim"{sloped, below}] &  \\
		p(\dom(\widehat{\theta_{x,U}y}_{\dom(\widehat{x}_U)})) \ar[r, "p(\widehat{\theta_{x,U}y}_{\dom(\widehat{x}_U)})"{below, yshift=-3}]& p(\dom(\widehat{x}_U)) \ar[r, "p(\widehat{x}_U)"']& p(U)
	\end{tikzcd}\]
	Hence, there is a canonical isomorphism that can compare $\widehat{x}_U\widehat{\theta_{x,U}y}_{\dom(\widehat{x}_U)}$ with $\widehat{xy}_U$: that is, a unique $\chi_{x,y,U}:\dom(\widehat{xy}_U)\isorightarrow \dom(\widehat{\theta_{x,U}y}_{\dom(\widehat{x}_U)})$\index{$\chi_{x,y,U}$} such that $\widehat{xy}_U= \widehat{x}_U\widehat{\theta_{x,U}y}_{\dom(\widehat{x}_U)}\chi_{x,y,U}$ and $p(\chi_{x,y,U})=\theta_{\theta_{x,U}y,\dom(\widehat{x}_U)}\theta_{xy,U}\inv$. 
	Let us moreover introduce the notation 
	$$\lambda_{x,y,U}:=\widehat{\theta_{x,U}y}_{\dom(\widehat{x}_U)}\chi_{x,y,U}:\dom(\widehat{xy}_U)\rightarrow \dom(\widehat{x}_U):$$ 
	the arrow $\lambda_{x,y,A}$\index{$\lambda_{x,y,A}$} can also be defined as the unique arrow satisfying $\widehat{x}_U\lambda_{x,y,U}=\widehat{xy}_U$ and $p(\lambda_{x,y,U})=\theta_{x,U}y\theta_{xy,U}\inv$. Since both $\widehat{x}_U$ and $\widehat{xy}_U$ are cartesian, $\lambda_{x,y,U}$ is also cartesian.
	In particular, it is easy to verify that the following identities also hold: $\lambda_{x,1_X,A}= 1_{\dom(\widehat{x}_A)}$, and $\lambda_{x,y,A}\lambda_{xy,z,A}=\lambda_{x,yz,A}$.
\end{remark}
We can define a 2-category of fibred categories over $\cbicat$:
\begin{defn}
	We will define the \emph{2-category of fibrations} over $\cbicat$, denoted by $\Fib_\cbicat$\index{$\Fib_\cbicat$, $\Fib_\cbicat\Gr$}, as the sub-2-category of $\CAT/\cbicat$ defined thus:
	\begin{itemize}
		\item[0-cells:] they are Street fibrations $p:\dbicat\rightarrow \cbicat$;
		\item[1-cells:] given two fibrations $p:\dbicat\rightarrow\cbicat$ and $q:\ebicat\rightarrow\cbicat$, a \emph{morphism of fibrations} is a pair $(F,\phi)$, where $F:\dbicat\rightarrow \ebicat$ is a functor mapping cartesian arrows to cartesian arrows and $\phi$ is a natural isomorphism $q\circ F\xRightarrow{\sim} p$;
		\item[2-cells:] given two fibrations $[p]$ and $[q]$ over $\cbicat$ and two morphisms of fibrations $(F,\phi)$, $(G,\gamma):[p]\rightarrow[q]$, a 2-cell $\alpha:(F,\phi)\Rightarrow (G,\gamma)$ is given by a natural transformation $\alpha:F\Rightarrow G$ such that $ \phi= \gamma (q\circ \alpha)$.
	\end{itemize}
	In particular, we will denote by $\Cl\Fib_\cbicat$\index{$\Cl\Fib_\cbicat$, $\Cl\Fib_\cbicat\Gr$} the full sub-2-category of cloven fibrations. We will denote the (non full) sub-2-category of Grothendieck fibrations by $\Fib_\cbicat\Gr$, and by $\mathcal{U}:\Fib_\cbicat\Gr\rightarrow \Fib_\cbicat$ the inclusion functor; analogously, $\Cl\Fib_\cbicat\Gr$ will denote the full sub-2-category of $\Fib_\cbicat\Gr$ of Grothendieck fibrations endowed with a Grothendieck cleavage.
\end{defn}

Therefore, the Grothendieck construction provides a 2-functor
\[\gbicat:\Ind_\cbicat\rightarrow \Cl\Fib\Gr_\cbicat.\]

Let us finally provide the notion of fibre of a Street fibration: to do so we shall exploit a 2-categorical notion of pullback.
\begin{defn}
	Given two functors $A:\abicat\rightarrow\cbicat$ and $B:\bbicat\rightarrow\cbicat$, their \emph{strict pseudopullback}\index{strict pseudopullback}\footnote{For a reference on the use of the name strict pseudopullback, see \cite{nlab:2-limit}} $\abicat\times_\cbicat \bbicat$ is the category whose objects are triples $(X,U,f)$, where $X$ is an object of $\abicat$, $U$ is an object of $\bbicat$ and $f:A(X)\isorightarrow B(U)$ in $\cbicat$, while morphisms are pairs $(r,s):(X,U,f)\rightarrow (Y,V,g)$ where $r:X\rightarrow Y$, $s:U\rightarrow V$ and $gA(r)=B(s)f$. There are two obvious forgetful functors from $\abicat\times_\cbicat\bbicat$ to $\abicat$ and $\bbicat$ and a natural isomorphism $\sigma$ as in the diagram
	\[
	\begin{tikzcd}
		\abicat\times_\cbicat\bbicat \ar[d, "\pi_\abicat"'] \ar[r, "\pi_\bbicat"] & \bbicat \ar[d, "B"]\\
		\abicat \ar[r, "A"']  \ar[ur, Rightarrow, "\sim"sloped, "\sigma"'] & \cbicat
	\end{tikzcd}\]
	The strict pseudopullback $\abicat\times_\cbicat\bbicat$ satisfies the following universal property: for every other pair of functors $P:\dbicat\rightarrow\abicat$ and $Q:\dbicat\rightarrow \bbicat$ and natural isomorphism $\tau: AP\Isorightarrow BQ$ there is a unique functor $H:\dbicat\rightarrow\abicat\times_\cbicat\bbicat$ such that $\pi_\abicat H=P$, $\pi_\bbicat H=Q$ and $\sigma\circ H=\tau$. Moreover, for any pair two such cones $(P,Q,\tau)$ and $(P',Q',\tau')$, if there are natural transformations $\alpha:P\Rightarrow P'$ and $\beta:Q\Rightarrow Q'$ such that $\tau'(A\circ \alpha)=(B\circ \beta)\tau$ then there exists a unique natural transformation $\eta:H\Rightarrow H'$ such that $\pi_\abicat\circ 1\eta=\alpha$ and $\pi_\bbicat \circ \eta=\beta$. 
\end{defn}

\begin{defn}
	Consider a functor $p:\dbicat\rightarrow \cbicat$: the \emph{essential fibre of $p$ at $X$}, which we will simply call \emph{fibre}\index{fibre}\index{$\dcat(X)$} and denote by $\dcat(X)$, is the strict pseudopullback 
	\[
	\begin{tikzcd}
		\dcat(X) \ar[d] \ar[r] & \dbicat \ar[d, "p"] \\
		\onecat \ar[r, "e_X"'] \ar[ur, Rightarrow, "\sim" sloped] & \cbicat
	\end{tikzcd},\]
	where the functor below is the constant functor with value $X$. In other words, $\dcat(X)$ is the category whose objects are the couples $(A, \alpha:X\isorightarrow p(A))$, and whose arrows $\gamma:(A,\alpha)\rightarrow(B,\beta)$ are indexed by arrows $\gamma:A\rightarrow B$ of $\dbicat$ such that $p(\gamma)\alpha=\beta$.
\end{defn}
\begin{remark}
	We will use strict pseudopullbacks in a moment to build from $p$ a pseudofunctor $\dcat:\cbicat\op\rightarrow \CAT$. One might be concerned with the fact that strict pseudopullbacks are stable under isomorphism of categories but not under equivalence; nonetheless, they act as canonical (and manageable) representatives of pseudopullbacks:
	indeed, any category which is equivalent to the strict pseudopullback $\abicat\times_\cbicat \bbicat$ will be a \textit{pseudopullback} of $A$ and $B$, \ie the induced functor $H$ above will be unique up to a unique 2-isomorphism, and pseudopullbacks are indeed stable under equivalence.
\end{remark}

\section{The equivalence between Street fibrations and indexed categories}

Our first purpose is to show that cloven Street fibrations are equivalent to pseudofunctors, generalizing a well known result about Grothendieck fibrations. First of all, we need a 2-functor from cloven fibrations to $\Ind_\cbicat$:
\begin{prop}
	There is a strict 2-functor $\Ifrak:\Cl\Fib_\cbicat\rightarrow \Ind_\cbicat$\index{$\Ifrak$} operating as follows:
	\begin{itemize}
		\item[0-cells:] consider a cloven fibration $p:\dbicat\rightarrow \cbicat$: the $\cbicat$-indexed category
		$$\Ifrak(p)=\dcat:\cbicat\op\rightarrow\CAT$$
		is defined on object by taking every 
		$X$ in $\cbicat$ to the fibre $\dcat(X)$, while for $y:Y\rightarrow X$ in $\cbicat$ the functor $\dcat(y):\dcat(X)\rightarrow\dcat(Y)$ is defined as
		\[
		\left[(A,\alpha)\xrightarrow{\omega}(B,\beta)  \right]\mapsto\left[(\dom(\widehat{\alpha y}_A), \theta_{\alpha y,A})\xrightarrow{\dcat(y)(w)} (\dom(\widehat{\beta y}_B), \theta_{\beta y,B}) \right]\]
		where $\dcat(y)(\omega):\dom(\widehat{\alpha y}_A)\rightarrow \dom(\widehat{\beta y}_B)$ is the unique arrow satisfying the identities $\widehat{\beta y}_B \dcat(y)(\omega)= \omega \widehat{\alpha y}_A$ and $p(\dcat(y)(\omega))=\theta_{\beta y,B}\theta_{\alpha y,A}\inv$.
		\item[1-cells:] A morphism of fibrations $(F,\phi):[p:\dbicat\rightarrow\cbicat]\rightarrow[q:\ebicat\rightarrow \cbicat]$ produces a $\cbicat$-indexed functor $$\Ifrak^{(F,\phi)}:\dcat\Rightarrow \ecat$$ 
		built as follows: for every $X$ in $\cbicat$, $\Ifrak^{(F,\phi)}_X:\dcat(X)\rightarrow \ecat(X)$ is defined as
		$$\Ifrak^{(F,\phi)}_X:\left[(A,\alpha)\xrightarrow{\gamma} (B,\beta)\right] \mapsto \left[ (F(A), \phi_A\inv \alpha)\xrightarrow{F(\gamma)} (F(B), \phi_B\inv \beta) \right]$$
		while for every arrow $y:Y\rightarrow X$ in $\cbicat$ the canonical isomorphism $\Ifrak^{(F,\phi)}_y:\ecat(y)\Ifrak^{(F,\phi)}_X\Isorightarrow \Ifrak^{(F,\phi)}_Y\dcat(y)$ is defined componentwise thus: the arrow $$\Ifrak^{(F,\phi)}_y(A,\alpha):\dom(\widehat{\phi_A\inv \alpha y}_{F(A)})\rightarrow F(\dom(\widehat{\alpha y}_A))$$ is the unique satisfying the identities $F(\widehat{\alpha y}_A)\Ifrak^{(F,\phi)}_y(A,\alpha)= \widehat{\phi_A\inv \alpha y}_{F(A)}$ and $q(\Ifrak^{(F,\phi)}_y(A,\alpha))= \phi\inv_{\dom(\widehat{\alpha y}_A)}\theta_{\alpha y,A}\theta\inv_{\phi\inv_A \alpha y, F(A)}$.
		\item[2-cells:]	Given a 2-cell of fibrations $\xi:(F,\phi)\Rightarrow (G,\gamma):[p]\rightarrow[q]$, the corresponding $\cbicat$-indexed natural transformation
		$$\Ifrak^{\xi}:\Ifrak^{(F,\phi)}\Rrightarrow \Ifrak^{(G,\gamma)}$$
		is defined componentwise as follows: for $X$ in $\cbicat$, the natural transformation $\Ifrak^\xi_X:\Ifrak^{(F,\phi)}_X\Rightarrow \Ifrak^{(G,\gamma)}_X$ is defined componentwise, for every $(A,\alpha)$ in $\dcat(X)$, as  $\Ifrak^\xi_X(A,\alpha)=\xi_A:(FA, \phi_A\inv\alpha)\rightarrow (GA, \gamma_A\inv \alpha)$.
	\end{itemize} 
\end{prop}
\begin{proof}
	We only provide definitions for the relevant structure, leaving all the verifications to the reader: we stress that all equalities of arrows in the fibrations are verified by exploiting what we said in the first item of Remark \ref{remark:proprietà_frecce_cartesiane}.
	
	The arrow $\dcat(y)(\gamma)$ is well defined, and it is easy to see that its unicity implies that $\dcat(y)$ is a functor. 
	To see that $\dcat$ is a pseudofunctor, the following canonical natural isomorphisms must be considered: for every $X$ in $\cbicat$, we define
	$\phi^\dcat_X:Id_{\dcat(X)}\Isorightarrow \dcat(1_X)$ componentwise by setting $\phi_X^\dcat(A,\alpha)$ equal to $$ (A,\alpha)\xrightarrow{\widehat{\alpha}_A\inv} (\dom(\widehat{\alpha}_A), \theta_{\alpha, A});$$
	for every $z:Z\rightarrow Y$ and $y:Y\rightarrow X$, we define
	$\phi^\dcat_{y,z}:\dcat(z)\dcat(y)\Isorightarrow \dcat(yz)$ componentwise by setting $\phi^\dcat_{y,z}(A,\alpha)$ equal to
	$$( \dom(\widehat{\theta_{\alpha y,A}z}_{\widehat{\alpha y}_A} ), \theta_{\theta_{\alpha y,A}z, \dom(\widehat{\alpha y}_A)} )\xrightarrow{\chi_{\alpha y,z,A}\inv} (\dom(\widehat{\alpha yz}_A), \theta_{\alpha yz, A}).$$
	We have already shown that the arrows $\chi$ are isomorphisms, while $\widehat{\alpha}_A$ is an isomorphism since it lifts the isomorphism $\alpha$. Here one needs to check that they are arrows of the fibres, that their components are natural and that the identities in the definition of a $\cbicat$-indexed category are satisfied.
	
	For the pseudonatural transformation $\Ifrak^{(F,\phi)}$, notice that $\Ifrak^{(F,\phi)}_y(A,\alpha)$ is well defined because $F$ preserves the cartesianity of arrows: indeed, it is easy to see that since $\widehat{\alpha y}_A$ lifts $\alpha y$ via $p$ then $F(\widehat{\alpha y}_A)$ lifts $\phi_A\inv\alpha y$ via $q$, and hence it is canonically isomorphic to $\widehat{\phi_A\inv \alpha y}_{F(A)}$. The verification that $\Ifrak^{(F,\phi)}$ is a pseudonatural transformation is a matter of computation.
	
	Finally, the verification that the arrows $\Ifrak^\xi_X(A,\alpha)$ provide a natural transformation is an explicit check, as is the verification that $\Ifrak^\xi$ is indeed a $\cbicat$-indexed natural transformation.	To conclude, one can check that for every $[p]$ and $[q]$ as above, we have defined a functor $\Cl\Fib_\cbicat([p],[q])\rightarrow \Ind_\cbicat(\dcat, \ecat)$, that $\Ifrak^{Id_\dbicat}=Id_{\dcat}$ and that for $(F,\phi):[p]\rightarrow [q]$ and $(G,\gamma):[q]\rightarrow[r]$, $\Ifrak^{(G,\gamma)(F,\phi)}=\Ifrak^{(G,\gamma)}\Ifrak^{(F,\phi)}$: this makes $\Ifrak$ into a strict 2-functor of 2-categories.
\end{proof}
In adherence with the standard language for Grothendieck fibrations, we will call the functors $\dcat(y)$ between fibres \emph{transition morphisms} of the fibration $p$.

Though a fibration may admit different cleavages, they are all in a certain sense equivalent, since they produce essentially the same pseudofunctor:
\begin{prop}
	The construction of a pseudofunctor from a cloven fibration is essentially independent from the choice of cleavage: more explicitly, given two different cleavages for $p:\dbicat\rightarrow\cbicat$, the corresponding pseudofunctors $\dcat$ and $\tilde{\dcat}$ are equivalent up to a pseudonatural isomorphism.
\end{prop}
\begin{proof}
	Consider $f:X\rightarrow p(A)$ in $\cbicat$: we will denote by $(\widehat{f}_A, \theta_{f,A})$ the cartesian lift for one cleavage and by $(\widetilde{f}_A, \tilde{\theta}_{f,A})$ the cartesian lift for the other cleavage.
	
	The proof is straightforward. First of all, $\dcat$ and $\tilde{\dcat}$ behave in exactly the same way on the objects of $\cbicat$: this happens because the definition of fibre is independent from the cleavage, which only affects the construction of the transition morphisms. But recall that, given $y:Y\rightarrow X$ and $(A,\alpha)$ in $\dcat(X)$, then 
	$$\dcat(y)(A,\alpha)=(\widehat{\alpha y}_A, \theta_{\alpha y,A}),\ \tilde{\dcat}(y)(A,\alpha)=(\widetilde{\alpha y}_A, \widetilde{\theta}_{\alpha y,A}):$$
	now, since $\widehat{\alpha y}_A$ and $\widetilde{\alpha y}_A$ are both cartesian lifts of $\alpha y$, there is a unique well defined isomorphism $z_y(A,\alpha):(\widehat{\alpha y}_A, \theta_{\alpha y,A})\isorightarrow(\widetilde{\alpha y}_A, \widetilde{\theta}_{\alpha y,A})$ in $\dcat(Y)$, which provides the components for a pseudonatural isomorphism from $\dcat$ to $\tilde{\dcat}$.
\end{proof}
Now, let us denote by 
\[\Ifrak\Gr:\Cl\Fib\Gr_\cbicat \leftrightarrows \Ind_\cbicat:\gbicat\]
the standard equivalence between Grothendieck cloven fibrations and pseudofunctors, which appears for instance in \cite[Section B1.3]{elephant}. When working with Grothendieck fibrations, we obtain essentially the same $\cbicat$-indexed category whether we apply $\Ifrak$ or $\Ifrak\Gr$:
\begin{prop}
	Consider a cloven Grothendieck fibration $p:\dbicat\rightarrow \cbicat$: then $\dcat\Gr:=\Ifrak\Gr(p)$ and $\dcat:=\Ifrak(p)$ are equivalent pseudofunctors. This extends to an equivalence of the two 2-functors
	\[\Ifrak\Gr,\ \Ifrak\mathcal{U}:\Cl\Fib_\cbicat\Gr\rightarrow\Ind_\cbicat, \]
	\ie there exists an invertible 2-natural isomorphism $\Ifrak\Gr\Isorightarrow \Ifrak\mathcal{U}$.
\end{prop}
\begin{proof}
	Preliminarily, notice that since the cleavage for $p$ is the cleavage of a Grothendieck fibration, all the lifting isomorphisms $\theta$ are actually identities.
	
	We can start by comparing the fibres. There is an obvious full and faithful functor $\dcat\Gr(X)\hookrightarrow \dcat(X)$ mapping $A$ to $(A, 1_{p(A)})$ and $g:A\rightarrow B$ to $g:(A, 1_{p(A)})\rightarrow (B, 1_{p(B)})$. There is also a functor $K_X:\dcat(X)\rightarrow \dcat\Gr(X)$ in the opposite direction, which acts as follows:
	\[
	\left[ (A,\alpha)\xrightarrow{g}(B,\beta)\right] \mapsto \left[ \dom(\widehat{\alpha}_A) \xrightarrow{\dcat(1_X)(g)} \dom(\widehat{\beta}_B) \right]
	\]
	By definition, $\dcat(1_X)(g)$ is the unique arrow such that $\widehat{\beta}_B\dcat(1_X)(g)=g\widehat{\alpha}_A$ and $p(\dcat(1_X)(g))=1_{X}$. It is now a quick check to see that these are functors, that they are one the quasi-inverse of the other, and that they are  compatible with the transition functors $\dcat(y):\dcat(X)\rightarrow \dcat(Y)$ and $\dcat\Gr(y):\dcat\Gr(X)\rightarrow \dcat\Gr(Y)$: in fact, the composites $\dcat(y) K_X$ and $K_Y\dcat\Gr(y)$ are exactly equal. This shows that $\dcat\Gr\cong \dcat$. 
	
	To prove that this extends to a pseudoequivalence of 2-functors, we should build, for every morphism $F:[p]\rightarrow [q]$ in $\Cl\Fib_\cbicat\Gr$, an invertible modification 
	\[\begin{tikzcd}[column sep=large]
		\dcat\Gr \ar[d, "\Ifrak^{Gr,F}"'] \ar[r, "K^{[p]}"] & \dcat \ar[d, "\Ifrak^{(F,1)}"] \ar[dl, triple, "\kappa^F", "\sim"{sloped}]\\
		\ecat\Gr \ar[r, "K^{[q]}"'] & \ecat
	\end{tikzcd},\]
	where $(F,1)$ is just $\mathcal{U}(F)$. Now, a computation shows that the two pseudonatural transformations $K^{[q]} \Ifrak^{Gr,F}$ and $\Ifrak^{(F,1)}K^{[p]}$ are actually equal, and thus $\kappa^F$ can be defined as the identity modification. This concludes the proof.
\end{proof}
\begin{cor}
	The two functors
	$$ Id_{\Ind_\cbicat},\ \Ifrak\mathcal{U}\gbicat:\Ind_\cbicat\rightarrow \Ind_\cbicat$$
	are equivalent.
\end{cor}
\begin{proof}
	It is immediate: $Id_{\Ind_\cbicat}\cong \Ifrak\Gr\gbicat$ is standard, and we have just shown that $\Ifrak\Gr\cong \Ifrak\mathcal{U}$, whence $Id_{\Ind_\cbicat}\cong \Ifrak\mathcal{U}\gbicat$.
\end{proof}
The converse is also true, as the following result shows:
\begin{prop} Consider a cloven fibration $p:\dbicat\rightarrow \cbicat$: then it is equivalent to the Grothendieck fibration $\pi:\gbicat(\dcat)\rightarrow \cbicat$. This extends to an equivalence of pseudofunctors
	\[ Id_{\Cl\Fib_\cbicat},\ \mathcal{U}\gbicat\Ifrak: \Cl\Fib_\cbicat\rightarrow \Cl\Fib_\cbicat, \]
	\ie there is an invertible 2-natural transformation $\mathcal{U}\gbicat\Ifrak\Isorightarrow Id_{\Cl\Fib_\cbicat}$.
\end{prop}
\begin{proof}
	We remark that objects of $\gbicat(\dcat)$ are couples $(X,(A,\alpha))$, where $X$ is an object of $\cbicat$ and $(A,\alpha)$ an object of $\dcat(X)$, \ie $\alpha: X\isorightarrow p(A)$ in $\cbicat$.
	
	We begin by defining a functor $T:\dbicat\rightarrow \gbicat(\dcat)$ as
	\[ 
	\left[A\xrightarrow {g} B \right]\mapsto \left[ (p(A), (A,1_{p(A)})) \xrightarrow{(p(g), \bar{g})} (p(B), (B,1_{p(B)}))  \right],
	\]
	where $\bar{g}:A\rightarrow \dom(\widehat{p(g)}_B)$ is the unique arrow such that $\widehat{p(g)}_B\bar{g}=g$ and $p(\bar{g})=\theta_{p(g),B}$: the verification that it is a functor is based on the unicity of the arrow $\bar{g}$ (notice in particular that $\bar{1_A}=\widehat{1_p(A)}_A\inv=\phi_A(A,1_{p(A)})$, which implies the preservation of identities). If we suppose that $g:A\rightarrow B$ is cartesian, it is also a lift for $p(g)$: this implies that $\bar{g}$ is an isomorphism and hence $(p(g),\bar{g})$ is a cartesian arrow of $\gbicat(\dcat)$. Moreover, it is immediate to see that $\pi T=p$, and hence we have a morphism of fibrations $(T, \id):[p]\rightarrow [\pi]$.
	
	We define its quasi-inverse $S:\gbicat(\dcat)\rightarrow \dbicat$ as
	\[
	\left[ (Y, (B,\beta))\xrightarrow{(y,g)}(X,(A,\alpha))  \right]\mapsto\left[ B \xrightarrow{\widehat{\alpha y}_Ag} A  \right]:
	\]
	again, the verification that it is a functor it lenghty but straightforward. It also maps cartesian arrows to cartesian arrows: indeed, $(y, g)$ is cartesian if $g$ is an isomorphism, which means that its image $\widehat{\alpha y}_Ag$ is also cartesian (as $\widehat{\alpha y}_A$ is). Finally, there is a natural isomorphism $\sigma: ps \Isorightarrow \pi$ defined componentwise as
	\[   
	\left[ pS(X,(A,\alpha))\xrightarrow{\sigma(X, (A,\alpha))}\pi(X,(A,\alpha))\right]=\left[ p(A) \xrightarrow{\alpha\inv} X   \right] \]
	which is natural. The two compositions $(S,\sigma)(T,\id)$ and $(T,\id)(S,\sigma)$ are indeed the two components of an equivalence of fibrations.
	
	Now consider again the morphisms $(S,\sigma):\mathcal{}\gbicat(\dcat)\rightarrow [p]$: we add an apex $S^{[p]},\ \sigma^{[p]}$ to specify that they stem from the fibration $p$. To extend these data into an invertible pseudonatural transformation we also need for each $(F,\phi):[p]\rightarrow [q]$ a 2-isomorphism
	\[
	\begin{tikzcd}[column sep=huge]
		\gbicat(\dcat) \ar[d, "{\gbicat\Ifrak^{(F,\phi)}}"'] \ar[r, "{(S^{[p]}, \sigma^{[p]})}"] & {[p]} \ar[d, "{(F,\phi)}"] \ar[dl, Rightarrow, "{\zeta^{(F,\phi)}}", "\sim"{sloped}]\\
		\gbicat(\ecat) \ar[r, "{(S^{[q]}, \sigma^{[q]})}"'] & {[q]}
	\end{tikzcd}
	\] 
	satisfying some compatibility axioms: luckily, it is easy to see that the the two composites $(F,\phi)(S^{[p]},\sigma^{[p]})$ and $(S^{[q]},\sigma^{[q]})\gbicat\Ifrak^{(F,\phi)}$ are actually equal, so $\zeta$ is an identity, and the axioms are quickly verified to hold.
\end{proof}
Combining the previous results, we obtain the equivalence we wanted:
\begin{cor}\label{cor:equivalenza_strfib_indcat}
	The two 2-functors
	$$\Ifrak:\Cl\Fib_\cbicat\rightarrow \Ind_\cbicat,\ \mathcal{U}\gbicat:\Ind_\cbicat\rightarrow \Cl\Fib_\cbicat$$
	form an equivalence of 2-categories.
\end{cor}
This has as corollary a fact we mentioned multiple times, namely that Street fibrations are `up-to-equivalence Grothendieck fibrations':
\begin{cor}\label{cor:strfib_grfib_equivalenti}
	Consider a functor $p:\dbicat\rightarrow \cbicat$: then it is a fibration if and only there are a Grothendieck fibration $q:\ebicat\rightarrow\cbicat$, an equivalence of categories $F:\dbicat\isorightarrow\ebicat$ and a natural isomorphism $\phi$ as in the diagram:
	\[
	\begin{tikzcd}
		\dbicat \ar[dr, "p"', ""{name=A,below}] \ar[r, "F"{yshift=0.8ex}, "\sim"{yshift=-0.3ex}] & \ebicat \ar[d, "q"] \ar[to=A, Rightarrow, "\phi", "\sim"{sloped, above}]\\& \cbicat
	\end{tikzcd}\]
\end{cor}
\begin{proof}
	We have shown previously that if $p$ is a fibration then it is equivalent to $\gbicat(\dcat)\rightarrow\cbicat$, which is Grothendieck. Conversely, suppose given $F$ and $\phi$: without loss of generality we may assume that $F$ is the left adjoint of an adjoint equivalence, so that the identities  $G(\epsilon)=\eta_G\inv$ and $F(\eta)=\epsilon_F\inv$ hold for the unit and counit of the adjunction. Consider an arrow $x:X\rightarrow p(D)$ in $\cbicat$: then the composite $\phi_D\inv x:X\rightarrow qF(D)$ admits a cartesian lift $f:Y\rightarrow F(D)$ through $q$. A check shows that the arrow $\hat{x}:=\eta_D\inv G(f):G(Y)\rightarrow D$ is still cartesian, and that $x= p(\hat{x})\phi_{GY}q(\epsilon_{Y}\inv)$, making $\hat{x}$ into a lift for $x$.
\end{proof}

To conclude this section, let us consider split Street fibrations. We recall that a Grothendieck cleavage is said to be a \emph{splitting}\index{fibration!split} if it is compatible with identities and compositions: that is equivalent to requiring that the indexed category corresponding to a split fibration be a strict functor. By suitably generalizing the notion of splitting, we obtain an analogue result for Street fibrations:
\begin{prop}
	Consider a cloven fibration $p:\dbicat\rightarrow \cbicat$, then the following are equivalent:\begin{itemize}
		\item The corresponding $\cbicat$-indexed category $\dcat:\cbicat\op\rightarrow\CAT$ is a strict functor (\ie all natural transformations $\phi^\dcat_X$ and $\phi^\dcat_{y,z}$ are identities);
		\item the cleavage for $p$ is a \emph{splitting}, meaning that the two following conditions are verified: \begin{enumerate}[(a)]
			\item for every $\alpha:X\isorightarrow p(A)$, $\widehat{\alpha}_A=1_A$;
			\item for every $z:Z\rightarrow Y$, $y:Y\rightarrow p(A)$, $\widehat{yz}_A=\widehat{y}_A\widehat{\theta_{y,A}z}_{\dom(\widehat{y}_A)}$ and $\theta_{yz,A}=\theta_{\theta_{y,A}z,\dom(\widehat{y}_A)}$.	
		\end{enumerate} 
	\end{itemize}
	This restricts the equivalence $\Cl\Fib_\cbicat\simeq \Ind_\cbicat$ to an equivalence $\Spl\Fib_\cbicat\simeq [\cbicat\op,\CAT]$\index{$\Spl\Fib_\cbicat$}, where $\Spl\Fib_\cbicat$ denotes the full subcategory of $\Cl\Fib_\cbicat$ of split fibrations.
\end{prop}
\begin{proof}
	This is immediate recalling the definition of the natural isomorphisms $\phi_X$ and $\phi_{y,z}$ above. Since $\phi_X^\dcat(A,\alpha):=\widehat{\alpha}
	_A\inv$, $\phi^\dcat$ is an identity if and only if the cleavage lifts of isomorphisms are identity arrows. For the second condition, $\phi_{y,z}^\dcat(A,\alpha)=\chi_{\alpha y,z,A}\inv$: by the definition of $\chi$, it is the identity if and only if for all choices of $\alpha$, $y$ and $z$ it holds that $\widehat{\alpha yz}_A=\widehat{\alpha y}_A\widehat{\theta_{\alpha y,A}z}_{\dom(\widehat{\alpha y}_A)}$ and $\theta_{\alpha yz, A}=\theta_{\theta_{\alpha y,A}z, \dom{\widehat{\alpha y}_A}}$, which is evidently equivalent to the condition (b) stated above.
\end{proof}

\section{The fibred Yoneda lemma}
The next significant result we will prove is the extension of Yoneda lemma to Street fibrations. It is really just a matter of computations, therefore we will only sketch how to define the equivalence (also to set our notation).
\begin{prop}[ (fibred Yoneda lemma)] \label{prop:fibered_yoneda}\index{fibred Yoneda lemma}
	Consider a cloven Street fibration $p:\dbicat\rightarrow\cbicat$: then there is an equivalence of categories 
	$$\Fib_\cbicat(\cbicat/X,\dbicat)\simeq \dcat(X)$$ which moreover is pseudonatural in both components, \ie for any $y:Y\rightarrow X$ in $\cbicat$ and $(F,\phi):\dbicat\rightarrow \ebicat$ in $\Cl\Fib_\cbicat$ the two squares
	\[
	\begin{tikzcd}
		\dcat(X) \ar[r, "\sim"] \ar[d, "\dcat(y)"'] & \Fib_\cbicat(\cbicat/X,\dbicat) \ar[d, "-\circ \fib y"]\\
		\dcat(Y) \ar[r, "\sim"] & \Fib_\cbicat(\cbicat/Y,\dbicat) 
	\end{tikzcd}\hspace{1cm}
	\begin{tikzcd}
		\dcat(X) \ar[d, "F_{|\dcat(X)}"] \ar[r, "\sim"] & \Fib_\cbicat(\cbicat/X,\dbicat)\ar[d, "{(F,\phi)\circ-}"]\\
		\ecat(X)\ar[r,"\sim"] & \Fib_\cbicat(\cbicat/X, \ebicat)
	\end{tikzcd}	
	\]
	commute up to canonical natural isomorphisms.
\end{prop}
\begin{proof}
	One direction of the equivalence is easy: starting with $(F,\phi):\cbicat/X\rightarrow \dbicat$ a morphism of fibrations, we can  consider the pair $(F([1_X]), \phi_{[1_X]}\inv:X\isorightarrow pF([1_X]))$, which is an object of $\dcat(X)$; for a 2-cell $\alpha:(F,\phi)\Rightarrow (G,\gamma)$, we set $\alpha_{[1_X]}:(F([1_X]), \phi_{[1_x]}\inv)\rightarrow (G([1_X]),\gamma_{[1_X]}\inv)$ as its image. This definition is evidently functorial, and hence we have obtained a functor $\Phi:\Fib_\cbicat(\cbicat/X,\dbicat)\rightarrow \dcat(X)$.
	
	The building of the quasi-inverse exploits the cleavage. Consider an object $(A, \alpha:X\isorightarrow p(A))$ in $\dcat(X)$, then we define a morphism of fibrations $(F_{(A,\alpha)},\phi_{(A,\alpha)}):\cbicat/X\rightarrow \dbicat$ as follows:
	\begin{itemize}
		\item we define $F_{(A,\alpha)}([y]):=\dom(\widehat{\alpha y}_A)$;
		\item for $z:[yz]\rightarrow [y]$ we set $F_{(A,\alpha)}(z):=\lambda_{\alpha y,z,A}$.
	\end{itemize}
	Since arrows of the form $\lambda$ were shown to be cartesian in Remark \ref{remark:frecce_chi_lambda}, this functor maps cartesian arrows to cartesian arrows. The natural isomorphism $\phi_{(A,\alpha)}:pF\Isorightarrow p_X$ is defined componentwise as $\phi_{(A,\alpha)}([y]):=\theta_{\alpha y, A}\inv: pF([y])\isorightarrow Y$ and it is indeed natural. 
	
	We want to define our functor on arrows now. Consider $\gamma:(A,\alpha)\rightarrow (B,\beta)$ in $\dcat(X)$, \ie an arrow $\gamma:A\rightarrow B$ of $\dbicat$ such that $p(\gamma)=\beta\alpha\inv$: calling $(F,\phi)$ and $(F', \phi')$ the images of $(A,\alpha) $ and $(B,\beta)$ for brevity, we want to build a 2-cell $F_\gamma$ between them. To do so, it is sufficient to see that for any $y:Y\rightarrow X$ the identity $p(\gamma\widehat{\alpha y}_A)= p(\widehat{\beta y}_B) \theta_{\beta y,B}\theta_{\alpha y,A}\inv$ holds: then by cartesianity there is a unique $F_\gamma([y]):\dom(\widehat{\alpha y}_A)\rightarrow\dom (\widehat{\beta y}_B)$ such that $\widehat{\beta y}_BF_\gamma([y])= \gamma \widehat{\alpha y}_A$ and $p(F_\gamma([y]))=\theta_{\beta y, B}\theta_{\alpha y, A}\inv$. The components $F_\gamma([y])$ define a natural transformation $F\Rightarrow F'$ (it is an easy check using the definition of the arrows $\lambda$), and the identity $\phi= \phi'(p\circ F_\gamma)$ is immediately verified: thus $F_\gamma$ is indeed a 2-cell of $\Fib_\cbicat$. The unicity in the definition of the $F_\gamma([y])$ assures us that this construction is functorial, and hence we have a functor $\Psi:\dcat(X)\rightarrow \Fib_\cbicat(\cbicat/X,\dbicat)$. 
	
	Now to show that the two functors are quasi-inverses. Starting from 
	$(A,\alpha)$ in $\dcat(X)$, consider  $\Phi\Psi(A,\alpha)=(F_{(A,\alpha)}([1_X]),\phi_{(A,\alpha)}({[1_X]})\inv)$ in $\dcat(X)$. Notice that the arrow $\widehat{\alpha}_A:F([1_X])\rightarrow A$ is an isomorphism, since it is a cartesian lift for $\alpha$, which is invertible: since moreover $\alpha=p(\hat{\alpha}_A)\theta_{\alpha,A}$, it is an isomorphism $\hat{\alpha}_A: (\dom(\widehat{\alpha}_A), \theta_{\alpha,A} )\rightarrow(A,\alpha)$. It is easy to check that it is also natural in $(A,\alpha)$, and hence we have a natural isomorphism $\Phi\Psi\Isorightarrow Id_{\dcat(X)}$.
	
	Conversely, start from $(F,\phi):\cbicat/X\rightarrow \dbicat$ morphism of fibrations, consider $\Phi(F,\phi)=(F([1_X]),\phi_{[1_X]}\inv)$ in $\dcat(X)$ and then $\Psi\Phi(F,\phi)=(G,\gamma):\cbicat/X\rightarrow \dbicat$. Notice that for any $y:[y]\rightarrow [1_X]$ in $\cbicat/X$, the arrow $F(y):F([y])\rightarrow F([1_X])$ is cartesian since $F$ is a morphism of fibrations; moreover, it lifts $\alpha y$ (this is an immediate check): thus there is a unique isomorphism $\kappa_{[y]}:F[y]\isorightarrow \dom(\widehat{\alpha y}_A)$, since both lift the same arrow. It is immediate to check that the $\kappa_{[y]}$ are the components of a natural transformation $\kappa: F\Rightarrow G$ and that it is in fact an invertible 2-cell $\kappa:(F,\phi)\Rightarrow (G,\gamma)$. The naturality in $(F,\phi)$ is also a straightforward check, and so we conclude that there is a natural isomorphism $\Psi\Phi\Isorightarrow Id_{\Fib_\cbicat(\cbicat/X,\dbicat)}$. 
	
	Finally, it is lengthy but straightforward to verify that there exist natural isomorphisms making the squares above commute, and this concludes the proof.
\end{proof}
The fibred Yoneda lemma has the corollary that all cloven Street fibrations are split, up to equivalence:
\begin{cor}\label{cor:eqv_fibr_e_splfibr}
	Every Street fibration is equivalent to a split Street fibration.
\end{cor}
\begin{proof}
	The fibred Yoneda lemma states precisely that $\dcat\simeq \Fib_\cbicat(\cbicat/-, \dbicat)$ and since $\Fib_\cbicat(\cbicat/-,\dbicat)$ is a functor it corresponds to a split Street fibration.
\end{proof}

\section{Limit-preserving pseudofunctors}
In the present section we shall recall some results about the existence of limits in a Grothendieck fibration, which will become instrumental later.
\begin{lemma}\label{lemma:pb_preserving_psft}
	Consider a pseudofunctor $\dcat:\cbicat\op\rightarrow\CAT$, then the following are equivalent:\begin{enumerate}[(i)]
		\item $\gbicat(\dcat)$ has pullbacks, and the pullback of a horizontal (resp. vertical) arrow is a horizontal (resp. vertical) arrow;
		\item every fibre of $\dcat$ has pullbacks, every transition morphism preserves them and every arrow $y:Y\rightarrow X$ in $\cbicat$ such that $\dcat(X)$ is non-empty has pullbacks.
	\end{enumerate}
\end{lemma}
\begin{proof}
	We can split the proof in three parts:
	\begin{enumerate}[(a)]
		\item A quick calculation shows that squares of the kind
		\[
		\begin{tikzcd}
			{(Y,\dcat(y)(V))} \ar[d, "{(1_Y, \dcat(y)(a))}"'] \ar[rr, "{(y,1_{\dcat(y)(V)})}"] & & {(X,V)} \ar[d, "{(1_X,a)}"]\\
			{(Y, \dcat(y)(U))} \ar[rr, "{(y,1_{\dcat(y)(U)})}"'] & & {(X,U)}
		\end{tikzcd}
		\]
		are always pullbacks in a fibration, therefore each vertical arrow admits a pullback along any horizontal arrow, and said pullbacks are still vertical (resp. horizontal). Therefore in item (i) we can reduce to assuming that each horizontal (resp. vertical) arrow has a horizontal (resp. vertical) pullback along any other horizontal (resp. vertical) arrow. General pullbacks are then obtained by gluing pullback squares.
		
		\item Consider the two squares
		\[
		\begin{tikzcd}
			{(P,\dcat(yp)(U))} \ar[d, "{(p,1)}"'] \ar[r, "{(q,1)}"] & {(Z, \dcat(z)(U))} \ar[d, "{(z,1)}"]\\
			{(Y, \dcat(y)(U))} \ar[r, "{(y,1)}"'] &{(X,U)}
		\end{tikzcd},\hspace{1cm}
		\begin{tikzcd}
			P \ar[d, "p"'] \ar[r, "q"] & Z \ar[d, "z"]\\
			Y\ar[r, "y"'] & X
		\end{tikzcd}:
		\]
		One can easily show that one is a pullback if and only if the other is a pullback: therefore $\gbicat(\dcat)$ has pullbacks of horizontal arrows along horizontal arrows, and they are again horizontal, if and only if $\cbicat$ has pullbacks of every arrow $y:Y\rightarrow X$ such that the fibre $\dcat(X)$ is non-empty (of course, if $\dcat(X)$ were empty we would not be able to lift the arrows $y$ and $z$ to $\gbicat(\dcat)$, in order to compute their pullback).
		
		\item Consider the following two squares, the left-hand one being in the fibre $\dcat(X)$:
		\[
		\begin{tikzcd}
			P \ar[d, "p"'] \ar[r, "q"] & W\ar[d, "b"]\\
			V \ar[r, "a"'] & U 	
		\end{tikzcd},\ 
		\begin{tikzcd}
			{(X,P)} \ar[d, "{(1,p)}"'] \ar[r, "{(1,q)}"] & {(X,W)}\ar[d, "{(1,b)}"]\\
			{(X,V)} \ar[r, "{(1,a)}"'] & {(X,U)}	
		\end{tikzcd}.
		\]
		A quick computation shows that if the right-hand one is a pullback square in $\gbicat(\dcat)$ then the left-hand one is a pullback square in the fibre $\dcat(X)$. Conversely, suppose that the left hand square is a pullback in $\dcat(X)$ and consider the following situation:
		\[
		\begin{tikzcd}
			{(Y,R)} \ar[ddr, "{(y,r)}"', bend right] \ar[drr, "{(y,s)}", bend left]&&\\
			&{(X,P)} \ar[d, "{(1,p)}"'] \ar[r, "{(1,q)}"] & {(X,W)}\ar[d, "{(1,b)}"]\\
			&{(X,V)} \ar[r, "{(1,a)}"'] & {(X,U)}	
		\end{tikzcd}.
		\]
		The equality $(1,a)(y,r)=(1,b)(y,s)$ implies that there is a commutative diagram
		\[
		\begin{tikzcd}
			R \ar[ddr, "r"', bend right] \ar[drr, "s", bend left] \ar[dr, dashed, "t"description]&&\\
			&\dcat(y)(P) \ar[d, "\dcat(y)(p)"'] \ar[r, "\dcat(y)(q)"] & \dcat(y)(W)\ar[d, "\dcat(y)(b)"]\\
			&\dcat(y)(V) \ar[r, "\dcat(y)(a)"'] & {\dcat(y)(U)}	
		\end{tikzcd}
		\]
		in $\dcat(Y)$: since $\dcat(y)$ preserves pullback squares, there exists a unique $t:R\rightarrow \dcat(y)(P)$ making it commutative, which allows us to define a unique $(y,t):(Y,R)\rightarrow (X,P)$ presenting the square in $\gbicat(\dcat)$ as a pullback square.
	\end{enumerate}
\end{proof}
In the equivalent hypotheses of the previous lemma the functors $i_X:\dcat(X)\rightarrow \gbicat(\dcat)$ and $p_\dcat:\gbicat(\dcat)\rightarrow\cbicat$ preserve and reflect pullbacks. Notice that this is not always the case for other limits: for instance, terminal objects are almost never preserved by the functors $i_X$. However, similar considerations to those above still hold for limits of a general shape: for instance, given a diagram in $\gbicat(\dcat)$ whose edges are all horizontal, if its limit legs are all horizontal then the limit is preserved by $p_\dcat$. We will not need results of such generality. Still, it is interesting to remark that the existence of limits in fibres and in the base category is enough for them to exist in the fibration:
\begin{prop}
	Consider two categories $\cbicat$ and $\ibicat$ and a pseudofunctor $\dcat:\cbicat\op\rightarrow\CAT$. Suppose that limits of shape $\ibicat$ exist in $\cbicat$ and in every fibre of $\dcat$, and that these latter ones are preserved by transition morphisms: then $\gbicat(\dcat)$ has limits of shape $\ibicat$, and they are preserved by $p_\dcat$.
\end{prop}
\begin{proof}
	Consider a diagram $D:\ibicat\rightarrow \gbicat(\dcat)$: let us denote the image of an object $i$ of $\ibicat$ by $(X_i, U_i)$, and the image of an arrow $s:i\rightarrow j$ by $(x_s,b_s):(X_i,U_i)\rightarrow (X_j,U_j)$. First of all, by composing $D$ with $p_\dcat:\gbicat(\dcat)\rightarrow\cbicat$ we obtain a diagram of shape $\ibicat$ in $\cbicat$, whose vertices are the $X_i$ and whose edges are the $x_s$: therefore, we can compute its limit, an object $X$, and denote by $l_i:X\rightarrow X_i$ the legs of its limit cone. Next, we consider a further diagram $D':\ibicat\rightarrow\dcat(X)$: we map each object $i$ of $\ibicat$ to $\dcat(l_i)(U_i)$, and each arrow $s:i\rightarrow j$ in $\ibicat$ to the arrow
	\[
	\dcat(l_i)(U_i)\xrightarrow{\dcat(l_i)(b_s)}\dcat(l_i)\dcat(x_s)(U_j)\isorightarrow \dcat(l_j)(U_j),
	\]
	where the last arrow is one of the canonical isomorphisms of $\dcat$. We can compute the limit of $D'$: we shall call it $U$, and denote by $\lambda_i:U\rightarrow \dcat(l_i)(U_i)$ the legs of its limit cone. It is now a matter of computations to verify that the arrows
	\[
	(l_i,\lambda_i):(X,U)\rightarrow(X_i,U_i)
	\]
	form the legs of a limit cone for the original diagram $D$. Indeed, consider any cone $(y,_i,a_i):(Y,V)\rightarrow (X_i,U_i)$ over it: the arrows $y_i:Y\rightarrow X_i$ form a cone over $p_\dcat\circ D$, and thus induce a unique $y:Y\rightarrow X$ such that $y_i=l_i\circ y$; on the other hand, since transition morphisms preserve limits for diagrams of shape $\ibicat$, the object $\dcat(y)(U)$ is the limit for the composite diagram $\dcat(y)\circ D'$, and thus the legs $a_i:V\rightarrow \dcat(y)(U_i)$ induce a unique $a:V\rightarrow \dcat(y)(U)$ such that $a_i=\dcat(y)(\lambda_i)\circ a$. This provides the two components of a unique arrow $(y,a):(Y,V)\rightarrow (X,U)$, showing that $(X,U)$ is the limit of $D$.
\end{proof}
\begin{cor}\label{cor:colimits_in_discretefib_from_base}
	Consider two categories $\cbicat$, a discrete fibration $p:\fib P \rightarrow \cbicat$ and a further non-empty category $\ibicat$: if $\cbicat$ is has limits of shape $\ibicat$ then $\fib P$ has limits of shape $\ibicat$, and $p$ preserves them.
	
	In particular, if $\cbicat$ is (finitely) complete and the fibre $P(1_\cbicat)$ is a singleton, then $\fib P$ is (finitely) complete category and $p$ preserves (finite) limits.
\end{cor}
\begin{proof}
	It follows immediately from the previous result, since the fibres of $p$ are discrete and thus admit all non-empty limits. The last claim holds because if $P(1_\cbicat)$ has a terminal object $T$ then the pair $(1_\cbicat, T)$ is terminal in $\fib P$, but a dicrete category has a terminal object if and only if it is a singleton. 
\end{proof}

\section{Grothendieck construction for a composite fibration}

Let ${\mathbb D}:{\cal C}\op\to \textup{\bf CAT}$ be a $\cal C$-indexed category and ${\mathbb E}:{\cal G}({\mathbb D})\op\to \textup{\bf CAT}$ a ${\cal G}({\mathbb D})$-indexed category. If we consider the fibrations $p_{\mathbb D}:{\cal G}({\mathbb D})\to {\cal C}$ and $p_{\mathbb E}:{\cal G}({\mathbb E})\to {\cal G}({\mathbb D})$: it is a well known fact that the composite functor $p_\dcat p_\ecat$ is still a fibration. We provide here, for future reference, an explicit description of the corresponding pseudofunctor:
\begin{prop}\label{prop:compositefibration}
	Given $\mathbb D$ and $\mathbb E$ as above, consider the $\cal C$-indexed category ${\mathbb E}_{\mathbb D}:{\cal C}\op\to \textup{\bf CAT}$\index{${\mathbb E}_{\mathbb D}$} defined as follows: 
	\begin{itemize}
		\item for any $X$ in $\cbicat$, ${\mathbb E}_{\mathbb D}(X)$ is the category having as objects the pairs $(U, H)$ where $U$ is an object in ${\mathbb D}(X)$ and $H$ in ${\mathbb E}(X, U)$ and as arrows $(U', H')\to (U, H)$ the pairs $(a, h)$ where $a:U'\to U$ in the category ${\mathbb D}(X)$ and $h:H'\to {\mathbb E}(1_{X}, a)(H)$ in the category ${\mathbb E}(X, U')$; 
		
		\item for any arrow $y:Y\to X$ in $\cbicat$,  
		\[
		{\mathbb E}_{\mathbb D}(y):{\mathbb E}_{\mathbb D}(X)\to {\mathbb E}_{\mathbb D}(Y)
		\]
		is the functor sending any object $(U, H)$ of ${\mathbb E}_{\mathbb D}(X)$ to the object $$({\mathbb D}(y)(U), {\mathbb E}(y, 1)(H))$$ of ${\mathbb E}_{\mathbb D}(Y)$ and an arrow $(a, h):(U', H')\to (U, H)$ in ${\mathbb E}_{\mathbb D}(X)$ to the arrow of ${\mathbb E}_{\mathbb D}(Y)$
		\[({\mathbb D}(y)(a), {\mathbb E}(y, 1)(h))\hspace{-0.5ex}:\hspace{-0.5ex}({\mathbb D}(y)(U'), {\mathbb E}(y, 1)(H'))\hspace{-0.5ex}\to\hspace{-0.5ex} ({\mathbb D}(y)(U), {\mathbb E}(y, 1)(H)).\]
	\end{itemize}
	Then the fibration associated with ${\mathbb E}_{\mathbb D}$ is isomorphic to the composite fibration $p_{\mathbb D}\circ  p_{\mathbb E}$. 
\end{prop} 

\begin{proof}
	The result trivial, once we observe how the two fibrations $p_\dcat p_\ecat:\gbicat(\ecat)\rightarrow\gbicat(\dcat)\rightarrow \cbicat$ and $p_{\ecat_\dcat}:\gbicat(\ecat_\dcat)\rightarrow\cbicat$ are made.
	
	The category $\gbicat(\ecat)$ has objects of the form $((X,U), H)$, where $(X,U)$ is an object of $\gbicat(\dcat)$ and $H$ is an object of $\ecat(X,U)$; a morphism $((y,a),k):((Y,V),K)\rightarrow ((X,U),H)$ is given by a an arrow $(y,a):(Y,V)\rightarrow (X,U)$ in $\gbicat(\dcat)$ and an arrow $k:K\rightarrow \ecat(y,a)(H)$ in $\ecat(Y,V)$. We can simplify this description by saying that the objects of $\gbicat(\ecat)$ are triples $(X,U,H)$ with $X$ in $\cbicat$, $U$ in $\dcat(X)$, $H$ in $\ecat(X,U)$; arrows $(y,a,k):(Y,V,K)\rightarrow (X,U,H)$ are given by an arrow $y:Y\rightarrow X$ of $\cbicat$, an arrow $a:V\rightarrow \dcat(y)(U)$ of $\dcat(Y)$ and an arrow $k:K\rightarrow \ecat(y,a)(U)$ of $\ecat(Y,V)$.
	
	In $\gbicat(\ecat_\dcat)$ the objects  are of the form $(X,(U,H))$, where $(U,H)$ is an object of $\ecat_\dcat(X)$, and a morphism $(y,(a,k')):(Y,(V,K))\rightarrow (X,(U,H))$ of $\gbicat(\ecat_\dcat)$ is indexed by an arrow $y:Y\rightarrow X$ in $\cbicat$ and an arrow $(a,k'):(V,K)\rightarrow \ecat_\dcat(y)(U,H)=(\dcat(y)(U), \ecat(y,1)(H))$ in $\ecat_\dcat(Y)$. We can simplify this description too: objects of $\gbicat(\ecat_\dcat)$ are triples $(X,U,H)$ with $X$ in $\cbicat$, $U$ in $\dcat(X)$ and $H$ in $\ecat(X,U)$ (they are in fact the exact same objects of $\gbicat(\ecat)$); while an arrow $(y,a,k'):(Y,V,K)\rightarrow (X,U,H)$ of $\gbicat(\ecat_\dcat)$ is the given of an arrow $y:Y\rightarrow X$ in $\cbicat$, an arrow $a:V\rightarrow \dcat(y)(U)$ in $\dcat(Y)$ and an arrow $k':K\rightarrow \ecat(1,a)\ecat(y,1)(H)$ of $\ecat(Y,V)$. It is now obvious that by composing $k'$ with the canonical isomorphism $\ecat(1,a)\ecat(y,1)(H)\simeq \ecat(y,a)(H)$ we can go from an arrow $(y,a,k'):(Y,V,K)\rightarrow (X,U,H)$ of $\gbicat(\ecat_\dcat)$ to an arrow $(y,a,k):(Y,V,K)\rightarrow (X,U,H)$ of $\gbicat(\ecat)$ and viceversa. The fact that this provides an equivalence (actually, an isomorphism) of the categories $\gbicat(\ecat_\dcat)$ and $\gbicat(\ecat)$, which is also compatible with the respective fibration functors to $\cbicat$, is a matter of straightforward computations.
\end{proof}

\section{Stacks}\label{sec:stack}
Consider a site $(\cbicat,J)$ and a presheaf $P:\cbicat\op\rightarrow \Set$: we recall that $P$ is called a $J$-separated presheaf (resp. $J$-sheaf) if for every object $X$ of $\cbicat$ and every $J$-covering sieve $m_S:S\rightarrowtail \yo(X)$, the map
$$[\cbicat\op,\Set](\yo(X),P)\xrightarrow{-\circ m_S} [\cbicat\op,\Set](S,P)$$
is injective (resp. a bijection). The extension of this definition to the fibrational context is immediate:
\begin{defn}\label{def:stack_fib}
	Consider a site $(\cbicat,J)$ and a fibration $p:\dbicat\rightarrow \cbicat$: then $p$ is a \emph{$J$-prestack}\index{prestack} (resp. \emph{$J$-stack}\index{stack}) if for every $J$-sieve $m_S:S\rightarrowtail \yo(X)$ the functor
	$$\Fib_\cbicat(\cbicat/X,\dbicat)\xhookrightarrow{- \circ \fib m_S} \Fib_\cbicat(\fib S,\dbicat)$$
	is full and faithful (resp. an equivalence).
\end{defn}
The notion of stack is a precise expansion of that of sheaf, as established by the next result:
\begin{prop}[{\cite[Proposition 4.9]{vistoli.stack}}]\label{prop:psh+stack=sheaf}
	Consider a site $(\cbicat,J)$ and a presheaf $P:\cbicat\op\rightarrow\Set$: then $P$ is $J$-separated (resp. $J$-sheaf) if and only if the fibration $\fib P\rightarrow\cbicat$ is a $J$-prestack (resp. $J$-stack).
\end{prop}
\begin{remark}
	Discrete fibrations corresponding to $J$-sheaves, \ie discrete $J$-stacks, are called $J$-gluing fibrations in \cite[Definition 4.65]{denseness}.
\end{remark}

As is custom, separated presheaves and sheaves are often defined in terms of matching families and amalgamations, a definition which is usually more `operatively' useful. The same can be done for (pre)stacks, if we restrict to cloven fibrations and translate the definition above into the $\cbicat$-indexed language: given a site $(\cbicat,J)$ and a $\cbicat$-indexed category $\dcat:\cbicat\op\rightarrow\CAT$, then $\dcat$ is a {$J$-prestack} (resp. {$J$-stack}) if and only if for every sieve $m_S:S\rightarrowtail \yo(X)$ the functor
$$\Ind_\cbicat(\yo(X), \dcat)\xrightarrow{-\circ m_S}\Ind_\cbicat(S, \dcat)$$
is full and faithful (resp. an equivalence), where both $\yo(X)$ and $S$ are interpreted as discrete $\cbicat$-indexed categories. This definition can be unwinded by expliciting what a pseudonatural transformation $\alpha:S\Rightarrow \dcat$ is.  It consists of an object $U_y\in \dcat(Y)$  for every $y\in S(Y)$, and of a family of isomorphisms $\alpha_{y,z}: \dcat(z)(U_y)\simeq U_{yz}$ of $\dcat(Z)$ for every $y\in S(Y)$ and every $z:Z\rightarrow Y$, such that the following identities hold for any $w:W\rightarrow Z$, $z:Z\rightarrow Y$, $y\in S(Y)$ in $\cbicat$:
$$\alpha_{y,1_Y} = \phi^\dcat_Y(U_y)\inv: \dcat(1_Y)(U_y)\rightarrow U_y$$
$$\alpha_{y,zw}\circ\phi^\dcat_{z,w}(U_y)=\alpha_{yz,w}\circ\dcat(w)(\alpha_{y,z}):\dcat(w)\dcat(z)(U_y)\rightarrow U_{yzw}$$
Such a collection $\alpha=(U_y, \alpha_{y,z})_{y\in S}$ is called a \emph{descent datum} for $\dcat$ and $S$. We can see this as the equivalent of a matching family, but with some further `elasticity' due to the fact that we are considering pseudonatural trasformations: for each $y$ in $S$ we have an object $U_y$ in $\dcat(\dom y)$, and these objects are mutually compatible up to canonical isomorphisms. 

The definition of a morphism in the category of descent data, \ie an arrow $\xi:(U_y,\alpha_{y,z})_{y\in S}\rightarrow (V_y, \beta_{y,z})_{y\in S}$, can be retrieved analogously by expliciting the definition of a {modification} $\xi:\alpha \Rrightarrow \beta$: it is the given of an arrow $\xi_y:U_y\rightarrow V_y$ for each $y\in S(Y)$, subject to the condition $\beta_{y,z}\circ \dcat(Z)(\xi_y)=\xi_{yz}\circ \alpha_{y,z}$. We can therefore consider the category of descent data for $\dcat$ and $S$, which we will denote by $\dcat(S)$.

Now, recall that by the fibred Yoneda lemma for Grothendieck fibrations, there is a pseudonatural equivalence $\Ind_\cbicat(\yo(X),\dcat)\simeq \dcat(X)$: then the functor $(-\circ m_S)$ can be expressed as a functor  $L_S:\dcat(X)\rightarrow \dcat(S)$ acting on objects as follows:
$$U\in \dcat(X)\mapsto (\dcat(y)(X), \phi_{y,z}^\dcat(U))_{y\in S}.$$

A descent datum $(U_y, \alpha_{y,z})_{y\in S}$ is said to be \emph{effective} if it lies in the essential image of $L_S$, \ie there are $U$ in $\dcat(X)$ and an isomorphism of descent data $(\dcat(y)(U), \phi^\dcat_{y,z}(U))_S\simeq(U_y, \alpha_{y,z})_S$: we may think the object $U$ as the generalized version of an amalgamation for a matching family. We end up with the following definition of (pre)stack:
\begin{defn}\label{def:stack_ind}
	Consider a site $(\cbicat,J)$ and a $\cbicat$-indexed category $\dcat:\cbicat\op\rightarrow\CAT$: $\dcat$ is called a \emph{$J$-prestack}\index{prestack!for pseudofunctors} (resp.  \emph{$J$-stack}\index{stack!for pseudofunctors}) if for every $J$-sieve $S\rightarrowtail \yo(X)$ the functor
	$$L_S:\dcat(X)\rightarrow \dcat(S)$$
	is full and faithful (resp. an equivalence). More explicitly, $\dcat$ is a stack if and only if for every sieve $S$ of $J$ all descent data for $\dcat$ and $S$ are effective.
	
	Stacks over a site $(\cbicat,J)$ form a 2-full and faithful subcategory of $\Ind_\cbicat$, which we will denote by $\St(\cbicat,J)$\index{$\St(\cbicat,J)$}. In particular, we shall denote the category of stacks on a topos $\Etopos$, with respect to the canonical topology on it, by $\St(\Etopos)$\index{$\St(\Etopos)$}.\footnote{These $2$-categories of stacks are Grothendieck $2$-toposes in the sense of \cite{Street.2dimsheaftheory}.}
\end{defn}
In short, for a stack compatible local data along the arrows of a covering sieve $S$ can be glued together into a global datum in the fibre over the codomain of said sieve, in a similar way to sheaves. We remark that if the category $\cbicat$ has pullbacks, the notion of descent data admits a more manageable definition using $J$-covering families: it can be found for instance in \cite[§ 4.1.2]{vistoli.stack}. 

A further similarity between stacks and sheaves is the fact that, similarly to the sheafification process for presheaves, there exists a stackification process for fibrations:
\begin{thm}\label{thm:stackification_esiste}
	Consider a site $(\cbicat,J)$. There exists a 2-functor
	\[
	s_J:\Fib_\cbicat\rightarrow\St(\cbicat,J)
	,\]
	called \emph{stackification}\index{stackification}\index{$s_J$} (or \emph{associated stack functor}), which is left adjoint to the inclusion $i_J:\St(\cbicat,J)\rightarrow \Fib_\cbicat$.\index{$i_J$}
\end{thm}
\begin{proof}
	See Chapter II, Section 2 of \cite{giraud.cohomologie}.
\end{proof}

We will now show that another well known result about Grothendieck fibrations that still holds for Street fibrations. We recall that for $p:\dbicat\rightarrow\cbicat$ a Grothendieck fibration, and for any $X$ in $\cbicat$ and any pair of objects $A$, $B$ in $\dcat(X)$, there is a presheaf $\Hom(A,B):(\cbicat/X)\op\rightarrow \Set$ defined, for an object $[y:Y\rightarrow X]$ of $\cbicat/X$, as $\Hom(A,B)([y]):=\dcat(Y)\left(\dcat(y)(A),\dcat(y)(B)\right)$. Then the following result holds:
\begin{prop}\label{prop:prestack_Homft_caso_Grothendieck}
	Consider a site $(\cbicat,J)$: a Grothendieck fibration $p:\dbicat\rightarrow\cbicat$ is a prestack if and only if for every $X$ in $\cbicat$ and $A$ and $B$ in $\dcat(X)$ the presheaf $\Hom(A,B):(\cbicat/X)\op\rightarrow\Set$ is a $J_X$-sheaf.
\end{prop}
The definition of $\Hom$-presheaf extends to Street fibrations in the following form: 
\begin{defn}
	Consider a cloven fibration $p:\dbicat\rightarrow\cbicat$ and two objects $(A,\alpha:X\isorightarrow p(A))$ and $(B, \beta:X\isorightarrow p(B))$ of $\dcat(X)$: we defined the $\Hom$-presheaf\index{$\Hom((A,\alpha), (B,\beta))$}
	\[\Hom((A,\alpha), (B,\beta)):(\cbicat/X)\op\rightarrow\Set\]
	as follows:
	\begin{itemize}
		\item for any $[y:Y\rightarrow X]$ in $\cbicat/X$, \[\Hom((A,\alpha), (B,\beta))([y]):=\dcat(Y)\left( \dcat(y)(A,\alpha), \dcat(y)(B,\beta)) \right)\]
		more explicitly, elements of $\Hom((A,\alpha),(B,\beta))([y])$ can be seen as arrows $\gamma:\dom(\widehat{\alpha y}_A)\rightarrow \dom(\widehat{\beta y}_B)$ such that $p(\gamma)\theta_{\alpha y,A}=\theta_{\beta y,B}$;
		\item for $z:[yz]\rightarrow [y]$ and $\gamma\in \Hom((A,\alpha),(B,\beta))([y])$, we define the arrow $\Hom((A,\alpha),(B,\beta))(z)(\gamma)$ as the composite $\chi_{\beta y,z,B}\inv \dcat(z)(\gamma)\chi_{\alpha y,z,A}$; explicitely, it is the unique arrow $\gamma':\dom(\widehat{\alpha yz}_A)\rightarrow \dom(\widehat{\beta yz}_B)$ satisfying the identities 
		\[ \widehat{\beta y}_B \gamma \lambda_{\alpha y,z,A}=\widehat{\beta yz}_B \gamma'\]
		\[p(\gamma')=\theta_{\beta yz,B}\theta_{\alpha yz,A}\inv \]
	\end{itemize}
\end{defn}

\begin{remark}
	Consider  $\Hom((A,\alpha),(B,\beta)):(\cbicat/X)\op\rightarrow\Set$ and a $J_X$-sieve $S$ over $[y:Y\rightarrow X]$ in $\cbicat/X$: let us explicit what a matching family ad an amalgamation are in this case.
	
	A \emph{matching family} for $\Hom((A,\alpha),(B,\beta))$ and $S$ is the given for every $f\in S$ of an arrow $\gamma_f:\dom(\widehat{\alpha yf}_A)\rightarrow \dom(\widehat{\beta yf}_B)$ of $\dbicat$ such that $p(\gamma_f)=\theta_{\beta yf,B}\theta_{\alpha yf,A}\inv$, with the condition that whenever $g$ is precomposable to $f$ then $\gamma_{fg}=\Hom((A,\alpha),(B,\beta))(g)(\gamma_f)$, \ie  $\gamma_{fg}:\dom(\widehat{\alpha yfg}_A)\rightarrow \dom(\widehat{\beta yfg}_B)$ is the unique arrow such that $\widehat{\beta yf}_B \gamma_f \lambda_{\alpha y,fg,A}=\widehat{\beta yfg}_B \gamma_{fg}$.
	
	An \emph{amalgamation} for this matching family is an arrow $\gamma:\dom(\widehat{\alpha y}_A)\rightarrow \dom(\widehat{\beta y}_B)$ such that $p(\gamma)=\theta_{\beta y,B}\theta_{\alpha y,A}\inv$ and that for every $f$ in $S$ the arrow $\gamma_f$ is the unique arrow such that $\widehat{\beta y}_B \gamma \lambda_{\alpha y,f,A}=\widehat{\beta yf}_B \gamma_{f}$.

\end{remark}
We now provide two technical lemmas about matching families for $\Hom$-functors.
\begin{lemma}
	Consider a site $(\cbicat,J)$, an arrow $y:Y\rightarrow X$ of $\cbicat$ and a sieve $S\in J(Y)$. Denote by $S_{[1_Y]}$ (resp. $S_{[y]}$) the $J_Y$-sieve over $[1_Y]$ (resp. $J_X$-sieve over $[y]$) whose arrows are those of $S$: then $\lan_{(\fib y)\op}(S_{[1_Y]})\simeq S_{[y]}$ naturally.
\end{lemma}
\begin{proof}
	Consider a presheaf $H:(\cbicat/X)\op\rightarrow \Set$: an arrow $\alpha:S_{[y]}\rightarrow H$ in $[(\cbicat/X)\op,\Set]$ is a matching family for $H$ and $S_{[y]}$, \ie the given for the arrows $z:[yz]\rightarrow [y]$ in $S_{[y]}$ of compatible elements $x_z\in H([yz])$. It is immediate to see this is the same as a matching family for $S_{[1_Y]}$ and $H\circ (\fib y)\op$, providing a natural bijection
	\[ [(\cbicat/X)\op,\Set](S_{[y]}, H)\simeq [(\cbicat/Y)\op,\Set](S_{[1_Y]}, H\circ (\fib y)\op)\]
	which implies $S_{[y]}\simeq \lan_{(\fib y)\op}(S_{[1_Y]})$.
\end{proof}
From this it follows that all matching families of the $\Hom$-functors can be interpreted, if we allow a change of slice category, as matching families over the terminal object of the slice:
\begin{cor}\label{cor:Hom_mfam_su_terminali}
	Consider a site $(\cbicat,J)$, a fibration $\dbicat\rightarrow \cbicat$, a $J$-sieve $S$ over $Y$ and two objects $(A,\alpha)$ and $(B,\beta)$ of $\dcat(X)$: a matching family for $\Hom((A,\alpha), (B,\beta))$ and $S$, seen as a $J_X$-sieve over $[y]$ in $\cbicat/X$, is the same as a matching family for $\Hom(\dcat(y)(A,\alpha), \dcat(y)(B,\beta))$ and the sieve $S$ seen as a $J_Y$-sieve over $[1_Y]$ in $\cbicat/Y$. The same holds for amalgamations of matching families. 
\end{cor}
\begin{proof}
	A rapid computation shows that $\Hom(\dcat(y)(A,\alpha), \dcat(y)(B,\beta))$ is isomorphic to $\Hom((A,\alpha), (B,\beta))\circ (\fib y)\op$: the claim then follows from previous lemma.
\end{proof}
The second technical lemma relates matching families for $\Hom$-functors with 2-cells of fibrations:
\begin{lemma}
	Consider a site $(\cbicat,J)$, a fibration $p:\dbicat\rightarrow \cbicat$, a $J$-sieve $R$ over $X$ and two objects $(A,\alpha)$ and $(B,\beta)$ of $\dcat(X)$: a matching family for $\Hom((A,\alpha), (B,\beta))$ and $R$, seen as a $J_X$-sieve over $[1_X]$ in $\cbicat/X$, is the same thing as a 2-cell of fibrations $\Psi(A,\alpha)\circ m_R\Rightarrow \Psi(B,\beta)\circ m_R$, where $m_R:\fib R\hookrightarrow\cbicat/X$ is the canonical inclusion functor. Analogously, an amalgamation for a matching family as above corresponds to a 2-cell of fibrations $\Psi(A,\alpha)\Rightarrow \Psi(B,\beta)$, \ie to a morphism $(A,\alpha)\rightarrow (B,\beta)$.
\end{lemma}
\begin{proof}
	Remember that $\Psi(A,\alpha)\circ m_R:\fib R\rightarrow \dbicat$ operates as follows: every $[y]$ object of $\fib R$, \ie every arrow $y$ in $R$, is sent to $\dom(\widehat{\alpha y}_A)$, and every morphism $z:[yz]\rightarrow [y]$ to $\lambda_{\alpha y,z,A}:\dom(\widehat{\alpha yz}_A)\rightarrow \dom(\widehat{\alpha y}_A)$. It is now immediate to see that the components of a matching family for $\Hom((A,\alpha), (B,\beta))$, being arrows $\gamma_y:\dom(\widehat{\alpha y}_A)\rightarrow \dom(\widehat{\beta y}_B)$, provide exactly the components for a 2-cell of fibrations $\Psi(A,\alpha)\circ m_R\Rightarrow \Psi(B,\beta)\circ m_R$, and viceversa.
\end{proof}
We are now ready to prove the following:
\begin{prop}
	Consider a site $(\cbicat,J)$ and a cloven fibration $p:\dbicat\rightarrow\cbicat$: then $p$ is a $J$-prestack if and only if if for every $X$ in $\cbicat$ and every $(A,\alpha)$, $(B,\beta)$ in $\dcat(X)$ the presheaf $\Hom((A,\alpha),(B,\beta)):(\cbicat/X)\op\rightarrow\Set$ is a $J_X$-sheaf.
\end{prop}
\begin{proof}
	The proof is a generalization of the usual argument for Grothendieck fibrations. In the following we will use the notations $(F,\phi):=\Psi(A,\alpha)$ and $(G,\gamma)=\Psi(B,\beta)$.
	
	First of all, by the previous lemma we may reduce to consider matching families over the terminal $[1_X]$ of $\cbicat/X$. So suppose $\dbicat$ is a prestack and consider a $J$-covering sieve $R$ over $X$ and a matching family for $R$ over $[1_X]$ and $\Hom((A,\alpha), (B,\beta))$: by the previous lemma, it corresponds to a 2-cell $\alpha:(F_{|R},\phi_{|R})\Rightarrow (G_{|R},\gamma_{|R})$ in $\Fib_\cbicat(\fib R, \dbicat)$. If $\dbicat$ is a prestack the functor $\dcat(X)\simeq\Fib_\cbicat(\cbicat/X, \dbicat)\rightarrow \Fib_\cbicat(\fib R, \dbicat)$ is full and faithful, therefore $\alpha$ is the image of a unique $\bar{\alpha}:(A,\alpha)\rightarrow (B,\beta)$: this in turn corresponds to an amalgamation for the original matching family and hence the $\Hom$-presheaf we were considering is a $J_X$-sheaf. 
	
	If conversely all $\Hom$-presheaves are sheaves, start by considering $(F,\phi)$ and $(G,\gamma)$ as above and a 2-cell $\alpha:(F_{|R},\phi_{|R})\Rightarrow (G_{|R},\gamma_{|R})$. The 2-cell $\alpha$ corresponds to a matching family for $R$ over $[1_X]$ and $\Hom((A,\alpha), (B,\beta))$, by the previous lemma, and that has a unique amalgamation. Such an amalgamation corresponds to a unique 2-cell $\bar{\alpha}:(F,\phi)\Rightarrow (G,\gamma)$ extending the original $\alpha$, and hence $\Fib_\cbicat(\cbicat/X,\dbicat)\rightarrow \Fib_\cbicat(\fib R,\dbicat)$ is fully faithful.
\end{proof}

It is fundamental to remark that every site is canonically associated with its stack, as follows:
\begin{defn}\label{def:stack_canonica_su_sito} \label{sec:canonical_stack}
	Consider a site $(\cbicat,J)$: the \emph{canonical fibration over $(\cbicat,J)$}\index{fibration!canonical}\index{stack!canonical}\index{$\canst_{(\cbicat,J)}$} is the fibration associated to the 2-functor $\canst_{(\cbicat,J)}:\cbicat\op\rightarrow\CAT$ such that
	\[
	\left[Y\xrightarrow{y}X\right]\xmapsto{\canst_{(\cbicat,J)}}\left[\Sh(\cbicat,J)/\ell_J(X)\xrightarrow{\ell_J(y)^*}\Sh(\cbicat,J)/\ell_J(Y)\right].
	\]
	Using the content of Section \ref{section:comorphisms} and in particular Proposition \ref{prop:fib_discreta_slice_topos} we can define $\canst_{(\cbicat,J)}$ in terms of the functors $\fib y: \cbicat/Y\rightarrow \cbicat/X$ as 
	\[
	\left[Y\xrightarrow{y}X\right]\xmapsto{\canst_{(\cbicat,J)}}\left[\Sh(\cbicat/X,J_X)\xrightarrow{C_{\fib y}^*}\Sh(\cbicat/Y,J_Y)\right]
	\]
	Notice in particular that $C_{\fib y}^*$ acts as $-\circ (\fib y)\op$ by Proposition \ref{prop:fib_e_morfib_sono_can.comorf}.
\end{defn}
\begin{remark}
	Notice that, since there is a canonical choice of pullbacks in $\Sh(\cbicat,J)$, the canonical fibration over $(\cbicat,J)$ is always cloven.
\end{remark}
We can use both definitions of the 2-functor $\canst_{(\cbicat,J)}$ in order to describe its fibration.
Using the first definition of $\canst_{(\cbicat,J)}$, it is immediate to see that the fibration $\gbicat(\canst_{(\cbicat,J)})\rightarrow \cbicat$ is the comma category $\comma{1_{\Sh(\cbicat,J)}}{\ell_J}\rightarrow \cbicat$: objects are arrows $h:H\rightarrow \ell_J(X)$ of $\Sh(\cbicat,J)$, and arrows $(k:K\rightarrow \ell_J(Y))\rightarrow (h:H\rightarrow \ell_J(X))$ are pairs $(g,y)$ where $g:K\rightarrow H$, $y:Y\rightarrow X$ and $\ell_J(y)k=hg$. In particular, $(g,y)$ is cartesian if the square formed with $h$ and $k$ is a pullback square in $\Sh(\cbicat,J)$.

Using instead the second description, we see that objects of $\gbicat(\canst_{(\cbicat,J)})$ are couples $(X, P:\cbicat/X\op\rightarrow \Set)$, where $X$ is an object of $\cbicat$ and $P$ is a $J_X$-sheaf. Arrows of $\gbicat(\canst_{(\cbicat,J)})$ are pairs $(y,\alpha):(Y,Q)\rightarrow (X,P)$, where $y:Y\rightarrow X$ in $\cbicat$ and $\alpha: Q\Rightarrow P\circ  (\fib y)\op$. In particular, a cartesian arrow of $\gbicat(\canst_{(\cbicat,J)})$ is of the form $(y,\alpha):(Y, Q)\rightarrow (X,P)$ with $\alpha$ an isomorphism. In the following we will denote by $\canst_{(\cbicat,J)}$ indifferently the fibration and the $\cbicat$-indexed category.
\begin{remark}
	Using the well known equivalence $\Etopos\simeq \Sh(\Etopos, J\can_\Etopos)$, one can immediately see that the canonical fibration over $(\Etopos, J\can_\Etopos)$ coincides with the fibration $\cod:\Mor(\Etopos)\rightarrow\Etopos$ in Example \ref{ex:indexedcategories}(iv).
\end{remark}
The following result will be fundamental later: 
\begin{thm}\label{thm:canst_stack}
	Consider a site $(\cbicat,J)$: the canonical fibration $\canst_{(\cbicat,J)}$ is a $J$-stack, called the \emph{canonical stack of $(\cbicat,J)$}.
\end{thm}
This is a well known result, which appears for instance as Proposition 3.4.4 in \cite{giraud.cohomologie}. It can be proved explicitely using descent data: in literature this is most commonly done in the simplifying hypothesis that the site has finite limits (see for instance Example 4.11 of \cite{vistoli.stack}). We will provide two alternative proofs of this result: one appears as Corollary \ref{cor:canst_via_immdiretta}, and sees $\canst_{(\cbicat,J)}$ as the image along a direct image functor of the canonical stack over $\Sh(\cbicat,J)$ (which we must suppose to be a stack); the other proof will rely instead on the fundamental adjunction, and appears as Corollary \ref{cor:canst_via_aggfond}.

\section{The truncation functor}\label{sec:truncation_functor}
Let us now consider in more depth the relationship between sheaves and stacks: it is provided by an adjunction which has multiple interpretations, both at the level of categories and of toposes. Since we will be comparing stacks with sheaves over a site, in the present section we will consider \emph{small} stacks, \ie those take values in $\Cat$ instead of $\CAT$. We shall denote the 2-category of small stacks over a site $(\cbicat,J)$ by $\St^s(\cbicat,J)$\index{$\St^s(\cbicat,J)$}.

We have already mentioned in Proposition \ref{prop:psh+stack=sheaf} that a presheaf over $\cbicat$ is a $J$-sheaf if and only if, when seen as a discrete $\cbicat$-indexed category, it is in fact a $J$-stack: this provides us with a functor 
\[
{j_J}:\Sh(\cbicat,J)\rightarrow \St^s(\cbicat,J),
\]\index{$j_J$}
which acts by seeing every $J$-sheaf as a discrete $J$-stack. In fact, if we consider the (2-)adjunction 
\[\begin{tikzcd}
	{\Set}\ar[r, "{Disc}"', hook, bend right, start anchor={south east}, end anchor={south west}] \ar[r, phantom, "\dashv"{rotate=270}] & {\Cat} \ar[l,"{\pi_0}"', bend right, start anchor={north west}, end anchor={north east}]
\end{tikzcd},\]
where $Disc$ maps each set to the corresponding discrete category while $\pi_0$ sends a category to its set of connected components, then $j_J$ acts by mapping a $J$-sheaf $P:\cbicat\op\rightarrow\Set$ to the composite $Disc\circ P$. By standard consederations about adjunctions and functor categories (\cfr for instance Lemma \ref{lemma:aggiunzione_tra_cat_funtori}),
we can conclude that $j_J:=(Disc\circ -)$ also has a left adjoint, and that it induces an adjunction at the level of sheaves and stacks: 
\begin{prop}
	Consider a site $(\cbicat,J)$: there is an adjunction
	\[
	\begin{tikzcd}
		{\phantom{\slash}\Sh(\cbicat,J)}\ar[r, "{j_J}"', hook, bend right, start anchor={south east}, end anchor={south west}] \ar[r, phantom, "\dashv"{rotate=270}] & {\St^s(\cbicat,J)\phantom{\slash}} \ar[l,"{\trunc_J}"', bend right, start anchor={north west}, end anchor={north east}]
	\end{tikzcd},\]
	where $j_J$ includes $J$-sheaves over $\cbicat$ as discrete stacks in $\St^s(\cbicat,J)$, while the \emph{$J$-truncation functor}\index{functor!truncation}\index{$\trunc_J$} $\trunc_J$ maps each small $J$-stack $\dcat:\cbicat\op\rightarrow\Cat$ to the $J$-sheaf $\sheafify_J(\pi_0\circ \dcat)$, where $\pi_0:\Cat\rightarrow\Set$ is the connected components functor. In other words, the $J$-truncation of a stack is computed first by considering its presheaf of connected components, and the by sheafifying it with respect to $J$. 
	
	In particular, when $J$ is the trivial topology over $\cbicat$, we shall denote the truncation-inclusion adjunction simply by $\trunc_\cbicat\dashv j_\cbicat$.
\end{prop}
\begin{proof}
	As we mentioned above, the functor $j_\cbicat:=(Disc\circ-)$ has a left adjoint $\trunc_\cbicat:=(\pi_0\circ -)$. Now consider the commutative diagram
	\[
	\begin{tikzcd}
		{\St^s(\cbicat,J)} & {\Ind_\cbicat} \\
		{\Sh(\cbicat,J)} & {[\cbicat\op,\Set]}
		\arrow["{\iota_J}"', shift right=2, hook, from=2-1, to=2-2]
		\arrow["{\sheafify_J}"', shift right=2, from=2-2, to=2-1]
		\arrow["\dashv"{anchor=center,rotate=270}, draw=none, from=2-1, to=2-2]
		\arrow["{j_\cbicat}"', shift right=2, hook, from=2-2, to=1-2]
		\arrow["{\trunc_\cbicat}"', shift right=2, from=1-2, to=2-2]
		\arrow["\dashv"{anchor=center}, draw=none, from=1-2, to=2-2]
		\arrow["{i_J}"', shift right=2, hook, from=1-1, to=1-2]
		\arrow["{s_J}"', shift right=2, from=1-2, to=1-1]
		\arrow["\dashv"{anchor=center, rotate=270}, draw=none, from=1-1, to=1-2]
		\arrow["{j_J}"', shift right=2, hook, from=2-1, to=1-1]
		\arrow["{\trunc_J}"', dashed, shift right=2, from=1-1, to=2-1]
		\arrow["\dashv"{anchor=center}, draw=none, from=2-1, to=1-1]
	\end{tikzcd}:
	\]
	The identity $j_\cbicat\circ \iota_J=i_J\circ j_J $ is obvious, and thus by standard arguments about adjoints (for instance by applying Lemma \ref{lemma:adj_restriction_subcat}) one immediately concludes that the composite $\sheafify_J\circ \trunc_\cbicat\circ i_J$ provides the dashed left adjoint to $j_J$.
\end{proof}
When multiple topologies are involved, truncation functors act naturally with respect to sheafification and stackification:
\begin{lemma}
	Consider a site $(\cbicat,J)$, a further topology $K\supseteq J$ and the diagram
	\[
	\begin{tikzcd}
		{\St^s(\cbicat,K)} & {\St^s(\cbicat,J)} \\
		{\Sh(\cbicat,K)} & {\Sh(\cbicat,J)}
		\arrow["{\iota_K}"', shift right=2, hook, from=2-1, to=2-2]
		\arrow["{\sheafify_J}"', shift right=2, from=2-2, to=2-1]
		\arrow["\dashv"{anchor=center,rotate=270}, draw=none, from=2-1, to=2-2]
		\arrow["{j_J}"', shift right=2, hook, from=2-2, to=1-2]
		\arrow["{\trunc_J}"', shift right=2, from=1-2, to=2-2]
		\arrow["\dashv"{anchor=center}, draw=none, from=1-2, to=2-2]
		\arrow["{i_K}"', shift right=2, hook, from=1-1, to=1-2]
		\arrow["{s_K}"', shift right=2, from=1-2, to=1-1]
		\arrow["\dashv"{anchor=center, rotate=270}, draw=none, from=1-1, to=1-2]
		\arrow["{j_K}"', shift right=2, hook, from=2-1, to=1-1]
		\arrow["{\trunc_K}"', shift right=2, from=1-1, to=2-1]
		\arrow["\dashv"{anchor=center}, draw=none, from=2-1, to=1-1]
	\end{tikzcd}.
	\]
	Then the following hold:
	\begin{align*}
		j_J\circ \iota_K &= i_K\circ j_K,\\
		\sheafify_J\circ \trunc_J &\cong \trunc_K\circ s_K,\\
		\sheafify_K&\cong \trunc_K\circ s_K\circ j_J.
	\end{align*}
\end{lemma}
\begin{proof}
	The first identity is obvious; since the composites in the second isomorphism are the left adjoints of those in the first equality, the second isomorphism must also hold; finally, the third isomorphism follows from the second one and the fact that $\trunc_J\circ j_J\cong \id_{\Sh(\cbicat,J)}$.
\end{proof}

As we already mentioned at the beginning of the section, the truncation functor can be described in other ways. We will briefly show that from the point of view of toposes the truncation of stacks can be built from our fundamental adjunction (see Chapter \ref{chap:fundadj}); on the other hand, at the level of sites it can be interpreted using the tools of \cite[Section 4.7]{denseness}. 

Let us begin by the topos-theoretic point of view: we shall need the adjunction of Corollary \ref{prop:fundadj_stack_topos},
\[\begin{tikzcd}
	{\phantom{\slash}\St(\cbicat,J)}\ar[r, "\Lambda'", bend left, start anchor={north east}, end anchor={north west}] \ar[r, phantom, "\dashv"{rotate=270}] & {\EssTopos\co/\Sh(\cbicat,J)} \ar[l,"\Gamma'", bend left, start anchor={south west}, end anchor={south east}]
\end{tikzcd},\]
where in particular $\Gamma'$ maps an essential $\Sh(\cbicat,J)$-topos $\Etopos$ to the $J$-stack 
\[
\EssTopos\co/\Sh(\cbicat,J)(\Sh(\cbicat/-, J_{(-)}), \Etopos):\cbicat\op\rightarrow \CAT,
\]
while $\Lambda'$ maps a $J$-stack $\dcat:\cbicat\op\rightarrow\CAT$ to Giraud's topos $\Gir_J(\dcat)\rightarrow \Sh(\cbicat,J)$ for $\dcat$ (see Definition \ref{def:Giraud_topology_classifying_topos}). This adjunction can be restricted to small stacks and relatively small toposes (see Definition \ref{def:relativelysmallmorphism}).

Consider now the (2-)adjunction
\[\begin{tikzcd}
	{\phantom{\slash}\Sh(\cbicat,J)}\ar[r, "{\Sh(\cbicat,J)/-}"', hook, bend right, start anchor={south east}, end anchor={south west}] \ar[r, phantom, "\dashv"{rotate=270}] & {\EssTopos\co/\Sh(\cbicat,J)} \ar[l,"L"', bend right, start anchor={north west}, end anchor={north east}]
\end{tikzcd}:\]
the right adjoint $\Sh(\cbicat,J)/-$ maps a $J$-sheaf $P$ to the local homeomorphism $\Sh(\cbicat,J)\hspace{-0.7pt}/P\hspace{-0.7pt}\rightarrow \Sh(\cbicat,J)$ (see Definition \ref{def:topos_etale}), and it is full and faithful by Lemma \ref{lemma:morfgeom_tra_topos_slice_su_base}; on the other hand, the left adjoint $L$ maps a topos $E:\Etopos\rightarrow \Sh(\cbicat,J)$ to the image $E_!(1_\Etopos)$ of the terminal object via the essential image. The fact that the two form an adjunction follows immediately from Lemma 4.59 of \cite{denseness}.
\begin{remark}
	Alternatively, we could derive this adjunction from the discrete adjunction of Proposition \ref{prop:fundadj_discrete}, by restricting it to $J$-sheaves and essential $\Sh(\cbicat,J)$-toposes.
\end{remark}

By composing the two adjunctions, we can recover $j_J$ and $\trunc_J$:
\begin{prop}
	Consider the adjunction
	\[\begin{tikzcd}
		{\phantom{\slash}\Sh(\cbicat,J)}\ar[r, "{\Gamma'\circ\Sh(\cbicat,J)/-}"', bend right, start anchor={south east}, end anchor={south west}] \ar[r, phantom, "\dashv"{rotate=270}] & {\St^s(\cbicat,J)\phantom{\slash}} \ar[l,"L\circ \Lambda'"', bend right, start anchor={north west}, end anchor={north east}]
	\end{tikzcd}:\]
	then $\Gamma'\circ \Sh(\cbicat,J)/-\cong j_J$ and $L\circ \Lambda'\cong \trunc_J$. In particular, the truncation of a stack $\dcat$ may be defined as either of the two $J$-sheaves 
	\begin{align*}
		\trunc_J(\dcat)&\simeq (C_{p_\dcat})_!(1_{\Gir_J(\dcat)})\\
		&\simeq \colim_{(X,U)\in \gbicat(\dcat)} \ell_{J}(X).
	\end{align*}
\end{prop}
\begin{proof}
	It is sufficient to notice that for a $J$-sheaf $P$ the following chain of equivalences holds:
	\begin{align*}
		(\Gamma'\hspace{-1.2pt}\circ\hspace{-1pt} \Sh(\cbicat,J)/-)(P)&:=\Gamma'(\Sh(\cbicat,J)/P)\\
		&=\EssTopos\co\hspace{-0.75pt}/\Sh(\cbicat,J)(\Gir_J(\cbicat/-), \Sh(\cbicat,J)/P)\\
		&\simeq \EssTopos\co\hspace{-0.75pt}/\Sh(\cbicat,J)(\Sh(\cbicat,J)/\ell_J(-), \Sh(\cbicat,J)/P)\\
		&\simeq \Sh(\cbicat,J)(\ell_J(-),P)\\
		&\simeq P.
	\end{align*}
	The third line is justified since for each $X$ there is a natural equivalence $\Gir_J(\cbicat/-)\simeq \Sh(\cbicat,J)/\ell_J(-)$, by Proposition \ref{prop:fib_discreta_slice_topos}; the fourth line holds by Lemma \ref{lemma:morfgeom_tra_topos_slice_su_base} and the final by Yoneda's lemma. Therefore, the right adjoint $\Gamma'\circ \Sh(\cbicat,J)/-$ acts as the inclusion $j_J:\Sh(\cbicat,J)\hookrightarrow\St(\cbicat,J)$, meaning that the composite $L\circ \Lambda'$ is isomorphic to the truncation functor $\trunc_J$: by spelling out explicitly the composite $L\circ \Lambda'$ we get the first description of $\trunc_J(\dcat)$ in the claim. The second expression is obtained by seeing the terminal $1_{\Gir_J(\dcat)}$ as the colimit of all the representables in $\Gir_J(\dcat)$, and exploiting the commutativity of $(C_{p_\dcat})_!$ with colimits and the equivalence $(C_{p_\dcat})_!\circ \ell_{J_\dcat}\cong \ell_J\circ p_\dcat$ we have that
	\begin{align*}
		\trunc_J(\dcat)&\simeq (C_{p_\dcat})_!\left(\colim_{(X,U)\in \gbicat(\dcat)}\ell_{J_\dcat}(X,U)\right)\\
		&\simeq \colim_{(X,U)\in \gbicat(\dcat)} \left( (C_{p_\dcat})_!\ell_{J_\dcat}(X,U)\right)\\
		&\simeq \colim_{(X,U)\in \gbicat(\dcat)} \ell_{J}(X).
	\end{align*}
\end{proof}
Let us recall from \cite[Propositions 4.62]{denseness} that an essential geometric morphism $f:\Etopos\rightarrow \Ftopos$ always admits a factorization into a terminally connected geometric morphism $f'$ followed by a local homeomorphism (see Definition \ref{def:topos_etale}): in particular, the local homeomorphism is given by the object $f_!(1_\Etopos)$ of $\Ftopos$, so that $f$ is isomorphic to the composite
\[
\Etopos\xrightarrow{f'} \Ftopos/f_!(1_\Etopos) \xrightarrow{\prod_{f_!(1_\Etopos)}} \Ftopos.
\]
Now, consider in particular a $J$-stack $\dcat:\cbicat\op\rightarrow\Cat$: when shall show later in Section \ref{section:comorphisms} that if $\gbicat(\dcat)$ is endowed with Giraud's topology $J_\dcat$, then the fibration  $p_\dcat:\gbicat(\dcat)\rightarrow \cbicat$ becomes a continuous comorphism of sites $p_\dcat:(\gbicat(\dcat), J_\dcat)\rightarrow (\cbicat,J)$ and it induces the canonical essential geometric morphism
\[
\Gir_J(\dcat) \xrightarrow{C_{p_\dcat}} \Sh(\cbicat,J).
\]
from Giraud's topos for $\dcat$. If we apply the (terminally connected, local homeomorphism)-factorization to $C_{p_\dcat}$, we have that the object of $\Sh(\cbicat,J)$ providing the local homeomorphism factor is $(C_{p_\dcat})_!(1_{\Gir_J(\dcat)})$, \ie the $J$-truncation of $\dcat$ by our last result. Thus the truncation of stacks is intimately connected with one of the many factorization systems for geometric morphisms.

This is not all, since the (terminally connected, local homeomorphism)-factorization can actually be presented at the level of sites (cfr. \cite[Proposition 4.70(ii)]{denseness}) using the $J$-comprehensive factorization of functors: this will provide a further way of interpreting the truncation of a stack, directly at the level of fibrations.
\begin{defn}[{\cite[Definition 4.67]{denseness}}] \label{def:comprehensive_factorization}
	Consider a site $(\cbicat,J)$. Every functor $p:\dbicat\rightarrow\cbicat$ admits a
	\emph{$J$-comprehensive (orthogonal) factorization}
	\[
	\begin{tikzcd}
		{\dbicat}\ar[rr,"p"] \ar[dr,"\bar{p}"'] && {\cbicat}\\
		& {\fib p_J} \ar[ur,"\pi"'] &
	\end{tikzcd},\]
	where
	\begin{itemize}
		\item[-] $p_J$ is the $J$-sheaf $\colim (\ell_J\circ p)$ and $\pi$ its associated discrete $J$-stack;
		\item[-] $\bar{p}$ is a $M^{\pi}_J$-cofinal functor (cfr. \cite[Definition 2.23]{denseness}), where $M^{\pi}_J$ is Giraud's topology for $p_J$ (see Definition \ref{def:Giraud_topology_classifying_topos}). 
	\end{itemize}
	This factorization is the unique (up to equivalence) factorization of $p$ into a cofinal functor followed by a discrete $J$-stack. Moreover, if $p$ is a continuous comorphism of sites, at the level of toposes its $J$-comprehensive factorization induces the (terminally connected, local homeomorphism)-factorization of $C_p$.
\end{defn}
Now, consider a $J$-stack $\dcat:\cbicat\op\rightarrow \Cat$: if we consider the functor $p_\dcat:\gbicat(\dcat)\rightarrow\cbicat$, the $J$-sheaf $p_J$ appearing in its $J$-comprehensive factorization is coincides with $(C_{p_\dcat})_!(1_{\Gir_J(\dcat)})$, \ie the $J$-truncation of $\dcat$. Therefore, we may conclude the following:
\begin{cor}\label{cor:truncation_fibrations}
	From a fibrational point of view, the $J$-truncation functor
	\[
	\trunc_J:\St^s(\cbicat,J)\rightarrow \Sh(\cbicat,J)
	\]
	acts by mapping a $J$-stack $p:\dbicat\rightarrow \cbicat$ to the second component $\pi$ in its $J$-comprehensive factorization
	\[
	\begin{tikzcd}
		\dbicat\ar[rr,"p_\dcat"] \ar[dr,"\overline{p_\dcat}"'] && {(\cbicat,J)}\\
		& {(\fib (p_\dcat)_J, M^{\pi}_J)} \ar[ur,"\pi"'] &
	\end{tikzcd}.\]
\end{cor}

\section{Localizations of fibrations}\label{sec:localization_fibrations}
The present section is dedicated to analysing the relationship between fibrations and localizations. We will see in the next section that localizations are a fundamental tool to compute pseudocolimits of categories.

We begin by proving that the category of fibrations is closed under localization with respect to vertical arrows:
\begin{prop}\label{prop:localizationoffibrations}
	Let $p:{\cal D}\to {\cal C}$ be a fibration, $W$ a class of arrows of $\cal D$ and $j_{W}$ the canonical functor ${\cal D}\to {\cal D}[W^{-1}]$. Then the following conditions are equivalent:
	\begin{enumerate}[(i)]
		\item There is a fibration $p_{W}:{\cal D}[W^{-1}] \to {\cal C}$ such that $j_{W}$ is a morphism of fibrations $p \to p_{W}$:
		\[\begin{tikzcd}
			{{\cal D}} & {{\cal D}[W^{-1}]} \\
			{{\cal C}}
			\arrow["p"', from=1-1, to=2-1]
			\arrow["{p_{W}}", from=1-2, to=2-1]
			\arrow["{j_{W}}", from=1-1, to=1-2]
		\end{tikzcd}\]
		
		\item Every arrow in $W$ is vertical with respect to $p$ and, denoting by $p_{W}$ the unique functor (determined by the universal property of the localization) ${\cal D}[W^{-1}] \to {\cal C}$ such that  $p_{W}\circ j_{W}=p$, $j_{W}$ sends arrows which are cartesian with respect to $p$ to arrows which are cartesian with respect to $p_{W}$.   
	\end{enumerate}
\end{prop}

\begin{proof}
	(i) $\imp$ (ii) Let $f$ be an arrow in $W$; then $j_{W}(f)$ is an isomorphism by definition of the localization $j_{W}$, so $p_W(j_{W}(f))$ is also an isomorphism by functoriality; but $p_W(j_{W}(f))=p(f)$, so $p(f)$ is an isomorphism, that is, $f$ is vertical. The fact that $j_{W}$ sends cartesian arrows to cartesian arrow follows from the fact that $j_{W}$ is a morphism of fibrations. 
	
	(ii) $\imp$ (i) Since every arrow in $W$ is vertical with respect to $p$, we have a functor ${\cal D}[W^{-1}] \to {\cal C}$ such that $p_{W}\circ j_{W}=p$. It remains to show that this functor is a fibration. But this follows immediately from the fact that $p$ is a fibration by using the fact that $j_{W}$ sends cartesian arrows to cartesian arrows. Indeed, the functor $j_{W}$ is essentially surjective by the construction of ${\cal D}[W^{-1}]$, and given an arrow $c \to p_{W}(j_{W}(d))=p(d)$ in $\cal C$ and a cartesian lift $g:d'\to d$ of it with respect to $p$, the arrow $j_{W}(g)$ is clearly a cartesian lift of it with respect to $p_{W}$, by the equality $p_{W}\circ j_{W}=p$.  
\end{proof}
In particular, when working with a pseudofunctor $\dcat:\cbicat\op\rightarrow\Cat$ and its fibration $\gbicat(\dcat)$, the request that $\gbicat(\dcat)$ is localized with respect to a family $W$ of vertical arrows can be understood as a localization which already takes place at the level of fibres. In this case, one can then verify that computing the fibration $\gbicat(\dcat)$ and then localizing with respect to $W$ is the same as performing a fibrewise localization of $\dcat$ and then moving to the corresponding fibration:	
\begin{lemma}\label{lemma:pointwiselocalization}
 	Consider a pseudofunctor $\dcat:\cbicat\op\rightarrow \Cat$ and suppose given for each category $X$ in $\cbicat$ a class of arrows $S_X$ of $\dcat(X)$ such that each transition morphism $\dcat(y):\dcat(X)\rightarrow \dcat(Y)$ restricts to a transition morphism $\dcat(X)[S_X\inv]\rightarrow \dcat(X)[S_Y\inv]$. Denote by $\bar{\dcat}:\cbicat\op\rightarrow\Cat$ the pseudofunctor obtained by the pointwise localization of $\dcat$: then $\gbicat(\bar{\dcat})$ is a localization of $\gbicat(\dcat)$ with respect to all arrows $(y,a):(Y,V)\rightarrow (X,U)$ such that $y$ is invertible and $a$ belongs to $S_Y$.
 \end{lemma}
 \begin{proof}
 	We set the notations as in the following diagram (where as usual $y:Y\rightarrow X$ in $\cbicat$):
 	\[\begin{tikzcd}
 		{\dcat(X)} && {\dcat(Y)} \\
 		& {\gbicat(\dcat)} &&& \hbicat \\
 		{\bar{\dcat}(X)} && {\bar{\dcat}(Y)} \\
 		& {\gbicat(\bar{\dcat})}
 		\arrow["{\dcat(y)}"{pos=0.2}, from=1-1, to=1-3]
 		\arrow[""{name=0, anchor=center, inner sep=0}, "{i_X}"', from=1-1, to=2-2]
 		\arrow["{i_Y}", from=1-3, to=2-2]
 		\arrow["{q_X}"'{pos=0.8}, from=1-1, to=3-1]
 		\arrow["{q_Y}"{pos=0.8}, from=1-3, to=3-3]
 		\arrow["{\bar{\dcat}(y)}"{pos=0.2}, from=3-1, to=3-3]
 		\arrow[""{name=1, anchor=center, inner sep=0}, "{\bar{\imath}_X}"', from=3-1, to=4-2]
 		\arrow["{\bar{\imath}_Y}", from=3-3, to=4-2]
 		\arrow["{\bar{\imath}_y}"{pos=0.4}, shorten <=7pt, shorten >=13pt, Rightarrow, from=3-3, to=1]
 		\arrow["h", from=2-2, to=2-5,crossing over]
 		\arrow["q"{pos=0.2}, from=2-2, to=4-2, crossing over]
 		\arrow["{\bar{h}}"', bend right, from=4-2, to=2-5]
 		\arrow["{i_y}"{pos=0.4}, shorten <=7pt, shorten >=7pt, Rightarrow, from=1-3, to=0]
 	\end{tikzcd}.\]
 	First of all, we know by Proposition \ref{prop:laxcolim_grothconstr} that $\colim_{lax}(\dcat)\simeq \gbicat(\dcat)$, with $i_X$ and $i_y$ the components of its colimit cocone. Notice that by hypothesis all the squares such as that in the background of the diagram commute up to isomorphism, and thus there exists an essentially unique functor $q$ simply by the universal property of colimits. Our aim is to show that $q$ is in fact the localization of $\gbicat(\dcat)$ with respect to the class of arrows $(y,a):(Y,V)\rightarrow (X,U)$ such that $y$ is invertible and $a\in S_Y$. Notice that it is actually enough to show that $q$ localizes with respect to all vertical arrows  $(1,a):(X,U)\rightarrow (X,U')$ with $a\in S_X$ for some $X$. To show this, consider a functor $h:\gbicat(\dcat)\rightarrow\hbicat$ such that every arrow $(1,a):(X,U)\rightarrow (X,U')$ with $a\in S_X$ is inverted: this means that the composite functor $h i_X$ factors through $\bar{\dcat}(X)$. If $h$  inverts the vertical arrows in $i_X(S_X)$ 
 	for each $X$ in $\cbicat$, we can therefore build a lax cocone under the diagram $\bar{\dcat}$, and thus an essentially unique functor $\bar{h}:\gbicat(\bar{\dcat})\rightarrow\hbicat$ which factors $h$. This entails that $q$ presents $\gbicat(\bar{\dcat})$ as the localization of $\gbicat(\dcat)$ we desired.
 \end{proof}

Localizations are conveniently calculated when the class $W$ of morphisms to be inverted admits a calculus of fractions. The following proposition shows that if $W$ admits a right calculus of fractions then $p_{W}$ is automatically a fibration and $j_{W}$ a morphisms of fibrations from $p$ to $p_{W}$:

\begin{prop}\label{propcalculusoffractions}
	Let $p:{\cal D}\to {\cal C}$ be a fibration and $W$ a class of vertical arrows of $\cal D$ admitting a right calculus of fractions. Then $p_{W}$ is a fibration and $j_{W}$ yields a morphism of fibrations from $p$ to $p_{W}$.	
\end{prop}	

\begin{proof}
	By Proposition \ref{prop:localizationoffibrations}, we only have to show that the canonical functor $j_{W}:{\cal D}\to {\cal D}[W^{-1}]$ sends $p$-cartesian arrows to $p_{W}$-cartesian arrows. 
	
	Let $f:A \to B$ be a $p$-cartesian arrow in $\cal D$. We want to show that $j_{W}(f):j_{W}(A)\to j_{W}(B)$ is $p_{W}$-cartesian. For this, we suppose that $g:p_{W}(j_{W}(C))\to p_{W}(j_{W}(A))$ is an arrow in $\cal C$ and $h:j_{W}(C)\to j_{W}(B)$ is an arrow in ${\cal D}[W^{-1}]$ such that $p(f)\circ g=p_{W}(h)$. We want to show the existence and uniqueness of an arrow $r:j_{W}(C)\to j_{W}(A)$ such that $g=p_{W}(r)$ and $j_{W}(f)\circ r=h$. Let us start with the existence proof.
	
	Let us represent $h$ as $j_{W}(h')\circ {j_{W}(v)}^{-1}$, where $h'$ is an arrow $C'\to B$ in $\cal D$ and $v$ is an arrow $C' \to C$ in $W$. Since $p(f)\circ g=p_{W}(h)$, composing both sides with $p(v)$ we get $p(f)\circ (g\circ p(v))=p(h')$, whence, since $f$ is $p$-cartesian, there is an arrow (in fact, a unique one) $k:C'\to A$ in $\cal D$ such that $f\circ k=h'$ and $g\circ p(v)=p(k)$: 

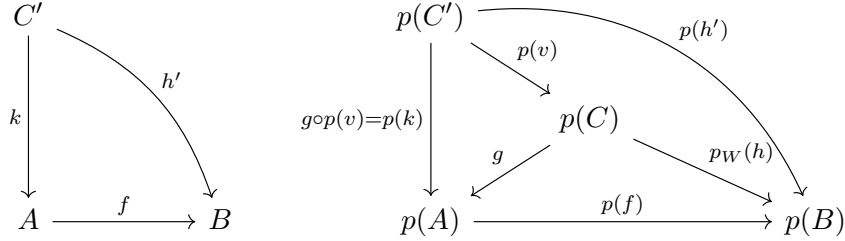
	\[\begin{tikzcd}
		{C'} &&&& {p(C')} \\
		&&&&& {p(C)} \\
		A && B && {p(A)} &&& {p(B)}
		\arrow["{p(f)}", from=3-5, to=3-8]
		\arrow["g"', from=2-6, to=3-5]
		\arrow["{p_{W}(h)}", from=2-6, to=3-8]
		\arrow["{p(v)}", from=1-5, to=2-6]
		\arrow["{g\circ p(v)=p(k)}"', from=1-5, to=3-5]
		\arrow["{p(h')}", bend left=36, from=1-5, to=3-8]
		\arrow["k"', from=1-1, to=3-1]
		\arrow["f", from=3-1, to=3-3]
		\arrow["{h'}", bend left=24, from=1-1, to=3-3]
	\end{tikzcd}\]
	
	Therefore the arrow $j_{W}(k)\circ (j_{W}(v))^{-1}$ satisfies the desired property.
	
	It now remains to prove uniqueness. We shall do so by showing that any arrow $r:j_{W}(C)\to j_{W}(A)$ such that $g=p_{W}(r)$ and $j_{W}(f)\circ r=h$ is necessarily equal to $j_{W}(k)\circ (j_{W}(v))^{-1}$ in ${\cal D}[W^{-1}]$. Let us represent $r$ as $j_{W}(r')\circ {j_{W}(v')}^{-1}$, where $r'$ is an arrow $C''\to A$ in $\cal D$ and $v$ is an arrow $C'' \to C$ in $W$. The equality $j_{W}(f)\circ r=h$ in ${\cal D}[W^{-1}]$ implies, by the construction of the localization at a class admitting a right calculus of fractions, that we can find arrows $u:C''' \to C$ in $\cal D$ and $z:C''' \to C'$ such that $v'\circ u=v\circ z\in W$ and the following diagram commutes:
	
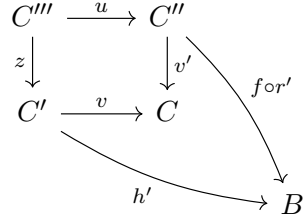
\[\begin{tikzcd}
		{C'''} & {C''} \\
		{C'} & C \\
		&& B
		\arrow["u", from=1-1, to=1-2]
		\arrow["{v'}", from=1-2, to=2-2]
		\arrow["v", from=2-1, to=2-2]
		\arrow["z"', from=1-1, to=2-1]
		\arrow["{h'}"', bend right=12, from=2-1, to=3-3]
		\arrow["{f\circ r'}", bend left=12, from=1-2, to=3-3]
	\end{tikzcd}\]
	
	That is, $f\circ r'\circ u=h'\circ z$. Now, consider the arrows $r'\circ u$ and $k \circ z$. We have	
	\begin{align*}
		p(r'\circ u) & = p_{W}(j_{W}(r')\circ j_{W}(u))=p_{W}(r\circ j_{W}(v') \circ j_{W}(u))\\
		&=p_{W}(r\circ j_{W}(v)\circ j_{W}(z)) =p_{W}(r)\circ p_{W}(j_{W}(v))\circ p_{W}(j_{W}(z))  \\
		& =g\circ p(v)\circ p(z)=p(k)\circ p(z)=p(k\circ z).
	\end{align*}
	Also, as remarked above, $f\circ (r'\circ u)=h'\circ z=f\circ k\circ z$. Therefore, as $f$ is $p$-cartesian, we can conclude that $r'\circ u=k \circ z$. This in turn implies that $r=j_{W}(k)\circ (j_{W}(v))^{-1}$ in ${\cal D}[W^{-1}]$, since the following diagram commutes:
	\[\begin{tikzcd}
		{C'''} & {C''} \\
		{C'} & C \\
		&& A
		\arrow["u", from=1-1, to=1-2]
		\arrow["{v'}", from=1-2, to=2-2]
		\arrow["v", from=2-1, to=2-2]
		\arrow["z"', from=1-1, to=2-1]
		\arrow["k"', bend right=12, from=2-1, to=3-3]
		\arrow["{r'}", bend left=12, from=1-2, to=3-3]
	\end{tikzcd}\]
\end{proof}
\begin{remark}\label{remarklocalizationoffibration}
	Notice that the family $W=\bigcup_{X\in \cbicat}S_X$ in $\gbicat(\dcat)$ satisfies condition (ii) in Proposition \ref{prop:localizationoffibrations}.
\end{remark}

The following corollary shows that, under some natural conditions on a pseudofunctor ${\mathbb D}$ and on the class $W$, the hypotheses of Proposition \ref{propcalculusoffractions} are satisfied. We shall be able to apply this result in our cases of interest.  

\begin{cor}\label{cor:rightcalculusfractions}
	Let ${\mathbb D}:{\cal C}^{\textup{op}}\to \Cat$ be a pseudofunctor satisfying the equivalent conditions of Lemma \ref{lemma:pb_preserving_psft} and such that each ${\mathbb D}(c)$ has (weak) equalizers, $p$ the associated fibration ${\cal G}({\mathbb D}) \to {\cal C}$ and $W$ a class of vertical arrows in ${\cal G}({\mathbb D})$ which contains all the identities and is stable under composition and pullback along arbitrary arrows of ${\cal G}({\mathbb D})$. Then $W$ admits a right calculus of fractions, whence $p_{W}$ is a fibration and $j_{W}$ yields a morphism of fibrations from $p$ to $p_{W}$. 
	
	In particular, this condition is satisfied if $W$ is the collection of \emph{all} the vertical arrows.  
\end{cor}

\begin{proof}
	The right Ore condition follows from the fact that pullbacks of arrows in $W$ along arbitrary arrows of ${\cal G}({\mathbb D})$ exist and yield arrows in $W$, while the last condition in the definition of right calculus of fractions holds since all the fibres of $\mathbb D$ have equalizers. Indeed, suppose that $(f, \alpha)$ and $(g, \beta)$ are two arrows $(c, x)\to (c', x')$ in ${\cal G}({\mathbb D})$ such that $v\circ (f, \alpha)=v\circ (g, \beta)$ where $v$ is a vertical arrow with domain $(c', x')$; since $v$ is vertical, the above equality entails the equality $f=g$. So we have two arrows $\alpha, \beta:x \to {\mathbb D}(f)(x')$ in the category ${\mathbb D}(c)$. By taking $e:z \mono x$ to be the equalizer of (or simply an arrows which equalizes) these two arrows in ${\mathbb D}(c)$, we have a vertical arrow $(1, e):(c, z)\to (c, x)$ in ${\cal G}({\mathbb D})$ such that $(f, \alpha)\circ (1, e)=(g, \beta)\circ (1, e)$, as required.    
	
	Concerning the last statement of the corollary, given a pseudofunctor $\mathbb D$ satisfying the conditions of Lemma \ref{lemma:pb_preserving_psft}, the collection of its vertical arrows clearly contains all the identities and is closed under composition; moreover, the lemma ensures that pullbacks of vertical arrows are vertical, so the collection of vertical arrows satisfies the hypotheses of the corollary. 
\end{proof}

\begin{cor}\label{cor:localizationdiscrete}
	Let $F:{\cal C}\to {\cal D}$ be a cartesian functor and $p:{\cal P}\to {\cal C}$ a discrete fibration. Then the collection of vertical arrows of $(1_{\cal D}\downarrow (F\circ p))\rightarrow \dbicat$ admits a right calculus of fractions. 
\end{cor}
\begin{proof}
	Since $\cal C$ is cartesian and $p$ is discrete, ${\cal P}$ has all non-empty limits and $p$ preserves them (by Corollary \ref{cor:colimits_in_discretefib_from_base}): in particular, it has pullbacks. Pullbacks in $\pbicat$ can be used, together with pullbacks in $\dbicat$, to compute pullbacks in $\comma{1_\dbicat}{F\circ p}$; moreover, in the fibration $\comma{1_\dbicat}{F\circ p}$ pullback of horizontal (resp. vertical) arrows are horizontal (resp, vertical), and thus it satisfies the equivalent conditions of Lemma \ref{lemma:pb_preserving_psft}. In a similar way, one can compute equalizers in $\comma{1_\dbicat}{F\circ p}$, and thus $(1_{\cal D}\downarrow (F\circ p))\to {\cal D}$ satisfies the hypotheses of Corollary \ref{cor:rightcalculusfractions}. 
\end{proof}
\begin{remark}
	Corollary \ref{cor:localizationdiscrete} can be notably applied to obtain a very concrete fibrational description of the inverse image of a sheaf along a cartesian morphism of sites.  
\end{remark}

\section{Weak colimits of categories}\label{sec:colimiti}

In the following we will meet many instances of bicategorical colimits, both when working with base change for stacks and in the context of the fundamental adjunction. Therefore, we shall devote the present section to some technical results about the theory of bicolimits, especially in $\Cat$. For a thorough tractation of these concept we refer to \cite[Chapter 5]{2dimcategories}; some of the following results may already have appeared in the literature, and we have cited them whenever possible: nonetheless a sketch of the explicit proofs is always provided, both for the sake of self-containment and as a warm-up for the reader. 

\subsection{Review of lax, pseudo and weighted colimits}

Let us start by recalling how colimits are defined in the 2-categorical setting:
\begin{defn} 
	Consider two weak 2-categories $\cbicat$ and $\kbicat$, a pseudofunctor $\dcat:\cbicat\op\rightarrow\CAT$ and a pseudofunctor $R:\cbicat\rightarrow \kbicat$: the \emph{$\dcat$-weighted lax colimit of $R$}\index{ lax/oplax/pseudo-colimit} is an object $L$ of $\kbicat$ such that there a pseudonatural equivalence
	\[\kbicat(L,K)\simeq [\cbicat\op,\CAT]_{lax}(\dcat, \kbicat(R(-),K)). \]
	Similarly, the \emph{$\dcat$-weighted oplax colimit} $L$ will satisfy the condition
	\[\kbicat(L,K)\simeq [\cbicat\op,\CAT]_{oplax}(\dcat, \kbicat(R(-),K)). \]
	If we further restrict to pseudonatural transformations, we obtain the notion of \emph{$\dcat$-weighted pseudocolimit}:
	\[\kbicat(L,K)\simeq [\cbicat\op,\CAT]_{ps}(\dcat, \kbicat(R(-),K)). \]
	We will denote these colimits respectively by $\colim_{lax}^\dcat R$, $\colim_{oplax}^\dcat R$ and $\colim_{ps}^\dcat R$\index{$\colim^\dcat_{\bullet}R$}.
	
	Any of these colimits is said to be if the weight is the constant pseudofunctor $\Delta\onecat:\cbicat\op\rightarrow\CAT$ with walue the terminal category $\onecat$: in this case we will omit mentioning the weight altogether. We will use the notations $\colim_{lax}R$, $\colim_{oplax}R$ and $\colim_{ps}R$.
\end{defn}
\begin{remark}
	In 2-categorical literature what we have just defined is usually called \textit{bicolimit}, while a \textit{colimit} is an object $L$ producing a natural \emph{isomorphism} of the hom-categories above; since we will not have to draw the distinction between the two concepts anywhere in the following, we have dropped the \textit{bi-} prefix.
\end{remark}
The lax/oplax/pseudonatural transformations appearing in the definition of colimit are called the \emph{lax/oplax/pseudonatural cocones with vertex $X$ under the diagram $R$}. Let us describe explicitly, for instance, the data of a lax transformation $F:\dcat\Rightarrow \kbicat(R-,K)$ of pseudofunctors from $\cbicat\op$ to $\CAT$. It consists: \begin{enumerate}[(i)]
	\item for every $X$ in $\cbicat$ of a functor $F_X:\dcat(X)\rightarrow \kbicat(R(X),K)$: that is, for every $X$ in $\cbicat$ and every $U$ in $\dcat(X)$ we have a 1-cell $F_X(U):R(X)\rightarrow K$ in $\kbicat$, and for every $a:U\rightarrow V$ in $\dcat(X)$ we have a 2-cell $F_X(a):F_X(U)\Rightarrow F_X(V)$ of $\kbicat$.
	\item for every arrow $y:Y\rightarrow X$ in $\cbicat$ of a natural transformation $F_y$ as in the diagram:
	\[
	\begin{tikzcd}
		{\dcat(X)} \ar[r, "F_X"] \ar[d, "\dcat(y)"'] & \kbicat(R(X),K) \ar[d, "-\circ R(y)"] \ar[dl, Rightarrow, "F_y"']\\
		{\dcat(Y)} \ar[r, "F_Y"'] & \kbicat(R(Y),K)
	\end{tikzcd}
	\]
	satisfying the same axioms of a pseudonatural transformation (see Definition \ref{def:indexed_category}). Therefore the component of $F_y$ at every $U$ of $\dcat(X)$ is a 2-cell $F_y(U):F_X(X)\circ R(y)\Rightarrow F_Y(\dcat(y)(U))$ of $\kbicat$.
\end{enumerate}
We can visualize the cocone $F$ in $\kbicat$ as in the following figure, for $y:Y\rightarrow X$ in $\cbicat$ and $a:U\rightarrow V$ in $\dcat(X)$:
\[
\begin{tikzcd}[row sep=10ex, column sep=10ex]
	R(X) \arrow[d, "F_X(V)"', ""{name=B}, bend right=40] \arrow[d, "F_X(U)", ""{name=A}, ""{name=D, xshift=-1ex}, bend left=40] &    R(Y)  \arrow[l, "R(y)"'] \arrow[ld, "F_Y(\dcat(y)(U))", ""{name=C, below, yshift=3ex, xshift=2ex}, bend left]  \ar[from=D, to=B, Rightarrow, "F_X(a)"'] \\
	K &  	\arrow["{F_y(U)}", Rightarrow, from={1-1}, shorten <=4pt, to=C]
\end{tikzcd}.\]
As we can see, the arrows of $\cbicat$ produce the usual triangles of the cocone, while the arrows in each $\dcat(X)$ produce a `spindle' underneath $R(X)$. The compatibility conditions that said arrows must satisfy, aside from the functoriality of $F_X$, are the following:
\begin{enumerate}[(i)]
	\item naturality of $F_y$: for each $a:U\rightarrow V$ in $\dcat(X)$ and each $y:Y\rightarrow X$,the two diagrams
	\[
	\begin{tikzcd}[row sep=10ex, column sep=10ex]
		R(X) \arrow[d, "F_X(U)"', ""{name=B}, bend right=40] \arrow[d, "F_X(V)", ""{name=A}, ""{name=D, xshift=-1ex}, bend left=40] &    R(Y)  \arrow[l, "R(y)"'] \arrow[ld, "F_Y(\dcat(y)(V))", ""{name=C, below, yshift=3ex, xshift=2ex}, bend left]  \ar[from=D, to=B, Leftarrow, "F_X(a)"'] \\
		K &  	\arrow["{F_y(V)}", Rightarrow, from={1-1}, shorten <=4pt, to=C]
	\end{tikzcd},\  
	\begin{tikzcd}[row sep=10ex, column sep=10ex]
		R(X) \ar[d, "F_X(U)"']& \ar[l, "R(y)"'] R(Y) \ar[dl, bend right, "F_Y(\dcat(y)(U))"{yshift=1ex}, ""{name=B, xshift=-1ex, yshift=-1ex}, ""{name=D, below}] \ar[dl, bend left=50, "F_Y(\dcat(y)(V))", ""{name=C, below, xshift=-1ex, yshift=-1ex}] \\
		K& \ar[from=1-1, to=D, Rightarrow,"F_y(U)"{xshift=-1ex}] \ar[from=B, to=C, "F_Y(\dcat(y)(a))"{ description,xshift=-0.5ex}, shorten >=2pt, Rightarrow]
	\end{tikzcd}
	\]
	coincide in $\kbicat$.
	\item lax transformation axioms: up to canonical 2-isomorphisms, for every $y:Y\rightarrow X$ and $z:Z\rightarrow Y$ in $\cbicat$ and every $U$ in $\dcat(X)$ the two diagrams
	\[
	\begin{tikzcd}
		R(X) \ar[dd, "F_X(U)"'] & R(Y) \ar[l, "R(y)"'] \ar[ddl, "F_Y(\dcat(y)(U))"{xshift=2.5ex, yshift=3ex}, ""{name=A, below}, ""{name=B}] & R(Z) \ar[l, "R(z)"'] \ar[lldd, "F_Z(\dcat(z)(\dcat(y)(U)))", bend left=40, ""{name=C, below, xshift=-1ex, yshift=1ex}]\\
		&&\\
		K && \ar[from=1-1, to=A, Rightarrow, "F_y(U)"{xshift=-0.5ex}] \ar[from=B, to=C, Rightarrow,"F_z(\dcat(y)(U))"{ xshift=-1ex}]
	\end{tikzcd}\] 
	and
	\[ 
	\begin{tikzcd}
		R(X) \ar[d, "F_X(U)"'] & R(Z) \ar[l, "R(yz)"'] \ar[dl, "F_Z(\dcat(yz)(U))",bend left, ""{name=A, below}] \ar[from=1-1, to=A, Rightarrow, "F_{yz}(U)"{xshift=-0.5ex}]\\
		K &
	\end{tikzcd}
	\] 	
	coincide in $\kbicat$. Moreover, up to canonical 2-isomorphisms, the 2-cell $F_{1_X}(U): F_X(U)\circ R(1_X)\Rightarrow F_X(\dcat(1_X)(U))$ coincides with the identity of $F_X(U)$.
\end{enumerate}
If we consider an \emph{oplax cocone}, what changes is the direction of all the natural trasformations of the kind $F_y(U)$; finally, if we consider \emph{pseudonatural transformations} all the $F_y(U)$ are natural isomorphisms. In particular if the weight were $\Delta\onecat:\cbicat\op\rightarrow\CAT$ we would have no `spindle' underneath each of the $R(X)$, only triangles: this explains why said 2-colimits are called conical.

Finally, weights and colimits play a symmetric role in colimits:
\begin{prop}\label{prop:commutativity_peso_diag_colimiti}\index{ lax/oplax/pseudo-colimit! commutativity of diagrams and weights}
	Consider a category $\cbicat$ and two pseudofunctors $\dcat:\cbicat\op\rightarrow\CAT$ and $\ecat:\cbicat\rightarrow \CAT$: there are isomorphisms of categories
	
	\[
	[\cbicat\op,\CAT]_{lax}(\dcat, \CAT(\ecat(-), \kbicat))\simeq [\cbicat,\CAT]_{oplax}(\ecat, \CAT(\dcat(-),\kbicat))
	\]
	\[
	[\cbicat\op,\CAT]_{oplax}(\dcat, \CAT(\ecat(-), \kbicat))\simeq [\cbicat,\CAT]_{lax}(\ecat, \CAT(\dcat(-),\kbicat))
	\]
	\[
	[\cbicat\op,\CAT]_{ps}(\dcat, \CAT(\ecat(-), \kbicat))\simeq [\cbicat,\CAT]_{ps}(\ecat, \CAT(\dcat(-),\kbicat))
	\]
	that are natural in $\kbicat$. This implies in particular the equivalences
	\[
	\colim_{lax}^\dcat \ecat\simeq \colim_{oplax}^\ecat\dcat,\ \colim_{oplax}^\dcat \ecat\simeq \colim_{lax}^\ecat\dcat,\ \colim_{ps}^\dcat \ecat\simeq \colim_{ps}^\ecat\dcat.
	\]
\end{prop}
\begin{proof}
	Let us consider the first isomorphism of categories. An object in the left-hand category is a  lax natural transformation $F:\dcat\Rightarrow \CAT(\ecat(-), \kbicat)$, \ie the given, for every $X$ in $\cbicat$, of a functor $F_X:\dcat(X)\rightarrow \CAT(\ecat(X), \kbicat)$, and for every arrow $y:Y\rightarrow X$ of a natural transformation
	\[
	\begin{tikzcd}
		\dcat(X)\ar[d, "\dcat(y)"'] \ar[r, "F_X"] & {\CAT(\ecat(X), \kbicat)}\ar[d, "-\circ \ecat(y)"] \ar[dl, "F_y", Rightarrow]\\
		\dcat(Y) \ar[r, "F_Y"'] & {\CAT(\ecat(Y), \kbicat)}
	\end{tikzcd}
	\] 
	satisfying suitable compatibility conditions. 
	
	Consider now any object $M$ in $\ecat(X)$: if we set the rule $R_X(M)(-):=F_X(-)(M)$, where the blank space stands either for an object or an arrow of $\dcat(X)$, then this defines a functor $R_X(M):\dcat(X)\rightarrow \kbicat$. One can also check functoriality in $M$ so that we have obtained a functor $R_X:\ecat(X)\rightarrow \CAT(\dcat(X), \kbicat)$: indeed, one can easily check that $R_X$ is a functor is and only if $F_X$ is a functor. We can now consider for $y:Y\rightarrow X$ in $\cbicat$ the components of $F_y$, \ie the natural transformations $F_y(U):F_X(U)\circ \ecat(y)\Rightarrow F_Y(\dcat(y)(U))$: if we fix $M$ in $\ecat(Y)$, we obtain arrows $F_y(U)(M):F_X(U)(\ecat(y)(M))\rightarrow F_Y(\dcat(y)(U))(M)$ in the category $\kbicat$. Using our definition of the functors $R_X$, notice that $F_y(U)(M)$ is an arrow $R_X(\ecat(y)(M))(U)\rightarrow R_Y(M)(\dcat(y)(U))$: setting $R_y(M)(U):=F_y(U)(M)$ we obtain a natural transformation $R_y(M):R_X(\ecat(y)(M))\Rightarrow R_Y(M)\circ \dcat(y)$, and all the natural transformations $R_y(M)$ in turn provide the components on one natural transformation
	\[
	\begin{tikzcd}
		\ecat(X) \ar[r, "R_X"] \ar[dr, Rightarrow, "R_y"]& {\CAT(\dcat(X),\kbicat)} \\
		\ecat(Y) \ar[u, "\ecat(y)"] \ar[r, "R_Y"'] & \CAT(\dcat(Y),\kbicat) \ar[u, "-\circ \dcat(y)"']
	\end{tikzcd}.
	\]
	Indeed, one can check that $R_y$ is a natural transformation if and only if $F_y$ is a natural transformation. Finally, one can also check that $F_y$ satisfies the axioms of a lax natural transformation if and only if $R_y$ satisfies those of an oplax natural transformation. A similar correspondence can be established between modifications, and this proves the first isomorphism of categories. The second and third isomorphism of categories are proved in the exact same fashion.
	
	Finally, the three equivalences between colimits are a straightforward consequence of the former three equivalences of hom-categories.
\end{proof}

\begin{remarks}
    \begin{enumerate}[(i)]
        \item This result appears in the enriched setting as Formula 3.9 of \cite[Section 3.1]{kelly2005}. More precisely, it is shown \emph{loc. cit.} that both $\colim^R\dcat$ and $\colim^\dcat R$ (in their strict- $\Cat$-enriched meaning) can be computed as the \textit{coend} of the functor $R\cdot\dcat:{\cal C}\times {\cal C}\op \to \textbf{Cat}$ acting as $(R\cdot \dcat)(X,Y):=R(X)\times \dcat(Y)$: thus the commutativity of the product in $\Cat$, which allows us to switch $R$ and $\dcat$, is the abstract reason behind the commutativity of weights and diagrams in colimits.
		\item Given a small category $\cbicat$ and a functor $A:\cbicat\rightarrow\Set$, there is a functor \[-\otimes_\cbicat A:[\cbicat\op,\Set]\rightarrow \Set\] acting as left Kan extension of $A$ along $\yo_\cbicat$, \ie the left adjoint to the functor \[R_A(-):=\Set(A(=),-):\Set\rightarrow[\cbicat\op,\Set]\] defined by $R_A(S):=\Set(A(-),S)$. Then for any presheaf $P:\cbicat\op\rightarrow\Set$ there is an isomorphism $$P\otimes_\cbicat A\cong A\otimes_{\cbicat\op}P$$ of sets (\cfr Section 5.1.4 of \cite{caramello.libro}). The commutativity of weights and diagrams is thus a generalization of the commutativity of the tensor of functors: indeed, if we introduce the notation
		\[-\otimes_\cbicat \ecat:=\colim_{ps}^{(-)} \ecat:[\cbicat\op,\CAT]_{ps}\rightarrow \CAT\] for the left adjoint of the 2-functor \[\CAT(\ecat(=),-):\CAT\rightarrow [\cbicat\op,\Set]_{ps},\] we have precisely that $$\dcat\otimes_\cbicat \ecat\simeq \ecat\otimes_{\cbicat\op}\dcat.$$
	\end{enumerate}
\end{remarks}

\subsection{The conification of op-/lax colimits}

Let us now analyse how weighted colimits can be `conified'\index{ lax/oplax/pseudo-colimit!conification}, \ie they can be interpreted as conical colimits: we will see in later sections that this provides an effective method of computing colimits in $\Cat$, for there is an easy way of computing lax conical colimits via the Grothendieck construction.

The idea behind conification is rather simple. Take a $\dcat$-weighted lax cocone under a diagram $R:\cbicat\rightarrow\Cat$: then we can open up the spindles under each node $R(X)$ into triangles, to obtain a conical cocone,
\[
\begin{tikzcd}[row sep=10ex, column sep=10ex]
	R(X) \arrow[d, "F_X(V)"', ""{name=B}, bend right=40] \arrow[d, "F_X(U)", ""{name=A}, ""{name=D, xshift=-1ex}, bend left=40] &    R(Y)  \arrow[l, "R(y)"'] \arrow[ld, "F_Y(\dcat(y)(U))", ""{name=C, below, yshift=3ex, xshift=2ex}, bend left]  \ar[from=D, to=B, Rightarrow, "F_X(a)"'] \\
	K &  	\arrow["{F_y(U)}", Rightarrow, from={1-1}, shorten <=4pt, to=C]
\end{tikzcd}\hspace{-10.5pt}\rightsquigarrow\hspace{-10.5pt}
\begin{tikzcd}[row sep=10ex, column sep=10ex]
	R(X) \ar[dr, "F_X(V)"', ""{name=B}]\ar[r, equal] &	R(X) \arrow[d, "F_X(U)", ""{name=A}, ""{name=D, xshift=-1ex, yshift=2pt}] &    R(Y)  \arrow[l, "R(y)"'] \arrow[ld, "F_Y(\dcat(y)(U))", ""{name=C, below, yshift=3ex, xshift=2ex}, bend left]  \ar[from=D, to=B, Rightarrow, "F_X(a)"{xshift=1ex, above}] \\
	&	K &  	\arrow["{F_y(U)}", Rightarrow, from={1-2}, shorten <=4pt, to=C]
\end{tikzcd}.
\]
Notice that our `conified' cocone is no longer under the diagram $R$, since the nodes are indexed by the pairs $(X,U)$ where $U$ belongs to $\dcat(X)$: it will instead be a cocone under a diagram over the category $\gbicat(\dcat)$. By making this process precise, we end up with the following two results:
\begin{prop}\label{prop:eqv_categorie_laxoplaxps}
	Consider two $\cbicat$-indexed categories $\dcat$, $\ecat:\cbicat\op\rightarrow\CAT$; in particular, by $\ecat\Vop$ we shall denote the $\cbicat$-indexed category operating as $X\mapsto \ecat(X)\op$ (see Definition \ref{def:Vop}), and by $p_\dcat:\gbicat(\dcat)\rightarrow\cbicat$ the Grothendieck fibration associated to $\dcat$. Then
	\[[\cbicat\op,\CAT]_{oplax}(\dcat, \ecat)\cong [\gbicat(\dcat)\op,\CAT]_{oplax}(\Delta\onecat, \ecat\circ p_\dcat\op),\]
	\[[\cbicat\op,\CAT]_{lax}(\dcat,\ecat)\cong [\gbicat(\dcat\Vop)\op,\CAT]_{lax}(\Delta\onecat, \ecat\circ p_{\dcat\Vop}\op).\]
	If moreover $\dcat$ is discrete, \ie it is in fact a presheaf $P:\cbicat\op\rightarrow\Set$, then the equivalences above reduce to
	\[[\cbicat\op,\CAT]_{ps}(P, \ecat)\cong [(\fib P)\op,\CAT]_{ps}(\Delta\onecat, \ecat\circ p_P).\]
\end{prop}
\begin{proof}
	An oplax transformation $F:\dcat\Rightarrow \ecat$ is the given of functors $F_X:\dcat(X)\rightarrow \ecat(X)$ for each $X$ in $\cbicat$, and of natural trasformations $F_y:F_Y\dcat(y)\Rightarrow \ecat(y)F_X$ for each $y:Y\rightarrow X$ satisfying suitable identities. On the other side, an oplax trasformation $\bar{F}:\Delta\onecat\rightarrow \ecat\circ p\op$ consists of an object $\bar{F}_{(X,U)}\in \ecat(X)$ for each $(X,U)$ in $\gbicat(\dcat)$, and of an arrow $\bar{F}_{(y,a)}:\bar{F}_{(Y,V)}\rightarrow \ecat(y)(\bar{F}_{(X,U)})$ for each $(y,a):(Y,V)\rightarrow (X,U)$, again satisfying suitable identities. Starting from $F$, we can define $\bar{F}$ as follows:
	\[\bar{F}_{(X,U)}:=F_X(U),\ \bar{F}_{(y,a)}:=F_y(U) F_Y(a)\mbox{ for }(y,a):(Y,V)\rightarrow (X,U) \]
	Conversely, starting from $\bar{F}$ we can define $F$ by setting
	\[ F_X(U):=\bar{F}_{(X,U)},\ F_X(a):= \phi^\ecat_X(\bar{F}_{(X,U')})\inv \bar{F}_{(1_X, \phi^\dcat_X(U')a)}\mbox{ for }a:U\rightarrow U' \]
	\[F_y(U):=\bar{F}_{(y, 1_{\dcat(y)(U)})} \] 
	We leave to the reader to check that the associations $F\mapsto \bar{F}$ and $\bar{F}\mapsto F$ can be extended to modifications, and that they provide the isomorphism of categories in the claim.
	
	The second identity follows from the first one by applying the isomorphism
	\begin{equation*}\label{eq:identita_Vop}
		[\cbicat\op,\CAT]_{lax}(\dcat, \ecat\Vop)\cong [\cbicat\op,\CAT]_{oplax}(\dcat\Vop,\ecat),
	\end{equation*}
	which can be readily checked, and the equality $(\ecat\circ p_{\dcat\Vop}\op)\Vop=\ecat\Vop\circ p_{\dcat\Vop}\op$.
	
	The last claim is an immediate consequence of the equivalence we defined above. Indeed, if $\dcat$ is a presheaf $P$, the only arrows in its fibres are identity morphisms. Then notice that $\bar{F}_{(y,1_{\dcat(y)(U)})}=F_y(U)$ holds, and hence $F_y$ is invertible if and only if every $\bar{F}_{(y,1_\dcat(y)(U))}$ is: but this means precisely that $F$ is pseudonatural if and only if $\bar{F}$ is. 
\end{proof}
Now, if we consider in particular a pseudofunctor $R:\cbicat\rightarrow\kbicat$, an object $K$ in $\kbicat$ and set $\ecat:=\kbicat(R(-),K):\cbicat\op\rightarrow\Cat$, the previous result has the following corollary:
\begin{cor}\label{cor:colimiti_pesati_conificati}
	Consider a $\cbicat$-indexed category $\dcat$, with $p_\dcat:\gbicat(\dcat)\rightarrow\cbicat$ its corresponding Grothendieck fibration, and a pseudofunctor $R:\cbicat\rightarrow \kbicat$: the $\dcat$-weighted lax colimit of $R$ is isomorphic to the conical colimit of $Rp_{\dcat\Vop}$:
	\[ \colim_{lax}^\dcat R\simeq \colim_{lax}Rp_{\dcat\Vop} \]
	Similarly, the $\dcat$-weighted oplax colimit of $R$ is isomorphic to the conical colimit of $Rp_\dcat$:
	\[ \colim_{oplax}^\dcat R\simeq \colim_{oplax} Rp_\dcat \]
	Finally, for a presheaf $P:\cbicat\op\rightarrow\Set$ with Grothendieck fibration $p:\fib P\rightarrow \cbicat$, the $P$-weighted pseudocolimit of $R$ is isomorphic to the conical colimit of $Rp$:
	\[ 
	\colim_{ps}^P R\simeq \colim_{ps} Rp
	\]
\end{cor}
The commutativity of weights and diagrams expressed by Proposition \ref{prop:commutativity_peso_diag_colimiti} can also be expressed in a conified version as follows:
\begin{cor}\label{cor:commute_weight_diagram_conified}
	Given two pseudofunctors $\dcat:\cbicat\op\rightarrow\CAT$ and $\ecat:\cbicat\rightarrow\CAT$: then 
	\[
	\colim_{lax}\ecat p_{\dcat\Vop}\simeq \colim_{oplax}\dcat p_\ecat.
	\]
\end{cor}

\begin{remark}\label{rmk:pseudocolim_non_conificano}
Proposition \ref{prop:eqv_categorie_laxoplaxps} and Corollary \ref{cor:colimiti_pesati_conificati} show that we can conify a pseudocolimit with a discrete weight, but this is not true for a general pseudocolimit. To understand this, consider two pseudofunctors $\dcat:\cbicat\op\rightarrow \Cat$, $R:\cbicat\rightarrow\Cat$ and a $\dcat$-weighted cocone under $R$ as the one below on the left. When we open the spindles of the pseudocolimit cocone,
    \[
    \begin{tikzcd}[row sep=10ex, column sep=10ex]
    	R(X) \arrow[d, "F_X(V)"', ""{name=B}, bend right=40] \arrow[d, "F_X(U)", ""{name=A}, ""{name=D, xshift=-1ex}, bend left=40] &    R(Y)  \arrow[l, "R(y)"'] \arrow[ld, "F_Y(\dcat(y)(U))", ""{name=C, below, yshift=3ex, xshift=2ex}, bend left]  \ar[from=D, to=B, Rightarrow, "F_X(a)"'] \\
    	K &  	\arrow["{F_y(U)}"{yshift=0.5ex, xshift=0.2ex},"\sim" sloped, Rightarrow, from={1-1}, shorten <=4pt, to=C]
    \end{tikzcd}\hspace{-10.5pt}\rightsquigarrow\hspace{-10.5pt}
    \begin{tikzcd}[row sep=10ex, column sep=10ex]
    	R(X) \ar[dr, "F_X(V)"', ""{name=B}]\ar[r, equal] &	R(X) \arrow[d, "F_X(U)", ""{name=A}, ""{name=D, xshift=-1ex, yshift=2pt}] &    R(Y)  \arrow[l, "R(y)"'] \arrow[ld, "F_Y(\dcat(y)(U))", ""{name=C, below, yshift=3ex, xshift=2ex}, bend left]  \ar[from=D, to=B, Rightarrow, "F_X(a)"{xshift=1ex, above}] \\
    	&	K &  	\arrow["{F_y(U)}"{yshift=0.5ex,xshift=0.2ex},"\sim"sloped, Rightarrow, from={1-2}, shorten <=4pt, to=C]
    \end{tikzcd},
\]
we can see that the natural transformations of the form $F_y(U)$ are invertible, but in general not those of the form $F_X(a)$: therefore, the cone on the right cannot be the cocone of a \textit{pseudo}-colimit. We thus can say that in general none of the colimits
\[
\colim_{ps}^\dcat \ecat,\ \colim_{ps}(\ecat\circ p_\dcat)\,  \colim_{ps}(\ecat\circ p_{\dcat\Vop})
\]
are equivalent. We have provided an explicit example of this using Giraud toposes in Remark \ref{rmk:esempio_colimite_pesato_non_conificabile}. This justifies our interest in \emph{lax} colimits: the commutativity of weights and diagrams in their conified expression allows for some extra elasticity in their expression. We will exploit this when computing weighted pseudocolimits as localizations of lax colimits, in the next section.
\end{remark}

\subsection{Computation of weighted pseudocolimits in $\Cat$}
It is now time to provide an explicit way of computing weighted pseudocolimits in $\Cat$. We begin by recalling that conical op-/lax colimits are easily computed by applying the Grothendieck construction:
\begin{prop}\label{prop:laxcolim_grothconstr}
	Consider a pseudofunctor $\dcat:\cbicat\op\rightarrow \Cat$: then \[\colim_{lax}\dcat\simeq \gbicat(\dcat),\]
	\[\colim_{oplax}\dcat\simeq \gbicat(\dcat\Vop)\op.
	\]
\end{prop}
\begin{proof}
	The first claim is proved in \cite[Section 10.2]{2dimcategories}, while the second is sketched for covariant pseudofunctors in the paragraph \textit{As an oplax colimit} of \cite{nlab:grothendieck_construction}.  We remark that the second claim can also be proved as a consequence of the first, since the following chain of natural equivalences holds: 
	\begin{align*}
		\Cat(\colim_{oplax}\dcat, \kbicat)
		&\simeq [\cbicat,\Cat]_{oplax}(\Delta\onecat, \Cat(\dcat(-),\kbicat))\\
		&\simeq [\cbicat,\Cat]_{oplax}((\Delta\onecat)\Vop, \Cat(\dcat(-),\kbicat))\\
		&\simeq [\cbicat,\Cat]_{lax}(\Delta\onecat, \Cat(\dcat(-),\kbicat)\Vop)\\
		&\simeq [\cbicat,\Cat]_{lax}(\Delta\onecat, \Cat(\dcat\Vop(-),\kbicat\op))\\
		&\simeq \Cat(\colim_{lax}(\dcat\Vop), \kbicat\op)\\
		&\simeq \Cat((\colim_{lax}(\dcat\Vop))\op, \kbicat).
	\end{align*}
\end{proof}
The colimit cocone of $\gbicat(\dcat)$ is made of the following triangles:
\[\begin{tikzcd}
	{\dcat(X)} && {\dcat(Y)} \\
	& {\gbicat(\dcat)}
	\arrow[""{name=0, anchor=center, inner sep=0}, "{i_X}"', from=1-1, to=2-2]
	\arrow["{i_Y}", from=1-3, to=2-2]
	\arrow["{\dcat(y)}", from=1-1, to=1-3]
	\arrow["{i_y}", shorten <=13pt, shorten >=7pt, Rightarrow, from=1-3, to=0]
\end{tikzcd}.\]
Each functor $i_X:\dcat(X)\rightarrow \gbicat(\dcat)$ is the usual inclusion of fibres, which maps an object $U$ to the object $(X,U)$ and acts on arrows accordingly. The natural transformation $i_y$ is defined componentwise, for $U$ in $\
\dcat(X)$, as the arrow 
\[
i_y(U):=(y, 1_{\dcat(y)(U)}): (Y,\dcat(y)(U)) \rightarrow (X,U). 
\]

On the other hand, objects of the category $\gbicat(\dcat\Vop)\op$ are still pairs $(X,U)$ with $X$ in $\cbicat$ and $U$ in $\dcat(X)$, but an arrow $(y,a):(X,U)\rightarrow (Y,V)$ is indexed by an arrow $y:Y\rightarrow X$ and an arrow $a:\dcat(y)(U)\rightarrow V$. The colimit cocone
\[
\begin{tikzcd}
	\dcat(X) \ar[dr, "j_X"'] \ar[rr, "\dcat(y)"]&& \dcat(Y)\ar[dl, "j_Y", ""{name=A, above}]\\
	& \colim_{oplax}\dcat \arrow[from={1-1},to=A,"j_y" near end, shorten <=2ex, Rightarrow]
\end{tikzcd}
\]
is defined as follows: the functor $j_X$ maps an object $U$ in $\dcat(X)$ to $(X,U)$, and acts on arrows accordingly. The natural transformation $j_y:j_X\Rightarrow j_Y\circ \dcat(y)$ is defined componentwise, for some $U$ in $\dcat(X)$, as
\[
j_y(U):=(y,1_{\dcat(y)(U)}):(X,U)\rightarrow(Y, \dcat(y)(U)).
\]

Now that we know how to compute conical op-/lax colimits, we can also compute weighted op-/lax colimits, simply by applying conification (see Corollary \ref{cor:colimiti_pesati_conificati}). Weighted \textit{pseudo}colimits can be computed instead by localizing op-/lax colimits, as the next result shows:
\begin{prop}\label{prop:pseudocolimiti_localizzazioni}
	Consider a category $\cbicat$ and two pseudofunctors $\dcat:\cbicat\op\rightarrow \CAT$ and $R:\cbicat\rightarrow \CAT$. By Corollary \ref{cor:colimiti_pesati_conificati} and Proposition \ref{prop:commutativity_peso_diag_colimiti}, the following two chains of equivalences hold:
	\[
	\colim_{lax} Rp_{\dcat\Vop}\simeq \colim_{lax}^\dcat R\simeq \colim^R_{oplax}\dcat\simeq \colim_{oplax} \dcat p_R,
	\]
	\[
	\colim_{oplax} Rp_{\dcat}\simeq \colim_{oplax}^\dcat R\simeq \colim^R_{lax}\dcat\simeq \colim_{lax} \dcat p_{R\Vop}.
	\]
	Then $\colim_{ps}^\dcat R$ can be presented as a localization of any of the eight categories above, as follows:
	\begin{enumerate}[(i)]
		\item as a localization of $\colim_{oplax}^\dcat R$ or of $\colim_{lax}^\dcat R$: if 
		\[
		\begin{tikzcd}[row sep=10ex, column sep=10ex]
			R(X) \arrow[d, "F_{U'}"', ""{name=B}, bend right=40] \arrow[d, "F_{U}", ""{name=A}, ""{name=D, xshift=-1ex}, bend left=40] &    R(Y)  \arrow[l, "R(y)"'] \arrow[ld, "F_{\dcat(y)(U)}", ""{name=C, below, yshift=3ex, xshift=2ex}, bend left]  \ar[from=D, to=B, Rightarrow, "F_a"'] \\
			\colim_{oplax}^\dcat R &  	\arrow["{F_{y,U}}", Leftarrow, from={1-1}, shorten >=4pt, to=C]
		\end{tikzcd}\]
		is the colimit cocone of $\colim_{oplax}^\dcat R$ (with $y:Y\rightarrow X$ in $\cbicat$ and $a:U\rightarrow U'$ in $\dcat(X)$), the essentially unique functor $\colim_{oplax}^\dcat R\rightarrow\colim_{ps}^\dcat R$ is a localization with respect to the components of each natural trasformation of the kind $F_{y,U}$. A similar statement holds by considering the colimit cocone of $\colim_{lax}^\dcat R$.
		\item as a localization of $\colim_{lax}^R\dcat$ or of $\colim_{oplax}^R\dcat$: if
		\[
		\begin{tikzcd}[row sep=10ex, column sep=10ex]
			\dcat(X) \arrow[d, "G_{A'}"', ""{name=B}, bend right=40] \arrow[d, "G_{R(y)(A)}", ""{name=A}, ""{name=D, xshift=-1ex}, bend left=40] &    \dcat(Y)  \arrow[l, leftarrow, "\dcat(y)"'] \arrow[ld, "G_{A}", ""{name=C, below, yshift=3ex, xshift=2ex}, bend left]  \ar[from=D, to=B, Rightarrow, "G_b"'] \\
			\colim_{lax}^R \dcat &  	\arrow["{G_{y,A}}"',Rightarrow, to={1-1}, shorten <=4pt, from=C]
		\end{tikzcd}\]
		is the colimit cocone of $\colim_{lax}^R \dcat$ (with $y:Y\rightarrow X$ in $\cbicat$, $A$ in $R(Y)$ and $b:R(y)(A)\rightarrow A'$ in $R(X)$), then the essentially unique functor $\colim_{lax}^R\dcat \rightarrow \colim_{ps}^\dcat R$ is a localization with respect to the components of each natural transformation $G_{y,A}$. A similar argument holds for $\colim_{oplax}^R\dcat$.
		\item as a localization of $\colim_{oplax} (Rp_\dcat)$ or of $\colim_{lax} (Rp_{\dcat\Vop})$: if
		\[
		\begin{tikzcd}[row sep=10ex]
			R(X) \ar[d, "H_{(X,U)}"', ""{name=A}] & R(Y) \ar[l, "R(y)"'] \ar[dl, "H_{(Y,V)}"] \ar[to=A, Rightarrow, "H_{(y,a)}"']\\
			\colim_{oplax}( Rp_\dcat)&
		\end{tikzcd}
		\]
		is the colimit cocone of $\colim_{oplax} (Rp_\dcat)$ (with $(y,a):(Y,V)\rightarrow (X,U)$ in $\gbicat(\dcat)$), then the induced functor $\colim_{oplax} (Rp_\dcat)\rightarrow\colim_{ps}^\dcat R$ is a localization with respect to the components of each natural trasformation of the kind $H_{(y,a)}$ such that $(y,a):(Y,V)\rightarrow (X,U)$ is a cartesian arrow of $\gbicat(\dcat)$ (\ie the component $a$ is invertible). The same considerations hold for $\colim_{lax}Rp_{\dcat\Vop}$.
		\item as a localization of $\colim_{lax} (\dcat p_{R\Vop})$ or of $\colim_{oplax} (\dcat p_R)$: if
		\[
		\begin{tikzcd}[row sep=10ex]
			\dcat(X) \ar[r, "\dcat(y)"] \ar[d, "K_{(X,B)}"', ""{name=A}] & \dcat(Y) \ar[dl, "K_{(Y,A)}"] \\
			\colim_{lax}\dcat p_{R\Vop}& \ar[to=A, from=1-2, Rightarrow, "K_{(y,b)}"']
		\end{tikzcd}
		\]
		is the colimit cocone of $\colim_{lax} (\dcat p_{R\Vop})$ (with $(y,b):(X,B)\rightarrow (Y,A)$ in $\gbicat(R\Vop)$), then the induced functor $\colim_{lax} (\dcat p_{R\Vop})\rightarrow \colim_{ps}^\dcat R$ is a localization with respect to the components of all natural transformations of the form $K_{(y,b)}$ where $(y,b)$ is a cartesian arrow of $\gbicat(R\Vop)$ (\ie $b$ is invertible). Similar considerations hold for $\colim_{oplax} (\dcat p_R)$.
	\end{enumerate}
\end{prop}
\begin{proof}
	We will prove the first item, and others will follow from it by expliciting how the colimit cocones relate to one another.
	\begin{enumerate}[(i)]
		\item Similarly to che colimit cocone of $\colim_{oplax}^\dcat R$ we provided above, let us denote by 
		\[
		\begin{tikzcd}[row sep=10ex, column sep=10ex]
			R(X) \arrow[d, "\bar{F}_{U'}"', ""{name=B}, bend right=40] \arrow[d, "\bar{F}_{U}", ""{name=A}, ""{name=D, xshift=-1ex}, bend left=40] &    R(Y)  \arrow[l, "R(y)"'] \arrow[ld, "\bar{F}_{\dcat(y)(U)}", ""{name=C, below, yshift=3ex, xshift=2ex}, bend left]  \ar[from=D, to=B, Rightarrow, "\bar{F}_a"'] \\
			\colim_{ps}^\dcat R &  	\arrow["{\bar{F}_{y,U}}", Leftarrow, from={1-1}, shorten >=4pt, to=C]
		\end{tikzcd}.\]
		the colimit cocone of $\colim_{ps}^\dcat R$. Notice that, being $\colim_{ps}^\dcat R$ a pseudocolimit, the components of each natural transformation $\bar{F}_y(U)$ are invertible. There is a (essentially) unique functor $\xi:\colim_{oplax}^\dcat R\rightarrow\colim_{ps}^\dcat R$: it satisfies in particular the identity $\xi\circ F_y(U)=\bar{F}_y(U)$, and thus it maps the components of every $F_y(U)$ to an invertible map. It is now obvious that any functor $K:\colim_{oplax}^\dcat R\rightarrow \kbicat$ to any category $\kbicat$ factors through $\colim_{ps}^\dcat R$ if and only if $K$ inverts all the components of the natural transformations of the form $F_y(U)$: therefore, $\colim_{ps}^\dcat R$ is obtained as a localization of $\colim_{oplax}^\dcat R$ with respect to all the components of all the natural transformations of the form $F_y(U)$.
		\item We can apply the natural equivalences in Proposition \ref{prop:commutativity_peso_diag_colimiti} in order to translate the colimit cocone of $\colim_{oplax}^\dcat R$ into the colimit cocone of $\colim_{lax}^R \dcat$: in particular, for any $U$ in $\dcat(X)$ and $A$ in $R(Y)$ the identity $G_{y,A}(U)=F_{y,U}(A)$ holds. Thus, since $\colim_{oplax}^\dcat R\rightarrow \colim_{ps}\dcat R$ localizes with respect to all the components of the natural transformations $F_{y,U}$, the functor $\colim_{lax}^R\dcat\rightarrow \colim_{ps}\dcat R$ localizes with respect to all the components of the natural transformations $G_{y,A}$.
		\item We can apply the natural equivalences in Proposition \ref{prop:eqv_categorie_laxoplaxps} in order to translate the colimit cocone of $\colim_{oplax}^\dcat R$ into the colimit cocone of $\colim_{oplax}( Rp_\dcat)$. In particular, the natural transformations $H_{(y,a)}$ are defined componentwise, for $A$ in $R(Y)$, as $H_{(y,a)}(A):= F_{y,U}(A)\circ F_a(A)$. Therefore, since $\colim_{oplax}^\dcat R\rightarrow \colim_{ps}^\dcat R$ localizes with respect to the components of each $F_{y,U}$ it follows that $\colim_{oplax}( Rp_\dcat)\rightarrow \colim_{ps}^\dcat R$ localizes with respect to the components of each $H_{(y,1)}$. Finally, we recall that if $a$ is invertible then $F_a$ is invertible, by functoriality, and thus localizing with respects to the components of the $H_{(y,1)}$ is exactly the same as localizing with respect to the components of the $H_{(y,a)}$ with $(y,a):(Y,V)\rightarrow (X,U)$ cartesian in $\gbicat(\dcat)$.
		\item We can apply the natural equivalences of Proposition \ref{prop:eqv_categorie_laxoplaxps} in order to translate the colimit cocone of $\colim_{lax}^R \dcat$ into the colimit cocone of $\colim_{lax} (\dcat p_{R\Vop})$. In particular each natural transformation $K_{(y,b)}$ is defined componentwise, for $U$ in $\dcat(X)$, as $K_{(y,b)}(U):= G_b(U)\circ G_{y,A}(U)$. With considerations similar to those of the previous item, we can conclude that $\colim_{lax}(\dcat p_{R\Vop})\rightarrow \colim_{ps}^\dcat R$ is a localization with respect to the components of all natural transformations $K_{(y,b)}$ with $(y,b):(X,B)\rightarrow (Y,A)$ cartesian in $\gbicat(R\Vop)$. 
	\end{enumerate}
\end{proof}
\begin{remark}
    Since $\gbicat(\dcat)\simeq \colim_{lax}\dcat$, $\colim_{ps}\dcat$ is the localization of $\gbicat(\dcat)$ with respect to its cartesian arrows: this result first appeared in Paragraph 6.4.0 of \cite[Exposé VI]{SGA4_II}, and was the motivation for our whole analysis of weighted colimits as localizations.
\end{remark}

The previous result allows us to choose, for a colimit $\colim_{ps}^\dcat R$, the representation which we deem more fitting or easier to compute. The two representations in items (i) and (ii) still maintain the distinction between weight and diagram, eventually exchanging them; on the other hand, the conified representations in items (iii) and (iv) are localizations of a `mixed structure' built from the data of both $R$ and $\dcat$, because they see $\colim_{ps}^\dcat R$ as a localization of any of the four colimits 
\[
\colim_{lax}(\dcat p_{R\Vop}),\ \colim_{lax}(R p_{\dcat\Vop}),\ \colim_{oplax}(Rp_\dcat),\ \colim_{oplax}(\dcat p_R).
\]
In the next proposition we shall compute explicitly the first of these colimits and the localization, to show what kind of `mixed structure' is at play: similar considerations hold for the other three cases.
\begin{prop}\label{prop:pscolim_localizz_struttura_mista}
Consider two pseudofunctors $\dcat:\cbicat\op\rightarrow \CAT$ and $R:\cbicat\rightarrow\CAT$: the colimit
\[
\colim_{lax}(\dcat \circ p_{R\Vop})
\]
is equivalent to the category whose objects are triples $(X,U,B)$, with $X$ in $\cbicat$, $U$ in $\dcat(X)$ and $B$ in $R(X)$, while an arrow
\[
(y,a,b):(Y,V,A)\rightarrow (X,U,B)
\]
is indexed by three arrows $y:Y\rightarrow X$ in $\cbicat$, $a:V\rightarrow \dcat(y)(U)$ in $\dcat(Y)$ and $b:R(y)(A)\rightarrow B$ in $R(X)$. 
There is a square
\[\begin{tikzcd}
{\colim_{lax}(\dcat\circ p_{R\Vop})}\ar[d, "p"'] \ar[r, "q"] & \gbicat(\dcat) \ar[d, "p_\dcat"]\\
\gbicat(R\Vop)\op \ar[r, "p_{R\Vop}\op"'] & \cbicat    
\end{tikzcd},\]
where $p$ is a Grothendieck fibration and $q$ maps $p$-cartesian arrows to $p_\dcat$-cartesian arrows. The functor $p$ forgets the second component, the functor $q$ the third and the diagonal both components. In particular, the square is a strict pullback.

The functor
\[
\colim_{lax}(\dcat\circ p_{R\Vop})\rightarrow \colim_{ps}^\dcat R
\]
acts by localizing $\colim_{lax}(\dcat\circ p_{R\Vop})$ with respect to all of the morphisms that are $p$-cartesian and that are mapped through $q$ to $p_\dcat$-cartesian arrows: that is, it localizes with respect to all morphisms $(y,a,b)$ where both $a$ and $b$ are invertible.
\end{prop}
\begin{proof}
The colimit $\colim_{lax}(\dcat\circ p_{R\Vop})$ can be computed, by Proposition \ref{prop:laxcolim_grothconstr}, as the Grothendieck fibration associated to the pseudofunctor
\[
\gbicat(R\Vop)\xrightarrow{p_{R\Vop}}\cbicat\op\xrightarrow{\dcat}\CAT.
\]
The fibration $p_{R\Vop}:\gbicat(R\Vop)\rightarrow\cbicat\op$ is made of objects $(X,B)$ with $B$ in $R(X)$, and arrows $(y,b):(X,B)\rightarrow (Y,A)$ such that $y:Y\rightarrow X$ and $b:R(y)(A)\rightarrow B$. Thus the fibration $p:\gbicat(\dcat\circ p_{R\Vop})\rightarrow\gbicat(R\Vop)\op$ is made of triples $((X,B),U)$, where $(X,B)$ belongs to $\gbicat(R\Vop)$ and $U\in \dcat(p_{R\Vop}(X,B))=\dcat(X)$, and an arrow
\[
((Y,A),V)\xrightarrow{((y,b),a)}((X,B),U)
\]
is indexed by an arrow $(y,a):(X,B)\rightarrow (Y,A)$ in $\gbicat(R\Vop)$ and another arrow $a:V\rightarrow \dcat(p_{R\Vop}(y,b))(U)=\dcat(y)(U)$ in $\dcat(Y)$: it is obvious that this description of $\colim_{lax}(\dcat\circ p_{R\Vop})$ is equivalent to the one provided in the claim. 

The composite functor
\[
\gbicat(\dcat\circ p_{R\Vop})\xrightarrow{p}\gbicat(R\Vop)\op\xrightarrow{p_{R\Vop}\op} \cbicat
\]
acts by forgetting the second and the third components. Notice that it is not a fibration, as it originates from the composition of the fibration $p$ with the opfibration $p_{R\Vop}\op$. This functor factors through $p_\dcat:\gbicat(\dcat)\rightarrow \cbicat$, and the factor
\[
q:\colim_{lax}(\dcat\circ p_{R\Vop})\rightarrow \gbicat(\dcat)
\]
acts by forgetting the third components in $\colim_{lax}(\dcat\circ p_{R\Vop})$. The check that the square above is a strict pullback of categories is straightforward.

Finally, if we explicit the colimit cocone 
\[
		\begin{tikzcd}[row sep=10ex]
			\dcat(X) \ar[r, "\dcat(y)"] \ar[d, "K_{(X,B)}"', ""{name=A}] & \dcat(Y) \ar[dl, "K_{(Y,A)}"] \\
			\colim_{lax}\dcat p_{R\Vop}& \ar[to=A, from=1-2, Rightarrow, "K_{(y,b)}"']
		\end{tikzcd},
\]
we see that the legs act thus, 
\[
K_{(X,B)}:\left[U\xrightarrow{\vphantom{(}a} U' \right] \mapsto \left[ (X,U,B) \xrightarrow{(1,a,1)} (X,U',B)\right],
\]
while for each $U$ in $\dcat(X)$ the component of $K_{(y,b)}$ at $U$ is the arrow
\[
K_{(y,b)}(U):= \left[ (Y,\dcat(y)(U),A) \xrightarrow{(y,1,b)} (X,U,B)\right].
\]
We have omitted, for the sake of readability, any reference to the canonical isomorphisms of the pseudofunctors $R$ and $\dcat$. 
Finally, we recall that the functor
\[
\colim_{lax}(\dcat\circ p_{R\Vop})\rightarrow \colim_{ps}^\dcat R
\]
localizes with respect to the components of the natural transformations $K_{(y,b)}$ such that $(y,b)$ is cartesian, \ie with respect to all the arrows $(y,1,b)$ with $b$ invertible: but this is the same as localizing with respect to the arrows in the statement.
\end{proof}
\begin{remark}\label{rmk:laxcolim_weighted_is_pb_grfibr}
	The pullback square above, which shows the connection between $\colim_{lax}(\dcat\circ p_{R\Vop})$ and $\gbicat(\dcat)$, is a particular instance of the concept of direct image of a $\cbicat$-indexed category along a functor, which we will analyse in Section \ref{sec:dir_imm_fibrations}: for further details, see Remark \ref{rmk:laxcolim_weighted_is_pb_grfibr2}(ii). 
\end{remark}
Every category $\cbicat$ is endowed in particular with the covariant pseudofunctor $\cbicat/-:\cbicat\rightarrow \CAT$, which maps each object $X$ to the slice over it. The previous result has the consequence that weighted pseudocolimits for $\cbicat/-$ can be interpreted as lax colimits of the weight:
\begin{cor}\label{cor:Grothconstruction}
	Consider $R:\cbicat\rightarrow \Cat/\cbicat$ mapping each $X$ to $\cbicat/X$: then the $\dcat$-weighted colimit of $R$ is equivalent to the lax colimit of $\dcat$.
\end{cor}
\begin{proof}
	For the sake of brevity, in this proof we shall denote by $\kbicat$ the category $\colim_{lax}(\dcat\circ p_{R\Vop})$, and also suppress any mention of $\dcat$'s structural isomorphisms.
    Applying the previous result to $R:=\cbicat/-:\cbicat\rightarrow \CAT$ provides the following description of $\kbicat$: objects are triples $(X,U\in\dcat(X),w:W\rightarrow X)$ with $X$,$W$ and $w$ in $\cbicat$ and $U$ in $\dcat(X)$, and an arrow
    \[
    (Y,V,z:Z\rightarrow Y) \xrightarrow{(y,a,b)}(X,U,w:W\rightarrow X) 
    \]
    is indexed by two arrows $y:Y\rightarrow X$, $b:Z\rightarrow W$ in $\cbicat$ such that $w\circ b=y\circ z$, and a further arrow $a:V\rightarrow \dcat(y)(U)$ in $\dcat(Y)$. We can then define a functor
    \[
    L:\kbicat\rightarrow \gbicat(\dcat)
    \]
    acting as follows:
    \[
    \left[(Y,V,[z])\xrightarrow{(y,a,b)}(X,U,[w])\right]\mapsto\left[(Z,\dcat(z)(V))\xrightarrow{(b,\dcat(z)(a))}(W, \dcat(w)(U)) \right].
    \]
    If we denote by $S$ the class of arrows $(y,a,b)$ in $\kbicat$ with both components $a$ and $b$ invertible, then $L$ inverts all arrows in $S$: if we prove that $L$ satisfies the universal property of the localization with respect to $S$, we can conclude that 
\[
\colim_{lax}\dcat \simeq\gbicat(\dcat)\simeq \kbicat[S\inv]\simeq \colim_{ps}^\dcat R,
\]
by the previous result and by Proposition \ref{prop:laxcolim_grothconstr}. To show this consider any functor $H:\kbicat\rightarrow\abicat$ mapping any arrow in $S$ to an invertible arrow in $\abicat$: we want to define an essentially unique functor
\[\bar{H}:\gbicat(\dcat)\rightarrow\abicat\]
such that $\bar{H}\circ L\cong H$. This condition forces the definition of $\bar{H}$, up to natural isomorphism: indeed, if we consider the invertible arrows 
\[
L(X,U,1_X)=(X,\dcat(1_X)(U))\xrightarrow{(1,1)} (X,U)
\]
of $\gbicat(\dcat)$, they must be mapped via $\bar{H}$ to isomorphisms
\[
\bar{H}\circ L(X,U,1_X)=H(X,U,1_X)\xrightarrow{\bar{H}(1,1)}\bar{H}(X,U).
\]
Therefore we can define $\bar{H}$ as follows: for any $(y,a):(Y,V)\rightarrow (X,U)$ in $\gbicat(\dcat)$, we set its image via $\bar{H}$ as the arrow
\[
H(Y,V, 1_Y)\xrightarrow{H(y,a,y)} H(X,U,1_X).
\]
Finally, we prove that $\bar{H}\circ L\cong H$. For this, notice that every object $(X,U,w:W\rightarrow X)$ in $\kbicat$ admits a canonical morphism
\[
(W, \dcat(w)(U),1_W) \xrightarrow{(w,1,1)} (X,U, w),
\]
which belongs to $S$: its image via $H$ is therefore an isomorphism
\[
H(W, \dcat(w)(U),1_W)\hspace{-0.5pt}=\hspace{-0.5pt}\bar{H}(W,\dcat(w)(U))=\bar{H}\circ L(X,U,w) \xrightarrow{(w,1,1)} H(X,U, w),
\]
and a quick naturality check proves that this is the component in $(X,U,w)$ of a natural isomorphism $\bar{H}\circ L\cong H$. 
\end{proof}
\begin{remarks}
\begin{enumerate}[(i)]
    \item This result can be seen as a consequence of Corollary \ref{cor:colimite_in_cat_su_C}, which shows that $\colim_{ps}^\dcat (\cbicat/-)\simeq \gbicat(\dcat)$ by exploiting the right adjoint of the 2-functor $\gbicat$. We will provide a further point of view on this equivalence in Remark \ref{rmk:Grothconstruction}, by using inverse images of fibrations.
    \item Curiously, the localization $L:\colim_{lax}(\dcat\circ p_{R\Vop})\rightarrow \colim_{ps}^\dcat R\simeq \gbicat(\dcat)$ acts by `restriction' of the diagram, and this happens because every fibre $R(X)=\cbicat/X$ of $R$ has a terminal object $[1_X]$.
    First of all, notice that
        \[
    \gbicat(R\Vop)\xrightarrow{p_{R\Vop}}\cbicat\op
    \]
    is the opposite of the codomain functor $\cod:\Mor(\cbicat)\rightarrow \cbicat$: objects in $\gbicat(R\Vop)$ are arrows $[w:W\rightarrow X]$, and a morphism $(y,b):[w:W\rightarrow X]\rightarrow [z:Z\rightarrow Y]$ is given by two arrows $y:Y\rightarrow X$ and $b:Z\rightarrow W$ such that $w\circ b=y\circ z$. The fact that each fibre of $R$ has a terminal implies that we can define a functor
    \[
    T:\cbicat\rightarrow \gbicat(R\Vop)\op=\Mor(\cbicat)
    \]
    by mapping every object $X$ to $[1_X]$ and every $y:Y\rightarrow X$ to $(y,y):[1_Y]\rightarrow [1_X]$. This functor is (a right adjoint and) a section of $\cod$, \ie $\cod \circ T= \id_\cbicat$: therefore, since $p_{R\Vop}\cong \cod\op$, we can conclude that
    \[
    \colim_{lax}\dcat \simeq \colim_{lax}(\dcat\circ p_{R\Vop}\circ T\op)=\colim_{lax}\left( (\dcat\circ p_{R\Vop})\circ T\op\right).
    \]
    We can really think of $\gbicat(\dcat)$ as obtained by restricting the cocone of $\colim_{lax} (\dcat\circ p_{R\Vop})$ along the functor $T$. As we have formulated it, it is evident that this happens everytime the covariant pseudofunctor $R$ we consider (not just $R=\cbicat/-$) has fibres with terminal objects.
\end{enumerate}
\end{remarks}    

To conclude this section, let us go back to conified pseudocolimits: we already mentioned in Remark \ref{rmk:pseudocolim_non_conificano} that one in general is not allowed to `conify' pseudocolimits, \ie suppose that $\colim^\dcat_{ps}R\simeq \colim_{ps}(R\circ p_\dcat)$ (it may still happen in some cases, such as when the weight is discrete: see the last claim of Corollary \ref{cor:colimiti_pesati_conificati}). One may however wonder how to express the conified pseudocolimit $\colim_{ps}(R\circ p_\dcat)$. Proposition \ref{prop:pseudocolimiti_localizzazioni} showed that a weighted pseudocolimit $\colim_{ps}^\dcat R$ can be recovered by localizing suitable conical op-/lax colimits stemming from the data $R$ and $\dcat$: for instance, one may localize $\colim_{oplax} (R\circ p_\dcat)$ with respect to the components of all the natural transformations $H_{(y,a)}$ of its colimit cocone such that $(y,a)$ is cartesian in $\gbicat(\dcat)$. If instead we localize $\colim_{oplax} (R\circ p_\dcat)$ with respect to the components of \emph{all} the natural transformations $H_{(y,a)}$, we obtain $\colim_{ps}(R\circ p_\dcat)$ instead.

Notice that the localization $\colim_{oplax}(R\circ p_\dcat)\twoheadrightarrow\colim_{ps}(R\circ p_\dcat)$ essentially crushes all the information coming from the fibres of $\dcat$: this allows us to express $\colim_{ps}(R\circ p_\dcat)$ in a further way, which combines commutation of weights and colimits with the fibrewise groupoidal localization of $\dcat$, as the next result shows.
\begin{prop}\label{prop:pseudocolim_conico_conncmpts}
Consider two pseudofunctors $\dcat:\cbicat\op\rightarrow\Cat$ and $R:\cbicat\rightarrow \Cat$. Denote by $\overline{\dcat}:\cbicat\op\rightarrow\Cat$ the pseudofunctor mapping each $X$ in $\cbicat$ to the groupoid obtained by inverting all arrows in $\dcat(X)$, with transition morphisms obtained by restriction of those of $\dcat$. Consider the colimit cocone
\[
		\begin{tikzcd}[row sep=10ex]
			\overline{\dcat}(X) \ar[r, "\overline{\dcat}(y)"] \ar[d, "\overline{K}_{(X,B)}"', ""{name=A}] & \overline{\dcat}(Y) \ar[dl, "\overline{K}_{(Y,A)}"] \\
			{\colim_{lax}(\overline{\dcat}\circ p_{R\Vop})}& \ar[to=A, from=1-2, Rightarrow, "\overline{K}_{(y,b)}"']
		\end{tikzcd}
		\]
where $(y,b):(X,U)\rightarrow (Y,A)$ is an arrow of $\gbicat(R\Vop)$: then the pseudocolimit $\colim_{ps}(R\circ p_\dcat)$ can be computed as the localization of $\colim_{lax}(\overline{\dcat}\circ p_{R\Vop})$ with respect to the components of all the natural transformations $K_{(y,b)}$ such that $(y,b)$ is cartesian (\ie $b$ is invertible).
\end{prop}
\begin{proof}
The passage from $\dcat$ to $\overline{\dcat}$ corresponds to a pointwise localization with respect to all the arrows of the fibres: therefore, by Lemma \ref{lemma:pointwiselocalization} the functor $\gbicat(\dcat)\simeq\colim_{lax}(\dcat\circ p_{R\Vop})\rightarrow \colim_{lax}(\overline{\dcat}\circ p_{R\Vop})$ is a localization with respect to all the vertical arrows in $\gbicat(\dcat\circ p_{R\Vop})$. Now it is sufficient to recall how we performed the passage of colimit cocones
\[
		\begin{tikzcd}[row sep=10ex]
			R(X) \ar[d, "H_{(X,U)}"', ""{name=A}] & R(Y) \ar[l, "R(y)"'] \ar[dl, "H_{(Y,V)}"] \ar[to=A, Rightarrow, "H_{(y,a)}"']\\
			\colim_{oplax}( R\circ p_\dcat)&
		\end{tikzcd}\rightsquigarrow
		\begin{tikzcd}[row sep=10ex]
			\dcat(X) \ar[r, "\dcat(y)"] \ar[d, "K_{(X,B)}"', ""{name=A}] & \dcat(Y) \ar[dl, "K_{(Y,A)}"] \\
			\colim_{lax}(\dcat\circ p_{R\Vop})& \ar[to=A, from=1-2, Rightarrow, "K_{(y,b)}"']
		\end{tikzcd},
\]
		where $(y,a):(Y,V)\rightarrow (X,U)$ in $\gbicat(\dcat)$ and $(y,b):(X,B)\rightarrow (Y,A)$ in $\gbicat(R\Vop)$. We defined each functor $K_{(X,A)}$ by mapping an arrow $a:U\rightarrow U'$ in $\dcat(X)$ to the arrow $H_{(1,a)}(A):H_{(X,U)}(A)\rightarrow H_{(X,U')}(A)$. On the other hand, for each $U$ in $\dcat(X)$ we defined the $K_{(y,b)}(U)$ as the composite
		\[
		H_{(Y,\dcat(y)(U))}(A)\xrightarrow{H_{(y,1)}(A)}H_{(X,U)}(R(y)(A))\xrightarrow{H_{(X,U)}(b)}H_{(X,U)}(B). 
		\]
		Therefore, if the functor $\colim_{oplax}(R\circ p_\dcat)\rightarrow \colim_{ps}(R\circ p_\dcat)$ inverts the components of all the natural transformations $H_{(y,a)}$, then the composite functor
		\[\colim_{lax}(\dcat\circ p_{R\Vop})\isorightarrow\colim_{oplax}(R\circ p_\dcat)\rightarrow \colim_{ps}(R\circ p_\dcat)
		\]
    must invert not only the components of each natural transformation $K_{(y,b)}$ with $b$ invertible, but also all the morphisms in the image of each functor of the form $K_{(X,B)}$, which are precisely the vertical arrows of $\colim_{lax}(\dcat\circ p_{R\Vop})$. Therefore, by the universal property of localizations we obtain a functor $\colim_{lax}(\overline{\dcat}\circ p_{R\Vop})\rightarrow \colim_{ps}({\dcat}\circ p_{R\Vop})$, which localizes with respect to the components of the natural transformations $\overline{K}_(y,b)$ indexed by cartesian arrows of $\gbicat(R\Vop)$.
\end{proof}

\section{Comorphisms of sites}
\label{section:comorphisms}
We conclude the chapter by recalling some basic aspects in the theory of comorphisms of sites. We have already anticipated that fibrations (and their morphisms) on a site $(\cbicat,J)$ can be naturally seen as continuous comorphisms of sites by exploiting what we will call \emph{Giraud topologies}: this is the topic of this and the following section. This point of view, first introduced by Jean Giraud in \cite{giraud.classifying} but not explored thoroughly in the literature, provides powerful tools to the study of stacks from a geometrical and logical point of view. In particular, one can associate to every fibration over $\cbicat$ its classifying topos, which will have a prominent role in defining the fundamental adjunction of Chapter \ref{chap:fundadj}. 

In order to define comorphisms of sites, we need first to recall that any functor between small categories induces an essential geometric morphism between presheaf toposes:
\begin{prop}[{\cite[Example A4.1.4]{elephant}}]\label{prop:lan_ran}\index{Kan extension}\index{$\lan_{p\op}$}\index{$\ran_{p\op}$}
	A functor $p:\dbicat\rightarrow\cbicat$ between small categories induces an adjoint triple $\lan_{p\op}\dashv p^* \dashv \ran_{p\op}:[\dbicat\op,\Set]\rightarrow[\cbicat\op,\Set]$, where $p^*:=(-\circ p\op)$ and its left and right adjoints are the \emph{left and right Kan extension functors along $p\op$}.
\end{prop}
In particular, this provides a functor from $\Cat$ to $\EssTopos$ mapping a small category $\cbicat$ to $[\cbicat\op,\Set]$ and a functor $p:\dbicat\rightarrow\cbicat$ to the essential geometric morphism $\ran_{p\op}$. If we wish to extend this to natural transformations we have to take into account that the involution $(-)\op$ reverses their direction, \ie if $\alpha:p\Rightarrow q:\dbicat\rightarrow \cbicat$ then $\alpha\op:q\op\Rightarrow p\op$. This means that there is an induced natural trasformation $\alpha^*:q^*\Rightarrow p^*:[\cbicat\op,\Set]\rightarrow[\dbicat\op,\Set]$ acting as precomposition with $\alpha\op$, and thus the functor $\Cat\rightarrow\EssTopos$ extends to a 2-functor $\Cat\rightarrow \EssTopos\co$.

One can now wonder what properties the functor $p$ must satisfy for the geometric morphism $p^*\dashv\ran_{p\op}$ to restrict toposes of sheaves. The answer lies precisely in the notion of comorphism of sites:
\begin{defn}[{\cite[Proposition C2.3.18]{elephant}}]\label{def:comorfismo}
	Consider two sites $(\cbicat, J)$ and $(\dbicat,K)$: a functor $p:\dbicat\rightarrow\cbicat$ is said to be a \emph{comorphism of sites}\index{comorphism} if one of the following equivalent conditions holds:
	\begin{enumerate}[(i)]
		\item The functor $p$ satisfies the \emph{covering-lifting property}, \ie for every $D$ in $\dbicat$ and every $J$-covering sieve $S$ over $p(D)$ there is a $K$-covering sieve $R$ over $D$ such that $p(R)\subseteq S$.
		\item The functor $\ran_{p\op}:[\dbicat\op,\Set]\rightarrow[\cbicat\op,\Set]$ maps $K$-sheaves to $J$-sheaves, \ie there exists a geometric morphism $C_p:\Sh(\dbicat,K)\rightarrow \Sh(\cbicat,J)$ satisfying $\iota_J{C_p}_*\simeq\ran_{p\op}\iota_K$.
		\item $p^*:[\cbicat\op,\Set]\rightarrow[\dbicat\op,\Set]$ maps $J$-dense monomorphisms to $K$-dense monomorphisms (we recall that a monomorphism $m$ is $J$-dense if $\sheafify_J(m)$ is an isomorphism).
	\end{enumerate}
\end{defn}
If we call $\Cosite$\index{$\Com$} the 2-category of small-generated sites, comorphisms and natural trasformations, the association $(\cbicat,J)\mapsto \Sh(\cbicat,J)$ and $p\mapsto C_p$ describes a functor $\Cosite\rightarrow\Topos$. This functor also extends to natural trasformations reversing their direction, and thus we have in fact a 2-functor 

\[C_{(-)}:\Cosite\rightarrow\Topos\co.\]\index{$C_{(-)}$}

A fact which will prove to be essential in the following is that every functor to a site can be made into a comorphism in a minimal way:
\begin{prop}[{\cite[Proposition C2.3.19(i)]{elephant}}]\label{prop:top_minima_comorfismo} Consider a site $(\cbicat,J)$ and a functor $A:\dbicat\rightarrow \cbicat$: there exists a smallest topology $M^A_J$\index{$M^A_J$} on $\dbicat$ such that $A:(\dbicat,M^A_J)\rightarrow (\cbicat,J)$ is a comorphism of sites. In particular, it is generated by the pullback-closed family (or \emph{coverage}) of sieves of the form $S_D:=\{f:\dom(f)\rightarrow D\ |\ p(f)\in S\}$ for any $D$ in $\dbicat$ and $S\in J(p(D))$.
\end{prop}
The topologies of the form $M^A_J$ satisfy another fundamental property:
\begin{lemma}[{\cite[Corollary 3.6]{denseness}}]
	\label{lemma:comorfismi_via_tplgir} Take a comorphism  $p:(\dbicat,K)\rightarrow (\cbicat,J)$ and two functors $A:\dbicat\rightarrow \ebicat$, $q:\ebicat\rightarrow\cbicat$: if $qA\cong p$ then $A$ is a comorphism of sites $(\dbicat,K)\rightarrow (\ebicat, M^q_J)$.
\end{lemma}
Let us denote by $\Com^s$ the category of small sites, with their comorphisms: then the minimal topologies of the form $M^A_J$ can be interpreted as providing a left adjoint, as follows. Denote by
\[\Giraud:\Cat/\cbicat\rightarrow \Cosite^s/(\cbicat,J)\]\index{$\Giraud$}
the 2-functor mapping a 0-cell $[q:\ebicat\rightarrow \cbicat]$ to $q:(\ebicat, M^q_J)\rightarrow (\cbicat,J)$. By Lemma \ref{lemma:comorfismi_via_tplgir}, this is a well-defined 2-functor, since every 1-cell in $\Cat/\cbicat$ is sent to a comorphism of sites. In fact, for \emph{any} comorphism $p:(\dbicat,K)\rightarrow (\cbicat,J)$, a 1-cell $(F,\phi):[p:\dbicat\rightarrow \cbicat]\rightarrow [q:\ebicat\rightarrow\cbicat]$ in $\Cat/\cbicat$ corresponds to a 1-cell $(F,\phi):[p:(\dbicat,K)\rightarrow (\cbicat,J)]\rightarrow [q:(\ebicat,M^q_J)\rightarrow (\cbicat,J)]$ of $\Cosite^s/(\cbicat,J)$. This proves the following:
\begin{cor}\label{cor:girtpl_aggiunto_sx}
	Consider a site $(\cbicat,J)$: then there is a 2-adjunction (see Definition \ref{def:2-adjoint})
	\[
	\begin{tikzcd}
		{\Cosite^s/(\cbicat,J)} \ar[r, bend left, "\For", start anchor={north east}, end anchor={north west}] \ar[r, phantom, "\vdash"{rotate=90}] & {\Cat/\cbicat} \ar[l, bend left, "\Giraud", start anchor={south west}, end anchor={south east}]
	\end{tikzcd}
	\]
	where $\For$ is the usual forgetful functor and $\Giraud$ is defined on 0-cells by mapping $[q:\ebicat\rightarrow \cbicat]$ to $[p:(\ebicat, M^q_J)\rightarrow (\cbicat,J)]$. The 2-functor $\Giraud$ is also locally fully faithful, \ie every functor $\Cat/\cbicat ([p], [q])\rightarrow \Cosite^s/(\cbicat,J)(\Giraud([p]), \Giraud([q]))$ is full and faithful.
\end{cor} 

For the subsequent results we will focus on a smaller class of comorphisms of sites, namely those that are also \emph{continuous} with respect to topologies:
\begin{defn}[{\cite[Definition 4.7]{denseness}}]
	Consider a site $(\cbicat,J)$ and a topos $\Etopos$: a functor $A:\cbicat\rightarrow\Etopos$ is \emph{$J$-continuous} if the induced right adjoint $R_A:\Etopos\rightarrow [\cbicat\op,\Set]$ factors through $\Sh(\cbicat,J)$. 
	More generally, given two sites $(\dbicat,K)$ and $(\cbicat,J)$, a functor $p:\dbicat\rightarrow\cbicat$ is \emph{$(K,J)$-continuous}\index{functor!$(K,J)$-continuous} if the composite functor $\ell_J\circ p:\dbicat\rightarrow \Sh(\cbicat,J)$ is $K$-continuous.
\end{defn}
\begin{prop}[{\cite[Propositions 4.8 and 4.13]{denseness}}]  \label{prop:ft_cont_caratterizzazione} Given two sites $(\cbicat,J)$ and $(\dbicat,K)$ and a functor $p:\dbicat\rightarrow\cbicat$, the following are equivalent:\begin{enumerate}[(i)]
		\item $p$ is $(K,J)$-continuous;
		\item $p^*=(-\circ p\op):[\cbicat\op,\Set]\rightarrow[\dbicat\op,\Set]$ maps $J$-sheaves to $K$-sheaves;
		\item $p$ is \textit{cover-preserving}, \ie if $S$ is a $K$-covering sieve then $p(S)$ is a $J$-covering family, and it satisfies the following \textit{cofinality condition}: for any $K$-covering
		sieve $S$ on an object $D$ and any commutative square 
		\[
		\begin{tikzcd}
			X \ar[d, "f"'] \ar[r, "g"] & p(E') \ar[d, "p(e')"]\\
			p(E) \ar[r, "p(e)"'] & p(D)
		\end{tikzcd}
		\]
		with $e$ and $e'$ belonging to $S$, there exists a $J$-covering family $\{y_i:Y_i\rightarrow X\ |\ i\in I \}$ such that for each $i$ the composites $fy_i$ and $gy_i$ belong to the same connected component of the comma category $\comma{Y_i}{D^p_S}$, where $D^p_S$ denotes the composite functor 
		\[\fib S \xrightarrow{\pi}\dbicat \xrightarrow{p}\cbicat.\]
	\end{enumerate}
\end{prop}
Every $(J,K)$-continuous functor $F:\cbicat\rightarrow \dbicat$ induces an adjunction $\Sh(F)^*\dashv \Sh(F)_*:\Sh(\dbicat,K)\rightarrow \Sh(\cbicat,J)$\index{$\Sh(F)$}: in adherence with \cite{denseness}, we shall call any such adjunction a \emph{weak geometric morphism} of toposes. We recall that a \emph{morphism of sites}\index{morphism of sites} $F:(\cbicat,J)\rightarrow (\dbicat,K)$ is a functor inducing a geometric morphism $\Sh(F):\Sh(\dbicat,K)\rightarrow \Sh(\cbicat,J)$ such that $\Sh(F)_*:=(-\circ F\op)$: therefore, every morphism of sites is in particular $(J,K)$-continuous. 

Notice that if $p$ is a $(K,J)$-continuous comorphism then the induced geometric morphism $C_p$ is \textit{essential}: this happens because the inverse image $C_p^*$ by continuity admits a further left adjoint, denoted by $(C_p)_!$. Continuous comorphisms of sites, whose 2-category we will denote by $\Com\cont$\index{$\Com\cont$}, are indeed a key ingredient in the theory of essential geometric morphisms, as the following result shows: 
\begin{prop}[\protect{\cite[Theorem 4.20]{denseness}}] \label{prop:eqv_essgeomorf_comorfcont} Consider a small-generated site $(\cbicat,J)$ and a Grothendieck topos $\Etopos$: then there is an equivalence of categories
	\[ 
	\EssTopos\co(\Sh(\cbicat,J), \Etopos)\simeq \Cosite\cont((\cbicat,J), (\Etopos, J\can_\Etopos))
	\]
	acting as follows: 
	\begin{itemize}
		\item An essential geometric morphism $F:\Sh(\cbicat,J)\rightarrow \Etopos$ is sent to the functor $F_!\ell_J:\cbicat\rightarrow \Etopos$, and a natural transformation $\Omega:F\Rightarrow G:\Sh(\cbicat, J)\rightarrow\Etopos$ to $\Omega_!\circ \ell_J$, where $\Omega_!:G_!\Rightarrow F_!$ is the 2-cell induced by $\Omega: F^*\Rightarrow G^*$ on the essential images;
		\item a $(J, J\can_\Etopos)$-continuous comorphism of sites $A:\cbicat\rightarrow \Etopos$ is sent to $C_A:\Sh(\cbicat,J)\rightarrow \Sh(\Etopos, J\can_\Etopos)$ and then composed with the canonical equivalence $\Sh(\Etopos, J\can_\Etopos)\simeq \Etopos$; the same holds for a natural transformation $\omega:A\Rightarrow B:\cbicat\rightarrow\Etopos$.
	\end{itemize}
\end{prop}
In Chapter \ref{chap:classification} we will provide a functorialization of this classification result, along with similar classification results for geometric morphisms.

We conclude this section by presenting a particular instance of topologies of kind $M^A_J$, namely when $A$ is a discrete fibration over $\cbicat$:
\begin{prop}\label{prop:fib_discreta_slice_topos}
	Consider a site $(\cbicat,J)$ and a presheaf $P:\cbicat\op\rightarrow\Set$ with corresponding Grothendieck fibration $\pi_P:\fib P\rightarrow \cbicat$. Denote by $J_P$\index{$J_P$} the topology $M^{\pi_P}_J$: a family of arrows $\{y_i:(Y_i, P(y_i)(U))\rightarrow (X, U)\ |\ i\in I \}$ in $\fib P$ is $J_P$-covering if and only if $\{y_i:Y_i\rightarrow X\ |\ i\in I \}$ is $J$-covering in $\cbicat$. 
	
	The composite functor $\fib P\xrightarrow{\pi_P}\cbicat\xrightarrow{\ell_J} \Sh(\cbicat,J)\rightarrow \Sh(\cbicat,J)/\sheafify_J(P)$ is flat and $J_P$-continuous, and it induces an equivalence of toposes
	$$\Sh(\fib P, J_P)\simeq \Sh(\cbicat,J)/\sheafify_J(P)$$
	which is pseudonatural in $P$ in the sense that for any $g:P\Rightarrow Q$ the square
	\[
	\begin{tikzcd}
		{\Sh(\fib Q, J_Q)} \ar[r, "\sim", no head]\ar[d, "C_{\fib g}^*"']& {\Sh(\cbicat,J)/\sheafify_J(Q)} \ar[d, "\sheafify_J(g)^*"]\\
		{\Sh(\fib P, J_P)} \ar[r, "\sim", no head] &{\Sh(\cbicat,J)/\sheafify_J(P)}  
	\end{tikzcd}\]
	is commutative up to natural isomorphism. Moreover, the induced equivalence makes the diagram
	\[\begin{tikzcd}[column sep=1ex]
		\Sh(\cbicat,J)/\sheafify_J(P) \ar[rr, "\sim", no head] \ar[dr, "\prod_{\sheafify_J(P)}"'] && \Sh(\fib P, J_P) \ar[dl, "C_{\pi_P}"]\\
		& \Sh(\cbicat,J)&
	\end{tikzcd}\]
	commutative up to isomorphism, where $\prod_{\sheafify_J(P)}$\index{$\prod_{f}$} denotes the canonical essential geometric morphism from a slice topos to its base topos (called the \emph{dependent product along $\sheafify_J(P)$}\index{functor!dependent product}\index{geometric morphism!dependent product}: see \cite{dependent}).
\end{prop}
\begin{proof}
	See \cite[Example 5.7]{denseness} and \cite[Theorem 2.2]{dependent}.
\end{proof}
\begin{remark}
	We remark that not only each functor $\pi_P$ is a comorphism of sites $(\fib P, J_P)\rightarrow (\cbicat,J)$, but it is also $(J_P,J)$-continuous; moreover, for every $g:P\rightarrow Q$ in $[\cbicat\op,\Set]$ the induced functor $\fib g:\fib P\rightarrow \fib Q$ is a $(J_P, J_Q)$-continuous comorphism of sites. These are all consequences of Lemma \ref{lemma:comorfismi_via_tplgir}.
\end{remark}

\begin{ex}\label{ex:topologia_su_slice}
	One particular instance of this minimal topology is when $P=\yo(X)$ for some $X$ in $\cbicat$: then as we have already mentioned $\fib \yo(X) \simeq \cbicat/X$. We will denote by $J_X$\index{$J_X$} the topology $J_{\yo(X)}$: in particular, a presieve $\fbicat=\{z_i:[yz_i]\rightarrow[y]\ |\ i\in I \}$ over $[y]$ in $\cbicat/X$ is $J_X$-covering if and only if $\{z_i:\dom(z_i)\rightarrow Y\ |\ i\in I \}$ is $J$-covering in $\cbicat$.
\end{ex}
In particular, we remark that the category of small-generated sites is `closed' with respect to discrete fibrations, in the following sense:
\begin{lemma}\label{lemma:discfib_sono_esssmall}
Given a small-generated site $(\cbicat,J)$ and a presheaf $P:\cbicat\op\rightarrow \Set$, the site $(\fib P, J_P)$ is small-generated.
\end{lemma}
\begin{proof}
    Denote by $\abicat\hookrightarrow\cbicat$ a small $J$-dense full subcategory of $\cbicat$. Consider the full subcategory $\bbicat\hookrightarrow\fib P$ whose objects are pairs $(A,U)$ with $A$ in $\abicat$: then $\bbicat$ as a set of objects and is locally small, therefore it is small. Now consider an object $(X,U)$ of $\fib P$: since $\abicat\hookrightarrow\cbicat$ is $J$-dense, there exists a $J$-covering family $\{y_i:A_i\rightarrow X\ |\ i\in I\}$ such that each $A_i$ belongs to $\abicat$. This implies that the family $\{y_i:(A_i, P(y_i)(U))\rightarrow (X,U)\ |\ i\in I\}$ is $J_P$-covering, and thus that $\bbicat$ is $J_P$-dense in $\fib P$.
\end{proof}
\begin{remark}
    We will introduce the notion of essential $J$-smallness for a fibration in Definition \ref{def:esssmall_fibration}: then the previous lemma is stating precisely that a discrete fibration $\fib P\rightarrow\cbicat$ is essentially $J$-small with respect to any topology on the base category $\cbicat$ making $(\cbicat,J)$ small-generated. 
\end{remark}

Another technical lemma that we will exploit later is the fact that, given a $J$-covering sieve $R\rightarrowtail \yo(X)$, the site $(\fib R, J_R)$ is Morita-equivalent to $(\cbicat/X, J_X)$:
\begin{lemma}\label{lemma:topos_su_sieve_eqv_topos_su_slice}
	Consider a site $(\cbicat,J)$ and a $J$-covering sieve $m_R:R\rightarrowtail \yo(X)$: then $\fib m_R:(\fib R, J_R)\rightarrow (\cbicat/X, J_X)$ induces an equivalence of toposes $C_{\fib m_R}: \Sh(\fib R, J_R)\isorightarrow \Sh(\cbicat/X, J_X)$.
\end{lemma}
\begin{proof}
	This is a corollary of Proposition \ref{prop:fib_discreta_slice_topos}, since the square
	\[
	\begin{tikzcd}
		\Sh(\fib R, J_R) \ar[d, "C_{\fib m_R}"'] \ar[r, "\sim"] & \Sh(\cbicat, J)/\sheafify_J(R) \ar[d, "\prod_{\sheafify_J(m_R)}", "\sim"{sloped, below}]\\
		\Sh(\cbicat/X, J_X) \ar[r, "\sim"] & \Sh(\cbicat, J)/\ell_J(X)
	\end{tikzcd}\]
	commutes up to natural isomorphism, where $\prod_{\sheafify_J(m_R)}$ is an equivalence since $\sheafify_J(m_R)$ is invertible. 
\end{proof}

\section{Giraud topologies and Giraud toposes}\label{sec:giraudtpl}
\begin{defn}[{\cite[\S 2.2]{giraud.classifying}}]\label{def:Giraud_topology_classifying_topos}
	Consider a site $(\cbicat,J)$ and a fibration $p:\dbicat\rightarrow \cbicat$: we call \emph{Giraud's topology for $p$}\index{topology!Giraud -} \index{site!Giraud -} the smallest topology $J_\dbicat$\index{$J_\dbicat$} over $\dbicat$ making $p$ into a comorphism of sites (see Proposition \ref{prop:top_minima_comorfismo}). The sheaf topos 
	\[
	C_p:\Sh(\dbicat, J_\dbicat)\rightarrow \Sh(\cbicat,J)
	\]
	will be called the \emph{classifying topos of the fibration $p$}\index{topos! classifying a fibration}\index{topos!Giraud -}, and it will be denoted by $\Gir_J(p)$\index{$\Gir_J(-)$}: it is the image of the fibration $p$ through the 2-functor
	\[
	\Fib_\cbicat \xrightarrow{\Giraud} \Cosite/(\cbicat,J)\xrightarrow{C_{(-)}}\Topos/\Sh(\cbicat,J)\co.\]
\end{defn}
When $\dbicat$ corresponds to a $\cbicat$-indexed fibration $\dcat$ we will use the notations $J_\dcat$ and $\Gir_J(\dcat)$.

\begin{remarks}
\begin{enumerate}[(i)]
    \item In particular, when $\dcat$ is a presheaf $P:\cbicat\op\rightarrow \Set$ we have denoted Giraud's topology by $J_P$, and when $\dcat=\yo(X)$ by $J_X$: see Proposition \ref{prop:fib_discreta_slice_topos} and the remark following it. In particular, Proposition \ref{prop:fib_discreta_slice_topos} shows that for every presheaf $P:\cbicat\op\rightarrow \Set$ there is an equivalence of functors
\[
\Sh(\cbicat,J)/\sheafify_J(P)\simeq \Gir_J(P),
\]
which is moreover natural in $P$.
    \item there is a size issue that needs to be addressed: there is no reason for the site $(\dbicat,J_\dbicat)$ to be small, or even small-generated, and thus for the category $\Gir_J(\dbicat):=\Sh(\dbicat, J_\dbicat)$ to be a topos. We will address this problem later in the context of the fundamental adjunction by introducing essentially $J$-small fibrations (see Definition \ref{def:esssmall_fibration}).
\end{enumerate}
\end{remarks}

Giraud's topologies can be explicitly described as follows:

\begin{prop}[{\cite[Theorem 3.13]{denseness}}]\label{prop:girtpl_descrizione_sieves}
	Given a fibration $p:\dbicat\rightarrow\cbicat$, a sieve $R$ is $J_\dbicat$-covering if and only if the collection of cartesian arrows in it is sent to a $J$-covering family through $p$. 
\end{prop}

Another fundamental result is the fact that when Giraud's topologies are involved, fibrations and their morphisms, are naturally \emph{continuous} comorphisms of sites:

\begin{prop}[{\cite[Theorem 4.44 and Corollary 4.47]{denseness}}]\label{prop:fib_e_morfib_sono_can.comorf}
	Consider a fibration $p:\dbicat\rightarrow\cbicat$ and a topology $J$ over $\cbicat$: then $p$ is a $(J_\dbicat,J)$-continuous comorphism of sites. More generally, given two fibrations $p:\dbicat\rightarrow\cbicat$ and $q:\ebicat\rightarrow\cbicat$ and a morphism of fibrations $(F,\phi):[p]\rightarrow [q]$ between them, then $F$ is a $(J_\dbicat, J_\ebicat)$-continuous comorphism of sites.
\end{prop}

The study of Giraud's topologies can provide some insight on the fibration considered. As a basic example of this, in the hypothesis of subcanonicity of $J$, the property of being a prestack can be checked directly by analysing Giraud's topology:
\begin{prop}\label{prop:sitosubcan-prestack_sse_Girtpl_subcanonica}
	Consider a subcanonical site $(\cbicat, J)$ and a cloven fibration $p:\dbicat\rightarrow\cbicat$: then $p$ is a prestack if and only if Giraud's topology $J_\dbicat$ is subcanonical.
\end{prop}
\begin{proof}
	We will provide only the definition of the relevant arrows, leaving all calculations to the reader.
	
	We start by supposing that $\dbicat$ is a prestack. Consider a $J_\dbicat$-covering family $R=\{f_i:\dom(f_i)\rightarrow D\ |\ i\in I\}$: then $p(R)=\{p(f_i)\}$ is a $J$-covering family of $p(D)$. We will call $S$ the $J_{p(D)}$-covering family $\{p(f_i):[p(f_i)]\rightarrow[1_{p(D)}]\ i\in I \}$ for $[1_{p(D)}]$ in $\cbicat/p(D)$. Consider now some other object $X$ of $\dbicat$: a matching family for $R$ and $\yo(X)$ is the given, for any $i\in I$, of an arrow $\alpha_i:\dom(f_i)\rightarrow X$, subject to the condition that for every span of arrows $h$ and $k$ such that $f_ih=f_jk$ it holds that $\alpha_ih=\alpha_jk$.
	
	We start by considering the family of arrows $\{p(\alpha_i):\dom(p(f_i))\rightarrow p(X) \}$: it is easy to verify that they are a matching family for $\yo(p(X))$ and $p(R)$, and hence since $J$ is subcanonical there exists a unique $\beta: p(D)\rightarrow p(X)$ such that $p(\alpha_i)=\beta p(f_i)$. 
	
	Using $\beta$ we can now `move over' $p(D)$ by considering the $\Hom$-functor $\Hom((D,1_{p(D)}), \dcat(\beta)(X,1_{p(X)})):(\cbicat/p(D))\op\rightarrow\Set$. We will now build a matching family for it and $S$. To do so, notice first that $f_i$ and $\widehat{p(f_i)}_D$ are both cartesian lifts of $p(f_i)$: therefore, there is a unique canonical isomorphism $\rho_i:\dom(\widehat{p(f_i)}_D)\rightarrow\dom(f_i)$ comparing them. There is also a canonical isomorphism $\sigma_i:\dom(\alpha_i)\rightarrow \dom(\widehat{p(\alpha_i)}_X)$. So we can consider the composite arrows $\gamma_i:=\chi_{\beta, p(f_i),X}\sigma_i\rho_i$: a lenghty calculation shows that the constitute a matching family for $\Hom((D,1_{p(D)}), \dcat(\beta)(X,1_{p(X)}))$ and $S$: if $\dbicat$ is a prestack they admit an amalgamation $\gamma:\dom(\widehat{1_{p(D)}}_D)\rightarrow\dom(\widehat{\theta_{\beta,X}}_{\dom(\widehat{\beta}_X)})$. Now, if we consider the arrow $\alpha$ defined as the composite
	\[D \xrightarrow{\widehat{1_{p(D)}}_D\inv}\dom(\widehat{1_{p(D)}}_D)\xrightarrow{\gamma}\dom(\widehat{\theta_{\beta,X}}_{\dom(\widehat{\beta}_X)})\xrightarrow{\widehat{\theta_{\beta,X}}_{\dom(\widehat{\beta}_X)}}\dom(\widehat{\beta}_X)\xrightarrow{\widehat{\beta}_X}X \]
	it is again lenghty but straightforward to see that it provides an amalgamation for the matching family $\{\alpha_i\}$, proving that $\yo(X)$ is a $J_\dbicat$-sheaf and hence that $J_\dbicat$ is subcanonical.
	
	Conversely, suppose $J_\dbicat$ is subcanonical and consider a $J_X$-covering sieve $S$ for it and $\Hom((A,\alpha),(B,\beta))$: without loss of generality we may assume that $S=\{f_i:[f_i]\rightarrow [1_X]\ |\ i\in I \}$ covers the terminal object of $\cbicat/X$, by Corollary \ref{cor:Hom_mfam_su_terminali}, so that the matching family consists of arrows $\gamma_i:\dom(\widehat{\alpha f_i}_A)\rightarrow\dom(\widehat{\beta f_i}_B)$ satisfying suitable compatibility conditions. We can now consider the $J_\dbicat$-covering family $R=\{\widehat{\alpha f_i}_A\ |\ i\in I \}$ over $A$, and the arrows $\mu_i:=\widehat{\beta f_i}_B \gamma_i:\dom(\widehat{\alpha f_i}_A)\rightarrow B$. A calculation shows that it is a matching family for $R$ and $\yo(B)$, and thus it admits a unique amalgamation $\gamma:A\rightarrow B$, which also proves to be a unique amalgamation for the original matching family $\{\gamma_i\}$: hence all $\Hom$-functors are sheaves and we conclude that $\dbicat$ is a prestack.  
\end{proof}

\begin{remark}
	This is a broad generalization of Proposition 2.1(3) of \cite{giraud.classifying}, where $\dbicat$ is supposed to be a lex stack (and a Grothendieck fibration) and the site is of the form $(\Etopos, J\can_\Etopos)$. 
\end{remark}
\begin{ex}
	It may happen that $K$ is subcanonical and $p$ is a prestack without $J$ begin subcanonical. Indeed, consider $\cbicat$ to be the arrow category $t:0\rightarrow 1$, $\dbicat$ the empty category and $p:\dbicat\rightarrow \cbicat$ the unique functor. Notice that $\dbicat$ corresponds to the discrete $\cbicat$-indexed category $\dcat:\cbicat\op\rightarrow \CAT$ with constant value the empty category. Now consider on $\cbicat$ the topology $J$ whose only non-trivial sieve is the singleton $\{t\}$: it is not subcanonical, since $\yo(0)$ is not a $J$-sheaf. $\dcat$ is a sheaf for this topology, and hence it is a $J$-stack; moreover, the induced topology $K$ on $\dbicat$ is the only topology on the empty category, which is trivially subcanonical.
\end{ex}

\chapter{Change of base functors}\label{chap:base_change}

It is well known that a continuous map of topological spaces $f:X\rightarrow Y$ induces a geometric morphism
\[
\Sh(f):\Sh(X)\rightarrow \Sh(Y)
\]
ot toposes of sheaves, acting as follows: \begin{itemize}
	\item given a sheaf $P\in\Sh(X)$, the \emph{direct image} of $P$ along $f$, \ie the presheaf $\Sh(f)_*(P)$, is defined for any $U\subseteq Y$ open as \[\Sh(f)_*(P)(U):=P(f\inv(U));\]
	\item starting from a sheaf $Q\in \Sh(Y)$, and denoted by $p_Q:E_Q\rightarrow Y$ its étale bundle over $Y$ (see Section \ref{sec:adj_topologica}), we can pull it back along $f:X\rightarrow Y$ in order to obtain an étale bundle over $X$: the corresponding sheaf is the \emph{inverse image} $\Sh(f)^*(Q)$ of $Q$ along $f$.
\end{itemize}
In fact, $\Sh(f)$ is the geometric morphism induced by homomorphism of frames $f\inv:\Ocal(Y)\rightarrow \Ocal(X)$, seen as a morphism of sites with respect to the canonical topologies. 

One can consider the same kind of change of base processes when working with fibrations and stacks (one standard reference for these constructions is provided by Sections I.2 and II.3 of \cite{giraud.cohomologie}). 

In this chapter we shall recall some results about these concepts, and establish a number of new results. For instance, we shall provide an explicit description of the inverse image of fibrations in terms of pseudocolimits of categories, and show that the inverse image of a fibration can be understood as a localization of a particular fibration of generalized elements. We will also study the relationship between these functors and the relative comprehensive factorization of functors introduced in \cite{denseness}. Moreover, for the purpose of extending to categories of stacks the constructions of the geometric morphisms induced by morphisms or comorphisms of sites, we shall introduce pseudo-Kan extensions and apply them to prove that, similarly to the sheaf-theoretic setting, such functors between sites also induce adjunctions between the corresponding categories of stacks. We also discuss the notion of continuous functor between sites from the point of view of stacks, and show that, not only composing with a continuous functor transforms sheaves into sheaves, but it also transforms stacks into stacks. Lastly, we show that pullbacks of Giraud toposes can already be computed at the level of sites, if one adopts the point of view of comorphisms of sites.

\section{Direct image of fibrations}\label{sec:dir_imm_fibrations}

The simplest change of base technique for fibrations is the \emph{direct image}, which, as the name suggests, consists in pulling back a fibration along a functor: indeed, it is well known that if $q:\ebicat\rightarrow\dbicat$ is a Grothendieck fibration, its pullback along a functor $F:\cbicat\rightarrow \dbicat$ provides a fibration over $\cbicat$. Let us show here that weakening to pseudopullbacks allows us to talk about the direct image of Street fibrations:
\begin{prop}\label{prop:directimage_Street_fibrations}
	Consider a functor $F:\cbicat\rightarrow\dbicat$, a fibration $q:\ebicat\rightarrow \dbicat$ and the strict pseudopullback
	\[
	\begin{tikzcd}
		\cbicat\times_\dbicat\ebicat \ar[d, "p"'] \ar[r, "\pi"] & \ebicat \ar[d, "q"] \ar[dl, equal, "\sim"sloped]\\
		\cbicat \ar[r, "F"'] & \dbicat
	\end{tikzcd}:\] 
	the functor $p$ is a fibration over $\cbicat$, called the \emph{direct image of $q$ along $F$}\index{functor!direct image}, and is denoted by $F_*(q)$; the functor $\pi$ maps $p$-cartesian arrows to $q$-cartesian arrows. In particular, if $q$ is cloven then its inverse image $p$ is also cloven.
	Computing the direct image of fibrations over $\dbicat$ along the functor $F$ provides a 2-functor \[F_*:\Fib_\dbicat\rightarrow \Fib_\cbicat,\]\index{$F_*$} defined by mapping a morphism $(A,\alpha):[q:\ebicat\rightarrow \dbicat]\rightarrow [r:\rbicat\rightarrow\dbicat]$ to the functor
	\[
	F_*(A,\alpha):\cbicat\times_\dbicat\ebicat\rightarrow \cbicat\times_\dbicat\rbicat,
	\]
	defined on objects as $F_*(A,\alpha)(X,E,f:q(E)\isorightarrow F(X)):=(X, A(E), f\alpha_E)$ and on arrows as $F_*(A,\alpha)(t,s)=(t, A(s))$. Similarly, a 2-cell $\gamma:(A,\alpha)\Rightarrow (B,\beta)$ of morphisms of fibrations is sent to the 2-cell $F_*(\gamma):F_*(A,\alpha)\Rightarrow F_*(B,\beta)$, defined componentwise by
	\[
	F_*(\gamma)(X,E,f):=(1, \gamma_E):(X, A(E), f\alpha_E)\rightarrow (X, B(E), f\beta_E).
	\]
	
	If we consider the pseudofunctor $\Ifrak([q])=\ecat$, then $\Ifrak(F_*([q]))\simeq \ecat\circ F\op$: that is, on $\dbicat$-indexed categories the direct image along $F$ can be described as precomposition with $F\op$.
\end{prop}
\begin{proof}
	Start with an object $(X,E,f:F(X)\isorightarrow q(E))$ of $\cbicat\times_\dbicat\ebicat$ and an arrow $y:Y\rightarrow X$: we want to build a cartesian lift of $y$ with codomain $(X,E,f)$. To do so, consider the composite arrow  $f F(y):F(Y)\rightarrow q(E)$ in $\dbicat$: since $q$ is a fibration, there are a cartesian arrow $\bar{y}:E'\rightarrow E$ in $\ebicat$ and an isomorphism $g:F(Y)\isorightarrow q(E')$ such that $q(\bar{y})g=fF(y)$. This provides an arrow $(y,\bar{y}):(Y, E',g)\rightarrow (X,E,f)$ such that $p(y,\bar{y})=y$. The verification that it is a cartesian lift is lenghty but straightforward.
	
	It is also obvious that if $q$ is cloven then $p$ is also cloven, since a canonical choice of lifting for $q$ entails a canonical choice of liftings for $p$.
	The fact that $F_*$ is a 2-functor is just a straightforward verification: notice in particular that the functor $F_*(A,\alpha)$ exists by the universal property of pseudopullbacks, while the fact that it preserves cartesian arrows follows immediately from the fact that $A$ does.
	
	Finally, the verification that $F_*$ acts on $\dbicat$-indexed categories as a precomposition is a lengthy but easy calculation. Let us sketch how it works by showing how to define on objects the fibre-wise equivalence $\ecat(F(X))\simeq \Ifrak(F_*([q]))(X)$, for $X$ an object of $\cbicat$ and $\ecat:=\Ifrak([q])$. By unwinding the definition of $\Ifrak$ we have that $\Ifrak(F_*([q]))(X)$ has as objects quadruples
	\[
	(Y,E, f:F(Y)\isorightarrow q(E), \gamma:X\isorightarrow Y),
	\]
	where $Y$ belongs to $\cbicat$ and $E$ to $\ebicat$. Then one can associate to each object $(E,\alpha:F(X)\isorightarrow q(E))$ of $\ecat(F(X))$ the object $(X,E, \alpha, 1_X)$ of $\Ifrak(F_*([q]))(X)$, and conversely to the object $(Y,E,f, \gamma)$ above one can associate the object $(E, fF(\gamma):F(X)\isorightarrow q (E))$ of $\dcat(F(X))$.
\end{proof}
\begin{remarks}\label{rmk:laxcolim_weighted_is_pb_grfibr2}
	\begin{enumerate}[(i)]
	\item By their very definition, we can now think of fibres as direct images. Given a fibration $p:\dbicat\rightarrow\cbicat$ and an object $X$ in $\cbicat$, the fibre of $p$ at $X$ is defined as the pseudopullback of $p$ along the functor from the one-object category,
	\[		e_X:\onecat\rightarrow \cbicat,
	\]
    which selects the object $X$. In short, the fibre of $p$ at $X$ is the domain of $e_X^*([p])$.
    \item The fact that the pullback of a fibration is still a fibration can be interpreted as a result about the compatibility of pullbacks and colimits.  Indeed, given a pseudofunctor $\ecat:\dbicat\op\rightarrow\CAT$, by Proposition \ref{prop:laxcolim_grothconstr} we know that $ \gbicat(\ecat)\simeq\colim_{lax}\ecat $ in $\CAT$, and thus \textit{a fortiori} in $\CAT/\dbicat$. If now we consider the previous result, we have that for any functor $F:\cbicat\rightarrow\dbicat$ the pseudopullback of $\colim_{lax}(\ecat)$ along $F$ is equivalent to the colimit $\colim_{lax}(\ecat\circ F\op)$.
    \item The argument in the previous item allows us to apply arguments about direct images of fibrations also to the topic of commutativity of weights and diagrams. Given two pseudofunctors $\dcat:\cbicat\op\rightarrow\CAT$ and $R:\cbicat\rightarrow\CAT$, we know by Corollary \ref{cor:commute_weight_diagram_conified} that there exists an equivalence $\colim_{oplax}(R\Vop\circ p_{\dcat})\simeq \colim_{lax}(\dcat\circ p_R)$. However, we can think of $\colim_{lax}(\dcat\circ p_R)$ as the fibration associated to the direct image of $\dcat$ along $p_{R}\op$; in other words, we have a pullback diagram
    \[
    \begin{tikzcd}[column sep=50]
        \colim_{lax}(\dcat\circ p_R) \ar[d] \ar[r, no head, "\sim"] & \gbicat(\dcat\circ p_R) \ar[d] \ar[r,"{(p_R\op)^*[p_\dcat]}"] & \gbicat(R)\op\ar[d,"p_R\op"] \\
        \colim_{lax}(\dcat)\ar[r,"\sim"]&\gbicat(\dcat) \ar[r, "p_\dcat"']& \cbicat \ar[ul, phantom, "\lrcorner" near end]
    \end{tikzcd}.
    \]
    If now we apply $\op$ to the square on the right, it is still a pullback square, but this time we can see it as computing the pullback of the fibration $p_R$ along the functor $p_\dcat\op$, \ie the inverse image of $R$ along $p_\dcat\op$:
    \[
    \begin{tikzcd}
        \colim_{lax}(R\circ p_\dcat)\ar[d, "\sim"{below, sloped}, no head] \ar[r] & \colim_{lax}(R)\ar[d,"\sim"sloped, no head] \\
        \gbicat(R\circ p_\dcat)\ar[d,"{(p_\dcat\op)^*[p_R]}"'] \ar[r] & \gbicat(R)\ar[d,"p_R"] \\
        \gbicat(\dcat)\op \ar[r, "p_\dcat\op"']& \cbicat\op \ar[ul, phantom, "\lrcorner" near end]
    \end{tikzcd}.
    \]
    This implies that there is an equivalence of categories
    \[\colim_{lax}(R\circ p_\dcat)\simeq \colim_{lax}(\dcat\circ p_R)\op,\]
    which expresses an alternative form of the commutativity of diagrams and colimits. Combining this with the identity from Corollary \ref{cor:commute_weight_diagram_conified} we also obtain the further equivalence
    \[
    \colim_{oplax}(R\Vop\circ p_\dcat)\simeq \colim_{lax}(R\circ p_\dcat)\op.
    \]
	\end{enumerate}
\end{remarks}

\section{Inverse image of fibrations}\label{sec:inv_imm_fibrations}
The behaviour of the direct image on pseudofunctors restricts to presheaves: its restriction
\[
(-\circ F\op):[\dbicat\op,\Set]\rightarrow [\cbicat\op,\Set]
\]
is a functor which we have already met in the context of comorphisms of sites, and which admits both a left and a right adjoint (see Proposition \ref{prop:lan_ran}). The same happens at the level of indexed categories, and this provides the notion of \emph{pseudo-Kan extensions} and of \emph{inverse image}:

\begin{prop}\label{prop:pseudoKan}\index{Kan extension! pseudo-}
	Denote by $\Ind_\cbicat^s$\index{$\Ind_\cbicat^s$} the sub-2-category of $\Ind_\cbicat$ of pseudofunctors with values in $\Cat$ (\ie `small' $\cbicat$-indexed categories). Consider any functor $F:\cbicat\rightarrow \dbicat$ and the \emph{direct image} 2-functor
	\[
	F^*:\Ind^s_\dbicat\rightarrow \Ind^s_\cbicat
	\]
	which acts by precomposition with $F\op$. The 2-functor $F^*$ has both a left and a right 2-adjoint, denoted respectively by $\Lan_{F\op}$\index{$\Lan_{F\op}$} and $\Ran_{F\op}$\index{$\Ran_{F\op}$}, which act as follows: 
	\begin{itemize}
		\item for any $D$ in $\dbicat$ denote by $\pi^D_F:\comma{D}{F}\rightarrow \cbicat$ the canonical projection functor: then for $\ecat:\cbicat\op\rightarrow \Cat$, its image $\Lan_{F\op}(\ecat):\dbicat\op\rightarrow\Cat$ is defined componentwise as
		\[
		\Lan_{F\op}(\ecat)(D):=\colim_{ps}\left( \comma{D}{F}\op\xrightarrow{(\pi^D_F)\op}\cbicat\op\xrightarrow{\ecat}\Cat \right)
		\] 
		The pseudofunctor $\Lan_{F\op}(\ecat)$ is called the \emph{inverse image of $\ecat$ along $F$}\index{functor!inverse image}.
		\item for any $D$ in $\dbicat$ denote by $\pi'^D_F:\comma{F}{D}\rightarrow \cbicat$ the canonical projection functor: then for $\ecat:\cbicat\op\rightarrow \Cat$, its image $\Ran_{F\op}(\ecat):\dbicat\op\rightarrow\Cat$ is defined componentwise as
		\[
		\Ran_{F\op}(\ecat)(D):=\lim_{ps}\left( \comma{F}{D}\op\xrightarrow{(\pi'^D_F)\op}\cbicat\op\xrightarrow{\ecat}\Cat \right)
		\] 
	\end{itemize}
\end{prop}
\begin{proof}
	The approach used to compute enriched Kan extensions in \cite[Chapter 4]{kelly2005} adapts immediately to this context. For instance, it is shown \textit{loc. cit.} that the left Kan extension $\Lan_{F\op}\ecat$ can be computed componentwise in $D$ as the pseudocolimit $\colim_{ps}^{\dbicat(D, F(-))}\ecat$; the equivalence with our expression above results from applying the last part of Corollary \ref{cor:colimiti_pesati_conificati} and noticing that $\pi^D_F$ is the opfibration corresponding to $\dbicat(D,F(-))$.
\end{proof}
Colimits also provide a way of expressing the inverse image of a pseudofunctor in fibrational terms. To formulate it, we must introduce the functor
\[
\tau_F:\cbicat\rightarrow\Cl\Fib_\dbicat,
\]\index{$\tau_F$}
which maps every object $X$ in $\cbicat$ to the discrete fibration $\comma{1_\dbicat}{F(X)}\rightarrow \dbicat$. We can also think of $\tau_F$ as a functor with values in $\Ind_\dbicat$, where each $X$ is mapped to the pseudofunctor $\comma{-}{F(X)}=\dbicat(-,F(X)):\dbicat\op\rightarrow \Cat$.

\begin{prop}\label{prop:inverseimage_as_weighted_colimit}
Consider a functor $F:\cbicat\rightarrow \dbicat$ and a pseudofunctor $\ecat:\cbicat\op\rightarrow \Cat$: the inverse image $\Lan_{F\op}(\ecat)$ is computed as the weighted pseudocolimit
\[
\colim_{ps}^\ecat(\tau_F).
\]
\end{prop}
\begin{proof}
This follows from the identity $\tau_F(-)=\dbicat/F(-)$ and Yoneda's lemma: for every $\fcat:\dbicat\op\rightarrow\Cat$ we have the chain of natural equivalences
 \begin{align*}
 \Ind_\dbicat(\colim_{ps}^\ecat \tau_F, \fcat )&\simeq \Ind_\cbicat(\ecat, \Ind_\dbicat(\tau_F(-), \fcat))\\
 &\simeq \Ind_\cbicat(\ecat, \Ind_\dbicat(\dbicat/F(-), \fcat))\\
 &\simeq \Ind_\cbicat(\ecat, \fcat(F(-)))\\
  &\simeq \Ind_\cbicat(\ecat, F^*(\fcat)),
\end{align*}
implying that $\Lan_{F\op}(\ecat)\simeq \colim^\ecat_{ps}\tau_F$. 
\end{proof}
\begin{remarks}\label{rmk:Grothconstruction}
\begin{enumerate}[(i)]
    \item This result justifies in particular the equivalence
    \[
     \colim_{ps}^\dcat \cbicat/- \simeq \colim_{lax}\dcat
    \]
     of Corollary \ref{cor:Grothconstruction}. One can easily see that $\cbicat/-$ is the functor 
    \[
    \tau_{1_\cbicat}:\cbicat\rightarrow \Cl\Fib_\cbicat,
    \]
    and thus the first colimit is the inverse image of $p_\dcat$ along the functor $1_\cbicat$: but obviously direct and inverse images along an identity functor are identity functors themselves, and hence we have that
    \[
    \colim_{ps}^\dcat \cbicat/- \simeq \Lan_{1_\cbicat\op}[p_\dcat]\simeq \gbicat(\dcat)\simeq \colim_{lax}\dcat
    \]
    as fibrations over $\cbicat$.
    \item Kan extensions of presheaves can be expressed as the (co)ends of the product of two $\Set$-valued functors, and thus they admit a double colimit representation. We shall generalize this double colimit representation in Section \ref{sec:computationKanextensions} by using lax colimits. 
    \item The restriction to small $\cbicat$-indexed categories is needed to ensure the existence of the colimits involved in the definition of $\Lan_{F\op}$. In principle, this prevents us from calculating inverse images of many stacks that naturally arise in mathematical practice, such as the canonical stack of a site; however, we will see that in practice many $\cbicat$-indexed categories can be `presented by a set of generators': in this case, the colimits involved will be computed on small subdiagrams by arguments of cofinality.
\end{enumerate}
\end{remarks}

The fibrewise description of $\Lan_{F\op}(\ecat)$ of Proposition \ref{prop:pseudoKan} can be recovered from Proposition \ref{prop:inverseimage_as_weighted_colimit} by exploiting evaluation functors. Once fixed an object $D$ of $\dbicat$, the \emph{evaluation 2-functor}
\[\Ev_D:\Cl\Fib_\dbicat\rightarrow \CAT\]\index{$\Ev_D$} 
acts by mapping each fibration $q:\qbicat\rightarrow\dbicat$ to its fibre in $D$; or analogously, each $\dbicat$-indexed category $\fcat$ to its fibre $\fcat(D)$. The 2-functor $\Ev_D$ behaves as the direct image along the functor $e_D:\onecat\rightarrow \dbicat$ which selects the object $D$ in $\dbicat$: from a fibrational point of view, this is true because the fibre at $D$ is computed precisely as the direct image along $e_D$ (see Remark \ref{rmk:laxcolim_weighted_is_pb_grfibr2}(i)); from an indexed point of view, because the inverse image of a pseudofunctor $\fcat:\dbicat\op\rightarrow\CAT$ along $e_D$ is the composite
\[
\onecat\op \xrightarrow{(e_D)\op}\dbicat\op\xrightarrow{\fcat}\CAT,
\]
which via the equivalence $[\onecat\op,\CAT]_{ps}\simeq \CAT$ is uniquely determined by its value $\fcat(D)$. This has the consequence that $\Ev_D$ has both a left and a right adjoint, \ie the right and left Kan extensions of Proposition \ref{prop:pseudoKan}, and thus that it preserves all kinds of limits and colimits:
\begin{lemma}
    Left 2-adjoints preserve any kind of bicolimit: more explicitly, given two 2-functors $L:\abicat\rightarrow\bbicat$ and $R:\bbicat\rightarrow\abicat$ such that $L\dashv R$, a category $\cbicat$, two pseudofunctors $\dcat:\cbicat\op\rightarrow \Cat$ and $I:\cbicat\rightarrow \abicat$, then
    \[
    L(\colim_\bullet^\dcat I)\simeq \colim_\bullet^\dcat (L\circ I),
    \]
    where $\bullet$ is either $lax$, $oplax$ or $ps$.
\end{lemma}
\begin{proof}
This is a consequence for the following chain of natural equivalences, which hold for every $B$ in $\bbicat$:
    \begin{align*}
        \bbicat\left(\colim_{\bullet}^\dcat (L\circ I) , B\right)&\simeq [\cbicat\op,\Cat]_\bullet \left(\vphantom{\colim^\dcat}\dcat, \bbicat((L\circ I)(-),B)\right)\\
        &\simeq [\cbicat\op,\Cat]_\bullet \left(\vphantom{\colim^\dcat}\dcat, \abicat(I(-),R(B))\right)\\
        &\simeq \abicat\left(\colim_\bullet^\dcat I, R(B)\right)\\
        &\simeq \abicat\left(L(\colim_\bullet^\dcat I), B\right). 
    \end{align*}
\end{proof}
By considering in particular the functor $\Ev_D:\Ind^s_\dbicat\rightarrow \Cat$, which is a left adjoint, the following equivalences hold for any diagram $I:\cbicat\rightarrow \Ind_\dbicat$ with weight $\dcat:\cbicat\op\rightarrow\Cat$:
\begin{align*}
\Ev_D(\colim^\dcat_{lax}(I))\simeq& \colim_{lax}(\Ev_D\circ I),\\ \Ev_D(\colim^\dcat_{oplax}(I))\simeq& \colim^\dcat_{oplax}(\Ev_D\circ I),\\ \Ev_D(\colim^\dcat_{ps}(I))\simeq& \colim^\dcat_{ps}(\Ev_D\circ I).    
\end{align*}
This means that colimits of any kind are computed fibrewise in $\Ind^s_\dbicat$.

Now, we can apply these considerations to diagrams of the kind
\[\tau_F^D:\cbicat\xrightarrow{\tau_F} \Cl\Fib_\dbicat \xrightarrow{\Ev_D}\Set.\]\index{$\tau_F^D$}
Notice that each of these functors takes values in $\Set$, since $\tau_F^D(X)=\dbicat(D,F(X))$: indeed, it is the fibre in $D$ of $\comma{1_\dbicat}{F(X)}$, which is a discrete fibration over $\dbicat$. Since evaluation functors commute with bicolimits, we have that
\begin{align*}
    \Lan_{F\op}(\ecat)(D)&:=\Ev_D(\Lan_{F\op}(\ecat))\\
    &=\Ev_D(\colim_{ps}^\ecat \tau_F)\\
    &\simeq \colim_{ps}^\ecat(\tau^D_F)\\
    &\simeq \colim_{ps}^{\tau_F^D}\ecat
\end{align*}
where the last step holds by the commutativity of weights and diagrams (see Proposition \ref{prop:commutativity_peso_diag_colimiti}): but the last colimit is precisely the one defined as $\Lan_{F\op}(\ecat)(D)$ in the proof of Proposition \ref{prop:pseudoKan}.

\section{Computation of inverse images}\label{sec:computationKanextensions}
Since an inverse image $\Lan_{F\op}(\ecat)$ is a weighted pesudocolimit, we have now the perfect opportunity to apply the techniques about colimits and localizations developed in Chapter \ref{chap:preliminaries}. First of all, notice that by Proposition \ref{prop:pseudocolimiti_localizzazioni} the colimit $\Lan_{F\op}(\ecat)\simeq \colim_{ps}^\ecat \tau_F$ can be seen as a localization of either of the following:
	\[
	\colim_{lax} \tau_Fp_{\ecat\Vop}\simeq \colim_{lax}^\ecat \tau_F\simeq \colim^{\tau_F}_{oplax}\ecat\simeq \colim_{oplax} \ecat p_{\tau_F},
	\]
	\[
	\colim_{oplax} \tau_Fp_{\ecat}\simeq \colim_{oplax}^\ecat \tau_F\simeq \colim^{\tau_F}_{lax}\ecat\simeq \colim_{lax} \ecat p_{\tau_F\Vop}.
	\]
Each of these localizations is performed with respect to the components of a certain class of natural transformations appearing in the colimit cocone, either indexed by arrows in the base category $\cbicat$ (for the weighted colimits) or by cartesian arrows in the $\gbicat(\ecat)$ or $\gbicat(\tau_F)$ (for the conified counterparts). Moreover, Remark \ref{rmk:laxcolim_weighted_is_pb_grfibr2}(iii) implies that $\colim_{ps}^\ecat \tau_F$ can also be seen as a localization of $\colim_{lax}(\tau_F\Vop\circ p_\ecat)\op$. In the following, we will focus on the description of the localization
\[
\colim_{lax}(\ecat \circ p_{\tau_F\Vop})\rightarrow \colim_{ps}^\ecat \tau_F
\]
appearing in Proposition \ref{prop:pscolim_localizz_struttura_mista}, which is based on an application of Proposition \ref{prop:laxcolim_grothconstr} to explicitly compute lax colimits using Grothendieck fibrations. We also remark that computing any of the other colimits using Proposition \ref{prop:laxcolim_grothconstr} yields exactly the same category: indeed, computing op-/lax colimits with the Grothendieck construction is an optimal way of computing them, and irons out all the conceptual differences between the possible different presentations.

First of, let us compute the fibrations associated with the functors $\tau_F$ and $\tau_F^D$. If we consider the functor 
\[
\tau_F:\cbicat\rightarrow \Cl\Fib_\dbicat,
\]
as a functor with values in $\Cat$, a quick computation shows that the opfibration corresponding to $\tau_F$, \ie the opposite of the functor $p_{\tau_F\Vop}:\gbicat(\tau_F\Vop)\rightarrow \cbicat\op$ (see Remark \ref{rmk:covGroth.i}(i)), is precisely the comma category
\[
\pi_\cbicat:\comma{1_\dbicat}{F}\rightarrow \cbicat,
\]
with $\pi_\cbicat$ forgetting the component in $\dbicat$. The same category is also fibred over $\dbicat$, by considering the functor
\[
\pi_\dbicat:\comma{1_\dbicat}{F}\rightarrow \dbicat
\]
which forgets the component in $\cbicat$. In a similar way, each functor
\[
\tau_F^D:\cbicat\rightarrow \Set
\]
is associated with the opfibration
\[
\pi^D_\cbicat:\comma{D}{F}\rightarrow\cbicat.
\]
We are now ready to prove the following:
\begin{prop}\label{prop:inverseimage_as_localization}
	Consider a pseudofunctor $\ecat:\cbicat\op\rightarrow\Cat$ and a functor $F:\cbicat\rightarrow\dbicat$. Consider the category
	$
	\comma{D}{F\circ p_\ecat},
	$
	whose objects are arrows $[d:D\rightarrow F\circ p_\ecat(X,U)]$ and whose morphisms $(y,a):[d':D\rightarrow F\circ p_\ecat(Y,V)]\rightarrow [d:D\rightarrow F\circ p_\ecat(X,U)]$ are indexed by arrows $(y,a):(Y,V)\rightarrow (X,U)$ such that $F(y)\circ d'=d$. Consider the class of arrows
	\begin{equation*}
	\begin{gathered}
	    S_D:=\left\{[d']\xrightarrow{(y,a)}[d]\ \middle| \ (y,a)\mbox{ cartesian in }\gbicat(\ecat)  \right\}:
	\end{gathered}
	\end{equation*}
	then 
	\[
    \Lan_{F\op}(\ecat)(D)\simeq\comma{D}{F\circ p_\ecat}[S_D\inv].
    \]
	
	From a fibred point of view, let $p:\pbicat\rightarrow\cbicat$ be a cloven fibration. Consider the fibration of generalized elements 
	\[\comma{1_\dbicat}{(F\circ p)}\xrightarrow{r}\dbicat\]
	of the functor $F\circ p$, whose objects are arrows $[d:D\rightarrow (F\circ p)(U)]$ of $\dbicat$, and a morphism
	\[
	(e,\alpha):[d':D'\rightarrow (F\circ p)(V)]\rightarrow[d:D\rightarrow (F\circ p)(U)]
	\]
	is indexed by an arrow $e:D'\rightarrow D$ in $\dbicat$ and an arrow $\alpha:V\rightarrow U$ in $\pbicat$ such that $(F\circ p)(\alpha)\circ d'=d\circ e$. Consider the class of arrows
	\[
    S:=\{(e,\alpha):[d']\rightarrow [d]\ |\ (e,\alpha)\mbox{ $r$-vertical, }\alpha\mbox{ cartesian in }\pbicat \}: 
	\]
	then 
	\[
	\Lan_{F\op}([p])\simeq\comma{1_\dbicat}{(F\circ p)}[S\inv].
	\]
\end{prop}
\begin{proof}
	Proposition \ref{prop:pseudocolimiti_localizzazioni} provided us with various possibilities to compute a pseudocolimit such as $\Lan_{F\op}(\ecat)(D):=\colim_{ps}^{\tau_F^D}\ecat$ by localizing a lax colimit: in this circumstance, exploiting Proposition \ref{prop:pscolim_localizz_struttura_mista} seems to provide the simplest and most suggestive description of the inverse image. Said result states that we can consider the pullback
\[\begin{tikzcd}
	{\colim_{lax}(\ecat\circ p_{(\tau_F^D)\Vop})}\ar[d, "(\pi^D)'"'] \ar[r, "\pi^D_\ecat"] & \gbicat(\ecat) \ar[d, "p_\ecat"]\\
	\gbicat((\tau_F^D)\Vop)\op \ar[r, "p_{(\tau_F^D)\Vop}\op"'] & \cbicat    
\end{tikzcd},\]
and $\colim_{ps}^{\tau_F^D}\ecat$ is the localization of $\colim_{lax}(\ecat\circ p_{(\tau_F^D)\Vop})$ with respect to the morphisms that are $(\pi^D)'$-cartesian and are mapped via $\pi^D_\ecat$ to $p_\ecat$-cartesian arrows. Now we can perform the following simplifications: first of all $(\tau_F^D)\Vop\cong\tau_F^D$, since it is a discrete functor; secondly, as we mentioned at the beginning of this section, the functor $p\op_{\tau_F^D}$ corresponds to the canonical projection
\[
\pi^D_\cbicat:\comma{D}{F}\rightarrow \cbicat;
\]
finally, the explicit description of $\colim_{lax}(\ecat\circ p_{(\tau_F^D)\Vop})$ of Proposition \ref{prop:pscolim_localizz_struttura_mista} provides, in this case, the category $\comma{D}{F\circ p_\ecat}$. Therefore, the pullback above can be rewritten as
\[\begin{tikzcd}
	{\comma{D}{F\circ p_\ecat}}\ar[d, "(\pi^D)'"'] \ar[r, "\pi^D_\ecat"] & \gbicat(\ecat) \ar[d, "p_\ecat"]\\
	{\comma{D}{F}} \ar[r, "\pi^D_F"'] & \cbicat    
\end{tikzcd}.\]
Finally, notice that an arrow \[
[d':D\rightarrow F\circ p_\ecat(Y,V)]\xrightarrow{(y,a)} [d:D\rightarrow F\circ p_\ecat(X,U)]
 \]
in $\comma{D}{F\circ p_\ecat}$ is $(\pi^D)'$-cartesian if and only if $(y,a)$ is $p_\ecat$-cartesian as an arrow in $\gbicat(\ecat)$, which is the same as to say that it is mapped to a $p_\ecat$-cartesian arrow via the forgetful functor $\pi^D_\ecat$: therefore, $\Lan_{F\op}(\ecat)(D)$ is computed localizing $\comma{D}{F\circ p_\ecat}$ with respect to all its arrows whose component in $\gbicat(\ecat)$ is cartesian, \ie with respect to the class $S_D$.

	Now, if we consider each category $\comma{D}{F\circ p}$ as the fibre in $D$ of the pseudofunctor
	\[
	\comma{-}{F\circ p}:\dbicat\op\rightarrow\Cat,
	\]
	then $\Lan_{F\op}[p](D)$ is a localization of it; moreover, said localization is compatible with transition morphisms, \ie given $g:D'\rightarrow D$ in $\dbicat$, the transition morphism
	\[
	-\circ e: \comma{D}{F\circ p}\rightarrow \comma{D'}{F\circ p}
	\]
	restricts to a functor
	\[
	\comma{D}{F\circ p}[S_D\inv]\rightarrow \comma{D'}{F\circ p}[S_{D'}\inv].
	\]
	Therefore, we may apply Lemma \ref{lemma:pointwiselocalization}: the fibration associated to the localized pseudofunctor
	\[
	\comma{-}{F\circ p}[S_{(-)}\inv]:\dbicat\op\rightarrow \Cat
	\]
	is the localization of the fibration associated to $\comma{-}{F\circ p}$ , which is the fibration of generalized elements 
	\[\comma{1_\dbicat}{F\circ p}\xrightarrow{r}\dbicat,\] 
	with respect to the class of arrows 
	\[
	(e,\alpha):[d':D'\rightarrow Fp(V)]\rightarrow[d:D\rightarrow Fp(U)]\] 
	that are $\pi_\dbicat$-vertical (\ie $e$ is invertible) and such that their component $\alpha$ in the fibre belongs to $S_{D'}$. This provides the description of $\Lan_{F\op}([p])$ in the statement.
\end{proof}	

\begin{remarks}
	\begin{enumerate}[(i)]
	    \item We could also have computed $\Lan_{F\op}([p])$ directly by using our results on colimits: it is however instructive to see that it is in fact a fibrewise localization of another pseudofunctor. This shows that, when computing inverse images, we can follow two different paths: we can first act a fibrewise localization of $\comma{D}{F\circ p_\ecat}$, and the `universalize' it, \ie compute the associated fibration; or alternatively, we can universalize and then localize the associated fibration.
	    \item The fibrational description of invere images could also be proven directly without resorting to the indexed formalism at all, but instead by showing explicitly the equivalence of categories
	    \[
	    \Cl\Fib_\cbicat([p],F^*[q])\simeq \Cl\Fib_\dbicat(\comma{1_\dbicat}{F\circ p}[S\inv], [q])
	    \]
	    for $p:\pbicat\rightarrow \cbicat$ and $q:\qbicat\rightarrow \dbicat$ fibrations. The proof goes as follows: by definition of pseudopullback, a pair $(G,\gamma)$ as in the diagram
	    \[
	    \begin{tikzcd}
	        \pbicat\ar[r, "G"] \ar[dr, "p"', ""{name=A}]&{ \cbicat\times_\dbicat\qbicat} \ar[d,"F^*(q)"]\ar[to=A, Rightarrow, "\gamma", "\sim"{above, sloped}]\\
	        &\cbicat
	    \end{tikzcd}
	    \]
	    corresponds to a pair $(G_1,\gamma_1)$ as in the diagram
	    \[
	    \begin{tikzcd}
	    \pbicat \ar["p"',d]  \ar[r, "G"]& \qbicat\ar[d,"q"] \ar[dl, "\gamma_1", "\sim"{above, sloped}, Rightarrow] \\
	    \cbicat\ar[r,"F"'] & \dbicat
	    \end{tikzcd};
	    \]
	    since $q$ is a cloven fibration we can apply Corollary \ref{cor:genelts_aggiunto_G}, and we obtain a unique pair $(G_2, \gamma_2)$ as in the diagram
	    \[
	    \begin{tikzcd}
	        {\comma{1_\dbicat}{F\circ p}} \ar[rd,"\pi_\dbicat"', ""{name=A}] \ar[r, "G_2"] & \qbicat\ar[d,"q"] \ar[to=A, Rightarrow, "\gamma_2", "\sim"{above, sloped}] \\
	        & \dbicat
	    \end{tikzcd},
	    \]
	    and moreover the functor $G_2$ is a morphism of fibrations. What we said so far is true for any \emph{functor} $G$: one can then verify explicitly that $G$ is a morphism of fibrations if and only if $G_2$ inverts all arrows in $S$, and thus factors through the localization as a functor $\bar{G}:\comma{1_\dbicat}{F\circ p}\rightarrow \qbicat$. The correspondence between $G$ and $\bar{G}$ describes (on objects) the equivalence between hom-categories given by the adjunction.
	    
	    \item The explicit description of inverse images can be notably applied to describe the inverse image of a stack along a morphism of sites, as well as the essential image of a stack along a continuous comorphism of sites (see Section \ref{sec:dirinv_image_stack}).
	\end{enumerate}
\end{remarks}
Let us sum up all the categories at hand in the following diagram:
\[\begin{tikzcd}[column sep=huge, row sep=huge]
	{\comma{D}{F\circ p}} & {\comma{1_\dbicat}{F\circ p}} & {\gbicat(\ecat)} \\
	{\comma{D}{F}} & {\comma{1_\dbicat}{F}} & \cbicat \\
	\onecat & \dbicat
	\arrow["p", from=1-3, to=2-3]
	\arrow["{\pi'}", from=1-2, to=2-2]
	\arrow["{\pi_\ecat}", from=1-2, to=1-3]
	\arrow["{\pi_\cbicat}", from=2-2, to=2-3]
	\arrow[hook, from=2-1, to=2-2]
	\arrow[hook, from=1-1, to=1-2]
	\arrow["{(\pi^D)'}"', from=1-1, to=2-1]
	\arrow["{!}"', from=2-1, to=3-1]
	\arrow["{\pi_\dbicat}", from=2-2, to=3-2]
	\arrow["{e_D}"', from=3-1, to=3-2]
	\arrow["\lrcorner"{anchor=center, pos=0.125}, draw=none, from=2-1, to=3-2]
	\arrow["\lrcorner"{anchor=center, pos=0.125}, draw=none, from=1-2, to=2-3]
	\arrow["\lrcorner"{anchor=center, pos=0.125}, draw=none, from=1-1, to=2-2]
	\arrow["r"'{pos=0.2}, bend right=50, from=1-2, to=3-2, crossing over]
	\arrow["{\pi^D_\cbicat}"'{pos=0.2}, bend right=10, from=2-1, to=2-3, end anchor={south west}, start anchor={south east}, crossing over]
	\arrow[from=1-1, to=1-3, bend left=10, start anchor={north east}, end anchor={north west}, "\pi^D_\ecat"]
\end{tikzcd},\]
where all the squares are pullbacks since they correspond to the computation of direct images of fibrations: this implies that all the vertical functors are all fibrations. On the other hand, $\pi_\cbicat=p_{\tau_F\Vop}\op$ is an opfibration, and thus its pullback $\pi_\ecat$ is also an opfibration. The fibration $\Lan_{F\op}[p]$ is the localization of $\comma{1_\dbicat}{F\circ p} $ at the class of arrows that are $r$-vertical and whose image via $\pi_\ecat$ is $p$-cartesian; similarly, the single fibre $\Lan_{F\op}(\ecat)(D)$ is the localization of $\comma{D}{F\circ p}$ at the class of arrows whose image via $\pi^D_\ecat$ is $p$-cartesian: notice that in this case there is no need to consider the vertical arrows, since every arrow in $\comma{D}{F\circ p}$ is vertical with respect to the unique functor to $\onecat$.

\section{Change of base for stacks}\label{sec:dirinv_image_stack}

We now reintroduce topologies into the discourse: starting from two sites $(\cbicat,J)$ and $(\dbicat,K)$, we wish to study direct and inverse images of stacks along a functor $F:\cbicat\rightarrow \dbicat$. In analogy with the case of presheaves, the main ingredient is the \emph{$(J,K)$-continuity} of $F$, which was necessary to restrict the adjunction $\lan_{F\op}\dashv F^*:[\dbicat\op,\Set]\rightarrow[\cbicat\op,\Set]$ to a weak geometric morphism $\Sh(F):\Sh(\dbicat,K)\rightarrow \Sh(\cbicat,J)$ (see Proposition \ref{prop:ft_cont_caratterizzazione}). 
 
Let us begin by the following technical lemma:
\begin{lemma}
	Consider two small-generated sites $(\cbicat,J)$ and $(\dbicat,K)$ and a functor $F:\cbicat\rightarrow \dbicat$: then $F$ is $(J,K)$-continuous if and only if for every $J$-covering sieve $m:R\rightarrowtail \yo(X)$ the arrow $\sheafify_K(\lan_{F\op}(m))$ of $\Sh(\dbicat,K)$ is an isomorphism. 
\end{lemma}	
\begin{proof}
	This can be immediately proven by considering, for every $J$-covering sieve $m:R\rightarrowtail \yo(X)$ and every $K$-sheaf $P$, the commutative square of Hom-sets
	\[
	\begin{tikzcd}
		{\Sh(\dbicat,K)(\sheafify_K\lan_{F\op}(\yo(X)),P)} \ar[r, no head, "\sim"] \ar[d, "{-\circ \sheafify_K\lan_{F\op}(m)}"']& {[\cbicat\op,\Set](\yo(X), P\circ F\op)} \ar[d, "{-\circ m}"]\\
		{\Sh(\dbicat,K)(\sheafify_K\lan_{F\op}(R),P)} \ar[r, no head, "\sim"]& {[\cbicat\op,\Set](R, P\circ F\op)}
	\end{tikzcd},
	\]
	where the horizontal isomorphisms are given by the adjunction $\sheafify_K\lan_{F\op} \dashv (-\circ F\op)i_K$. The left-hand vertical map is an isomorphism if and only if $\sheafify_K\lan_{F\op}(m)$ is an isomorphism, while the right-hand one if and only if $F\circ R\op$ is a $J$-sheaf, \ie if and only if $F$ is $(J,K)$-continuous.
\end{proof}
\begin{remark}
	This result can also be obtained from Exposé II, Proposition 5.3 and Exposé III, Proposition 1.2 of \cite{SGA4_I}. 
\end{remark}
We are now ready to state that continuous functors are precisely those whose direct image preserves the property of being a stack, extending the content of Proposition \ref{prop:ft_cont_caratterizzazione}(i)-(ii):
\begin{prop}\label{prop:ftcont_preservano_stack}
	Consider two sites $(\cbicat,J)$ and $(\dbicat,K)$ and a functor $F:\cbicat\rightarrow\dbicat$: then $F$ is $(J,K)$-continuous functor if and only if $F^*:\Ind_\dbicat\rightarrow \Ind_\cbicat$ restricts to a 2-functor $\St(\dbicat,K)\rightarrow \St(\cbicat,J)$.
\end{prop}
\begin{proof}
	One implication is obvious: if $F_*$ maps $K$-stacks to $J$-stacks, in particular it maps $K$-sheaves to $J$-sheaves, and thus $F$ is $(J,K)$-continuous.
	Now suppose instead that $F$ is $(J,K)$-continuous, and consider a $K$-stack $\ecat:\dbicat\op\rightarrow\CAT$: we have to show that $\ecat\circ F\op:\cbicat\op\rightarrow\CAT$ is a $J$-stack, \ie that for every $J$-covering sieve $m:S\rightarrowtail \yo(X)$ the functor
	\[
	\Ind_\cbicat(\yo(X), \ecat\circ F\op)\xrightarrow{-\circ m} \Ind_\cbicat(S, \ecat\circ F\op)
	\]
	is an equivalence (see Definition \ref{def:stack_ind}). To do so, we exploit the 2-adjunction $s_K\Lan_{F\op}\dashv (-\circ F\op)i_K$: the functor $-\circ m$ translates into the functor
	\[
	L'_S:\St(\dbicat,K)(s_K\Lan_{F\op}(\yo(X)),\ecat)\rightarrow \St(\dbicat,K)(s_K\Lan_{F\op}(S),\ecat)
	\]
	acting by precomposition with $-\circ s_K\Lan_{F\op}(m)$. The 2-functors $s_K$ and $\Lan_{F\op}$ act on presheaves respectively like $\sheafify_K$ and $\lan_{F\op}$: but by the previous lemma $F$ is $(J,K)$-continuous if and only if each $\sheafify_K\lan_{F\op}(m)\simeq s_K\Lan_{F\op}(m)$ is an isomorphism of sheaves, and thus $L_S'$ is an equivalence.
\end{proof}
\begin{remark}
	An explicit proof of this, though burdened by lengthy calculations, is not hard to provide. Recall from Section \ref{sec:stack} that a pseudonatural transfomation $S\Rightarrow \ecat\circ F\op$ can be interpreted as \emph{descent data} $(U_y, \alpha_{y,z})_{y\in S}$, \ie the given of objects $U_y\in \ecat(F(Y))$  for every $y:Y\rightarrow X$ in $ S(Y)$, and of a family of isomorphisms $\alpha_{y,z}: \ecat(F(z))(U_y)\simeq U_{yz}$ of $\ecat(F(Z))$ for every $y$ in $S$ and every $z:Z\rightarrow Y$, satisfying suitable compatibility properties. Since $F$ is continuous it is in particular cover-preserving, and hence $F(S)$ generates a $K$-covering sieve $R$ over $F(X)$ in $\dbicat$. Now, starting from the descent data $(U_y, \alpha_{y,z})$, one can obtain data $(U'_f, \alpha'_{f,g})_{f\in R}$ for $\ecat$. Indeed, any arrow $f\in R$ is of the kind $F(y)h$ for some $y\in S$, and thus we set $U'_f=\ecat(h)(U_{y})$, and the isomorphisms $\alpha'$ are then defined in the obvious way: the fact that this yields well-defined descent data for $\ecat$ can be proved exploiting the cofinality condition for $(J,K)$-continuous functors in Proposition \ref{prop:ft_cont_caratterizzazione}(iii). Now we can work with $\ecat$, which we assumed to be a stack: therefore the descent data $(U'_f, \alpha'_{f,g})_{f\in R
	}$ admit a `gluing' $U'\in \ecat(F(X))$, which one can the prove is also a gluing for the original descent datum $(U_y, \alpha_{y,z})$ of the indexed category $\ecat\circ F\op$. A similar argument works for morphisms of descent data, proving that the functor $-\circ m$ in the previous proposition is in fact an equivalence.
\end{remark}
We can now conclude that, similarly to what happens with sheaves, $(J,K)$-continuous functors and in particular morphisms of sites induce an adjunction between categories of stacks:
\begin{cor}\label{cor:morphsites_functor_between_stacks}
    Consider a morphism of sites $F:(\cbicat,J)\rightarrow(\dbicat,K)$, or more generally a $(J,K)$-continuous functor: it induces a 2-adjunction
    \[
    \begin{tikzcd}
	{\St^s(\cbicat,J)} \arrow[r, "\St(F)^*", bend left, start anchor={north east}, end anchor={north west}] &   {\St^s(\dbicat,K)} \arrow[l, "\St(F)_*", bend left, start anchor={south west}, end anchor={south east}] \ar[l, "\dashv"{rotate=270}, phantom]
    \end{tikzcd},
    \]    
    whose pair we shall refer to simply by $\St(F)$\index{$\St(F)$}.
    The 2-functor $\St(F)_{\ast}$ is called the \emph{direct image of stacks along $F$} and acts as the precomposition \[F^*:=(-\circ F\op):\Ind_\dbicat\rightarrow \Ind_\cbicat;\] In terms of fibrations, a stack $q:\ebicat\rightarrow\dbicat$ is mapped by $\St(F)_*$ to its strict pseudopullback $p:\pbicat\rightarrow\cbicat$ along $F$.
    The left adjoint $\St(F)^{\ast}$ is the \emph{inverse image of stacks along $F$} and acts as the composite
    \[
    \St^s(\cbicat,J) \xrightarrow{i_J} \Ind^s_\cbicat \xrightarrow{\Lan_{F\op}}\Ind^s_\dbicat\xrightarrow{s_K}\St^s(\dbicat,K),
    \]
    where $s_K$ denotes the stackification functor and $\Lan_{F\op}$ can be computed as in Proposition \ref{prop:pseudoKan}. In terms of fibrations, a stack $p:\pbicat\rightarrow\cbicat$ is mapped by $\St(F)^*$ to the stackification of its inverse image $\Lan_{F\op}([p])$ along $F$, which can be computed as in Proposition \ref{prop:inverseimage_as_localization}.
    
    In particular, any geometric morphism $f:\Ftopos\to \Etopos$ induces a pair of adjoint functors, denoted by $\St(f)$: 
	\[
	\begin{tikzcd}
		{\St^s(\Etopos)} \arrow[r, "\St(f)^*", bend left, start anchor={north east}, end anchor={north west}] &   {\St^s(\Ftopos)} \arrow[l, "\St(f)_*", bend left, start anchor={south west}, end anchor={south east}] \ar[l, "\dashv"{rotate=270}, phantom]
	\end{tikzcd}.
	\] 	
\end{cor}
\begin{proof}
    The fact that $s_K\circ \Lan_{F\op}\circ i_J$ is adjoint to $\St(F)_*$ is a standard argument. Proposition \ref{prop:ftcont_preservano_stack} showed that the precomposition $-\circ F\op$ maps stacks to stacks when $F$ is $(J,K)$-continuous, and justifies the description of $\St(F)_*$ in $\cbicat$-indexed terms; its description in fibrational terms as a pseudopullback comes instead from  Proposition \ref{prop:directimage_Street_fibrations}. The last claim holds since any geometric morphism $f:\Ftopos\to \Etopos$ can be seen as a morphism of sites $f^{\ast}:({\Etopos}, J\can_\Etopos)\to (\Ftopos, J\can_\Ftopos)$.
\end{proof}

As already mentioned in Subsection \ref{sec:canonical_stack}, one important consequence of the fact that the direct image along continuous functors preserves stacks is the fact that all canonical fibrations are stacks, if those of toposes are:
\begin{cor}\label{cor:canst_via_immdiretta}
	Suppose that for every topos $\Etopos$ its canonical fibration $\canst_{(\Etopos, J\can_\Etopos)}$ is a stack: then for any site $(\cbicat,J)$ its canonical fibration $\canst_{(\cbicat,J)}$ is a stack.
\end{cor}
\begin{proof}
	Consider the topos $\Etopos=\Sh(\cbicat,J)$: then one can check immediately that the direct image of $\canst_{(\Etopos, J\can_\Etopos)}$ along the morphism of sites
	\[
	(\cbicat,J) \xrightarrow{\ell_J} (\Etopos, J\can_\Etopos),
	\]
	is the canonical fibration $\canst_{(\cbicat,J)}$. But $\ell_J$ is $(J, J\can_\Etopos)$-continuous and the canonical fibration $\canst_{(\Etopos, J\can_\Etopos)}$ is a $J\can_\Etopos$-stack, thus $\canst_{(\cbicat,J)}$ is a $J$-stack.
\end{proof}

As we know, when $F:(\cbicat,J)\rightarrow (\dbicat,K)$ is not only $(J,K)$-continuous but a morphism of sites, the adjunction $\Sh(F)$ is actually a geometric morphism, \ie the inverse image $\Sh(F)^*$ is a lex functor: the same holds for $\St(F)^*$, as we will show in a moment using internal categories. We have avoided exploiting the theory of internal categories as much as possible for the reasons already expressed in the introduction: even though internal categories in a topos are closely related to stacks over the site of definition (we will recall how in the proof of the next result), the passage from stacks to internal categories requires a process of `rigidification' which is not only cumbersome for many practical uses, but also unnatural in a theory that was born precisely to give more `elasticity' to the theory of sheaves. However, in this case the formalism of internal categories proves to be much more efficient than an explicit check on the limit preservation of $\Lan_{F\op}$, which is why we will exploit it.

Let us recall briefly that given any category $\ebicat$ with finite limits, an \emph{internal category} $\acat$ in $\ebicat$ is the given of an object of objects $A_0$ and of an object of arrows $A_1$, along with morphisms representing domain, codomain etc., pasted together into some commutative diagrams (for more details see \cite[Section B2.3]{elephant}). Internal categories of $\ebicat$ form a 2-category $\Cat(\ebicat)$: since internal categories can be described uniquely by commutative diagrams and pullbacks, they are preserved by any lex functor\footnote{From a logical standpoint, this holds as internal categories are models for a cartesian theory: for more details see \cite[Example D1.2.15(e)]{elephant}}. When $\ebicat$ is a topos of sheaves $\Sh(\cbicat,J)$, objects of $\Cat(\Sh(\cbicat,J))$ can be seen as split small discrete $J$-stacks of $\cbicat$ by a form of `externalization' which associates $\acat$ with a strict functor $\bar{\acat}:\cbicat\op\rightarrow\Cat$, by defining, for every $X$ in $\cbicat$, the category $\bar{\acat}(X)$ as the category whose set of objects is $A_0(X)$ and whose set of arrows is $A_1(X)$.

With this in mind, we are now ready to prove the following result:
\begin{prop}\label{prop:morphsites_stacks_intcategories}
	Consider a morphism of sites $F:(\cbicat,J)\rightarrow (\dbicat,K)$: then the adjoints $\St(F)^*\dashv \St(F)_*:\St^s(\cbicat,J)\rightarrow \St^s(\dbicat,K)$ act on internal categories as the adjoints $\Sh(F)^*$ and $\Sh(F)_*$.
	
	In particular, given two sites $(\cbicat,J) $ and $(\dbicat,K)$, a functor $F:\cbicat\rightarrow \dbicat$ is a morphism of sites if and only if $\St(F)^*$ preserves finite pseudolimits of small stacks.
\end{prop}
\begin{proof}
	Consider the diagram
	\[\begin{tikzcd}
		{\St^s(\dbicat,K)} & {\St^s(\cbicat,J)} \\
		{\Spl\St^s(\dbicat,K)} & {\Spl\St^s(\cbicat,J)} \\
		{\Cat(\Sh(\dbicat,K))} & {\Cat(\Sh(\cbicat,J))}
		\arrow["\sim"sloped, no head, from=3-1, to=2-1]
		\arrow[from=2-1, to=1-1]
		\arrow["\sim"{sloped, below}, no head, from=3-2, to=2-2]
		\arrow[from=2-2, to=1-2]
		\arrow["{\Sh(F)_*}"', from=3-1, to=3-2]
		\arrow["{\St(F)_*}", from=1-1, to=1-2]
		\arrow["{\overline{(-)}}", from=3-1, to=1-1, bend left=70]
		\ar["{\overline{(-)}}"', from=3-2, to=1-2, bend right=70]		
	\end{tikzcd},\]
	where the bent arrows represent the externalization of an internal category, and $\Spl\St$ denotes the 2-category of split stacks (in analogy with $\Fib$ and $\Spl\Fib$). By an abuse of notation we call $\St(F)_*$ the lower horizontal arrow, which is actually the action of the lex functor $\St(F)_*$ on internal categories. Given an internal category $\acat$ in $\Sh(\dbicat,K)$, we can now compute the two stacks $\St(F)_*(\bar{\acat})$ and $\overline{\Sh(F)_*(\acat)}:\cbicat\op\rightarrow \Cat$: for each object $X$ in $\cbicat$, one immediately checks that 
	both categories $\St(F)_*(\bar{\acat})(X)$ and $\overline{\Sh(F)_*(\acat)}(X)$ have the set of objects $A_0(F(X))$, because the action of $\St(F)_*$ and $\Sh(F)_*$ is just precomposition with $F\op$; the same holds for the set of morphisms, and all the relevant data (domains, codomains, etc.) are the same. This means that the diagram above is commutative, meaning that $\St(F)_*$ and $\Sh(F)_*$ act on internal categories essentially in the same way. In turn, this entails that the two left adjoints $\St(F)^*$ and $\Sh(F)^*$ act essentially in the same way on internal categories.
	
	The last claim is now obvious: if $\St(F)^*$ preserves finite limits, its restriction to sheaves $\Sh(F)^*$ preserves \emph{a fortiori} finite limits of sheaves, and thus $F$ is a morphism of sites. Conversely, if $F$ is a morphism of sites $\Sh(F)^*$ preserves finite limits: in particular so it does with finite limits of internal categories. But since the action of $\Sh(F)^*$ and $\St(F)^*$ is essentially the same on internal categories, it follows that $\St(F)^*$ must also preserve said limits.
\end{proof}

\begin{remarks}
	\begin{enumerate}[(i)]
		\item 	This is an extension of Corollaire 3.2.9 of \cite[Chapter II]{giraud.cohomologie}, where it is shown that if $F:(\cbicat,J)\rightarrow (\dbicat,K)$ is a morphism of sites then $\St(F)^*$ preserves finite products of stacks.
		\item Consider a functor $A$ from a site $\cbicat$ to a Grothendieck topos $\Etopos$: then $A$ is $(J, J\can_\Etopos)$-continuous if and only if it is $J$-continuous in the sense of Section VII.7 of \cite{maclanemoerdijk}: \ie, it maps $J$-covering sieves to colimit cocones, or analogously if for each $E$ in $\Etopos$ the presheaf
		\[
		\Etopos(A(-), E):\cbicat\op\rightarrow \Set
		\]
		is a $J$-sheaf. Now, in light of the equivalence
		\(
		{\Etopos}\simeq \Sh({\Etopos}, J\can_{\Etopos}),
		\)
		we can consider instead presheaves of the form
		\[
		{\Sh(\Etopos, J\can_{\Etopos})}((\ell\circ A)(-), W),
		\]
		where $W$ is a $J\can_{\Etopos}$-sheaf on $\Etopos$ and $\ell:{\Etopos}\to \Sh(\Etopos, J\can_\Etopos)$  is the canonical functor given by the Yoneda embedding. Now let us denote by $L:\Etopos\to \St(\Etopos)$ the pseudofunctor sending any $E$ in $\Etopos$ to the presheaf $\yo_\Etopos(E)$, which is a $J\can_\Etopos$-sheaf and hence a $J\can_\Etopos$-stack: We could formulate an analogous condition of `$J$-stack-continuity' for $A$, asking that for any \emph{stack} $\dcat$ on $\Etopos$, the pseudofunctor
		\[
		{\St({\Etopos})}((L\circ A)(-), \dcat):\cbicat\op\rightarrow \Cat
		\]
		is a $J$-stack. Unsurprisingly, the two conditions are equivalent. Of course, $J$-stack-continuity implies immediately $J$-continuity. On the other hand, suppose that $A$ is $J$-continuous: this means that $\St(A)_*=(-\circ A\op)$ maps $J\can_\Etopos$-stacks to $J$-stacks. But it follows immediately from the definitions and from fibred Yoneda's lemma that
		\[
		\St(\Etopos)((L\circ A)(-),\dcat)\simeq \dcat(A(-))= \St(A)_*(\dcat),
		\]
		and thus $J$-continuity implies $J$-stack-continuity holds.
	\end{enumerate}
\end{remarks}
Proposition \ref{prop:morphsites_stacks_intcategories} has also the following fundamental consequence:
\begin{cor}
Given a site $(\cbicat,J)$, the category of $J$-stacks over $\cbicat$ is equivalent to the category of stacks over $\Sh(\cbicat,J)$ for the canonical topology: 
\[
\St^s(\cbicat,J)\simeq \St^s(\Sh(\cbicat,J), J\can_{\Sh(\cbicat,J)}).
\]
\end{cor}
\begin{proof}
It is well known that the functor $\ell_J:\cbicat\rightarrow \Sh(\cbicat,J)$ is a $(J,J\can_{\Sh(\cbicat,J)})$-morphism of sites inducing an equivalence
\[
\Sh(\ell_J): \Sh(\Sh(\cbicat,J), J\can_{\Sh(\cbicat,J)})\isorightarrow \Sh(\cbicat,J).
\]
This implies that $\Sh(\ell_J)$ acts as an equivalence on the internal categories of the two toposes: but since every $J$-stack is equivalent to split $J$-stack, and hence to an internal category, by fibred Yoneda's lemma, the functor
\[
\St(\ell_J): \St^s(\Sh(\cbicat,J), J\can_{\Sh(\cbicat,J)})\rightarrow \St^s(\cbicat,J);
\]
is an equivalence.
\end{proof}

In a similar way to morphisms of sites, comorphisms of sites also induce an adjunction between categories of stacks, this time by restricting the action of the right Kan extension:
\begin{prop}\label{prop:comorphism_action_on_stacks}
Consider a comorphism of sites $F:(\cbicat,J)\rightarrow(\dbicat,K)$: it induces a 2-adjunction
\[
    \begin{tikzcd}
	{\St^s(\dbicat,K)} \arrow[r, "(C_F^\St)^*", bend left, start anchor={north east}, end anchor={north west}] &   {\St^s(\cbicat,J)} \arrow[l, "(C_F^\St)_*", bend left, start anchor={south west}, end anchor={south east}] \ar[l, "\dashv"{rotate=270}, phantom]
    \end{tikzcd},
    \]
whose pair we shall refer to by $C_F^\St$\index{$C_F^\St$}. The right adjoint $(C_F^\St)_*$ acts by restriction of the right Kan extension $\Ran_{F\op}$ (see Proposition \ref{prop:pseudoKan}) to stacks; on the other hand, the left adjoint $(C_F^\St)^*$ acts as the composite 2-functor
\[
\St^s(\dbicat,K)\xrightarrow{i_K} \Ind^s_\dbicat \xrightarrow{F^*} \Ind^s_\cbicat \xrightarrow{s_J} \St^s(\cbicat,J),
\]
where $F^*:=(-\circ F\op)$ and $s_J$ is the stackification functor. Moreover, there is an equivalence
\[(C_F^\St)^*\circ s_K\cong s_J\circ F^*.
\]
\end{prop}
\begin{proof}
We only need to show $F$ is a comorphism of sites if and only if the 2-functor $\Ran_{F\op}:\Ind_\cbicat\rightarrow\Ind_\dbicat$ restricts to a 2-functor $(C_F^\St)_*:\St(\cbicat,J)\rightarrow\St(\dbicat,K)$: by standard considerations on adjoints (see Lemma \ref{lemma:adj_restriction_subcat} below) one verifies that $(C_F^\St)^*:=s_J\circ F^*\circ i_K$ provides a left adjoint to $(C_F^\St)_*$. 

Suppose first that $\Ran_{F\op}$ maps $J$-stacks to $K$-stacks: then in particular it maps $J$-sheaves to $K$-sheaves and hence $F$ is a comorphism by condition (ii) of Definition \ref{def:comorfismo}. Conversely, suppose that $F$ is a comorphism, consider a $J$-stack $\dcat:\cbicat\op\rightarrow\Cat$, a $K$-sieve $m_S:S\rightarrowtail \yo(D)$ and the diagram
\[
\begin{tikzcd}
{\Ind_\dbicat(\yo(D), \Ran_{F\op}(\dcat))} \ar[d, "{-\circ m_S}"'] \ar[r, "\sim"] &
{\Ind_\cbicat(\yo(C)\circ F\op, \dcat)}\ar[d, "{-\circ (m_S\circ F\op)}"]
\\
{\Ind_\dbicat(S, \Ran_{F\op}(\dcat))} \ar[r, "\sim"] &
{\Ind_\cbicat(S\circ F\op, \dcat)}
\end{tikzcd},
\]
where the horizontal equivalences come from the adjunction $F^*\dashv \Ran_{F\op}$. By condition (iv) of Definition \ref{def:comorfismo} the precomposition $F^*:=(-\circ F\op):[\dbicat\op,\Set]\rightarrow[\cbicat\op,\Set]$ maps any $K$-dense monomorphism to a $J$-dense monomorphism: in other words, if $m_S:S\rightarrowtail \yo(D)$ is a $K$-sieve, then $F^*(m_S)=(-\circ (m_S\circ F\op):S\circ F\op\rightarrowtail \yo(C)\circ F\op$ is $J$-dense. Since $F^*(m_S)$ is $J$-dense and $\dcat$ is a $J$-stack, the vertical arrow on the right is an equivalence, and thus the one on the left is, showing that $\Ran_{F\op}(\dcat)$ is a $K$-stack.

The last claim is justified by the fact that both $(C_F^\St)^*\circ s_K$ and $s_J\circ F^*$ are left adjoint to $i_K\circ (C_F^\St)_*\cong \Ran_{F\op}\circ i_J$.
\end{proof}
\begin{cor}\label{cor:continuouscomorphismsstacks}
Consider a $(J,K)$-continuous comorphism of sites \[F:(\cbicat,J)\rightarrow (\dbicat,K):\] 
it induces a triple of adjoints
\[\begin{tikzcd}[column sep=10ex]
{\St^s(\cbicat,J)} \arrow[r, "{(C_F^\St)_!}", bend left=45, start anchor={north east}, end anchor={north west}] 
\arrow[r, "{(C_F^\St)_*}"', bend right=45, start anchor={south east}, end anchor={south west}]
& {\St^s(\dbicat,K)} 
\arrow[l, "{(C_F^\St)^*}"description] \ar[l, "\dashv"{rotate=270, xshift=2.5ex}, phantom] \ar[l, "\dashv"{rotate=270, xshift=-2.5ex}, phantom] 
\end{tikzcd},
\]
where $(C_F^\St)_*$ acts by restriction of $\Ran_{F\op}$, $(C_F^\St)^*=\St(F)_*$ acts by restriction of $F^*$ and $(C_F^\St)_!=\St(F)^*$ is the composite 2-functor 
\[\St^s(\cbicat,J)\xrightarrow{i_J}\Ind^s_\cbicat\xrightarrow{\Lan_{F\op}}\Ind^s_\dbicat \xrightarrow{s_K} \St^s(\dbicat,K).\] 
\end{cor}
\begin{proof}
This is an immediate consequence of the previous results: the action of $\Ran_{F\op}$ restricts to stacks since $F$ is a comorphism, while that of $F^*$ restricts to stacks since $F$ is continuous. 
\end{proof}

\section{Operations on sheaves \emph{versus} operations on stacks}

In Section \ref{sec:truncation_functor} we introduced the truncation-inclusion adjunction
\[
\begin{tikzcd}
 {\Sh(\cbicat,J)} 
   \ar[r, hook, "j_J"', yshift=-1ex]& 
 {\St^s(\cbicat,J)},
   \ar[l, "\trunc_J"', yshift=1ex] 
   \ar[l,phantom,"\vdash"{rotate=90}]
\end{tikzcd}
\]
which describes the connection between $J$-sheaves and $J$-stacks: the same functors relate the action of morphisms and comorphisms of sites on stacks to their action on sheaves. To study this relation, let us first recall a technical result about adjunctions:
\begin{lemma}\label{lemma:adj_restriction_subcat}
Consider a diagram of 2-functors
\[
\begin{tikzcd}[column sep=large]
\abicat \ar[r, "R"', yshift=-1ex] \ar[dr, "R'"', start anchor={south east}] & \cbicat \ar[l, "L"', yshift=1ex] \ar[l,phantom, "\vdash"{rotate=90}]\\
 & \dbicat \ar[u, "i"', hook]
\end{tikzcd}
\]
and suppose $R\cong i\circ R'$: then $L\circ i\dashv R'$.
\end{lemma}
This can be applied to a $(J,K)$-continuous functor as follows:
\begin{prop}\label{prop:contft_stacks_sheaves}
Consider a continuous functor $F:(\cbicat,J)\rightarrow(\dbicat,K)$. The the inverse image
$\Sh(F)^*:\Sh(\cbicat,J)\rightarrow\Sh(\dbicat,K)$ is isomorphic to the composite functor
\[
\Sh(\cbicat,J) \xhookrightarrow{j_J} \St^s(\cbicat,J) \xrightarrow{\St(F)^*} \St^s(\dbicat,K)\xrightarrow{\trunc_K} \Sh(\dbicat,K).
\]
In particular, if we denote by $\trunc_\cbicat\dashv j_\cbicat:[\cbicat\op,\Set]\hookrightarrow\Ind_\cbicat$ the truncation-inclusion adjunction in the case of the trivial topology on $\cbicat$, then $\lan_{F\op}:[\cbicat\op,\Set]\rightarrow[\dbicat\op,\Set]$ is isomorphic to the composite 
\[
[\cbicat\op,\Set] \xhookrightarrow{j_\cbicat} \Ind^s_\cbicat \xrightarrow{\Lan_{F\op}} \Ind^s_\dbicat\xrightarrow{\trunc_\dbicat} [\dbicat\op,\Set].
\]
\end{prop}
\begin{proof}
Since the 2-functor $\St(F)_*:\St(\dbicat,K)\rightarrow \St(\cbicat,J)$ acts by precomposition with $F\op$, its action restricts to that of $\Sh(F)_*$ on sheaves: but then we can consider the diagram
\[
\begin{tikzcd}
    {\Sh(\dbicat,K)} \ar[drr, "\Sh(F)_*"', start anchor={south}, bend right=10] \ar[r, hook, "j_K"', yshift=-1ex] & {\St^s(\dbicat,K)} \ar[r, "{\St(F)_*}"', yshift=-1ex] \ar[l, "\trunc_K"', yshift=1ex] 
   \ar[l,phantom,"\vdash"{rotate=90}]& {\St^s(\cbicat,J)} \ar[l, "{\St(F)^*}"', yshift=1ex] 
   \ar[l,phantom,"\vdash"{rotate=90}] \\
    && {\Sh(\cbicat,J)} \ar[u, hook, "j_J"']
\end{tikzcd},
\]
and by applying Lemma \ref{lemma:adj_restriction_subcat} it follows that the left adjoint $\Sh(F)^*$ is isomorphic to the composite $\trunc_K\circ \St(F)^*\circ j_J$.
\end{proof}
Similarly, the geometric morphism induced by a comorphism of sites is obtained by truncating the adjunction between categories of stacks as follows:
\begin{prop}
Consider a comorphism of sites $F:(\cbicat,J)\rightarrow(\dbicat,K)$. The inverse image $C_F^*:\Sh(\dbicat,K)\rightarrow \Sh(\cbicat,J)$ of the geometric morphism induced by $F$ is isomorphic to the composite functor
\[
\Sh(\dbicat,K) \xhookrightarrow{j_K}\St^s(\dbicat,K)\xrightarrow{(C^\St_F)^*}\St^s(\cbicat,J)\xrightarrow{\trunc_J}\Sh(\cbicat,J).\]
\end{prop}
\begin{proof}
Notice that the two inclusions
\[
\Sh(\cbicat,J)\xhookrightarrow{\iota_J}[\cbicat\op,\Set]\xhookrightarrow{j_\cbicat}\Ind_\cbicat,\ \Sh(\cbicat,J)\xhookrightarrow{j_J}\St(\cbicat,J)\xhookrightarrow{i_J}\Ind_\cbicat
\]
are equal, and thus their left adjoints are isomorphic: $\sheafify_J\circ \trunc_\cbicat\cong \tau_J \circ s_J$. But then the following holds:
\begin{align*}
    \trunc_J\circ (C_F^\St)^*\circ j_K & = \trunc_J\circ s_J\circ F^*\circ i_K\circ j_K\\
    &\cong \sheafify_J\circ \trunc_\cbicat\circ F^*\circ j_\dbicat\circ \iota_K .
\end{align*}
As we know, the 2-functor $F^*:\Ind_\dbicat\rightarrow\Ind_\cbicat$ restricts to presheaves: if we adopt (with a slight abuse of notation) the same symbol for the restriction $F^*:[\dbicat\op,\Set]\rightarrow[\cbicat\op,\Set]$, we have thus the isomorphism $F^*\circ j_\dbicat\cong j_\cbicat\circ F^*$, and hence
\[
\sheafify_J\circ \trunc_\cbicat\circ F^*\circ j_\dbicat\circ \iota_K\cong \sheafify_J\circ \trunc_\cbicat\circ j_\cbicat\circ F^* \circ \iota_K
\]
Finally, since $j_\cbicat$ is fully faithful with left adjoint $\trunc_\cbicat$, the composite $\trunc_\cbicat\circ j_\cbicat\cong$ is isomorphic to the identity functor of $[\cbicat\op,\Set]$, and therefore we may conclude that
\[
\trunc_J\circ (C_F^\St)^*\circ j_K\cong \sheafify_J\circ F^* \circ \iota_K\cong C_F^*.
\]

\end{proof}

\subsection{A fibrational description of $\Sh(F)$}

Applying our knowledge about base change of stacks, we can describe in a fibrational fashion base change functors for sheaves.
\begin{lemma}\label{lemma:sheafify_lan_=_compr_fact}
Let $F:{\cal C}\to {\cal D}$ be a functor, $K$ a Grothendieck topology on $\cal D$ and $\ell_K:\dbicat\rightarrow \Sh(\dbicat,K)$ be the canonical functor. Then the functor $a_{K}\circ \lan_{F\op}:[{\cal C}\op, \Set]\to \Sh({\cal D}, K)$ admits the following fibrational description: for any presheaf $P$ on $\cal C$,
\[
\sheafify_{K}\circ \lan_{F\op}(P)=\colim(\ell_K\circ F\circ \pi_{P}).
\]
In other words, $\sheafify_{K}\circ \lan_{F\op}(P)$ coincides with the discrete part of the $K$-comprehensive factorization of the composite functor $F\circ \pi_{P}$. 
\end{lemma}
\begin{proof}
From $P\simeq \colim (\yo_\cbicat\circ \pi_P)$ it follows that
\begin{align*}
    \sheafify_K\circ \lan_{F\op}(P) &\simeq \sheafify_K\circ \lan_{F\op}(\colim(\yo_\cbicat\circ \pi_P) )\\
    &\simeq \colim(\sheafify_K\circ \lan_{F\op}\circ \yo_\cbicat\circ \pi_P) \\
    &\simeq \colim(\sheafify_K\circ \yo_\dbicat\circ F \circ \pi_P) \\
    &\simeq \colim(\ell_K\circ F \circ \pi_P).
    \end{align*}
The last claim is just the definition of $K$-comprehensive factorization (see Definition \ref{def:comprehensive_factorization}).
\end{proof}

\begin{prop}\label{prop:relativecomprehensivefactsheaves}
Let $F:({\cal C}, J)\to ({\cal D}, K)$ be a $(J, K)$-continuous functor.
Then the inverse image $\Sh(F)^*: \Sh({\cal C}, J)\to \Sh({\cal D}, K)$ of the weak morphism of toposes $\Sh({\cal D}, K)\to \Sh({\cal C}, J)$ induced by $F$ admits the following fibrational description: for any $J$-sheaf $P$ on $\cal C$, $\Sh(F)^*(P)$ coincides with the discrete part of the $K$-comprehensive factorization of the composite functor $F\circ \pi_{P}$. 

In particular, if $F$ is a continuous comorphism of sites $({\cal C}, J)\to ({\cal D}, K)$, the essential image $(C_{F})_{!}(P)$ of a $J$-sheaf $P$ coincides with the discrete part of the $K$-comprehensive factorization of the composite functor $F\circ \pi_{P}$. From a topos-theoretic point of view, the sheaf $(C_F)_!(P)$ corresponds to the second component of the (terminally connected, local homeomorphism)-factorization of the geometric morphism $C_F\circ \prod_P$:
\[
\begin{tikzcd}[column sep=70pt]
    {\Sh(\cbicat,J)/P} \ar[r, "\mbox{term. conn.}"] \ar[d, "\prod_P"']& {\Sh(\dbicat,K)/(C_F)_!(P)} \ar[d, "\prod_{(C_F)_!(P)}"]\\
    {\Sh(\cbicat,J)} \ar[r, "C_F"']& {\Sh(\dbicat,K)}
\end{tikzcd}.
\]
\end{prop}
\begin{proof}
The two functors $\Sh(F)^*$ and $(C_F)_!$ are defined as the composite
\[
\Sh(\cbicat,J)\xhookrightarrow{\iota_J}[\cbicat\op,\Set]\xrightarrow{\lan_{F\op}} [\dbicat\op,\Set]\xrightarrow{\sheafify_K}\Sh(\dbicat,K),
\]
and then we showed in the previous result that $\sheafify_K\circ \lan_{F\op}$ acts by mapping any presheaf $P$ to the discrete part of the $K$-comprehensive factorization of $F\circ \pi_P$. We have mentioned, after Definition \ref{def:comprehensive_factorization}, that the $K$-comprehensive factorization of a continuous comorphism of sites $p$ induces the (terminally connected, local homeomorphism)-factorization of the corresponding geometric morphism $C_p$: therefore, from the $K$-comprehensive factorization of the continuous comorphism of sites
\[
(\fib P, J_P)\xrightarrow{\pi_P} (\cbicat,J) \xrightarrow{F} (\dbicat,K)
\]
as
\[
(\fib P, J_P) \xrightarrow{\bar{F}} (\fib (C_F)_!(P), J_{(C_F)_!(P)}) \xrightarrow{\pi_{(C_F)_!(P)}} (\dbicat,K),
\]
we can deduce that the (terminally connected, étale)-factorization of the geometric morphism $C_{F\circ \pi_P}\cong C_F\circ C_{\pi_P}$ is given by
\[
\Sh(\fib P, J_P)\xrightarrow{C_{\bar{F}}} \Sh(\fib (C_F)_!(P), J_{(C_F)_!(P)}) \xrightarrow{C_{\pi_{(C_F)_!(P)}}} \Sh(\dbicat,K). 
\]
Finally, from Proposition \ref{prop:fib_discreta_slice_topos} we recall that $\Sh(\fib P, J_P)\simeq \Sh(\cbicat,J)/P$ and $C_{\pi_P}\cong \prod_P$, and similar identities hold for the $K$-sheaf $(C_F)_!(P)$. 
\end{proof}

In particular, when topologies are not involved the change of base can be describe in fibrational terms as follows:
\begin{prop}\label{propdirectinverseimagessheaves}
Consider a functor $F:\cbicat\rightarrow \dbicat$, and two presheaves $P:\cbicat\op\rightarrow\Set$ and $Q:\dbicat\op\rightarrow \Set$ with associated fibrations $\pi_P:\fib P\rightarrow \cbicat$ and $\pi_Q:\fib Q\rightarrow \dbicat$:
\begin{itemize}
    \item the direct image of $Q$ along $F$ is computed as the strict pullback of $\pi_Q$ along $F$:
    \[
    \begin{tikzcd}
        {\fib(F^*(Q))} \ar[d] \ar[r] & {\fib Q} \ar[d, "\pi_Q"] \\
        \cbicat \ar[r, "F"'] & \dbicat \ar[ul, phantom, "\lrcorner" near end]
    \end{tikzcd}
    \]
    \item consider the fibration $r:\comma{1_\dbicat}{F\circ \pi_P}\rightarrow\dbicat$, and denote its class of $r$-vertical arrows by $S$: the fibration associated to the presheaf $\Lan_{F\op}(P)$ is computed as the localization $\bar{r}:\comma{1_\dbicat}{F\circ \pi_P}[S\inv]\rightarrow \dbicat$, and the fibration associated to $\lan_{F\op}(P)$ by taking the second component of the comprehensive factorization of $\bar{r}$:
    \[
    \begin{tikzcd}
    {\comma{1_\dbicat}{F\circ \pi_P}} \ar[rr,"r"] \ar[dr] & &  {\dbicat}\\
    & {\comma{1_\dbicat}{F\circ \pi_P}[S\inv]} \ar[ru, "\bar{r}"] \ar[dr] & \\
    & & {\fib (\lan_{F\op}(P))}. \ar[uu, "\pi_{\lan_{F\op}(P)}"', "\scriptsize\mbox{compr. fact.}"{sloped, above}] 
    \end{tikzcd}
    \]
\end{itemize}
\end{prop}
\begin{proof}
The description of the direct image follows directly from its fibrational definition as a pullback. Notice that the pseudopullback of $\pi_P$ along $F$ is in general a Street fibration, but one can show immediately that it is equivalent to the strict pullback of $\pi_P$ along $F$, which is instead a Grothendieck fibration.

The description of $\Lan_{F\op}(P)$ is the one appearing in Proposition \ref{prop:inverseimage_as_localization}: $\Lan_{F\op}(P)$ is computed by localizing $\comma{1_\dbicat}{F\circ \pi_P}$ with respect to the class of its vertical arrows whose component in $\int P$ is cartesian. But, since every arrow in $\int P$ is cartesian, for $P$ is discrete, we conclude that $\bar{r}$ is the fibration associated with $\lan_{F\op}(P)$. Finally, we know by Proposition \ref{prop:contft_stacks_sheaves} that the fibration of $\lan_{F\op}$ can be recovered as the truncation of $\bar{r}$, and by Corollary \ref{cor:truncation_fibrations} the truncation functor $t_\dbicat:\Ind_\dbicat\rightarrow [\dbicat\op,\Set]$ acts on fibrations by mapping to the second component of the comprehensive factorization.
\end{proof}

\subsection{Inverse images and pullbacks}\label{sec:invimage_pb}

At the very beginning of the chapter we have recalled the definition of inverse image for sheaves over topological spaces, whose construction appears for instance as Theorem 2 of \cite[Chapter II, Section 9]{maclanemoerdijk}: given a continuous function $f:X\to Y$ between topological spaces, the inverse image $\Sh(f)^{\ast}$ of the geometric morphism
\[
\Sh(f):\Sh(X)\to \Sh(Y)
\]
induced by $f$ corresponds under the equivalences 
\[
\Sh(X)\simeq \Etale\slash X
\]
and 
\[
\Sh(Y)\simeq \Etale\slash Y
\]
to the pullback operation along $f$ in the category $\Top$ of topological spaces and continuous maps between them. In this section we shall establish natural extensions of this result in the context of arbitrary sites and toposes, where the natural topos-theoretic analogue of the notion of continuous map between topological spaces is the notion of geometric morphism. 

Giraud's article \cite{giraud.classifying} shows in particular that classifying toposes of inverse images of cartesian stacks can be used to compute pullbacks of toposes; his result was later extended to the language of internal categories by Diaconescu as follows:
\begin{thm}[{\cite[Theorem 5.1]{diaconescu.pullback}}]
Consider a geometric morphism $f:\Ftopos\rightarrow \Etopos$ and an internal category $\ccat$ in $\Etopos$. Denote by $[\ccat\op,\Etopos]$ the category of internal presheaves for $\ccat$, and similarly by $[f^*\ccat\op,\Ftopos]$ the category of internal presheaves for $f^*\ccat$ (it is still an internal category in $\Ftopos$, since $f^*$ preserves finite limits). Then the square
\[
\begin{tikzcd}
    {[f^*\ccat\op,\Ftopos]} \ar[r, "\bar{f}"] \ar[d] & {[\ccat\op,\Etopos]} \ar[d]  \\
    {\Ftopos} \ar[r, "f"'] & {\Etopos} 
\end{tikzcd}
\]
is a pullback in $\Topos$, where $\bar{f}$ is the geometric morphism whose inverse image acts as $f^*$ on internal presheaves while the two vertical geometric morphisms are the usual global sections functors.
\end{thm}\qed

This result can be expressed externally using stacks. We have recalled in the proof of Proposition \ref{prop:morphsites_stacks_intcategories} that every stack over a site $(\cbicat,J)$ is equivalent to a split stack, which is nothing but an internal category of $\Sh(\cbicat,J)$; with this in mind, Proposition \ref{prop:morphsites_stacks_intcategories} showed that the action of $\St(F)$ on stacks, seen as internal categories, corresponds to the action of $\Sh(F)$. Finally, given an internal category $\dcat$ in a topos $\Sh(\cbicat,J)$ the equivalence $\Gir_J(\dcat)\simeq [\dcat\op,\Sh(\cbicat,J)]$ holds (\cfr Definition \ref{def:Giraud_topology_classifying_topos} and Proposition C2.5.4 of \cite{elephant}), and hence we can deduce the topos-theoretic analogue of the above-mentioned pullback characterization of the inverse image operation on topological sheaves:
\begin{prop}\label{prop:pb_Giraudtopos_is_inverseimage}
Let $F:(\cbicat,J)\rightarrow (\dbicat,K)$ be a morphism of sites and $\dcat$ a small $J$-stack: then the square
\[
	\begin{tikzcd}[column sep=30]
		\Gir_{K}(\St(F)^*(\dcat)) \arrow[r] \arrow[d, "C_{p_{(\St(f)^*(\dcat))}}"']   & \Gir_{J}({\dcat}) \arrow[d, "C_{p_{{\dcat}}}"]     \\
		{\Sh(\dbicat,K)} \arrow[r, "\Sh(F)"']  &  {\Sh(\cbicat,J)} 
	\end{tikzcd}
\]
is a pullback of toposes. In particular, for a geometric morphism $f:\Ftopos\to \Etopos$ and a stack $\dcat$ over $\Etopos$, the following square is a pullback of toposes:
\begin{equation*}
	\begin{tikzcd}[column sep=30]
		\Gir_{\Ftopos}(\St(f)^{\ast}({\mathbb D})) \arrow[r] \arrow[d, "C_{p_{(\St(f)^{\ast}({\mathbb D}))}}"']  &  \Gir_{\Etopos}({\mathbb D}) \arrow[d, "C_{p_{{\mathbb D}}}"]      \\
		{\Ftopos} \arrow[r, "f"'] &   {\Etopos} 
	\end{tikzcd}.
\end{equation*}  
\end{prop}\qed

\begin{remark}
The generalization of Giraud's construction of pullbacks of toposes for arbitrary stacks can be deduced directly from a relative Diaconescu's equivalence theorem for arbitrary stacks, extending Corollary 2.5 of \cite{giraud.classifying}. Such a result, and its application to the construction of pullbacks of toposes, will be proved in a future version of this work. 
\end{remark}

In particular, for a geometric morphism $C_F$ induced by a continuous comorphism of sites then the above pullback of toposes can be seen as induced by a pseudopullback in $\Cat$:
\begin{prop}
Let $F:({\cal C}, J)\to ({\cal D}, K)$ be a continuous comorphism of sites and $\mathbb D$ a $K$-stack on $\cal D$. Then the diagram 
\begin{equation*}
	\begin{tikzcd}[column sep=30]
		{\cal G}((C_F^\St)^*({\mathbb D})) \arrow[r] \arrow[d, "p_{(C_F^\St)^*({\mathbb D})}"']  &    {\cal G}({\mathbb D}) \arrow[d, "p_{\mathbb D}"]      \\
		{\cal C} \arrow[r, "F"'] &  {\cal D} 
	\end{tikzcd}
\end{equation*}	  
(where the unnamed horizontal functor is the obvious one), is a pseudopullback in $\Cat$, which is sent by the $2$-functor $C$ to a pullback in $\Topos$:
\begin{equation*}
	\begin{tikzcd}[column sep=30]
		\Gir_{J}((C_F^\St)^*({\mathbb D})) \arrow[r] \arrow[d, "C_{p_{(C_F^\St)^*({\mathbb D})}}"']  &   \Gir_{K}({\mathbb D}) \arrow[d, "C_{p_{\mathbb D}}"]      \\
		\Sh({\cal C}, J) \arrow[r, "C_{F}"'] &  \Sh({\cal D}, K) 
	\end{tikzcd}.
\end{equation*}	  
\end{prop}
\begin{proof}
As $F$ is continuous, by Corollary \ref{cor:continuouscomorphismsstacks} the functor $(C^\St_F)^*$ acts by precomposition with $F\op$, \ie like the direct image functor $F^*:\Ind_\dbicat\rightarrow \Ind_\cbicat$: therefore, since by Proposition \ref{prop:directimage_Street_fibrations} the direct image $F^*$ acts on fibrations by pulling back along $F$, the first square is a pseudopullback. On the other hand, the second square is a pullback of toposes as a consequence of Proposition \ref{prop:pb_Giraudtopos_is_inverseimage}.
\end{proof}

\begin{remark}\label{rem_noncontinuouscomorphism}
If $F$ is a non-necessarily continuous comorphism of sites, then $(C^\St_F)^*(\dcat)$ is not simply $\dcat\circ F\op$ but rather its $J$-stackification, by the definition of $C^\St_F$ provided in Proposition \ref{prop:comorphism_action_on_stacks}. Therefore, we may compute first of all the pullback of $p_\dcat$ along $F$, which provides the fibration associated to $\dcat\circ F\op$, and then $J$-stackify it.
\end{remark}

In light of Remark \ref{rem_noncontinuouscomorphism}, we can apply this expression for the inverse image stack in the context of an arbitrary geometric morphism, induced by a \emph{morphism} of sites, in order to compute pullbacks of toposes along it. The idea is to represent the morphism by means of the associated comorphism of sites as in \cite{denseness}. Still, this requires the replacement of the original site of definition for the domain topos of the morphism with a Morita-equivalent one, admitting a functor from it. It therefore remains the problem of turning a stack on this site into a stack on the original site. For this, a generalization to stacks of the relative comprehensive factorization of a functor, as described in the next section, is useful, in allowing a generalization to stacks of Proposition \ref{prop:relativecomprehensivefactsheaves}.

\begin{prop}\label{prop:inversedirect}
Let $F:({\cal C}, J)\to ({\cal D}, K)$ be a morphism of small-generated sites and $\mathbb D$ a $J$-stack on $\cal C$. Then, given the pseudopullback diagram 
\begin{equation*}
	\begin{tikzcd}
		{\cal G}(\pi_{\cal C}^*({\mathbb D})) \arrow[rr] \arrow[d, "p_{{\pi_{\cal C}}^*({\mathbb D})}"']  & &   {\cal G}({\mathbb D}) \arrow[d, "p_{\mathbb D}"]      \\
		(1_{\cal D}\downarrow F) \arrow[rr, "\pi_{\cal C}"] & &  {\cal C} 
	\end{tikzcd}
\end{equation*}	  
in $\Cat$, $\St(F)^*({\mathbb D})$ is the $K$-stack on $\cal D$ (equivalently, the stack on $\Sh({\cal D}, K)$ with respect to the canonical topology) corresponding under the equivalence
\[
C_{\pi_{\cal D}}:\Sh((1_{\cal D}\downarrow F), \widetilde{K})\simeq \Sh({\cal D}, K) 
\]
to the $\widetilde{K}$-stackification of the fibration $p_{{\pi_{\cal C}}^*({\mathbb D})}$.

If $\mathbb D$ is a $J$-sheaf on $\cal C$ then $\Sh(F)^{\ast}({\mathbb D})$ is the $K$-sheaf on $\cal D$ corresponding under the equivalence $C_{\pi_{\cal D}}$ to the $\widetilde{K}$-sheafification of the fibration $p_{{\pi_{\cal C}}^*({\mathbb D})}$; in other words, it coincides with the $K$-comprehensive factorization of the composite functor $\pi_{\cal D}\circ p_{{\pi_{\cal C}}^*({\mathbb D})}$. 
\end{prop}

\begin{proof}
Notice that the functor $\pi_\cbicat: \comma{1_\dbicat}{F} \rightarrow \cbicat$ has a right adjoint $i:\cbicat\rightarrow \cbicat\rightarrow \comma{1_\dbicat}{F}$ defined by mapping an object $X$ to the arrow $1_FX:F(X)\rightarrow F(X)$ in $\comma{1_\dbicat}{F}$: this implies that $\Lan_{i\op}\simeq \pi_\cbicat\op$. Moreover, $F=\pi_\dbicat\circ i$. Thus we have that
\begin{align*}
    \St(F)^*&\simeq s_K\circ \Lan_{F\op} \circ i_J\\
    &\simeq s_K\circ \Lan_{\pi_\dbicat\op}\circ \Lan_{i\op} \circ i_J\\
    &\simeq s_K\circ \Lan_{\pi_\dbicat\op}\circ \pi_\cbicat^* \circ i_J\\
    &\simeq \St(\pi_\dbicat)^*\circ s_{\widetilde{K}}\circ \pi_\dbicat^*\circ i_J
\end{align*}
where the first line is justified by Corollary
\ref{cor:morphsites_functor_between_stacks} and the second and third line by the identities above; the fourth line holds because $\St(\pi_\dbicat)^*\circ s_{\widetilde{K}}$ and $s_K\circ \Lan_{\pi_\dbicat\op}$ are both left adjoints to $i_{\widetilde{K}}\circ \St(\pi_\dbicat)_*\simeq \pi_\dbicat^*\circ i_K$, and therefore they are equivalent. Finally, $\St(\pi_\dbicat)$ is an equivalence, since $\Sh(\pi_\dbicat)\cong C_{\pi_\dbicat}$ is an equivalence.

The statement for sheaves can be proved in a perfectly analogous way. The second assertion follows from Proposition \ref{prop:relativecomprehensivefactsheaves} by observing that, since $\pi_{\cal D}$ is a comorphism of sites, the following diagram commutes:
\begin{equation*}
	\begin{tikzcd}
		{[{\cal C}\op, \Set]} \arrow[rr] \arrow[d, "\sheafify_{J}"']  & &  {[{\cal D}\op, \Set]} \arrow[d, "\sheafify_{K}"]      \\
		\Sh({\cal C}, J) \arrow[rr, "C_{F}"] & &  \Sh({\cal D}, K) 
	\end{tikzcd}
\end{equation*}
\end{proof}

\begin{cor}
Let $f:{\cal F}\to {\cal E}$ be a geometric morphism and $\mathbb D$ a stack on $\cal E$. Then, given the pseudopullback diagram 
\begin{equation*}
	\begin{tikzcd}
		{\cal G}(\pi_{\cal C}^*({\mathbb D})) \arrow[rr] \arrow[d, "p_{{\pi_{\cal E}}^*({\mathbb D})}"']  & &   {\cal G}({\mathbb D}) \arrow[d, "p_{\mathbb D}"]      \\
		(1_{\cal F}\downarrow f^{\ast}) \arrow[rr, "\pi_{\cal E}"] & &  {\cal E} 
	\end{tikzcd}
\end{equation*}	  
in $\Cat$, $f^*({\mathbb D})$ is the stack on $\cal F$ corresponding under the equivalence
\[
C_{\pi_{\cal F}}:\Sh((1_{\cal F}\downarrow f^{\ast}), \widetilde{J^\textup{can}_{\cal F}})\simeq {\cal F} 
\]
to the stackification of the fibration $p_{{\pi_{\cal E}}^*({\mathbb D})}$.
\end{cor}

\chapter{Classification of geometric morphisms via comorphisms of sites}\label{chap:classification}

We devote the present chapter to the study of presentations of geometric morphisms in terms of generators. The first section will provide the functorialization of some results of Section \ref{section:comorphisms}, showing that the category of essential toposes is adjoint to that of sites and continuous comorphisms. The second section will be devoted instead to the study of slice categories of $\Topos$: specifically, we will show that essential geometric morphisms over a base topos, whose domain is presented by a continuous comorphism of sites, admit a site-theoretic description radically different from the usual presentation in terms of flat functors. 

The point of view developed here will become instrumental later in Chapter \ref{chap:discrete}. We will see that the fundamental adjunction allows to interpret the sheafification $\sheafify_J:[\cbicat\op,\Set]\rightarrow \Sh(\cbicat,J)$ precisely in terms of sets of essential geometric morphisms over a base topos whose domain is induced by a continuous comorphism of sites: therefore, combining that with the results of this chapter will allow for various possible descriptions of the sheafification functor (see Section \ref{sec:point_of_view_sheafify}).

\section{Essential geometric morphisms and continuous comorphisms of sites}
The content of this section is basically a functorialization of two results of \cite{denseness}. 

The first result is the equivalence
\[
\Com\cont((\dbicat,K),(\Etopos, J\can_\Etopos)) \simeq \EssTopos\co(\Sh(\dbicat,K),\Etopos)
\]
which we recalled in Proposition \ref{prop:eqv_essgeomorf_comorfcont}. We can show that this equivalence presents a 2-adjunction between the category of sites and continuous comorphisms and that of toposes and essential geometric morphisms. One of the two adjoints is the functor
\[C_{(-)}:\Cosite\cont\rightarrow \EssTopos\co\]
that we have already know. Its right adjoint maps a topos $\Etopos$ to the site $(\Etopos, J\can_\Etopos)$, an essential geometric morphism $H:\Etopos\rightarrow\Ftopos$ to the functors $H_!$ and a 2-cell $\omega:H\Rightarrow K$ of geometric morphisms to the 2-cell $\omega_!: K_!\Rightarrow H_!$ induced on the essential images. The only thing we need to prove is that $H_!$ is in fact a continuous $(J\can_\Etopos, J\can_\Ftopos)$-comorphism of sites: in fact, we will show that at the level of sheaf toposes it induces precisely the geometric morphism $H$, \ie that there is a commutative diagram
\[
\begin{tikzcd}
	{[\Etopos\op,\Set]} \ar[r, "\ran_{H_!\op}"] & {[\Ftopos\op,\Set]}\\
	\Etopos \ar[u, hook] \ar[r, "H"]& \Ftopos \ar[u, hook]
\end{tikzcd}
\]. First of all, notice that the adjunction $H_!\dashv H^*\dashv H_*$ implies the two equivalences
\[(-\circ (H^*)\op) \simeq \lan_{H_*\op},\ 
(-\circ H_!\op)\simeq \lan_{(H^*)\op}.\]
It is well known that for any functor $A:\abicat\rightarrow\bbicat$ the identity $\lan_{A\op}\yo_\abicat\simeq\yo_\bbicat A$ holds, so in particular we have that, for any $X$ in $\Etopos$ and $Y$ in $\Ftopos$, there are natural isomorphisms \[ 
\yo(Y)\circ H_!\op\simeq \yo(H^*(Y)),\ \ran_{H_!\op}(\yo(X))\simeq \yo(H_*(X))
\]
which express precisely the fact that the adjunction $\lan_{H_!\op}\dashv (-\circ H_!\op)\dashv \ran_{H_!\op}$ restricts to the essential geometric morphism $H:\Etopos\rightarrow\Ftopos$, \ie that $H_!$ is a $(J\can_\Etopos, J\can_\Ftopos)$-continuous comorphism of sites. One can then verify that the correspondence is pseudonatural in $(\dbicat,K)$ and $\Etopos$, and hence we obtain the following:
\begin{cor}\label{cor:aggiunzione_cositi_esstopos}
	There is a 2-adjunction
	\[
	\begin{tikzcd}
		\Cosite\cont \ar[r, bend left, "C_{(-)}", start anchor={north east}, end anchor={north west}] \ar[r,phantom, "\dashv"{rotate=270}]& \EssTopos\co \ar[l, bend left, "(-)_!", start anchor={south west}, end anchor={south east}]
	\end{tikzcd}\]
	acting as follows:
	\begin{itemize}
		\item $C_{(-)}$ maps a site $(\cbicat,J)$ to the sheaf topos $\Sh(\cbicat,J)$; a continuous comorphism of sites $A:(\cbicat,J)\rightarrow (\dbicat,K)$ is sent to the essential geometric morphism $C_A:\Sh(\cbicat,J)\rightarrow \Sh(\dbicat,K)$, and a natural transformation $\alpha:A\Rightarrow B$ of comorphisms of sites to the natural transformation $C_\alpha: C_B\Rightarrow C_A$.
		\item $(-)_!$ maps a Grothendieck topos $\Etopos$ to the site $(\Etopos, J\can_\Etopos)$, an essential geometric morphism $H:\Etopos\rightarrow\Ftopos$ to the continuous comorphism of sites $H_!:(\Etopos, J\can_\Etopos)\rightarrow(\Ftopos, J\can_\Ftopos)$ and a natural trasformation $\omega: H\Rightarrow K$ to the natural transformation $\omega_!:K_!\Rightarrow H_!$ between the essential images.
	\end{itemize}
	Moreover, the functor $(-)_!$ is 2-full and faithful.
\end{cor}

In Section 4.4 of \cite{denseness} a result similar to that of Proposition \ref{prop:eqv_essgeomorf_comorfcont} is sketched, where essential geometric morphisms $F:\Sh(\cbicat,J)\rightarrow \Sh(\dbicat,K)$ are presented using liftings of topologies to presheaf toposes. More precisely, it is shown that given a site $(\cbicat,J)$ there is a topology $\widehat{J}$\index{$\widehat{J}$} on $[\cbicat\op,\Set]$, called the \emph{presheaf lifting of $J$}\index{presheaf lifting of $J$}, which is the topology coinduced by $J$ along $\yo_\cbicat:\cbicat\hookrightarrow[\cbicat\op,\Set]$ (in the sense of \cite[Proposition 6.11]{denseness}): the topology $\widehat{J}$ makes $\yo_\cbicat$ into a comorphism of sites such that \[C_{\yo_\cbicat}:\Sh(\cbicat,J)\rightarrow \Sh([\cbicat\op,\Set], \widehat{J})\] is an equivalence of toposes. One can also check that the inclusion \[\Sh([\cbicat\op,\Set], \widehat{J})\hookrightarrow [[\cbicat\op,\Set]\op, \Set]\] can be seen as the inclusion 
\[\Sh(\cbicat,J) \hookrightarrow [\cbicat\op,\Set] \xhookrightarrow{\yo_{[\cbicat\op,\Set]}} [[\cbicat\op, \Set]\op, \Set].\]
Using this topology one can prove that essential geometric morphisms $\Sh(\dbicat,K)\rightarrow \Sh(\cbicat,J)$ correspond to $J$-equivalence classes of continuous comorphisms of sites $(\dbicat,K)\rightarrow ([\cbicat\op,\Set], \widehat{J})$, where two comorphisms of sites to $[\cbicat\op,\Set]$ are \emph{$J$-equivalent}\index{comorphism!$J$-equivalence of -} if and only if they induce the same geometric morphism (up to equivalence). The correspondence extends to natural transformations, and thus we end up with the following: 
\begin{prop}\label{prop:eqv_essgeomorf_comcont_pshlifting}
	There is an equivalence of categories
	\[ \EssTopos\co( \Sh(\dbicat,K), \Sh(\cbicat,J)) \simeq \Cosite\cont\Jeqv((\dbicat,K), ([\cbicat\op,\Set],\widehat{J} )),\]
	where the category on the right is the category of continuous comorphisms of sites $(\dbicat,K)\rightarrow([\cbicat\op,\Set], \widehat{J})$ up to $J$-equivalence\index{$\Cosite\cont\Jeqv$}, acting as follows:
	\begin{itemize}
		\item an essential geometric morphism $F:\Sh(\dbicat,K)\rightarrow \Sh(\cbicat,J)$ is sent to the functor $\iota_JF_!\ell_K:\dbicat\rightarrow [\cbicat\op,\Set]$, which proves to be a $(K, \widehat{J})$-continuous comorphism of sites; a natural trasformation $\omega: F\Rightarrow G$ induces $\omega_!:G_!\Rightarrow F_!$, and thus we map it to the composite $\iota_J \circ \omega_!\circ \ell_K$.
		\item a continuous comorphism of sites $A:(\dbicat,K)\rightarrow ([\cbicat\op,\Set], \widehat{J})$ will produce an essential geometric morphism $\Sh(\dbicat,K)\rightarrow \Sh([\cbicat\op,\Set], \widehat{J})\simeq \Sh(\cbicat,J)$, and the same holds for natural trasformations.
	\end{itemize}	
	Moreover, the equivalence is pseudonatural in $(\dbicat,K)$ and $(\cbicat,J)$.
\end{prop}
\begin{remark}
	Combining this result with Proposition \ref{prop:eqv_essgeomorf_comorfcont}, one gets in particular an equivalence of categories
	\[
	\begin{tikzcd}
		{\Cosite\cont((\dbicat,K), (\Sh(\cbicat,J), J\can_{\Sh(\cbicat,J)}))} \ar[d, bend right, "{\iota_J\circ-}"']\\
		{\Cosite\cont\Jeqv((\dbicat,K), ([\cbicat\op,\Set], \widehat{J}) )} \ar[u, bend right, "{\sheafify_J\circ -}"'].
	\end{tikzcd}.
	\]
\end{remark}

\section{Site-theoretic classification of geometric morphisms over a base}

Let us now tackle the following problem: is there a way to classify, in a fashion similar to that of the previous section, geometric morphisms \emph{over a base topos $\Sh(\cbicat,J)$}, if we know that their domain or codomain is induced by a comorphism of sites $p:(\dbicat,K)\rightarrow(\cbicat,J)$? A first answer to this question is Proposition 2.4 of \cite{giraud.classifying}; we will provide a more general answer using the results in the previous section. 

We start by classifying essential geometric morphisms whose domain is induced by a continuous comorphism of sites:
\begin{prop}\label{prop:classthm_essenziali_comorfismo_dominio}
	Consider $p:(\dbicat,K)\rightarrow (\cbicat,J)$, a continuous comorphism of sites, and an essential geometric morphism $E:\Etopos\rightarrow \Sh(\cbicat,J)$: then there a pseudonatural equivalence of categories 
	\[
	\EssTopos\co\hspace{-0.5ex}\sslash\hspace{-0.5ex}\Sh(\cbicat,J) (\Sh(\dbicat,K), \Etopos)\simeq\Cosite\cont((\dbicat,K), (\Etopos, J\can_\Etopos))/E^*\ell_Jp
	\] 
	Moveover, this restricts to an equivalence
	between $$\EssTopos\co/\Sh(\cbicat,J) (\Sh(\dbicat,K), \Etopos)$$ and the full subcategory of $\Cosite\cont((\dbicat,K), (\Etopos, J\can_\Etopos))/E^*\ell_Jp$ whose 1-cells are the natural transformations $\xi:A\Rightarrow E^*\ell_Jp$ such that the composite
	\[ 
	E_!A\xRightarrow{E_!\circ\xi}E_!E^*\ell_Jp\xRightarrow{\epsilon\circ \ell_Jp} \ell_Jp
	\]
	is an isomorphism, where $\epsilon$ is the counit of $E_!\dashv E^*$.
\end{prop}
\begin{proof}
	First of all, let us specify how $\EssTopos\co\sslash\Sh(\cbicat,J) (\Sh(\dbicat,K), \Etopos)$ is made. Notice that the presence of ${}\co$ forces all the 2-cells appearing in the definition of the lax slice to change direction: so for instance an object would be a pair $(F,\phi):[C_p]\rightarrow [E]$, where $F$ is an essential geometric morphism $\Sh(\dbicat,K)\rightarrow\Etopos$ and $\phi$ a natural transformation $ C_p^*\Rightarrow F^*E^*$. However, we can exploit the presence of the essential images to write the 2-cells in the `usual' direction: objects of $\EssTopos\co\sslash\Sh(\cbicat,J) (\Sh(\dbicat,K), \Etopos)$ are pairs $(F,\phi)$ with $F:\Sh(\dbicat,K)\rightarrow\Etopos$ essential and $\phi: E_!F_!\Rightarrow (C_p)_!$. Similarly, given two such 1-cells $(F,\phi)$ and $(G,\gamma)$, an arrow $\omega:(F,\phi)\Rightarrow (G,\gamma)$ is given by a natural transformation $\omega: F_!\Rightarrow G_!$ satisfying the identity $\gamma(E_!\circ\omega)=\phi$. 
	Notice that we can exploit the adjunction $E_!\dashv E^*$ to obtain from $\phi$ a natural transformation $\bar{\phi}: F_!\Rightarrow E^*(C_p)_!$: then the identity $\gamma(E_!\circ\omega)=\phi$ translates into $\bar{\phi}=\bar{\gamma}\omega$. Finally, notice that $F_!$, $\bar{\phi}$ and $\omega$ are defined (up to isomorphism) by their values on the generators of $\Sh(\dbicat,K)$: \ie, they are uniquely defined by the composites $F_!\ell_K$, $\bar{\phi}\circ \ell_K$ and $\omega\circ \ell_K$. But $F_!\ell_K$ is a continuous comorphism of sites presenting $F$, while $E^*(C_p)_!\ell_K\simeq E^*\ell_Jp$: thence we end up with the equivalence
	\[
	\EssTopos\co\sslash\Sh(\cbicat,J) (\Sh(\dbicat,K), \Etopos)\simeq\Cosite\cont((\dbicat,K), (\Etopos, J\can_\Etopos))/E^*\ell_Jp
	\] 
	Now, notice that $\phi: E_!F_!\Rightarrow (C_p)_!$ can be recovered from $\bar{\phi}$ as $\phi=(\epsilon\circ (C_p)_!)(E_!\circ \bar{\phi})$. Restricting again to the generators of $\Sh(\dbicat,K)$ and setting $\xi:=\bar{\phi}\circ \ell_K$, then $\phi$ is an isomorphism if and only if $(\epsilon\circ \ell_Jp)(E_!\circ \xi)$ is.
\end{proof}
\begin{remark}
	There is a slight, but innocuous, abuse of notation in the previous result: we wrote the right-hand term as a slice category of $\Cosite\cont$, but the functor $E^*\ell_Jp$ in general is \textit{not} a continuous comorphism of sites. 
\end{remark}
If instead we work with the equivalence of Proposition \ref{prop:eqv_essgeomorf_comcont_pshlifting}, we can obtain a similar result for relative geometric morphisms whose domain and codomain are both induced by continuous comorphisms of sites: 
\begin{prop}\label{prop:classthm_essenziali_pshlifting}
	Consider two continuous comorphisms $p:(\dbicat,K)\rightarrow (\cbicat,J)$ and $q:(\ebicat,T)\rightarrow (\cbicat,J)$: then there a pseudonatural equivalence of categories between
	\[
	\EssTopos\co\sslash\Sh(\cbicat,J) (\Sh(\dbicat,K), \Sh(\ebicat,T))\]
	and\[\Cosite\cont\Teqv((\dbicat,K), ([\ebicat\op,\Set], \widehat{T}))/q^*\yo_\cbicat p.
	\] 
	Moveover, this restricts to an equivalence
	between $$\EssTopos\co/\Sh(\cbicat,J) (\Sh(\dbicat,K), \Etopos)$$ and the full subcategory of $\Cosite\cont\Teqv((\dbicat,K), ([\ebicat\op,\Set], \widehat{T})/q^*\yo_\cbicat p$ whose 1-cells are the natural transformations $\tau:B\Rightarrow q^*\yo_\cbicat p$ such that the composite
	\[ \lan_{q\op}B \xRightarrow{\lan_{q\op}\circ\tau} \lan_{q\op}q^*\yo_\cbicat p \xRightarrow{\epsilon'\circ \yo_\cbicat p} \yo_\cbicat p.
	\]
	is sent by $\sheafify_J$ to an isomorphism, where $\epsilon'$ is the counit of $\lan_{q\op}\dashv q^*$.
\end{prop}
\begin{proof}
	Let us again use the shorthands $\Sh(\cbicat,J):=\widetilde{\cbicat}$ and $[\cbicat\op,\Set]:=\widehat{\cbicat}$. The previous proposition tells us that 
	\[
	\EssTopos\co\sslash\widetilde{\cbicat} (\widetilde{\dbicat}, \widetilde{\ebicat})\simeq\Cosite\cont((\dbicat,K), (\widetilde{\ebicat}, J\can_{\widetilde{\ebicat}}))/C_q^*\ell_Jp
	\]
	which starting from $(F,\phi)$ provides a natural transformation $\bar{\phi}: F_!\ell_K\Rightarrow C_q^*\ell_Jp$. First of all, notice that $C_q^*\ell_J p= C_q^* \sheafify_J \yo_\cbicat p\simeq \sheafify_T q^*\yo_\cbicat p$. Second of all, set $A:=\iota_T F_!\ell_K$ and consider the composite $\iota_T\circ \bar{\phi}: A\Rightarrow \iota_T\sheafify_Tq^*\yo_\cbicat p$: we are now in $\widehat{\ebicat}$, and we can perform the pullback of $\iota_T\circ \bar{\phi}$ componentwise along $\eta\circ q^*\yo_\cbicat p: q^*\yo_\cbicat p\Rightarrow \iota_T\sheafify_T q^*\yo_\cbicat p$, where $\eta$ is the unit of $\sheafify_T\dashv \iota_T$. We obtain a natural trasformation $\tau:B\Rightarrow q^*\yo_\cbicat p$, where $B$ is the functor $B:\dbicat\rightarrow \widehat{\ebicat}$ mapping avery object $D$ to the pullback of $\bar{\phi}(D)$ and $\eta_{q^*\yo_\cbicat(p(D))}$ in $[\ebicat\op,\Set]$. Notice that the natural trasformation $\omega:B\Rightarrow A$ satisfies the condition that $\sheafify_{T}\circ \omega$ is an isomorphism, since it is the pullback of $\eta$: this forces $B$ to be a $(K, \widehat{T})$-continuous comorphism of sites which induces the same geometric morphism as $A$, namely $F$. To see this, we recall that $\Sh(\widehat{\ebicat}, \widehat{T})\simeq \widetilde{\ebicat}$ with the inclusion $\iota_{\widehat{T}}:\Sh(\widehat{\ebicat}, \widehat{T})\hookrightarrow[\widehat{\ebicat}\op,\Set]$ being the composite functor $\yo_{\widehat{\ebicat}}\iota_T$. Let us first see that $(-\circ B\op): [\widehat{\ebicat}\op,\Set]\rightarrow \widehat{\dbicat}$ restricts to sheaves. For any $T$-sheaf $W$
	\begin{align*}
		(-\circ B\op)\yo_{\widehat{\ebicat}}\iota_T(W):=\widehat{\ebicat}(B(-), \iota_T(W))&\simeq \widetilde{\ebicat}(\sheafify_TB(-), W)\\&\simeq \widetilde{\ebicat}(\sheafify_TA(-), W)\\&\simeq \widehat{\ebicat}(A(-), \iota_T(W)),	
	\end{align*}
	meaning that $(-\circ B\op)\circ\iota_{\widehat{T}}\simeq (-\circ A\op)\circ\iota_{\widehat{T}}$. But the latter functor factors through $\widetilde{\dbicat}$, since $A$ is $(K, \widehat{T})$-continuous, and hence so does the first: therefore $B$ is $(K, \widehat{T})$-continuous. This immediately implies that $\sheafify_K (-\circ B\op)\iota_{\widehat{T}}\simeq C_A^*$, which has a right adjoint, and thus $B$ is a comorphism of sites. Therefore, for any $(F,\phi)$ in 
	\[
	\EssTopos\co\sslash\Sh(\cbicat,J) (\Sh(\dbicat,K), \Sh(\ebicat,T)),\]
	we may map it to $\tau:B\Rightarrow p^*\yo_\cbicat p$ in 
	\[\Cosite\cont\Teqv((\dbicat,K), ([\ebicat\op,\Set], \widehat{T}))/q^*\yo_\cbicat p,
	\] 
	and this provides our equivalence.
	
	Finally, we know that $\phi$ is an isomorphism if and only if
	\[ 
	(C_q)_!F_!\ell_K\xRightarrow{(C_q)_!\circ\bar{\phi}}(C_q)_!C_q^*\ell_Jp\xRightarrow{\epsilon\circ \ell_Jp} \ell_Jp
	\]
	is an isomorphism, where $\epsilon$ is the counit of $(C_q)_!\dashv C_q^*$, but a routine computation shows that $ (\epsilon\circ \ell_Jp)((C_q)_!\circ\bar{\phi})$ is the image through $\sheafify_J$ of 
	\[ \lan_{q\op}B \xRightarrow{\lan_{q\op}\circ\tau} \lan_{q\op}q^*\yo_\cbicat p \xRightarrow{\epsilon'\circ \yo_\cbicat p} \yo_\cbicat p.
	\]
	
\end{proof}
The next result will instead classify \emph{all} geometric morphisms whose codomain is of the form $[C_p]$ is the codomain we are able to classify in a similar fashion \emph{all} geometric morphisms. We will need for this a couple of technical lemmas: 
\begin{lemma}[{ \cite[Proposition 4.4.6]{riehlcontext}}]\label{lemma:aggiunzione_tra_cat_funtori}
	Consider a pair of adjoint functors $F\dashv G:\abicat\rightarrow\bbicat$: then for any category $\cbicat$ there is an adjunction
	\[
	\begin{tikzcd}
		{[\bbicat,\cbicat]}\ar[r, "{(-\circ G)}"', bend right, start anchor={south east}, end anchor={south west}] \ar[r, phantom, "\dashv"{rotate=90}] & {[\abicat,\cbicat]} \ar[l,"{(-\circ F)}"', bend right, start anchor={north west}, end anchor={north east}]
	\end{tikzcd}.\]
\end{lemma}
\begin{cor}\label{cor:ft_da_subcat_riflessiva}
	Consider a reflective subcategory $\sheafify\dashv i:\abicat\hookrightarrow \cbicat$ ad another category $\Ftopos$: then for any pair of functors $F,G:\abicat\rightrightarrows \Ftopos$, there is a natural bijection 
	$$[\abicat,\Ftopos](F,G)\cong [\cbicat,\Ftopos](F\sheafify, G\sheafify)$$
\end{cor}
\begin{proof}
	We recall that a right adjoint is full and faithful if and only if the counit of the adjunction is a natural isomorphism (see for instance \cite[Lemma 4.5.13]{riehlcontext}): thus the counit $\epsilon$ of $\sheafify\dashv i$ is a natural isomorphism. Therefore we have the chain of natural bijections
	$$[\abicat,\Ftopos](U,V)\cong [\abicat,\Ftopos](U\sheafify i,V)\cong [\cbicat,\Ftopos](U\sheafify, V\sheafify)$$
	where the bijection on the left is just composition with $U\epsilon$, while the bijection on the right comes from the previous lemma.
\end{proof}

\begin{thm}\label{thm:classthm_morfismi_siti}
	Consider a comorphism of sites $p:(\dbicat,K)\rightarrow(\cbicat,J)$ and a geometric morphism $E:\Etopos\rightarrow \Sh(\cbicat,J)$: then there is a pseudonatural equivalence of categories
	$$\Topos\sslash \Sh(\cbicat,J)([E], [C_p])\simeq \Site((\dbicat,K),(\Etopos, J\can_\Etopos))/E^*\ell_Jp$$
	which is pseudonatural in both $[E]$ and $p$.
\end{thm}
\begin{proof}
	We start by defining the equivalence at the level of objects. An object of $\Topos\sslash \Sh(\cbicat,J)([E], [C_p])$ is a pair $(F,\phi)$ where $F:\Etopos\rightarrow \Sh(\dbicat,K)$ is a geometric morphism and $\phi: F^*C_p^*\Rightarrow E^*$. By Corollary \ref{cor:ft_da_subcat_riflessiva} this is the same as a natural transformation $\phi\sheafify_J: F^*C_p^*\sheafify_J\Rightarrow E^*\sheafify_J$; but since $C_p^*\sheafify_J\simeq \sheafify_K p^*$, we have for now a natural transformation $\phi\circ\sheafify_J:F^*\sheafify_Kp^*\Rightarrow E^*\sheafify_J$. Now, since $\lan_{p\op}\dashv p^*$ we have that $-\circ p^*\dashv -\circ \lan_{p\op}$ by Lemma \ref{lemma:aggiunzione_tra_cat_funtori}, and thus $\phi\circ\sheafify_J$ corresponds to a natural transformation $\phi': F^*\sheafify_K\Rightarrow E^*\sheafify_J\lan_{p\op}$. Finally, such a natural transformation is uniquely determined (up to isomorphism) by its values on the generators of $\Sh(\dbicat,K)$, \ie by the composite $\bar{\phi}:=\phi'\circ \yo_\dbicat:F^*\ell_K\Rightarrow E^*\sheafify_J\lan_{p\op}\yo_\dbicat\simeq E^*\ell_Jp$. Notice that $F^*\ell_K$ is the flat $K$-continuous functor $\dbicat\rightarrow \Etopos$ that generates the geometric morphism $F$, and thus it is indeed a morphism of sites $(\dbicat,K)\rightarrow(\Etopos, J\can_\Etopos)$.
	
	To extend the equivalence to arrows, consider two 1-cells $(F,\phi)$ and $(G,\gamma):[E]\rightarrow[C_p]$ and 2-cell $\omega:(F,\phi)\Rightarrow (G,\gamma)$, \ie a natural transformation $\omega:F^*\Rightarrow G^*$ such that $\gamma (\omega\circ C_p^*)=\phi$. Notice that $\omega$ is uniquely defined by its restriction $\bar{\omega}=\omega\circ \ell_K$ to the generators of $\Sh(\dbicat,K)$: then a rapid computation shows that $\gamma (\omega\circ C_p^*)=\phi$ holds if and only if $\bar{\gamma}\bar{\omega}=\bar{\phi}$. Thus the association $\omega\mapsto \bar{\omega}$ defines the equivalence
	\[
	\Topos\sslash \Sh(\cbicat,J)([F],[C_p])\simeq \Site((\dbicat,K),(\Etopos, J\can_\Etopos))/E^*\ell_Jp
	\]
	on arrows. The pseudonaturality is lengty but straightforward to check.
\end{proof}
\begin{remark}\label{rmk:classthm_morfismi_relazione_2celle}
	Let us explicit the relationship between the two natural transformations $\phi: F^*C_p^*\Rightarrow E^*$ and $\bar{\phi}: F^*\ell_K\Rightarrow E^*\ell_L p$, for it will come in handy later. Denote by $\eta$ and $\epsilon$ the unit and counit of $\lan_{p\op}\dashv p^*$. Starting from $\phi$, we obtain $\bar{\phi}$ as the composite
	\[
	\begin{tikzcd}
		F^*\ell_K  \ar[r, equal] \ar[ddd, "\bar{\phi}"', Rightarrow]   &  F^*\sheafify_K\yo_\dbicat \ar[d, Rightarrow, "F^*\sheafify_K\circ \eta\circ \yo_\dbicat"]\\
		& F^*\sheafify_K p^*\lan_{p\op}\yo_\dbicat \ar[d, Rightarrow, "\sim"sloped]\\
		& F^*C_p^* \sheafify_J\lan_{p\op}\yo_\dbicat \ar[d, Rightarrow, "\phi\circ\sheafify_J\lan_{p\op}\yo_\dbicat"]\\
		E^*\ell_J p& E^*\sheafify_J\lan_{p\op}\yo_\dbicat \ar[l, "\sim"', Rightarrow]
	\end{tikzcd}\]
	Conversely, start from $\bar{\phi}: F^*\ell_K\Rightarrow E^*\ell_Jp$, which we can see as a 2-cell $F^*\sheafify_K\yo_\dbicat\Rightarrow E^*\sheafify_J\lan_{p\op}\yo_\dbicat$: then $\bar{\phi}$ induces a natural trasformation $\tilde{\phi}: F^*\sheafify_K\Rightarrow  E^*\sheafify_J\lan_{p\op}$. We can then consider the composite
	\[
	F^*C_p^*\sheafify_J \cong F^*\sheafify_K p^* \xRightarrow{\tilde{\phi}\circ p^*} E^*\sheafify_J\lan_{p\op}p^*\xRightarrow{E^*\sheafify_J\circ \epsilon} E^*\sheafify_J:	
	\]
	its restriction to sheaves (which coincides with its composition with the functor $\iota_K:\Sh(\dbicat,J)\rightarrow [\dbicat\op,\Set]$) provides the natural trasformation $\phi: F^*C_p^*\Rightarrow E^*$.
	
	The components of $(E^*\sheafify_J\circ \epsilon)(\tilde{\phi}\circ p^*)$, and thus those of $\phi$, can be stated directly in terms of the components of $\bar{\phi}$ using colimits. Let us start by considering a representable presheaf $\yo(X):\cbicat\op\rightarrow\Set$: the composite 
	\[
	F^*\sheafify_K p^*\yo(X) \xrightarrow{\tilde{\phi}(p^*(\yo(X)))} E^*\sheafify_J\lan_{p\op}p^*(\yo(X))\xrightarrow{E^*\sheafify_J(\epsilon_{\yo(X)})} E^*\sheafify_J\yo(X)	
	\]	
	can be described using the fact in $[\dbicat\op,\Set]$ the presheaf $p^*\yo(X)$ can be presented as the colimit of representables $p^*(\yo(X))\simeq \colim_{y:p(D)\rightarrow X}\yo(D)$. First of all, we recall that $\epsilon_{\yo(X)}$ 
	is the map
	\[\epsilon_{\yo(X)}:\colim_{y:P(D)\rightarrow X}\yo(p(D))\rightarrow \yo(X)
	\]
	induced by the cocone whose $y$-indexed leg is the map
	\[
	\yo(y):\yo(p(D))\rightarrow \yo(X). \]
	The composite $E^*\sheafify_J(\epsilon_{\yo(X)})$ is computed thus as the arrow
	\[\colim_{y:P(D)\rightarrow X}E^*\ell_J(p(D))\rightarrow  E^*\ell_J(X),
	\]
	induced by the cocone whose $y$-indexed leg is the arrow
	\[
	E^*\ell_J(y): E^*\ell_J(p(D))\rightarrow E^*\ell_J(X).
	\]
	On the other hand, the arrow
	\[
	\tilde{\phi}(p^*\yo(X)): F^*\sheafify_Kp^*\yo(X)\rightarrow E^*\sheafify_J\lan_{p\op}p^*(\yo(X))\]
	is the morphism
	\[
	\colim_{y:p(D)\rightarrow X} F^*\ell_K(D) \rightarrow \colim_{y:p(D)\rightarrow X} E^* \ell_J(p(D)) \]
	induced by colimit property by the maps 
	\[\bar{\phi}(D):F^*\ell_K(D) \rightarrow E^*\ell_J(p(D)):\]
	therefore, globally we have that
	\[ 
	(E^*\sheafify_J\circ \epsilon)(\tilde{\phi}\circ p^*)(\yo(X)): F^*\sheafify_Kp^*\yo(X)\rightarrow E^*\sheafify_J\yo(X)
	\]
	is an arrow
	\[
	\colim_{y:p(D)\rightarrow X} F^*\ell_K(D)\rightarrow E^*\ell_J(X) \]
	induced by the cocone whose $y$-indexed leg is the arrow
	\[
	F^*\ell_K(D)\xrightarrow{\bar{\phi}(D)} E^*\ell_J(p(D))\xrightarrow{E^*\ell_J(y)} E^*\ell_J(X).
	\]
	If now we take any presheaf $H:\cbicat\op\rightarrow\Set$, we can exploit the colimit $H\simeq\colim_{x\in H(X)} \yo(X)$: then the arrow \[
	(E^*\sheafify_J\circ \epsilon)(\tilde{\phi}\circ p^*)(H):F^*C_p^*\sheafify_J(H)\rightarrow E^*\sheafify_J(H)
	\]
	is a morphism 
	\[
	\colim_{x\in H(X)} F^*\sheafify_Kp^*\yo(X) \rightarrow \colim_{x\in H(X)} E^*\ell_J(X)
	\]
	induced componentwise by the arrows $(E^*\sheafify_J\circ \epsilon)(\tilde{\phi}\circ p^*)(\yo(X))$ we described above: thus we can conclude that $(E^*\sheafify_J\circ \epsilon)(\tilde{\phi}\circ p^*)(H)$ is
	induced by colimit property by the arrows 
	\[
	\alpha_{x,y}:F^*\ell_K(D) \xrightarrow{\bar{\phi}(D)} E^*\ell_J(p(D)) \xrightarrow{E^*\ell_J(y)}  E^* \ell_J(X)
	\]
	indexed by $x\in H(X)$ and $y:p(D)\rightarrow X$.
\end{remark}
The previous results admit an alternative formulation using comma categories:
\begin{defn}
	Consider two functors $A:\abicat\rightarrow \cbicat$ and $B:\bbicat\rightarrow\cbicat$: the \emph{comma category}\index{category!comma} $\comma{A}{B}$\index{$\comma{-}{-}$} is the category whose objects are triples $(X\in \abicat,\ Y\in \bbicat,\ f:A(X)\rightarrow B(Y))$ and whose morphism $(\alpha,\beta):(X,Y,f)\rightarrow (X',Y', f')$ are pairs of arrows $\alpha:X\rightarrow X'$ and $\beta: Y\rightarrow Y'$ such that $B(\beta)\circ f= f'\circ A(\alpha)$.
	
	The comma category has two obvious canonical projections to $\abicat$ and $\bbicat$ and a natural transformation $\phi:A\circ p_{{\cal A}}\to B\circ p_{\cal B}$ such that $\phi(X,Y,f)$ is the arrow $A(X)\xrightarrow{f} B(Y)$ (for any object $(X,Y,f)$ of $\comma{A}{B}$):
	\[
	\begin{tikzcd}
		{\comma{A}{B}} \ar[d, "p_A"'] \ar[r, "p_B"] & \bbicat \ar[d, "B"]\\
		\abicat \ar[r, "A"'] \ar[ur, Rightarrow,  "\phi"]&\cbicat
	\end{tikzcd}.\]
	The comma category $\comma{A}{B}$ satisfies a strict $2$-limit universal property in $\CAT$: for every other pair of functors $F_A:\dbicat\rightarrow \abicat$, $F_B:\dbicat\rightarrow \bbicat$ and natural transformation $\psi: AF_A\Rightarrow BF_B$ there is a unique functor $F:\dbicat\rightarrow \comma{A}{B}$ such that $F_A=p_AF$, $F_B=p_BF$ and $\psi= \phi\circ F$. This universal property extends immediately to 2-cells.
\end{defn}
Using comma categories we can provide an alternative description of the category $\Site((\dbicat,K), (\Etopos, J\can_\Etopos))/E^*\ell_Jp$. Let us set $E^*\ell_J:=A$: then a 1-cell $[\xi:f\Rightarrow Ap]$ corresponds to a unique $\bar{\xi}:\dbicat\rightarrow \comma{1_\Etopos}{A}$ as in the following diagram:
\[
\begin{tikzcd}
	\dbicat \ar[ddr, "p"', bend right] \ar[dr, "\bar{\xi}"description]\ar[drr, "f", bend left]&&\\
	&{\comma{1_\Etopos}{A}} \ar[d, "\pi_\cbicat"] \ar[r, "\pi_\Etopos"] & \Etopos \ar[d, "1_\Etopos"]\ar[dl, Rightarrow, "\kappa"]\\
	&\cbicat \ar[r, "A"']& \Etopos 
\end{tikzcd}\]
and a similar correspondence holds for the 2-cells. We only need to take into account that we want the composite $\pi_\Etopos \bar{\xi}$ to be a morphism of sites $(\dbicat,K)\rightarrow (\Etopos, J\can_\Etopos)$. To do so, we shall exploit the following result, which appears as Theorem 3.16 of \cite{denseness}:
\begin{thm}\label{thm:morfgeom_presentato_commacategory}
	Consider a morphism of sites $A:(\cbicat,J)\rightarrow (\ebicat, K)$. Consider the topology $\bar{K}$\index{$\bar{J}$} over the comma category $\comma{1_\ebicat}{A}$, whose covering sieves are exactly those whose image in $\ebicat$ is $K$-covering: then \begin{itemize}
		\item the projection $\pi_\cbicat:\comma{1_\ebicat}{A}\rightarrow \cbicat$ is a comorphism of sites,
		\item the projection $\pi_\ebicat:\comma{1_\ebicat}{A}\rightarrow \ebicat$ is a morphism and a comorphism of sites inducing an equivalence of toposes, 
		\item the diagram of geometric morphisms
		\[
		\begin{tikzcd}
			{\Sh(\ebicat,K)} \ar[dr, "\Sh(A)"']\ar[rr, no head, "\sim"]&& {\Sh(\comma{1_\ebicat}{A}, \bar{K})}\ar[dl, "C_{\pi_\cbicat}"]\\
			&{\Sh(\cbicat,J)}&
		\end{tikzcd}
		\] 
		is commutative.
	\end{itemize}
\end{thm} 
Seeing $A$ as a morphism of sites $(\cbicat,J)\rightarrow (\Etopos, J\can_\Etopos)$, we can apply the theorem above to obtain an equivalence $\Sh( \comma{1_\Etopos}{A}, \overline{J\can_\Etopos})\simeq \Etopos$ that identifies the geometric morphisms $E$ and $C_{\pi_\cbicat}$. One verifies immediately that $\pi_\Etopos\bar{\xi}$ is a morphism of sites if and only if $\bar{\xi}$ is, since at the level of toposes the two functors induce essentially the same geometric morphism, and so we end up with the following result:
\begin{prop}
	Consider a geometric morphism $E:\Etopos\rightarrow \Sh(\cbicat,J)$ with corresponding flat $J$-continuous functor $A:\cbicat\rightarrow \Etopos$ and a comorphism of sites $p:(\dbicat,K)\rightarrow (\cbicat,J)$: then there is an equivalence of categories between \[\Topos\sslash \Sh(\cbicat,J)([E], [C_p])\] and the full subcategory of $\Site((\dbicat,K), (\comma{1_\Etopos}{A}, \overline{J\can_\Etopos}) )$ whose objects are the morphisms of sites $\xi:(\dbicat,K)\rightarrow (\comma{1_\Etopos}{A}, \overline{J\can_\Etopos})$ such that $\pi_\cbicat {\xi}=p$. 
\end{prop}

With the same argument we can also derive the following corollary of Proposition \ref{prop:eqv_essgeomorf_comorfcont}:
\begin{prop}\label{prop:essgeom_equivalgono_comcont_verso_eltigeneralizzati}
	Consider a geometric morphism $E:\Etopos\rightarrow \Sh(\cbicat,J)$ with corresponding flat $J$-continuous functor $A:\cbicat\rightarrow \Etopos$ and a comorphism of sites $p:(\dbicat,K)\rightarrow (\cbicat,J)$: then there is an equivalence of categories between \[\EssTopos\co\sslash \Sh(\cbicat,J)([C_p],[E])\] and the full subcategory of $\Com\cont((\dbicat,K),(\comma{1_\Etopos}{A}, \overline{J\can_\Etopos}))$ whose objects are the continuous comorphisms of sites $\xi:(\dbicat,K)\rightarrow (\comma{1_\Etopos}{A}, \overline{J\can_\Etopos})$ such that $\pi_\cbicat\xi=p$.
\end{prop}

\chapter{The fundamental adjunction}\label{chap:fundadj}

In the past chapters we have collected all the necessary ingredients to build a 2-adjunction 
\[
\begin{tikzcd}
	\phantom{/}\Ind_\cbicat \arrow[r,bend left, ""{below, name=A}, start anchor={north east}, end anchor={north west}] &   \Topos\co/\Sh(\cbicat,J) \arrow[l, bend left, ""{above, name=B}, start anchor={south west}, end anchor={south east}] \ar[from=A, to=B, "\top"{rotate=180}, phantom]
\end{tikzcd},
\]
which we shall call the \emph{fundamental adjunction}. This comes as a broad generalization of many adjunctions between categories of `$\cbicat$-indexed entities' on one side, and `entities over $\cbicat$' on the other: one classical instance is the adjunction between presheaves over a topological spaces and bundles over the same space, which we will recall later in Section \ref{sec:adj_topologica}. The fundamental adjunction stems from the usual duality between indexed categories and fibrations, but the passage to toposes takes into account the topological information given by the site: in the authors' opinion, a restriction of this adjunction to suitable sub-2-categories of relative toposes may allow for a geometric formulation of the stackification process, in a similar way to how the topological adjunction allows us to recover the sheafification functor.

The chapter is split into four sections. After a preliminary result about colimits of toposes, which acts as a motivation for the subsequent results, we investigate the many adjunctions that exist between $\Ind_\cbicat$, $\CAT/\cbicat$ and $\Com/(\cbicat,J)$: in particular, they provide expressions for weighted colimits of sites and comorphisms. The third section shows that the passage to toposes, performed by the 2-functor
\[
C_{(-)}:\Com\cont/(\cbicat,J)\rightarrow \Topos\co/\Sh(\cbicat,J),
\]
preserves said colimits and allows us to extend the site-theoretic adjunction to our fundamental topos-theoretic adjunction. The final section analyses the role of the canonical stack over $(\cbicat,J)$ as a dualizing object for this adjunction.

\section{Colimits of toposes}\label{sec:colimits_of_toposes}
Let us begin with a preliminary result about colimits of toposes, namely the fact that a topos of sheaves is in a canonical way a colimit of étale toposes. We will provide here an abstract proof of the result and a sketch of the explicit proof, which we will later generalize to weighted pseudocolimits of toposes in Section \ref{sec:adj_for_toposes}.
\begin{prop}\label{prop:colimit_topos}
	Let $(\cbicat, J)$ be a small-generated site. Then $\Sh(\cbicat, J)$ is the conical pseudocolimit in the category $\Topos$ of the diagram $\cbicat\to \Topos$ sending any object $X$ of $\cbicat$ to $\Sh(\cbicat, J)\slash \ell_J(X)$ and any arrow $y:Y\rightarrow X$ in $\cbicat$ to the geometric morphism $\prod_{\ell(y)}:\Sh(\cbicat, J)\slash \ell_J(Y)\to \Sh(\cbicat, J)\slash \ell_J(X)$ induced by the arrow $\ell_J(y)$.
\end{prop}
\begin{proof}
	Clearly, the terminal object of the topos $\Sh(\cbicat, J)$ is the colimit of the diagram $\ell_J:\cbicat\to \Sh(\cbicat, J)$. On the other hand, $D$ is the composite of $\ell_J$ with the functor $\Sh(\cbicat, J)\to \Topos$ sending any object $A$ of $\Sh(\cbicat, J)$ to the slice topos $\Sh(\cbicat, J)\slash A$. This functor, by \cite[Proposition 6.3.5.14]{lurie}, preserves colimits (cf. also section 6.3.2 of \cite{lurie}), whence our thesis follows.
	
	For an explicit proof, we resort to the equivalences $\Sh(\cbicat,J)/\ell_J(X)\simeq \Sh(\cbicat/X,J_X)$ introduced in Proposition \ref{prop:fib_discreta_slice_topos} and Example \ref{ex:topologia_su_slice}. A pseudococone under the diagram $D$ of vertex $\Etopos$ would be the given of geometric morphisms $F_X$ for every $X$ in $\cbicat$ and natural isomorphisms $F_y:F_X\circ C_{\fib y}\Isorightarrow F_Y$ for every $y:Y\rightarrow X$ as in the following diagram:
	\[
	\begin{tikzcd}
		{\Sh(\cbicat/X,J_X)} \arrow[rd, "F_X"'] &         & {\Sh(\cbicat/Y, J_Y)} \arrow[ll, "C_{\fib y}"'] \arrow[ld, "F_Y", ""{name=A, xshift=-1ex, yshift=1ex}] \ar[from=A, to=1-1, Leftarrow, "F_y", "\sim"sloped]\\
		& \Etopos &                                                                  
	\end{tikzcd}
	\]
	such that the following two conditions are satisfies: $F_{1_X}$ must be the canonical isomorphism $F_X\cong F_X\circ C_{\fib 1_X}$; for any pair of arrows $z:Z\rightarrow Y$, $y:Y\rightarrow X$, $F_{yz}$ must be equal to $F_z (F_y\circ C_{\fib z})$ up to the canonical isomorphism $C_{\fib y}C_{\fib z}\cong C_{\fib yz}$. We want to build from these data a geometric morphism $H:\Sh(\cbicat,J)\rightarrow \Etopos$. It is rather easy working with inverse images, since $C_{\fib y}^*:= -\circ (\fib y)\op$ because $\fib y$ is a continuous comorphism of sites (see Proposition \ref{prop:ft_cont_caratterizzazione}). So for any $E$ in $\Etopos$ we can define $H^*(E):\cbicat\op\rightarrow \Set$ on an object $X$ in $\cbicat$ as
	\[H^*(E)(X):=F^*_X(E)([1_X]),\]
	and on an arrow $y:Y\rightarrow X$ by setting $H^*(E)(y)$ equal to
	
	\[
	\begin{tikzcd}[column sep=12ex, row sep=4.5ex]
		{F^*_X(E)([1_X])} \ar[r, "{F^*_X(E)(y)}"] & {F^*_X(E)([y])} \arrow[ 
		rounded corners, 
		to path={ 
			-- ([xshift=2ex]\tikztostart.east) 
			-- ([yshift=-4.3ex, xshift=2ex]\tikztostart.east) 
			-| ([xshift=-2ex]\tikztotarget.west) 
			-- (\tikztotarget)
		}, anchor=center]{dl}\\
		{C_{\fib y}^*F_X^*(E)([1_Y])} \ar[r, "{F_y(E)([1_Y])}"] \ar[r, phantom, "\sim"{yshift=5ex}] & {F_Y^*(E)([1_Y])}
	\end{tikzcd}\]
	while for an arrow $g:E\rightarrow E'$ in $\Etopos$ we set $H(g):H(E)\Rightarrow H(E')$ componentwise as
	\[H(g)(X):=\left[F^*_X(E)([1_X])\xrightarrow{F_X^*(g)([1_X])} F_X^*(E')([1_X])\right]. \]
	We leave all verifications to the reader: well-definedness of $H^*$, showing that it is the inverse image of a geometric morphism and finally that the correspondence between cocones and geometric morphisms $\Sh(\cbicat,J)\rightarrow\Etopos$ is an equivalence.
\end{proof}
The legs of the colimit cocone are the geometric morphisms $C_{p_X}$, where $p_X:\cbicat/X\rightarrow \cbicat$ is the usual fibration:
\[
\begin{tikzcd}
	{\Sh(\cbicat/X,J_X)} \arrow[rd, "C_{p_X}"'] &         & {\Sh(\cbicat/Y, J_Y)} \arrow[ll, "C_{\fib y}"'] \arrow[ld, "C_{p_Y}", ""{name=A, xshift=-1ex, yshift=1ex}] \ar[from=A, to=1-1, equal, "\sim"sloped]\\
	& \Sh(\cbicat,J) &                                                                  
\end{tikzcd}
\]
This means that objects and arrows of $\Sh(\cbicat,J)$ can be determined up to isomorphism by the given of local data:
\begin{itemize}
	\item an object $H$ of $\Sh(\cbicat,J)$ is determined by the given for every $X$ in $\cbicat$ of a $J_X$-sheaf $H_X:(\cbicat/X)\op\rightarrow \Set$, and for every $y:Y\rightarrow X$ in $\cbicat$ of an isomorphism $H_y:H_Y\isorightarrow H_X\circ (\fib y)\op$ of $J_Y$-sheaves satisfying the following conditions:
	\begin{itemize}
		\item for every $X$ the arrow $H_{1_X}:H_X\isorightarrow H_X \circ (\fib 1_X)\op$ is the identity of $H_X$;
		\item for every pair of arrows $z:Z\rightarrow Y$, $y:Y\rightarrow X$ the arrow $H_{yz}:F_Z\isorightarrow H_X\circ (\fib yz)\op$ is equal to the arrow $F_Z\xrightarrow{F_z} F_Y\circ (\fib z)\op \xrightarrow{F_y\circ (\fib z)\op}F_X\circ (\fib y)\op\circ (\fib z)\op$.		
	\end{itemize}
	\item an arrow $\alpha:H\rightarrow K$ is determined by the given for every $X$ in $\cbicat$ of an arrow $\alpha_X:H_X\rightarrow K_X$ satisfing the identity
	\( G_y \alpha_Y= (\alpha_X\circ (\fib y)\op)F_y \).
\end{itemize}
Thus, a $J$-sheaf on $\cbicat$ is the gluing of from $J_X$-sheaves on the slice categories $\cbicat/X$ which are naturally compatible with each other: this is yet another instance of the moral that in a topos `things' exist - or, from a logical viewpoint, things are true - as long as they exist, or they are true, \emph{locally}.

Finally, let us show briefly that colimits of slice toposes are stable under pullback, for this will come in handy later. Consider a diagram of slices of $\Sh(\cbicat,J)$: we can interpret it as a composite 
\[I\xrightarrow{D} \Sh(\cbicat,J)\xrightarrow{\Sh(\cbicat,J)/-}\Topos.
\]
As we mentioned in the proof of the previous result, the functor ${\Sh(\cbicat,J)/-}$ preserves colimits: this means that if $L\simeq \colim D$ then $\Sh(\cbicat,J)/L\simeq \colim( \Sh(\cbicat,J)/D(-))$, \ie the colimit cocone on the left is mapped to the colimit cocone on the right.
\[
\begin{tikzcd}
	D(i) \ar[r, "D(s)"] \ar[rd, "\lambda_i"'] & D(j) \ar[d, "\lambda_j"]\\
	& L
\end{tikzcd}\xmapsto{\Sh(\cbicat,J)/-} 
\begin{tikzcd}
	\Sh(\cbicat,J)/D(i) \ar[r, "\prod_{D(s)}"] \ar[dr, "\prod_{\lambda_i}"'] & \Sh(\cbicat,J)/D(j) \ar[d, "\prod_{\lambda_j}"]\\
	& \Sh(\cbicat,J)/L
\end{tikzcd}
\]
This tells us in particular that a conical colimit of slices of $\Sh(\cbicat,J)$ is again a slice of $\Sh(\cbicat,J)$. 
Now, consider a geometric morphism from $\Sh(\cbicat,J)/Q$ to $\Sh(\cbicat,J)/L$ over $\Sh(\cbicat,J)$: we know that up to equivalence it is the dependent product geometric morphism $\prod_\alpha$ induced by a morphism $\alpha:Q\rightarrow L$ of $\Sh(\cbicat,J)$ (see Proposition \ref{lemma:morfgeom_tra_topos_slice_su_base}). The inverse image of $\prod_\alpha$ acts as the pullback along $\alpha$: hence it preserves colimits and thus $Q\simeq \colim (\alpha^* D)$. Now we can move to colimits of toposes applying $\Sh(\cbicat,J)/-:\Sh(\cbicat,J)\rightarrow \Topos$ to conclude that $\Sh(\cbicat,J)/Q\simeq \colim( \Sh(\cbicat,J)/H^*(D(-)))$: \ie, the colimit cocone over $\Sh(\cbicat,J)/P$ is preserved by pullback along $\prod_\alpha$.
\[\begin{tikzcd}
	\Sh(\cbicat,J)/H^*(D(i)) \ar[r] \ar[d, "\prod_{H^*(\lambda_i)}"] & \Sh(\cbicat,J)/D(i) \ar[d, "\prod_{\lambda_i}"]\\
	\Sh(\cbicat,J)/Q \ar[r, "\prod_\alpha"] & \Sh(\cbicat,J)/P \ar[ul, phantom, "\lrcorner"near end]
\end{tikzcd}\]
\begin{remark}
	We have just shown that a conical pseudocolimit of slice toposes is a slice topos, but we converse also holds: for any presheaf $P:\cbicat\op\rightarrow\Set$, the canonical colimit
	\[
	P\simeq \colim(\fib P\xrightarrow{p_P}\cbicat\hookrightarrow{\yo_\cbicat}[\cbicat\op,\Set])= \colim_{x\in P(X)} \yo(X)
	\]
	tells us that $\Sh(\cbicat,J)/\sheafify_J(P) \simeq \colim_{x\in P(X)} \Sh(\cbicat,J)/\ell_J(X)$. Restricting to slices over representables and using the equivalence in Proposition \ref{prop:fib_discreta_slice_topos}, we can alternatively say a topos is a conical pseudocolimit of a diagram
	\[
	I\xrightarrow{} \cbicat\xrightarrow{\Sh(\cbicat/-,J_{(-)})}\Topos
	\]
	if and only if it is the topos of sheaves $\Sh(\fib P, J_P)$ for some (pre)sheaf $P:\cbicat\op\rightarrow\Set$. The fundamental adjunction of Section \ref{sec:adj_for_toposes} will generalize this by showing that a topos over $\Sh(\cbicat,J)$ is the pseudocolimit of toposes of the kind $\Sh(\cbicat/-,J_{(-)})$ if and only if it is the classifying topos associated to a pseudofunctor $\dcat:\cbicat\op\rightarrow\CAT$ (see Definition \ref{def:Giraud_topology_classifying_topos}).
\end{remark}

\section{The adjoints to the Grothendieck construction}\label{sec:adjoints_to_Grothendieck}
We already know that the 2-functor $\gbicat:\Ind_\cbicat\rightarrow\Cl\Fib_\cbicat$ in an equivalence between 2-categories, and in a moment we will see that by embedding the codomain $\Cl\Fib_\cbicat$ into $\CAT/\cbicat$ the equivalence extends to an adjunction. If we also consider a topology $J$ on $\cbicat$, by seeing fibrations over $\cbicat$ as comorphisms of sites to $(\cbicat,J)$ we also obtain other adjoints to the Grothendieck construction; this will imply in particular that the Giraud site associated to a $\cbicat$-indexed category is a colimit in the category of sites and (continuous) comorphisms.

We start by recalling the definition of 2-adjunction:
\begin{defn}\label{def:2-adjoint}
	Consider two 2-categories $\abicat$ and $\bbicat$ and two 2-functors $L:\abicat\rightarrow\bbicat$ and $R:\bbicat\rightarrow\abicat$: then there is a 2-adjunction\index{2-adjunction} $L\dashv R$ if and only if there is an equivalence
	\[ \bbicat(L(X),Y)\simeq \abicat(X, R(Y)) \]
	pseudonatural in $X$ and $Y$.
\end{defn}
From now on we shall consider the diagram $\cbicat/-:\cbicat\rightarrow\Cl\Fib_\cbicat$ mapping $X$ to $p_X:\cbicat/X\rightarrow\cbicat$ and $y:Y\rightarrow X$ to $\fib y: \cbicat/Y\rightarrow\cbicat/X$; the corresponding diagram $\cbicat\rightarrow\Ind_\cbicat$ operates by mapping $X$ to the presheaf $\yo(X)$, and $y$ to the arrow of presheaves $\yo(y)$. The first thing to remark is that the fibred Yoneda lemma (Proposition \ref{prop:fibered_yoneda}) presents a cloven fibration $p:\dbicat\rightarrow \cbicat$ as a $\dcat$-weighted colimit in $\Cl\Fib_\cbicat$ of the diagram $\cbicat/-:\cbicat\rightarrow \Cl\Fib_\cbicat$: 
\[\Cl\Fib_\cbicat(\dbicat, \xbicat)\simeq \Ind_\cbicat(\dcat, \xcat)\simeq \Ind_\cbicat(\dcat, \Cl\Fib_\cbicat( \cbicat/-, \xbicat)) \]
considering the image of the identity of $\dbicat$ via this equivalence, we obtain the following colimit cocone:
\[
\begin{tikzcd}[row sep=10ex]
	\cbicat/X \ar[dr, "F_{(B,\beta)}"', bend right, ""{name=B}] \ar[dr, "F_{(A,\alpha)}"{yshift=2ex, xshift=-2.5ex}, bend left, ""{name=A, below}, ""{name=D}] && \cbicat/Y \ar[ll, "\fib y"'] \ar[dl, "F_{\dcat(y)(A,\alpha)}", ""{name=C, below}]  \\
	&\dbicat& \arrow[from=C, to=D, Rightarrow, "F^y_{(A,\alpha)}"{yshift=3ex, xshift=4.5ex}, "\sim"{sloped, yshift=-2ex},  shorten <=1ex] \ar[from=A, to=B, Rightarrow, "F_\gamma"']
\end{tikzcd}.\]
We recall that the morphism of fibrations  $(F_{(A,\alpha)},\phi_{(A,\alpha)})\in \Cl\Fib_\cbicat(\cbicat/X,\dbicat)$\index{$F_{(A,\alpha)}$} is defined thus: $F_{(A,\alpha)}[y]:=\dom(\widehat{\alpha y}_A)$, and for $z:[yz]\rightarrow[y]$ we have set $F_{(A,\alpha)}(z):=\lambda_{\alpha y,z,A}$, and $\phi_{(A,\alpha)}([y]):=\theta_{\alpha y,A}\inv$. To define $F_\gamma:F_{(A,\alpha)}\Rightarrow F_{(B,\beta)}$\index{$F_{\gamma}$}, notice that the image of $\gamma \widehat{\alpha y}_A$ via $p$ factors through the image of $\widehat{\beta y}_B$: thus there is a unique $F_\gamma([y]):\dom(\widehat{\alpha y}_A)\rightarrow \dom(\widehat{\beta y}_B)$ induced by the property of cartesian arrows. Finally, to define the components of $F^y_{(A,\alpha)}$, notice that for $[z]$ in $\cbicat/Y$ it holds that $F_{(A,\alpha)}\fib y([z])= \dom (\widehat{\alpha yz}_A)$ and $F_{\dcat(y)(A,\alpha)}([z])= \dom (\widehat{\theta_{\alpha y,A}z}_{\dom(\widehat{\alpha y})_A})$: thus we define $F^y_{(A,\alpha)}([z]):=\chi_{\alpha y,z,A}\inv$\index{$F^y_{(A,\alpha)}$}.

\begin{remark}
	The fact that any cloven fibration is a colimit of the fibrations $\cbicat/X$ is an evident generalization of the fundamental result that any presheaf is a colimit of representable presheaves (once we recall that $\fib \yo(X)\simeq \cbicat/X$).
\end{remark}
Let us now start considering adjoints to $\gbicat$. We begin by providing the following result, which is known but somehow hard to find the in the literature, showing that the Grothendieck construction is part of an adjoint triple: 
\begin{prop}\label{prop:aggiunti_di_G}
	Denote by $\Lambda_{\CAT/\cbicat}$\index{$\Lambda_{\CAT/\cbicat}\dashv \Gamma_{\CAT/\cbicat}$} the 2-functor 
	\[
	\Ind_\cbicat\xrightarrow{\gbicat} \Cl\Fib_\cbicat\xrightarrow{\For}\CAT/\cbicat:
	\] 
	it admits a left adjoint $\Lfrak$\index{$\Lfrak$}, defined on objects as
	\[
	\Lfrak:[F:\dbicat\rightarrow \cbicat]\mapsto\left[
	\comma{-}{F}:\cbicat\op\rightarrow\CAT\right].
	\]
	Moreover, if $\cbicat$ is small, then the 2-functor 
	\[
	\Lambda_{\CAT/\cbicat}:\Ind_\cbicat\xrightarrow{\gbicat}\CAT/\cbicat
	\] 
    also admits a right adjoint $\Gamma_{\CAT/\cbicat}$, defined on objects as
	\[ \Gamma_{\CAT/\cbicat}:[F:\dbicat\rightarrow\cbicat]\mapsto \left[ \CAT/\cbicat(\cbicat/-,[F]):\cbicat\op\rightarrow \CAT \right].
	\]
\end{prop}
\begin{proof}
	These results are mentioned in \cite{nlab:grothendieck_construction}; one reference where both are proved in a much wider context using coend calculus is \cite[Proposition 2.1, Definition 2.5 and Proposition 3.2]{hollander}. Notice that the restriction on the size of $\cbicat$ is needed so that $\Gamma_{\CAT/\cbicat}([F])$ is a locally small $\cbicat$-indexed category.
\end{proof}
We remark in particular the composite 2-functor $\gbicat\Lfrak$ maps any functor $F:\dbicat\rightarrow\cbicat$ to its fibration of generalized elements $\comma{1_\cbicat}{F}\rightarrow\cbicat$ (see \cite[Section 3.4.4]{denseness}). In the context of cloven fibrations the adjunction $\Lfrak\dashv \gbicat$ entails the following universal property of the fibration of generalized elements:
\begin{cor}\label{cor:genelts_aggiunto_G}
	Let $F:\dbicat\rightarrow\cbicat$ be a functor. Then	
	\begin{enumerate}[(i)]
		\item The fibration $ {\pi_{\cbicat}^{F}}:\comma{1_{\cbicat}}{F} \to \cbicat$ satisfies the following universal property: given a factorization of $F$ through a fibration $q$, \ie a functor $G:\dbicat\rightarrow\ebicat$ such that $qG\cong F$, there is a unique morphism of fibrations $\chi:{\pi_{\cbicat}^{F}} \to q$ such that $G=\chi\circ i^F$: 
		\begin{equation*}
			\begin{tikzcd}
				{\cal D} \arrow[rr, "F"]   \arrow[dr, "i^F"] \arrow[ddr, "G"{below, xshift=-2ex}]  & &  {\cal C}       \\
				& {\comma{1_{\cal C}}{F}} \ar[d, dashed, "\chi"] \ar[ur, "{\pi_{\cal C}^{F}}"{xshift=1ex}] & \\
				& {\cal E}  \ar[uur, "q"{below, xshift=1ex}]  &  
			\end{tikzcd}
		\end{equation*}	
		\item The functor $i^{F}$ is an equivalence if and only if for any $d\in {\cal D}$ and any $c\in {\cal C}$, every arrow $\alpha:c\to F(d)$ in $\cal D$ is an isomorphism. 
	\end{enumerate} 
\end{cor}

Let us apply Proposition \ref{prop:aggiunti_di_G} to a $\cbicat$-indexed category $\dcat$ over a small category $\cbicat$: we obtain the equivalence
\[
\CAT/\cbicat (\gbicat(\dcat), \kbicat)\simeq \Ind_\cbicat(\dcat, {\CAT}/\cbicat(\cbicat/-, \kbicat)),
\]
which presents $\gbicat(\dcat)$ as the $\dcat$-weighted colimit of the diagram $\cbicat/-:\cbicat\rightarrow \CAT/\cbicat$. The legs of the colimit cocone are given by the functors $F_{(A,\alpha)}:\cbicat/X\rightarrow \gbicat(\dcat)$ introduced earlier in this section:
\begin{cor}\label{cor:colimite_in_cat_su_C}
	Consider $\cbicat$ small and a cloven fibration $\dbicat\rightarrow \cbicat$ with corresponding pseudofunctor $\dcat:\cbicat\op\rightarrow\CAT$: then $\gbicat(\dcat)$ is the $\dcat$-weighted colimit of the diagram $\cbicat/-:\cbicat\rightarrow \CAT/\cbicat$. Moreover, the forgetful functor $\Cl\Fib_\cbicat\rightarrow \CAT/\cbicat$ creates the $\dcat$-weighted colimit of the diagram $\cbicat/-:\cbicat\rightarrow \Cl\Fib_\cbicat$.
\end{cor}
\begin{remark}\label{remark:morfismo_indotto_colimite_fibrazione}
	Consider a cocone under the diagram $\cbicat/-$ with vertex $\abicat$:
	\[
	\begin{tikzcd}[row sep=10ex]
		\cbicat/X \ar[dr, "H_{(B,\beta)}"', bend right, ""{name=B}] \ar[dr, "H_{(A,\alpha)}"{yshift=2ex, xshift=-2.5ex}, bend left, ""{name=A, below}, ""{name=D}] && \cbicat/Y \ar[ll, "\fib y"'] \ar[dl, "H_{\dcat(y)(A,\alpha)}", ""{name=C, below}]  \\
		&\abicat& \ar[from=C, to=D, Rightarrow, "H^y_{(A,\alpha)}"{yshift=3ex, xshift=4.5ex}, shorten <=1ex]\ar[from=A, to=B, Rightarrow, "H_\gamma"']
	\end{tikzcd}\]
	Let us write explicitly the (essentially unique) functor $h:\dbicat\rightarrow\abicat$ induced by the universal property of colimits. Knowing that the colimit cocone is given by the functors $F_{(A,\alpha)}$, and that for every $D$ in $\dbicat$ there is a canonical isomorphism $D\simeq F_{(D,1_{p(D)})}([1_{p(D)}])$, the request that $H_{(A,\alpha)}\simeq hF_{(A,\alpha)}$ forces the definition of $h$: indeed, $h(D)\simeq hF_{(D,1_{p(D)})}([1_{p(D)}])\simeq H_{(D,1_{p(D)})}([1_{p(D)}])$, and hence we can set 
	\[
	h(D):=H_{(D,1_{p(D)})}([1_{p(D)}]).
	\]
	The definition of the action of $h$ on arrows is trickier, but in fact there is only one possible canonical way of defining, given $g:D\rightarrow E$ in $\dbicat$, an arrow $h(g):h(D)\rightarrow h(E)$. If we denote by $v_g:D\rightarrow \dom(\widehat{p(g)}_E)$ the unique arrow such that $\widehat{p(g)}_Ev_g=g$ and $p(v_g)=\theta_{p(g), E}$, we have an arrow $v_g:(D, 1_{p(D)})\rightarrow (\dom(\widehat{p(g)}_E), \theta_{p(g), E})=\dcat(p(g))(E, 1_{p(E)})$ in $\dcat(p(D))$; then $h(g)$ is the composite arrow
	\[
	\begin{tikzcd}[column sep=12ex, row sep=5ex]
		{H_{(D, 1_{p(D)})}([1_{p(D)}])} \ar[r, "{H_{v_g}([1_{p(D)}])}"] & {H_{\dcat(p(g))(E, 1_{p(E)})}} \arrow[
		rounded corners, 
		to path={ 
			-- ([xshift=2ex]\tikztostart.east) 
			--node[right]{\scriptsize $H^{p(g)}_{(E, 1_{p(E)})}([1_{p(D)}])$} ([yshift=-5ex, xshift=2ex]\tikztostart.east) 
			-| ([xshift=-2ex]\tikztotarget.west) 
			-- (\tikztotarget)
		}, anchor=center]{dl}\\
		{H_{(E, 1_{p(E)})}([p(g)])} \ar[r, "{H_{(E,1_{p(E)})}(p(g))}"] & {H_{(E,1_{p(E)})}([1_{p(E)}])}
	\end{tikzcd}.\]
\end{remark}
So far we have only dealt with raw categories: let us now study the interaction of the adjoints to $\gbicat$ with Giraud's topologies, when $\cbicat$ is endowed with a topology $J$. For any cloven fibration $p:\dbicat\rightarrow\cbicat$, every leg $F_{(A,\alpha)}$ of the colimit cocone
\[
\begin{tikzcd}[row sep=10ex]
	\cbicat/X \ar[dr, "F_{(B,\beta)}"', bend right, ""{name=B}] \ar[dr, "F_{(A,\alpha)}"{yshift=2ex, xshift=-2.5ex}, bend left, ""{name=A, below}, ""{name=D}] && \cbicat/Y \ar[ll, "\fib y"'] \ar[dl, "F_{\dcat(y)(A,\alpha)}", ""{name=C, below}]  \\
	&\dbicat& \ar[from=C, to=D, Rightarrow, "F^y_{(A,\alpha)}"{yshift=3ex, xshift=4.5ex}, shorten <=1ex, "\sim"{sloped, yshift=-2ex}] \ar[from=A, to=B, Rightarrow, "F_\gamma"']
\end{tikzcd}.\]
is $(J_X, J_\dbicat)$-continuous by Proposition \ref{prop:fib_e_morfib_sono_can.comorf}, therefore the cocone 
lives in the slice category  $\Cosite\cont/(\cbicat,J)$. What is relevant is that it is still a colimit cocone, as an immediate consequence of the following lemma:
\begin{lemma} With the notations of Remark \ref{remark:morfismo_indotto_colimite_fibrazione}, and considering topologies $J$ on $\cbicat$ and $T$ on $\abicat$:
	\begin{enumerate}[(i)]
		\item Suppose that the legs $H_{(A,\alpha)}$ and $q:\abicat\rightarrow\cbicat$ are comorphisms of sites: then $h$ is a comorphism of sites $(\dbicat, J_\dbicat)\rightarrow (\abicat,T)$; 
		\item Suppose that all the legs $H_{(A,\alpha)}$ 
		are $(J_{p(D)},T)$-continuous functors: then $h$ is a $(J_\dbicat,T)$-continuous functor. 
	\end{enumerate}
\end{lemma}
\begin{proof}
	\begin{enumerate}[(i)]
		\item Take $D$ in $\dbicat$ and $R\in T(h(D))$: since $D=F_{(D,1_{p(D)})}([1_{p(D)}])$ and $hF_{(D,1_{p(D)})}\simeq H_{(D,1_{p(D)})}$, which is a comorphism of sites, there is a sieve $S\in J_{p(D)}([1_{p(D)}])$ such that $H_{(D,1_{p(D)})}(S)\subseteq R$. But now we recall that a continuous functor between sites is always cover-preserving \cite[Proposition 4.13]{denseness}: in particular $F_{(D,1_{p(D)})}:(\cbicat/p(D), J_{p(D)})\rightarrow (\dbicat, J_\dbicat)$ is so, and thus the sieve $S'=\langle F_{(D,1_{p(D)})}(S)\rangle$ is $J_\dbicat$-covering. But then $h(S')\subseteq R$, up to isomorphism, and hence $h$ is a comorphism of sites.
		\item Consider a $T$-sheaf $W:\abicat\op\rightarrow\Set$: we wish to show that $W\circ h\op:\dbicat\op\rightarrow \Set$ is a $J_\dbicat$-sheaf. To do so we exploit Lemma \ref{lemma:fascio_su_fib_sse_fascio_su_slice} stating that $W\circ h\op$ is a $J_\dbicat$-sheaf if and only if every composite $W\circ h\op\circ F_{(A,\alpha)}\op$ is a $J_{X}$-sheaf: but this is true because $hF_{(A,\alpha)}\simeq H_{(A,\alpha)}$, which is a continuous with respect to the relevant topologies. 
	\end{enumerate}
\end{proof}
This entails immediately the following results:
\begin{cor}\index{$\Lambda_{\Cosite/(\cbicat,J)}\dashv \Gamma_{\Cosite/(\cbicat,J)}$} \index{$\Lambda_{\Cosite\cont/(\cbicat,J)}\dashv \Gamma_{\Cosite\cont/(\cbicat,J)}$}
	Consider a small category $\cbicat$ and the four 2-functors 
	\[
	\Lambda_{\Cosite/(\cbicat,J)}: \Ind_\cbicat \xrightarrow{\gbicat}\Cl\Fib_\cbicat\xrightarrow{\Giraud
	}\Cosite/(\cbicat,J),
	\]
	\[
	\Lambda_{\Cosite\cont/(\cbicat,J)}: \Ind_\cbicat \xrightarrow{\gbicat}\Cl\Fib_\cbicat\xrightarrow{\Giraud}\Cosite\cont/(\cbicat,J),
	\]
	\[
	\Gamma_{\Cosite/(\cbicat,J)}: \Cosite/(\cbicat,J)\xrightarrow{[p]\mapsto \Cosite/(\cbicat,J)(\cbicat/-,[p])  }\Ind_\cbicat,
	\]
	\[
	\Gamma_{\Cosite\cont/(\cbicat,J)}: \Cosite\cont/(\cbicat,J)\xrightarrow{[p]\mapsto \Cosite\cont/(\cbicat,J)(\cbicat/-,[p])  }\Ind_\cbicat:
	\]
	then there are 2-adjunctions
	\[
	\begin{tikzcd}
		\phantom{\slash}\Ind_\cbicat \ar[r, bend left, "\Lambda_{\Cosite/(\cbicat,J)}", start anchor={north east}, end anchor={north west}] \ar[r, phantom, "\dashv"{rotate=270}]& \Com/(\cbicat,J) \ar[l, bend left, "{\Gamma_{\Cosite/(\cbicat,J)}}", start anchor={south west}, end anchor={south east}]
	\end{tikzcd},\ \begin{tikzcd}
		\phantom{\slash}\Ind_\cbicat \ar[r, bend left, "\Lambda_{\Cosite\cont/(\cbicat,J)}", start anchor={north east}, end anchor={north west}] \ar[r, phantom, "\dashv"{rotate=270}]& \Com\cont/(\cbicat,J) \ar[l, bend left, "{\Gamma_{\Cosite\cont/(\cbicat,J)}}", start anchor={south west}, end anchor={south east}]
	\end{tikzcd}.
	\]
	In particular, given a cloven fibration $p:\dbicat\rightarrow\cbicat$, the site $(\dbicat, J_\dbicat)$ is the $\dcat$-weighted diagram of the functor $\Giraud\circ (\cbicat/-):\cbicat\rightarrow\Fib_\cbicat\rightarrow \Cosite\cont/(\cbicat,J)$. Moreover, the forgetful functors
	\[\Cosite\cont/(\cbicat,J)\rightarrow\Cosite/(\cbicat,J)\rightarrow \CAT/\cbicat. \]
	reflect and preserve the $\dcat$-weighted colimit of $\Giraud\circ (\cbicat/-)$.
\end{cor}
\begin{proof}
	Consider a comorphism of sites $q:(\ebicat, T)\rightarrow (\cbicat,J)$: by the previous lemma the functor
	\[ \Ind_\cbicat(\dcat, \CAT/\cbicat(\cbicat/-,[q]))
	\rightarrow \CAT/\cbicat( \gbicat(\dcat), [q]) \]
	restricts to a functor
	\[\Ind_\cbicat( \dcat, \Cosite/(\cbicat,J)(\cbicat/-,[q]))\rightarrow\Cosite/(\cbicat,J)( \gbicat(\dcat), [q]), \]
	and its quasi-inverse (precomposition with the colimit cocone) obviously restricts to $\Cosite$ too. The same considerations hold for $\Cosite\cont$.
\end{proof}

The adjunction $\Lfrak\dashv\Lambda_{\CAT/\cbicat}$, on the other hand, does not share the same interesting behaviours. We can still say something when working with small sites and small fibrations though:
\begin{prop}
Denote by $\Ind_\cbicat^s$ the small $\cbicat$-indexed categories, \ie those with values in $\Cat$: then the functor 
\[
\Lambda_{\Cosite^s/(\cbicat,J)}:\Ind^s_\cbicat\xrightarrow{\gbicat}\Cl\Fib_\cbicat\xrightarrow{\Giraud}\Com^s/(\cbicat,J)
\] 
admits as left adjoint the composite 
\[
\Cosite^s/(\cbicat,J)\xrightarrow{\For} \Cat/\cbicat\xrightarrow{\Lfrak} \Ind^s_\cbicat.
\]
\end{prop}
\begin{proof}
    In Corollary \ref{cor:girtpl_aggiunto_sx} we proved that $\For:\Cosite^s/(\cbicat,J)\rightarrow \Cat/\cbicat$ has $\Giraud:\Cat/\cbicat\rightarrow \Cosite^s/(\cbicat,J)$ as its left adjoint: thus by composing the two adjunction $\For\dashv \Giraud$ and $\Lfrak\dashv \Lambda_{\Cat/\cbicat}$ we obtain $\Lfrak\For\dashv \Lambda_{\Cosite^s/(\cbicat,J)}$. 
\end{proof}

\section{Giraud toposes as weighted colimits}\label{sec:adj_for_toposes}

The purpose of this section is to extend the adjunctions of the previous section to toposes and show that, given a small-generated site $(\cbicat,J)$ and a cloven Street fibration $p:\dbicat\rightarrow \cbicat$, the topos $\Gir_J(p)$ is a $\dcat$-weighted colimit in the category of Grothendieck toposes over $\Sh(\cbicat,J)$. In particular this will provide an adjunction between $\Fib_\cbicat$ and $\Topos\co/\Sh(\cbicat, J)$, as we will see at the end of the section.

In the previous sections, the canonical functors 
\[
F_{(A,\alpha)}:\cbicat/X\rightarrow \dbicat,
\]
indexed by objects $(A,\alpha)$ in the fibres $\dcat(X)$, provided the legs of our colimit cocone: since they are all morphisms of fibrations, they are continuous comorphisms of sites (\cfr Proposition \ref{prop:fib_e_morfib_sono_can.comorf}) and thus induce geometric morphisms $C_{F_{(A,\alpha)}}:\Sh(\cbicat/X,J_X)\rightarrow \Gir_J(\dcat)$.  
To ease the notation, we will denote each $C_{F_{(A,\alpha)}}$ with $C_{(A,\alpha)}$\index{$C_{(A,\alpha)}$}, and for every arrow $\gamma:(A,\alpha)\rightarrow(B,\beta)$ of $\dcat(X)$ we will denote by $C_\gamma$\index{$C_\gamma$} the induced natural transformation $C_{\gamma}:C_{(B,\beta)}\Rightarrow C_{(A,\alpha)}$.

Let us now provide some technical results that will come in handy later. The first result is about the family of functors $C_{(A,\alpha)}^*$:
\begin{prop}
	The functors $$-\circ (F_{(A,\alpha)})\op: [\dbicat\op,\Set]\rightarrow[(\cbicat/X)\op,\Set]$$ are jointly conservative, and the same holds for the functors $C_{(A,\alpha)}^*$.
\end{prop}
\begin{proof}
	Since the functors $C_{(A,\alpha)}^*$ act by restricting the functors $-\circ (F_{(A,\alpha)})\op$, the second claim follows from the first one. So consider $r:H\Rightarrow K$ in $[\dbicat\op,\Set]$ and suppose that for every $X$ in $\cbicat$ and every $(A,\alpha)$ in $\dbicat(X)$ the arrow $r\circ (F_{(A,\alpha)})\op$ is invertible. In particular, notice that $$(r\circ F^{(D,1_p(D))})([1_{p(D)}]):=r(\dom(\widehat{1_{p(D)}}_D)):$$
	but $\dom(\widehat{1_{p(D)}}_D)$ is canonically isomorphic to $D$, since both $\widehat{1_{p(D)}}_D$ and $1_D$ are cartesian lifts for $1_{p(D)}$, and hence in particular the arrows $(r\circ F^{(D,1_p(D))})([1_{p(D)}])$ and $r(D)$ are canonically isomorphic. This implies that all components of the natural transformation $r$ are bijective, and hence $r$ is invertible.
\end{proof}
Let us also state a property of jointly conservative functors that we will apply later to the family of the $C_{(A,\alpha)}^*$:
\begin{lemma}\label{lemma:jcons_e_colimiti}
	Consider a functor $F:\abicat\rightarrow \bbicat$ and a family of jointly conservative functors $C_i:\bbicat\rightarrow \cbicat_i$. Suppose that $\abicat$, $\bbicat$ and all the $\cbicat_i$ have (co)limits of shape $\ibicat$, and that all the functors $C_i$ and $C_iF$ preserve said (co)limits: then $F$ also preserves said (co)limits.
\end{lemma}
\begin{proof}
	We prove this for colimits, but the argument is the same for limits. Consider a diagram $D:\ibicat\rightarrow \abicat$ and consider in $\bbicat$ the unique arrow $m:\colim(FD)\rightarrow F(\colim D)$. If we apply $C_i$ we obtain $C_i(m):C_i(\colim(FD))\rightarrow C_iF(\colim D)$: but both $C_i$ and $C_iF$ commute with limits of shape $\ibicat$, thus $C_i(m)$ is equivalent up to canonical isomorphisms to the identity of $\colim(C_iFD)$. Since the functors $C_i$ are jointly conservative this immediately implies that $m$ is an isomorphism, and hence $F$ preserves colimits of shape $\ibicat$.
\end{proof}
Finally, the following lemma shows that the property of being a $J_\dbicat$-sheaf over $\dbicat$ can be checked `locally', by moving to the slice categories $\cbicat/X$:
\begin{lemma}\label{lemma:fascio_su_fib_sse_fascio_su_slice}
	A presheaf $W:\dbicat\op\rightarrow\Set$ is a $J_\dbicat$-sheaf if and only if for every $X$ in $\cbicat$ and every $(A,\alpha)$ in $\dcat(X)$ the presheaf $W\circ (F_{(A,\alpha)})\op$ is a $J_X$-sheaf.
\end{lemma}
\begin{proof}
	Of course if $W$ is a $J_\dbicat$-sheaf then $W\circ (F_{(A,\alpha)})\op=:C_{{(A,\alpha)}}^*(W)$ is a $J_X$-sheaf. 
	
	Conversely, suppose that all composites as above are $J_X$-sheaves: what we will do is building, from a matching family for $W$ and a $J_\dbicat$-covering family $R$ over $D$, a matching family for $W\circ (F^{(D,1_{p(D)})})\op$, and shows that it the latter one admits an amalgamation so does the first one. Consider a $J_\dbicat$-covering family $R=\{f_i:\dom(f_i)\rightarrow D\ |\ i\in I \}$, \ie the datum of $\alpha_i\in W(\dom(f_i))$ satisfying the usual compatibility condition. By definition of $J_\dbicat$-covering family, all the arrows $f_i$ are cartesian and the projection $\{p(f_i) \}$ is a $J$-covering family for $p(D)$: we can then lift it to a $J_{p(D)}$-covering family $S=\{p(f_i):[p(f_i)]\rightarrow [1_{p(D)}]\ |\ i\in I \}$ in $\cbicat/p(D)$. Now, notice that $f_i$ and $\widehat{p(f_i)}_D$ are both cartesian lifts of $p(f_i)$, and therefore there are canonical isomorphisms $\gamma_i:\dom(\widehat{p(f_i)}_D)\rightarrow \dom(f_i)$ between their domains. We can define $\beta_i\in W(\dom(\widehat{p(f_i)}_D))=(W\circ F^{(D,1_{p(D)})})\op([p(f_i)])$ as $\beta_i=W(\gamma_i)(\alpha_i)$: it is now immediate to check that it is a matching family for $S$ and the composite $W\circ (F^{(D,1_{p(D)})})\op$. Since $W\circ (F^{(D,1_{p(D)})})\op$ is a $J_{p(D)}$-sheaf, that matching family admits a unique amalgamation $\beta\in W(\dom(\widehat{1_{p(D)}}_D))$. Finally, since $\widehat{1_{p(D)}}_D:\dom(\widehat{1_{p(D)}}_D)\rightarrow D$ is an isomorphism, the element $\alpha:=W(\widehat{1_{p(D)}}_D\inv)(\beta)\in W(D)$ provides an amalgamation for the matching family $\{\alpha_i\}$.
\end{proof}
\begin{remark}
	Notice that in both previous results we could restrict to considering only those functors of the kind $F^{(D,1_{p(D)})}$: this is an indication that these results, duly reformulated, would hold also in the setting of Grothendieck fibrations.
\end{remark}
What we will show in a moment is that the colimit $(\dbicat, J_\dbicat)\simeq \colim_{ps}^\dcat \cbicat/-$ in $\Cosite/(\cbicat,J)$, which we have introduced in the previous section, is preserved by the pseudofunctor
$ \Cosite/(\cbicat,J)\xrightarrow{C_{(-)}} \Topos/\Sh(\cbicat, J)\co$: this will have as a consequence the existence of a general adjunction between indexed categories and toposes over $\Sh(\cbicat,J)$.

Let us recall the adjoint functor theorem for toposes:
\begin{thm}\index{adjoint functor theorem for toposes}
    A functor $F:\Etopos\rightarrow \Ftopos$ between toposes is a left adjoint if and only if it preserves arbitrary colimits.
\end{thm}
We will use it to build the geometric morphism from the colimit to an arbitrary topos by building its inverse image and showing that it preserves arbitrary colimits. There is however a size issue that needs to be resolves: even if the base site $(\cbicat,J)$ is small-generated, or even small, nothing grants that $\Gir_J(\dcat)$ for an arbitrary fibration $\dcat:\cbicat\op\rightarrow\CAT$ is a Grothendieck topos. To circumvent this, we propose the following definition:

\begin{defn}\label{def:esssmall_fibration}
Given a small-generated site $(\cbicat,J)$, a $\cbicat$-indexed category $\dcat:\cbicat\op\rightarrow \CAT$ is essentially $J$-small\index{$\cbicat$-indexed category! essentially $J$-small} if the Giraud site $(\gbicat(\dcat), J_\dcat)$ is small-generated. We will denote by $\Ind_\cbicat^J$\index{$\Ind_\cbicat^J$} and $\St^J(\cbicat,J)$\index{$\St^J(\cbicat,J)$} the categories of essentially $J$-small fibrations and $J$-stacks.
\end{defn}
\begin{remarks}
\begin{enumerate}[(i)]
    \item From the point of view of fibrations, we can define a fibration $p:\dbicat\rightarrow\cbicat$ to be essentially $J$-small\index{fibration! essentially $J$-small} if and only if the site $(\dbicat, J_\dbicat)$ is small-generated.
    \item If a fibration over $\cbicat$ (or a $\cbicat$-indexed category) is essentially $J$-small, then its Giraud topos is in fact a topos: this answers to the size issue posed in the remark following Definition \ref{def:Giraud_topology_classifying_topos}.
    \item In particular, a small fibration is essentially $J$-small for any choice of $J$. Denote by $\abicat$ the small $J$-dense subcategory of $\cbicat$: then every object $X$ in $\cbicat$ admits a $J$-covering family $y_i:A_i\rightarrow X$ whose domains all lie in $\abicat$. Then consider the full subcategory $\bbicat\hookrightarrow \gbicat(\dcat)$ whose objects are of the form $(A,U)$, where $A$ is an object of $\abicat$ and $U\in \dcat(A)$: it is a small category, and every object $(X,U)$ in $\gbicat(\dcat)$ admits a $J_\dcat$-covering family $(A_i, \dcat(y_i)(U))\rightarrow (X,U)$ of objects in $\bbicat$ (where we used the explicit description of $J_\dcat$ given in Proposition \ref{prop:girtpl_descrizione_sieves}). Therefore, $\bbicat$ is a small $J_\dcat$-dense subcategory of $\gbicat(\dcat)$.
    \item In particular, Lemma \ref{lemma:discfib_sono_esssmall} proved that if $(\cbicat,J)$ is small-generated then every discrete fibration $\fib P\rightarrow \cbicat$ is essentially $J$-small.
    \item Corollary \ref{cor:canonicalstackrelativesite} implies that for every small-generated site $(\cbicat,J)$ the canonical stack $\canst_{(\cbicat,J)}$ is essentially $J$-small.
\end{enumerate}
\end{remarks}

We can now prove that the Giraud topos of an essentially $J$-small fibration is a weighted colimit of étale toposes:
\begin{thm}\label{thm:topos_Giraud_colimite_pesato}
	Given an essentially $J$-small cloven fibration $p:\dbicat\rightarrow \cbicat$ with corresponding $\cbicat$-indexed category $\dcat$, then the topos of sheaves $\Gir_J(p):=\Sh(\dbicat, J_\dbicat)$ is the $\dcat$-weighted colimit of the diagram
	\[
	L:\cbicat\xrightarrow{\cbicat/-}\Cl\Fib_\cbicat\xrightarrow{\Giraud}\Cosite/(\cbicat,J)\xrightarrow{C_{(-)}} \Topos\co/\Sh(\cbicat,J):
	\]
	more explicitly, for any $\Sh(\cbicat,J)$-topos $\Etopos$ there is an equivalence between
	\[\Topos\co/{\Sh(\cbicat,J)}\left( {\Gir_J(p)}, \Etopos\right)\]
	
	and \[ \Ind_\cbicat\left(\dcat, \Topos\co/\Sh(\cbicat,J)\left({\Sh(\cbicat/(-), J_{(-)})},\Etopos \right)\right),\]
	which moreover is pseudonatural in $\Etopos$.
\end{thm}
\begin{proof}
	Since in this proof we will not consider different topologies over the same category, there is no risk of confusion in writing $\widetilde{\abicat}$ for any sheaf topos $\Sh(\abicat,K)$.
	
	Let us start by considering the pseudonatural equivalence in Yoneda's lemma and compose it with $C_{(-)}\Giraud$:
	\[ \dcat(X)\isorightarrow \Fib_\cbicat(\cbicat/X,\dbicat)\rightarrow\Topos\co/\widetilde{\cbicat}(\widetilde{\cbicat/X},\widetilde{\dbicat}) \]
	We obtain a transformation, which is pseudonatural in $X$, acting by mapping $(A,\alpha)$ in $\dcat(X)$ to the geometric morphism $C_{{(A,\alpha)}}$, and $\gamma:(A,\alpha)\rightarrow(B,\beta)$ to the natural transformation $C_{\gamma}:C_{(B,\beta)}\Rightarrow C_{(A,\alpha)}$. The pseudonaturality condition implies the existence for any $y:Y\rightarrow X$ in $\cbicat$ of natural isomorphisms $C_y^{(A,\alpha)}:C_{(A,\alpha)}C_{\fib y}\cong C_{\dcat(y)(A,\alpha)}$ satisfying the usual compatibility conditions. This provides us with a $\dcat$-weighted cocone under $L$ with vertex $\widetilde{\dbicat}$, which we will denote by $\underline{C}$. Let us draw a sketch of it, for the sake of clarity:
	\[
	\begin{tikzcd}
		\widetilde{\cbicat/X} \arrow[rrdd, "{C_{(A,\alpha)}}"{xshift=-1ex}, bend left, ""{name=D, xshift=2ex, yshift=-1ex}, ""{name=A, below}] \arrow[rrdd, "{C_{(B,\beta)}}"', ""{name=B}, bend right] \ar[from=B, to=A, Rightarrow, "C_\gamma"]&  &                     &  & \widetilde{\cbicat/Y} \arrow[llll, "C_{\fib y}"'] \arrow[lldd, "{C_{\dcat(y)(A,\alpha)}}", ""{name=C, yshift=3ex}] \ar[Leftarrow, from=C, to=D, "\sim", "C^{(A,\alpha)}_y"{above,xshift=1ex}] \\
		&  &                     &  &                                                                                            \\
		&  & \widetilde{\dbicat} &  &                                                                                           
	\end{tikzcd}
	\]
	where we are omitting all structural geometric morphisms to $\widetilde{\cbicat}$.
	
	We want to show that $\underline{C}$ is indeed the colimit cocone. More explicitly, we can consider the functor 
	\[(-\circ\underline{C}):\Topos\co/\widetilde{\cbicat}\left( \widetilde{\dbicat}, \Etopos\right)\rightarrow \Ind_\cbicat\left(\dcat, \Topos\co/\widetilde{\cbicat}\left(\widetilde{\cbicat/-},\Etopos \right)\right)  \]
	which starting from $(F,\phi):\widetilde{\dbicat}\rightarrow\Etopos$ maps it to the cocone 
	$$\dcat\xRightarrow{\underline{C}}\Topos\co/\widetilde{\cbicat}(\widetilde{\cbicat/-}, \widetilde{\dbicat})\xRightarrow{(F,\phi)\circ-}\Topos\co/\widetilde{\cbicat}(\widetilde{\cbicat/-}, \Etopos),$$
	\ie composes the geometric morphism $(F,\phi)$ with the legs of the cocone $\underline{C}$: then $\underline{C}$ is a colimit cocone if the functor $(-\circ \underline{C})$ is an equivalence (pseudonatural in $\Etopos$).
	
	To build a quasi-inverse for $(-\circ \underline{C})$, we start by considering a $\dcat$-weighted cone $\underline{G}$ under $L$ with vertex $\Etopos$: 
	
	\[
	\begin{tikzcd}
		\widetilde{\cbicat/X} \arrow[rrdd, "{G_{(A,\alpha)}}"{xshift=-1ex}, bend left, ""{name=D, xshift=2ex, yshift=-1ex}, ""{name=A, below}] \arrow[rrdd, "{G_{(B,\beta)}}"', ""{name=B}, bend right] \ar[from=B, to=A, Rightarrow, "G_\gamma"]&  &                     &  & \widetilde{\cbicat/Y} \arrow[llll, "C_{\fib y}"'] \arrow[lldd, "{G_{\dcat(y)(A,\alpha)}}", ""{name=C, yshift=3ex}] \ar[Leftarrow, from=C, to=D, "\sim", "G^{(A,\alpha)}_y"{above,xshift=1ex}] \\
		&  &                     &  &                                                                                            \\
		&  & \Etopos &  &                                                                                           
	\end{tikzcd}
	\]
	In the following, we will work with inverse images of geometric morphisms: this makes things a bit easier, since we have already mentioned that the behaviour of $C_{(A,\alpha)}^*$ is simply precomposition with the functor $(F_{(A,\alpha)})\op$. We want to build from the data of the cocone $\underline{G}$ an inverse image functor $H:\Etopos\rightarrow \widetilde{\dbicat}$ so that its composition with the cocone $\underline{C}$ results again in $\underline{G}$.
	
	Consider an object $E$ of $\Etopos$. We recall once again that for every $D$ in $\dbicat$ it holds that $D\simeq F^{(D,1_{p(D)})}([1_{p(D)}])$, from which we infer that $H(E)(D)\simeq H(E)F^{(D,1_{p(D)})}([1_{p(D)}])$. Since we want $H(E)F^{(D,1_{p(D)})}=C_{(A,\alpha)}^*(H(E))$ to be isomorphic to $G^*_{(D,1_{p(D)})}\hspace{-0.65pt}(E)$, we can set $H(E)(D)\hspace{-2.5pt}:=\hspace{-1.5pt}G^*_{(D,1_{p(D)})}\hspace{-0.65pt}(E)([1_{p(D)}])$. The definition of $H(E)$ is a bit more intricate on arrows. Consider $r:D'\rightarrow D$ in $\dbicat$: if we consider the cartesian lift $\widehat{p(r)}_D$, then by cartesianity $r$ must factor through it with a unique $v_r:D'\rightarrow \dom(\widehat{p(r)}_D)$ such that moreover $p(v_r)=\theta_{p(r),D}$, and this provides an arrow $v_r:(D', 1_{p(D')})\rightarrow (\dom(\widehat{p(r)}_D), \theta_{p(r),D})=\dcat(p(r))(D,1_{p(D)})$. Then we can define $H(E)(r)$ as the following composite arrow:
	\[\begin{tikzcd}[column sep=12ex, row sep=5ex]
		{G^*_{(D,1_{p(D)})}(E)[1_{p(D)}]} \ar[dd, "H(E)(r)"] \ar[r, "{G^*_{(D,1_{p(D)})}(E)(p(r))}"] & {G^*_{(D,1_{p(D)})}(E)([p(r)])} \ar[d, equal]
		\\
		& {C_{\fib p(r)}^*G^*_{(D,1_{p(D)})}(E)([1_{p(D')}])} \ar[d, "{G^{(D,1_{p(D)})}_{p(r)}([1_{p(D')}])}"]\\
		G^*_{(D',1_{p(D')})}(E)([1_{p(D')}])  & {G^*_{(\dom(\widehat{p(r)}_D),\theta_{p(r),D})}(E)([1_{p(D')}])} \arrow[l, "{G_{v_r}(E)([1_{p(D')}])}"]
	\end{tikzcd}.\]
	The functoriality of $H(E)$ can be checked through a (rather lenghty) computation. It is easier to define $H$ on arrows: indeed, given $g:E\rightarrow E'$ in $\Etopos$, $H(g):H(E)\rightarrow H(E')$ is defined componentwise, where for $D$ in $\dbicat$ we set $H(g)(D)$ as
	$$G^*_{(D,1_{p(D)})}(g)([1_{p(D)}]):G^*_{(D,1_{p(D)})}(E)([1_{p(D)}])\rightarrow G_{(D,1_{p(D)})}(E')([1_{p(D)}]).$$
	It is very easy to see that this defined a natural transformation $H(E)\Rightarrow H(E')$ and that the association $g\mapsto H(g)$ is functorial.
	
	Now, by the very definition of $H(E)$ it holds that for every $X$ in $\cbicat$ and every $(B,\beta)$ we have $C_{(B,\beta)}^*H(E)\simeq G_{(B,\beta)}(E)$: all these are sheaves, and by applying Lemma \ref{lemma:fascio_su_fib_sse_fascio_su_slice} we have that $H(E)$ is a $J_\dbicat$-sheaf.
	
	For now we have a functor $H:\Etopos\rightarrow\widetilde{\dbicat}$ such that the cone $\underline{G}$ is essentially equivalent to the composite of $H$ with the cone $\underline{C}$: we want it to be the inverse image of a geometric morphism. But this follows from Lemma \ref{lemma:jcons_e_colimiti}: since all the functors $C_{(A,\alpha)}$ and $C_{(A,\alpha)}H\simeq G_{(A,\alpha)}$ preserve finite limits and arbitrary colimits, and since the $C_{(A,\alpha)}$ are jointly conservative, it follows that $H$ must also preserve finite limits and arbitrary colimits. By the adjoint functor theorem for Grothendieck toposes we conclude that $H$ is the inverse image of a geometric morphism $\widetilde{\dbicat}\rightarrow\Etopos$. It is also immediate to check that it is a morphism in the slice $\Topos\co/\widetilde{\cbicat}$: thus we have defined the behaviour of the functor
	$$\Ind_\cbicat(\dcat, \Topos\co/\widetilde{\cbicat}(\widetilde{\cbicat/-},\Etopos))\rightarrow \Topos\co/\widetilde{\cbicat}(\widetilde{\dbicat},\Etopos)$$
	on objects. The definition on arrows is rather easy: start with two cocones $\underline{G}$ and $\underline{G}'$, corresponding to inverse images functors $H:\Etopos\rightarrow\widetilde{\dbicat}$ and $H':\Etopos\rightarrow\widetilde{\dbicat}$. An arrow $\xi:\underline{G}\Rrightarrow \underline{G}'$ is a modification, and we want to build from it a natural trasformation $\eta:H'\Rightarrow H$ of geometric morphisms. To define $\eta$, notice that $\xi$ corresponds to the given for every $(A,\alpha)$ in $\dcat(X)$ of a natural transformation $\xi_{(A,\alpha)}:G'_{(A,\alpha)}\Rightarrow G_{(A,\alpha)}$ of geometric morphisms (the direction is reversed by the $\co$), satisfying some naturality conditions: then we can set $\eta(E):H'(E)\Rightarrow H(E):\dbicat\op\rightarrow\Set$ componentwise as
	\[\eta(E)(D):=G'^*_{(D,1_{p(D)})}(E)([1_{p(D)}])\xrightarrow{\xi_{(D,1_{p(D)})}(E)([1_{p(D)}])} G^*_{(D,1_{p(D)})}(E)([1_{p(D)}])	\]
	It is easy to check using the definition of modification that $\eta(E)(D)$ is natural in both components, and thus provides a 2-cell of geometric morphisms $\eta:H'\Rightarrow H$.
	
	The verification that we have defined a pseudonatural equivalence is lenghty but straightforward.
\end{proof}
\begin{remark}
	Notice in particular that if $\Etopos=\Sh(\abicat, T)$ and all the geometric morphisms $G_{(A,\alpha)}$ of the cocone are induced by continuous comorphisms of sites, then $H:\Gir_J(p)\rightarrow\Etopos$ is induced, up to isomorphism, by the continuous comorphism of sites $h:\dbicat\rightarrow \abicat$ we defined in Corollary \ref{cor:colimite_in_cat_su_C}.
\end{remark}

A further result shows that Giraud toposes for presheaves are canonically seen as conical colimits, in a way that generalizes Proposition \ref{prop:colimit_topos}:

\begin{cor}\label{cor:topos_discfib_colimite_conico}
	Consider a small-generated site $(\cbicat,J)$: for any presheaf $P:\cbicat\op\rightarrow\Set$, the topos $\Gir_J(P)$ is the conical pseudocolimit of the diagram
	\[
	\textstyle	D:\fib P \xrightarrow{p_P}\cbicat\xrightarrow{\cbicat/-}\Cl\Fib_\cbicat\xrightarrow{\Giraud}\Cosite/(\cbicat,J)\xrightarrow{C_{(-)}}\Topos\co/\Sh(\cbicat,J);\]
	in particular, if $P$ is the terminal presheaf we have that  $\Sh(\cbicat,J)$ is the conical pseudocolimit of the diagram $C_{(-)}\Giraud(\cbicat/-)$.
	
	If we consider the 1-category $\Topos/_1\Sh(\cbicat,J)$, where geometric morphisms are identified up to equivalence, then $\Gir_J(P)$ is also the 1-colimit of $D$ in $\Topos/_1\Sh(\cbicat,J)$. 
\end{cor}
\begin{proof}
	The first claim follows from Corollary \ref{cor:colimiti_pesati_conificati}, which stated that a colimit weighted by a discrete indexed category can be described as a conical colimit, so that $$\Gir_J(P)\simeq \colim_{ps}^P(C_{(-)}\Giraud (\cbicat/-))\simeq \colim_{ps}(C_{(-)}\Giraud (\cbicat/-)p_P).$$
	Its colimit cocone is that of the geometric morphisms
	\[
	\begin{tikzcd}
		{\Sh(\cbicat/X, J_X)} \arrow[rd, "{C_{(X,s)}}"{below}, ""{name=A}] &                     & \Sh({\cbicat/Y, J_Y}) \arrow[ll, "C_{\fib y}"'] \arrow[ld, "{C_{(Y, P(y)(s))}}", ""{name=C, above}] \ar[from=C, to=A, Rightarrow, "\sim"'] \\
		&  \Gir_J(P) &                     
	\end{tikzcd}
	\]
	indexed by the objects $(X,s)$ in $\fib P$. If in particular we consider the terminal presheaf $1:\cbicat\op\rightarrow \Set$, it is immediate to see that the corresponding Grothendieck fibration is the identity functor of $\cbicat$, and thus $\Sh(\cbicat,J)\simeq \colim_{ps}(C_{(-)}\Giraud (\cbicat/-))$.
	
	For the last part, consider a 1-cocone $G_{(X,s)}:\Sh(\cbicat/X, J_X)\rightarrow \Ftopos$ under the diagram $D$. In the proof of Theorem \ref{thm:topos_Giraud_colimite_pesato} the definition of the induced geometric morphism $H:\Gir_J(P)\rightarrow \Ftopos$ was forced by the conditions $G_{(X,s)}\simeq H C_{(X,s)}$; if we want to build a 1-colimit we must strenghten them to the equalities $G_{(X,s)}=H C_{(X,s)}$, and this determines $H$ uniquely up to isomorphism: thus $\Gir_J(P)$ is a 1-colimit of toposes.
\end{proof}

\begin{remark}\label{rmk:esempio_colimite_pesato_non_conificabile}
	For an arbitrary weight $\dcat$ the two colimits $\colim_{ps}^\dcat D$ and $\colim_{ps}(Dp_\dcat)$ are in general different. Take $\cbicat=\onecat$, the terminal category, and endow it with the trivial topology: then $\Sh(\onecat, J^{tr}_\onecat)\simeq \Set$ and $D:\onecat\rightarrow \Topos\co/\Set$ maps the unique object of $\onecat$ to $\Set$. Consider now the category $\twocat$ that has two objects $0$ and $1$ and an arrow $t:0\rightarrow 1$ between them: the unique functor $p:\twocat\rightarrow \onecat$, is a fibration and the Giraud topology over $\twocat$ is the trivial topology, hence $\colim_{ps}^\twocat(D)\simeq \Sh(\twocat, J_\twocat)\simeq [\twocat, \Set]$. On the other hand, $Dp:\twocat\rightarrow\Topos\co/\Set$ maps the two objects of $\twocat$ to the topos $\Set$ and the unique arrow between them to the identity geometric morphism of $\Set$: thus $\colim_{ps}(Dp)\simeq \Set$.
	
\end{remark}

Finally, from Theorem \ref{thm:topos_Giraud_colimite_pesato} we can deduce the \emph{fundamental adjunction} between essentially $J$-small cloven fibrations over $\cbicat$ and toposes over $\Sh(\cbicat,J)$:

\begin{cor}\label{cor:fundamental_adjunction}\index{fundamental adjunction}
	For any small-generated site $(\cbicat,J)$, the two pseudofunctors
	\[\Lambda_{\Topos\co/\Sh(\cbicat,J)}:\Cl\Fib^J_\cbicat\xrightarrow{\Giraud}\Cosite/(\cbicat,J) \xrightarrow{C_{(-)}} \Topos\co/\Sh(\cbicat,J),\]
	\[\left[[p:\dbicat\rightarrow \cbicat]\xrightarrow{(F,\phi)}[q:\ebicat\rightarrow \cbicat]\right]\mapsto \left[ [\Gir_J(p)] \xrightarrow{(C_F,C_\phi)}[\Gir_J(q)]\right], \]
	and
	\[\Gamma_{\Topos\co/\Sh(\cbicat,J)}:\Topos\co/\Sh(\cbicat,J)\rightarrow \Ind^J_\cbicat\simeq\Cl\Fib^J_\cbicat,\]\index{$\Lambda_{\Topos\co/\Sh(\cbicat,J)}\dashv\Gamma_{\Topos\co/\Sh(\cbicat,J)}$}
 	which acts by mapping a geometric morphism $E:\Etopos\rightarrow \Sh(\cbicat,J)$ to
	\[  \Topos\co/\Sh(\cbicat,J)(\Sh(\cbicat/-, J_{(-)}), [E]):\cbicat\op\rightarrow\CAT ,\]
	are the two components of a 2-adjunction 
	\[
	\begin{tikzcd}
		\phantom{\slash}\Cl\Fib^J_\cbicat \arrow[r, "\Lambda_{\Topos\co/\Sh(\cbicat,J)}", bend left, start anchor={north east}, end anchor={north west}] \ar[r,phantom, "\vdash"{rotate=90}]&   \Topos\co/\Sh(\cbicat,J)\phantom{{}^J_\cbicat} \arrow[l, "\Gamma_{\Topos\co/\Sh(\cbicat,J)}", bend left,  start anchor={south west}, end anchor={south east}]
	\end{tikzcd}
	\]
\end{cor}

Finally, as the 2-adjunction $\Lambda_{\CAT/\cbicat}\dashv \Gamma_{\CAT/\cbicat}$ restricts to sites and continuous comorphisms, so the 2-adjunction $\Lambda_{\Topos\co/\Sh(\cbicat,J)}\dashv\Gamma_{\Topos\co/\Sh(\cbicat,J)}$ can also be formulated on a smaller class of toposes. Indeed, since $\Giraud$ maps all fibrations in $\Cl\Fib_\cbicat$ and their morphisms to continuous comorphisms of sites, it follows that the geometric morphisms in the image of $\Lambda_{\Topos\co/\Sh(\cbicat,J)}$ are all essential geometric morphisms, and thus the colimit cocone lives in $\EssTopos$. Moreover, if we go back to the proof of Theorem \ref{thm:topos_Giraud_colimite_pesato} and supposed that all the legs $G^*_{(A,\alpha)}$ of the cone of the inverse images preserve arbitrary limits, $H$ does too by Lemma \ref{lemma:jcons_e_colimiti}, which is to say that if all the legs $G_{(A,\alpha)}$ are essential geometric morphisms then the induced functor $H$ is an essential geometric morphism. This proves the following result:
\begin{cor}\label{cor:fundadj_cfib_esstopos}\index{$\Lambda_{\EssTopos\co/\Sh(\cbicat,J)}\dashv\Gamma_{\EssTopos\co/\Sh(\cbicat,J)}$}
	Given a cloven fibration $p:\dbicat\rightarrow \cbicat$ with corresponding $\cbicat$-indexed category $\dcat$, then there is a 2-adjunction
	\[
	\begin{tikzcd}
		\phantom{\slash}\Cl\Fib^J_\cbicat \arrow[r, "\Lambda_{\EssTopos\co/\Sh(\cbicat,J)}", bend left, start anchor={north east}, end anchor={north west}] \ar[r,phantom, "\vdash"{rotate=90}]&   \EssTopos\co/\Sh(\cbicat,J)\phantom{{}^J_\cbicat} \arrow[l, "\Gamma_{\EssTopos\co/\Sh(\cbicat,J)}", bend left,  start anchor={south west}, end anchor={south east}]
	\end{tikzcd}
	\]
	where $$\Gamma_{\EssTopos\co/\Sh(\cbicat,J)}(\Etopos):=\EssTopos\co/\Sh(\cbicat,J)(\widetilde{\cbicat/-},\Etopos):\cbicat\op\rightarrow\CAT,$$
	$$\Lambda_{\EssTopos\co/\Sh(\cbicat,J)}(p):=[C_{p_\dbicat}:\Gir_J(P) \rightarrow \Sh(\cbicat,J)].$$
\end{cor}

\begin{remark}
	This could also be derived from the adjunction $C_{(-)}\dashv (-)_!$ between $\EssTopos\co$ and $\Cosite\cont$ which we introduced in Corollary \ref{cor:aggiunzione_cositi_esstopos}. Indeed, since $C_{(-)}$ is a left adjoint it preserves weighted pseudocolimits: thus from $(\gbicat(\dcat), J_\dcat)\simeq \colim_{ps}^\dcat (\cbicat/-)$  it follows that $$\Gir_J(\dcat):=\Sh(\gbicat(\dcat), J_\dcat)\simeq\colim_{ps}^\dcat \Sh(\cbicat/-, J{(-)})$$ in $\EssTopos\co$. Moreover, the functor $C_{(-)}$ restricts to a functor of slice categories \[C_{(-)}:\Cosite\cont/(\cbicat,J)\rightarrow\EssTopos\co/\Sh(\cbicat,J),\] so the diagram $\Sh(\cbicat/-, J_{(-)})$ has image in the slice $\EssTopos\co/\Sh(\cbicat,J)$. Finally, since the forgetful functor $\EssTopos\co/\Sh(\cbicat,J)\rightarrow \EssTopos\co$ reflects colimits, the topos $\Gir_J(\dcat)$ is also the colimit of $\Sh(\cbicat/-, J_{(-)})$ in $\EssTopos\co/\Sh(\cbicat,J)$. 
\end{remark}

In fact, the codomain of $\Gamma$ can also be restricted:
\begin{lemma}\label{lemma:immagine_gamma_in_Stack}
	Given a site $(\cbicat,J)$, the functor 
	\[\Gamma_{\Topos\co/\Sh(\cbicat,J)}:\Topos\co/\Sh(\cbicat,J) \rightarrow \Cl\Fib_\cbicat \]
	factors through $\St(\cbicat,J)$.
\end{lemma}
\begin{proof}
	Indeed, consider $X$ in $\cbicat$, $m_R:R\rightarrowtail \yo(X)$ in $J(X)$ and any $\Sh(\cbicat,J)$-topos $\Etopos$: then the diagram
	\[
	\begin{tikzcd}[column sep=huge]
		\Fib_\cbicat(\cbicat/X, \Gamma(\Etopos)) \ar[d,no head, "{-\circ \fib m_R}"'] \ar[r, no head, "\sim"]& \Topos\co/\Sh(\cbicat,J)(\Sh(\cbicat/X, J_{X}), \Etopos) \ar[d, no head, "\sim"sloped, "{-\circ C_{\fib m_R}}"']\\
		\Fib_\cbicat(\fib R, \Gamma(\Etopos)) \ar[r, no head, "\sim"]& \Topos\co/\Sh(\cbicat,J)(\Sh(\fib R, J_{R}), \Etopos)
	\end{tikzcd}\]
	shows that the functor $-\circ \fib m_R$ is an equivalence, and hence $\Gamma(\Etopos)$ is a $J$-stack.
\end{proof}
The previous lemma entails that we may restrict the fundamental adjunction to stacks, by applying \ref{lemma:adj_restriction_subcat}:
\begin{prop}\label{prop:fundadj_stack_topos}
    There are 2-adjunctions
\[\begin{tikzcd}
	{\St^J(\cbicat,J)}\ar[r, bend left, start anchor={north east}, end anchor={north west}] \ar[r, phantom, "\dashv"{rotate=270}] & {\Topos\co/\Sh(\cbicat,J)\phantom{^J}} \ar[l,"\Gamma", bend left, start anchor={south west}, end anchor={south east}]
\end{tikzcd},\]    
\[\begin{tikzcd}
	{\St^J(\cbicat,J)}\ar[r, bend left, start anchor={north east}, end anchor={north west}] \ar[r, phantom, "\dashv"{rotate=270}] & {\EssTopos\co/\Sh(\cbicat,J)\phantom{^J}} \ar[l,"\Gamma'", bend left, start anchor={south west}, end anchor={south east}]
\end{tikzcd},\]
where 
\begin{itemize}
    \item in both cases the left adjoint acts by mapping a $J$-stack $\dcat:\cbicat\op\rightarrow\CAT$ to its classifying topos $C_{p_\dcat}:\Sh(\gbicat(\dcat), J_\dcat)\rightarrow \Sh(\cbicat,J)$;
    \item the right adjoint $\Gamma$ maps a $\Sh(\cbicat,J)$-topos $\Etopos$ to the $J$-stack
    \[
    \Topos\co/\Sh(\cbicat,J)(\Sh(\cbicat/-,J_{(-)}),\Etopos):\cbicat\op\rightarrow\CAT
    \]
    \item the right adjoint $\Gamma'$ maps an essential $\Sh(\cbicat,J)$-topos $\Etopos$ to the $J$-stack
    \[
    \EssTopos\co/\Sh(\cbicat,J)(\Sh(\cbicat/-,J_{(-)}),\Etopos):\cbicat\op\rightarrow\CAT.
    \]
\end{itemize}
\end{prop}

\section{The canonical fibration as a dualizing object}

An immediate consequence of Theorem \ref{thm:topos_Giraud_colimite_pesato} is the fact that Giraud toposes are equivalent to morphisms of fibrations to the canonical stack of the base site:
\begin{cor}\label{correlativepresheaves}
	Consider a $\cbicat$-indexed category $\dcat$, and denote by $\dcat\Vop$ its composite with $(-)\op$ (see Definition \ref{def:Vop}): then
	\[\Gir_J(\dcat) \simeq \Ind_\cbicat(\dcat, \canst_{(\cbicat,J)}\Vop)\op \simeq \Ind_\cbicat(\dcat\Vop, \canst_{(\cbicat,J)}).\]
\end{cor}
\begin{proof}
	If we denote by $\widetilde{\cbicat}[\mathbb{O}]$ the object classifier over $\Sh(\cbicat,J)$ (which exists by \cite[Example B3.2.9]{elephant}), then the following chain of natural equivalences holds:
	\begin{align*}
		\Gir_J(\dcat):&= \Sh(\gbicat(\dcat), J_\dcat)\\
		&\simeq\Topos/\widetilde{\cbicat}(\widetilde{\gbicat(\dcat)}, \widetilde{\cbicat}[\mathbb{O}])\\
		&\simeq \Topos\co/\widetilde{\cbicat}(\widetilde{\gbicat(\dcat)}, \widetilde{\cbicat}[\mathbb{O}])\op\\
		&\simeq \Ind_\cbicat(\dcat, \Topos\co/\widetilde{\cbicat}(\widetilde{\cbicat/-},\widetilde{\cbicat}[\mathbb{O}]))\op\\
		&\simeq \Ind_\cbicat(\dcat, \widetilde{\cbicat/-}\Vop)\op\\
		&\simeq \Ind_\cbicat(\dcat\Vop, \canst_{(\cbicat,J)})
	\end{align*}
	Where we exploited the fact that $\Ind_\cbicat(\dcat\Vop, \ecat)\simeq \Ind_\cbicat(\dcat, \ecat\Vop)\op$ (which is immediate to check) and that by definition $\canst_{(\cbicat,J)}:=\widetilde{\cbicat/-}$.
\end{proof}

\begin{remarks}\label{remrelativepresheaves}
\begin{enumerate}[(i)]
    \item 	This result generalizes Proposition 2.3 of \cite{giraud.classifying}, where it is formulated for lex stacks over lex site.
    \item We have already mentioned in Section \ref{sec:invimage_pb} that when $\dcat$ is an internal category of $\Sh(\cbicat,J)$ then the Giraud topos $\Gir_J(\dcat):=\Sh(\gbicat(\dcat),J_\dcat)$ is the topos of internal presheaves for $\dcat$:
    \begin{equation}\label{eq:prefasci_interni_girtopos}
    \Gir_J(\dcat)\simeq [\dcat\op,\Sh(\cbicat,J)].\end{equation} 
    The Corollary above provides the external intuition for that: if we think of $\canst_{(\cbicat,J)}$ as the embodiment of the topos $\Sh(\cbicat,J)$ in terms of stacks over $\cbicat$, then we have an immediate parallelism between \ref{eq:prefasci_interni_girtopos} and
    \[
    \Gir_J(\dcat)\simeq \Ind_\cbicat\left(\dcat\Vop, \canst_{(\cbicat,J)}\right)
    .\]
    Thinking of Giraud's topos as the topos of presheaves over a fibration will be the starting point to develop relative topos theory, by generalizing the usual notions at the level of sites (sheaves, flat functors, morphisms, etc.) to stacks.
\end{enumerate}	
\end{remarks}
As a corollary we obtain an alternative proof of the fact that the canonical fibration $\canst_{(\cbicat,J)}$ is a stack:
\begin{cor}\label{cor:canst_via_aggfond}
	Consider a site $(\cbicat,J)$: then its canonical fibration $\canst_{(\cbicat,J)}:\cbicat\op\rightarrow\CAT$ is a $J$-stack.
\end{cor}
\begin{proof}
	Consider a $J$-covering sieve $m_R:R\rightarrowtail \yo(X)$: Lemma \ref{lemma:topos_su_sieve_eqv_topos_su_slice} showed that
	\[ 
	\Sh(\fib R, J_R)\simeq \Sh(\cbicat/X, J_X)
	\]
	via the geometric morphism $C_{m_R}$. Therefore we have a commutative square
	\[ 
	\begin{tikzcd}
		{\Sh(\fib R, J_R)} \arrow[r, "\sim", no head]                      & {\Ind_\cbicat( R\Vop, \canst_{(\cbicat,J)})}                               \\
		{\Sh(\cbicat/X, J_X)} \arrow[r, "\sim", no head] \arrow[u, "\sim"{sloped}] & {\Ind_\cbicat( \yo(X)\Vop, \canst_{(\cbicat,J)})} \arrow[u, "-\circ m_R"']
	\end{tikzcd}\]
	Finally, it is enough to notice that since both $R$ and $\yo(X)$ are discrete they are left unchanged by ${(-)}\Vop$, \ie $R\simeq R\Vop$ and $\yo(X)\simeq \yo(X)\Vop$, to conclude the proof.
\end{proof}

The equivalence $\Gir_J(\dcat)\simeq\Ind_\cbicat(\dcat\Vop, \canst_{(\cbicat,J)})$ gives us a way of seeing $J_\dcat$-sheaves as gluing of local data, since a $J_\dcat$-sheaf $W:\gbicat(\dcat)\op\rightarrow\Set$ corresponds to a morphism of Grothendieck fibrations $\gbicat(\dcat\Vop)\rightarrow \gbicat(\canst_{(\cbicat,J)})$, \ie to the following data:
\begin{itemize}
	\item for every $X$ in $\cbicat$ and $U$ in $\dcat(X)$, a $J_X$-sheaf $H_U:(\cbicat/X)\op\rightarrow\Set$;
	\item for every $U$ in $\dcat(X)$ and every pair of arrows $y:Y\rightarrow X$ in $\cbicat$ and $a:\dcat(y)(U)\rightarrow V$ in $\dcat(Y)$, a morphism of presheaves $h_{(y,a)}:H_V\rightarrow H_U\circ (\fib y)\op$ such that the association $(y,a)\mapsto h_{(y,a)}$ is functorial and moreover whenever $a$ is invertible then $h_{(y,a)}$ is invertible.  
\end{itemize}
A similar description of arrows of $\Gir_J(\dcat)$ in terms of local data can be given.

The last result also implies that the canonical stack $\canst_{(\cbicat,J)}$ has somewhat the role of a \emph{dualizing object} between the two categories $\Ind_\cbicat$ and $\Topos\co/\Sh(\cbicat,J)$, because it allows us to express both $\Gamma$ and $\Lambda$ as hom-functors in $\canst_{(\cbicat,J)}$ as the two following equivalences show:
\begin{align*}
	\Gamma_{\Topos\co/\Sh(\cbicat,J)} (\Etopos):&=\Topos\co/\Sh(\cbicat,J) (\Sh(\cbicat/-, J_{(-)}), \Etopos)\\
	&\simeq \Topos\co/\Sh(\cbicat,J)(\canst_{(\cbicat,J)}(-), \Etopos),\\
	\Lambda_{\Topos\co/\Sh(\cbicat,J)}(\dcat):&=\Gir_J(\dcat)\\
	&\simeq \Ind_\cbicat (\dcat\Vop, \canst_{(\cbicat,J)}).
\end{align*}

\chapter{The discrete setting}\label{chap:discrete}

In this chapter, we specialize the two adjunctions
\[
\begin{tikzcd}
	\Cl\Fib_\cbicat \arrow[r, "\Lambda_{\CAT/\cbicat}", bend left, start anchor={north east}, end anchor={north west}] &   \CAT/\cbicat \arrow[l, "\Gamma_{\CAT/\cbicat}", bend left, start anchor={south west}, end anchor={south east}] \ar[l, "\dashv"{rotate=270}, phantom]
\end{tikzcd},\ 
\begin{tikzcd}
	\Cl\Fib^J_\cbicat \arrow[r, "\Lambda_{\Topos/\Sh(\cbicat,J)\co}", bend left, start anchor={north east}, end anchor={north west}] &     \Topos\co/\Sh(\cbicat,J)\phantom{{}^J_\cbicat} \arrow[l, "\Gamma_{\Topos/\Sh(\cbicat,J)\co}", bend left, ""{above, name=B}, end anchor={south east}, start anchor={south west}] \ar[l, "\dashv"{rotate=270}, phantom]
\end{tikzcd}
\]
of Chapter \ref{chap:fundadj} to the discrete setting, that is, we study their behaviour on presheaves. The restriction of the first adjunction is not so interesting \textit{per se}, but it will be put to good use in the context of preorder categories; on the other hand, the fundamental adjunction provides a broad generalization of the topological presheaf-bundle adjunction which we will review in Section \ref{sec:adj_topologica}: toposes over $\Sh(\cbicat,J)$ play the role of topological spaces over $X$, and we can again recover the sheafification functor as the composite of the two adjoints. The example of topological spaces will be motivating for the study of the fundamental adjunction in the context of preordered categories, which is carried out in Section \ref{sec:preorder}; in this setting, the adjunction can be notably formulated in purely site-theoretic language, without any reference to geometric morphisms. Further, we investigate the possibility of building structure sheaves for algebraic structures through the fundamental adjunction, establishing a number of results useful in this regard. More specifically, we undertake a systematic study of the ways for equipping the domain of a map towards a given topological space with a topology making it a local homeomorphism; such a classification involves the consideration of sections of the given map defined on open subsets of the base topological space, and can be profitably applied in cases where one disposes of an \ac algebraic' presentation of the topos of sheaves on the given space providing a basis for it (we discuss two well-known examples of structure sheaves for algebraic structures, namely the Zariski structure sheaf for a commutative ring with unit and the structure sheaf for a MV-algebra introduced in \cite{dubuc_poveda}, to illustrate this point). We also make a refined analysis of the interplay between the point-free and point-set perspectives in the setting of the classical presheaf-bundle adjunction for topological spaces.

Section \ref{sec:point_of_view_sheafify} exploits the discrete fundamental adjunction to list various possible descriptions of the sheafification functor. 

\section{The specialization to presheaves}\label{sec:discrete_adj}
Suppose that $\cbicat$ is small: then, given any presheaf $P:\cbicat\op\rightarrow\Set$, its discrete fibration $\int P$ is again a small category. When defining the Grothendieck construction $\gbicat$ we mentioned that its image in $\CAT/\cbicat$ is contained in the sub-2-category of strictly commutative triangles: this means in particular that presheaves are sent through $\gbicat$ to the 1-categorical slice $\Cat/_1\cbicat$. With this in mind, we can easily show the following:
\begin{prop}\label{prop:aggiunzione_prefasci_e_cat_su_C}\index{$\Lambda_{\Cat/_1\cbicat}\dashv \Gamma_{\Cat/_1\cbicat}$}
	Consider a small category $\cbicat$: the adjunction of Proposition \ref{prop:aggiunti_di_G}restricts to an adjunction
	\[\begin{tikzcd}
		{[\cbicat\op,\Set]}  \ar[r, bend left, "\Lambda_{\Cat/_1\cbicat}", start anchor={north east}, end anchor={north west}] \ar[r, phantom, "\vdash"{rotate=90}]& {\Cat/_1\cbicat} \ar[l, bend left, "{\Gamma_{\Cat/_1\cbicat}}", start anchor={south west}, end anchor={south east}]	
	\end{tikzcd}\]
	where:\begin{itemize}
		\item $\Cat/_1\cbicat$ denotes the 1-categorical slice of $\Cat$ over $\cbicat$ introduced in Definition \ref{def:slice};
		\item $\Lambda_{\Cat/_1\cbicat}$ is the restriction of $\Lambda_{\CAT/\cbicat}$: it maps a presheaf $P$ to the functor $\fib P\rightarrow\cbicat$, and an arrow $g:P\rightarrow Q$ of presheaves to the functor $\fib g:\fib P\rightarrow\fib Q$;
		\item $\Gamma_{\Cat/_1\cbicat}$ maps a functor $p:\dbicat\rightarrow\cbicat$ to the presheaf $\Cat/_1\cbicat(\cbicat/-,\dbicat):\cbicat\op\rightarrow\Set$.
	\end{itemize}
	All objects of $[\cbicat\op,\Set]$ are fixed points of the adjunction, while the fixed points of $\Cat/_1\cbicat$ are exactly the discrete Grothendieck fibrations over $\cbicat$.
\end{prop}
\begin{proof}
	The hypothesis that both $\cbicat$ and $\dbicat$ are small implies that $\Gamma_{\Cat/_1\cbicat}$ has image in $[\cbicat\op,\Set]$. Thus for $P$ in $ [\cbicat\op,\Set]$ and $p:\dbicat\rightarrow \cbicat$ we can consider the natural isomorphism
	\[
	[\cbicat\op,\Set](P, \Cat/_1\cbicat(\cbicat/-, [p]))\simeq \Cat/_1 \cbicat( \fib P, [p])
	\]
	given by the 2-adjunction $\Lambda_{\CAT/\cbicat}\dashv \Gamma_{\CAT/\cbicat}$. It maps an arrow $g:P\rightarrow \Cat/_1/\cbicat(\cbicat/X, [p])$ to $\bar{g}:\fib P\rightarrow \dbicat$ defined on objects by $\bar{g}(X,s):=g_X(s)$; conversely, $h:\fib P\rightarrow\dbicat$ is mapped to $\tilde{h}:P\rightarrow \Cat/_1\cbicat(\cbicat/-,[p])$, $h_X(s):\cbicat/X\rightarrow [p]$ being defined on objects by $h_X(s)([y]):=h(Y, P(y)(s))$.
	
	The unit of the adjunction for a presheaf $P$ is the arrow $\eta_P:P\rightarrow \Gamma\Lambda(P)$ such that for every $X$ in $\cbicat$ the function $\eta_P(X):P(X)\rightarrow \Cat/_1\cbicat(\cbicat/X, \fib P)$ maps any $s\in P(X)$ to the functor $\cbicat/X\rightarrow \fib P$ defined on objects as $[y:Y\rightarrow X]\mapsto (Y, P(y)(s))$: one immediately verifies that $\eta_P$ is in fact the isomorphism $P\isorightarrow\Cat/_1\cbicat(\cbicat/-, \fib P)$ given by fibered Yoneda lemma, and hence every presheaf $P$ is a fixed point. Conversely, for any $p:\dbicat\rightarrow \cbicat$ the counit $\epsilon_{[p]}:\fib (\Cat/_1\cbicat(\cbicat/-,[p]))\rightarrow [p]$ is defined by mapping a pair $(X, F:\cbicat/X\rightarrow \dbicat)$ to $F([1_X])$: then $\epsilon_{[p]}$ is invertible if and only if $\dbicat$ is isomorphic to the discrete fibration $\fib (\Cat/_1\cbicat(\cbicat/-, [p]))$.
\end{proof} 

Let us now consider the topos-theoretic adjunction. The first thing we need to take care of is a size problem: in general, the hom category between two toposes is \emph{not} a set, and hence we have to refine the definition of the right adjoint $\Gamma_{\Topos/\Sh(\cbicat,J)\co}$. This is rather immediate:
\begin{defn}\label{def:relativelysmallmorphism}
	We call a geometric morphism $F:\Ftopos\to \Sh(\cbicat, J)$ \emph{small relative to $\Sh(\cbicat,J)$}\index{topos!small relatively to a base} if for any $J$-sheaf $P:\cbicat\op\rightarrow\Set$ the geometric morphisms $\Sh(\cbicat,J)/P \to \Ftopos$ over $\Sh(\cbicat,J)$ form a set (up to equivalence of geometric morphisms): more compactly, if the category $$\Topos/_1\Sh(\cbicat,J)(\Sh(\cbicat,J)/P, \Ftopos)$$ is small. 
	We denote by $\Topos^s/_1\Sh(\cbicat,J)$\index{$\Topos^s/_1\Sh(\cbicat,J)$} the full subcategory of the 1-category $\Topos/_1\Sh(\cbicat,J)$ whose objects are the small geometric morphisms relative to $\Sh(\cbicat,J)$.
\end{defn}
\begin{remark}
	In fact, one can reduce to checking the smallness of all categories $\Topos/_1\Sh(\cbicat,J)(\Sh(\cbicat,J)/\ell_J(X), \Ftopos)$ for $X$ an object of $\cbicat$: this because the topos $\Sh(\cbicat,J)/P\simeq \Sh(\fib P, J_P)$ is a conical colimit of toposes of the form $\Sh(\cbicat/X, J_X)\simeq\Sh(\cbicat,J)/\ell_J(X)$, by Corollary \ref{cor:topos_discfib_colimite_conico}.
\end{remark}
The adjunction between presheaves and bundles over a topological space (see Section \ref{sec:adj_topologica}), which inspired these results, restricts to an equivalence between sheaves and étale bundles over the space. Also the fundamental adjunction, when restricted to the discrete case, induces an equivalence between the category $\Sh(\cbicat,J)$ and a special class of toposes over $\Sh(\cbicat,J)$, which are unsurprisingly the étale toposes:
\begin{defn}\label{def:topos_etale}
	A geometric morphism $F:\Ftopos\rightarrow\Etopos$ is said to be \emph{étale}, or a \emph{local homeomorphism}\index{geometric morphism!étale}\index{topos!étale} \index{geometric morphism!local homeomorphism}, if there is some $E$ in $\Etopos$ such that $F$ is isomorphic to the canonical geometric morphism $\prod_E:\Etopos/E\rightarrow\Etopos$.
	
	We denote by $\Topos\etale/_1\Sh(\cbicat,J)$\index{$\Topos\etale/_1\Sh(\cbicat,J)$} the 1-category of étale $\Sh(\cbicat,J)$-toposes.
\end{defn}
\begin{remark}
    We recall that the category $\Topos\etale/_1\Sh(\cbicat,J)$ embeds fully inside $\Topos_1\Sh(\cbicat,J)$, for it is well-known that a geometric morphism between étale toposes is itself étale.
\end{remark}
Étale toposes are the way in which one can externalize objects of a topos $\Etopos$ as toposes over $\Etopos$, and significantly this process is full and faithful: 
\begin{lemma}\label{lemma:morfgeom_tra_topos_slice_su_base}
	Consider a topos $\Etopos$: then the functor $\Etopos\rightarrow \Topos/_1\Etopos$ of 1-categories mapping each $E$ in $\Etopos$ to $\prod_E:\Etopos/E\rightarrow \Etopos$ and each $g:X\rightarrow Y$ to $\prod_g:\Etopos/X\rightarrow \Etopos/Y$ presents $\Etopos$ as the full subcategory of étale geometric morphisms over $\Etopos$. More explicitly, there is an isomorphism
	\[\Etopos(X,Y)\simeq \Topos/_1\Etopos(\Etopos/X, \Etopos/Y). \]
\end{lemma}
\begin{proof}
	This is a well known result, but let us sketch the correspondence for later use: starting from $g:X\rightarrow Y$, we consider the functor $\prod_g$. Viceversa, consider the arrow $\Delta_{Y}:Y\rightarrow Y\times Y$ in $\Etopos$: it is also an arrow $\Delta_Y:1_{\Etopos/Y}\rightarrow Y^*(Y)$ of $\Etopos/Y$, where $Y^*:\Etopos\rightarrow \Etopos/Y$ is the usual functor mapping any $Z$ to $[\pi:Y\times Z\rightarrow Y]$. Now take any $F:\Etopos/X\rightarrow \Etopos/Y$: if we consider the arrow $F^*(\Delta_Y):1_{\Etopos/X}\rightarrow X^*(Y)$, \ie $F^*(\Delta_Y):[1_X]\rightarrow [X\times Y\rightarrow X]$, it is an arrow $\langle 1,g\rangle:X\rightarrow X\times Y$ of $\Etopos$ and this provides our arrow $g:X\rightarrow Y$.
\end{proof}
\begin{remarks}\begin{enumerate}[(i)]
		\item 	Every local homeomorphism to $\Sh(\cbicat,J)$ is small relative to $\Sh(\cbicat,J)$: this holds because the geometric morphisms over $\Sh(\cbicat,J)$ from a local homeomorphisms $\Sh(\cbicat,J)\slash P $ to a local homeomorphism $\Sh(\cbicat,J)\slash Q$ correspond precisely to the arrows $P\to Q$ in $\Sh(\cbicat,J)$. 
		\item Suppose that $\Etopos\simeq \Sh(\cbicat,J)$: since $\Sh(\int P, J_P)\simeq \Sh(\cbicat,J)/\sheafify_J(P)$ by Proposition \ref{prop:fib_discreta_slice_topos}, the previous lemma is telling us that a geometric morphism $\Sh(\fib P, J_P)\rightarrow \Sh(\fib Q, J_Q)$ is presented by an arrow $\alpha:\sheafify_J(P)\rightarrow \sheafify_J(Q)$. The fact that we can replace geometric morphisms with arrows in the topos is very relevant, since it will imply that we will be able to work at a lower level of complexity by `hiding' the topos-theoretic content of our adjunction. This is not an uncommon feature: the same happens for topological spaces (see Section \ref{sec:adj_topologica}), where the right adjoint $\Gamma$ can be described as a hom functor at the level of topological spaces, and it happens more generally for preordered categories (see Section \ref{sec:preorder}). In general, this is a feature of classes of geometric morphisms that can be presented at the level of sites.
	\end{enumerate}
\end{remarks}
We have now all the ingredients to restrict the fundamental adjunction to presheaves, generalizing at the same time the topological presheaf-bundle adjunction of Section \ref{sec:adj_topologica} to all sites:
\begin{prop}\label{prop:fundadj_discrete}\index{presheaf-bundle adjunction!for sites}\index{$\Lambda_{\Topos^s/_1\Sh(\cbicat,J)}\dashv \Gamma_{\Topos^s/_1\Sh(\cbicat,J)}$} Consider a small-generated site $(\cbicat,J)$:
	\begin{enumerate}[(i)]
		\item There is an adjunction of 1-categories
		\[
		\begin{tikzcd}
			{[\cbicat\op,\Set]} \arrow[r, "\Lambda_{\Topos^s/_1\Sh(\cbicat,J)}", bend left, start anchor={north east}, end anchor={north west}] &     \Topos^s/_1\Sh(\cbicat,J) \arrow[l, "\Gamma_{\Topos^s/_1\Sh(\cbicat,J)}", bend left, start anchor={south west}, end anchor={south east}] \ar[l, "\dashv"{rotate=270}, phantom]
		\end{tikzcd}.\]
		The functors $\Lambda_{\Topos^s/_1\Sh(\cbicat,J)}$ is the restriction of $\Lambda_{\Topos/\Sh(\cbicat,J)\co}$: \ie it maps a presheaf $P$ to $[\prod_{\sheafify_J(P)}:\Sh(\cbicat,J)/\sheafify_J(P)\rightarrow \Sh(\cbicat,J)]$ and an arrow $g:P\rightarrow Q$ to  $\prod_{\sheafify_J(g)}:\Sh(\cbicat,J)/\sheafify_J(P)\rightarrow \Sh(\cbicat,J)/\sheafify_J(Q)$; exploiting Proposition \ref{prop:fib_discreta_slice_topos} we can rephrase this in terms of comorphisms of sites, , as
		$\Lambda(P):=[C_{p_P}:\Sh(\fib P, J_P)\rightarrow \Sh(\cbicat,J)]$ and $\Lambda(g):=C_{\fib g}: \Sh(\fib P, J_P)\rightarrow \Sh(\fib Q, J_Q)$.
		
		The functor $\Gamma_{\Topos^s/_1\Sh(\cbicat,J)}$ acts like a Hom-functor by mapping an object $[F:\Ftopos\rightarrow \Sh(\cbicat,J)]$ of $\Topos^s/_1\Sh(\cbicat,J)$ to the presheaf $$\Topos^s/_1\Sh(\cbicat,J)(\Sh(\cbicat,J)/\ell_J(-), \Ftopos):\cbicat\op\rightarrow \Set.$$
		\item The image of $\Lambda_{\Topos^s/_1\Sh(\cbicat,J)}$ factors through $\Topos\etale/\Sh(\cbicat,J)$, and the image of $\Gamma_{\Topos^s/_1\Sh(\cbicat,J)}$ factors through $\Sh(\cbicat,J)$;
		\item for any $J$-sheaf $Q$ it holds that $\Gamma_{\Topos^s/_1\Sh(\cbicat,J)}([\prod_Q])\simeq Q$, implying that the fixed points of $\Topos^s/_1\Sh(\cbicat,J)$ are precisely the étale geometric morphisms, while those of $[\cbicat\op,\Set]$ are $J$-sheaves; in particular, the composite functor $\Gamma_{\Topos^s/_1\Sh(\cbicat,J)}\Lambda_{\Topos^s/_1\Sh(\cbicat,J)}$ is naturally isomorphic to $$i_J\sheafify_J:[\cbicat\op,\Set]\rightarrow \Sh(\cbicat,J)\rightarrow[\cbicat\op,\Set];$$
		\item the adjunction $\Lambda_{\Topos^s/_1\Sh(\cbicat,J)}\dashv \Gamma_{\Topos^s/_1\Sh(\cbicat,J)}$ restricts to an equivalence 
		\[\Sh(\cbicat,J)\simeq \Topos\etale/_1\Sh(\cbicat,J).\]
	\end{enumerate}
\end{prop}
\begin{proof} Let us once again adopt the notation $\widetilde{\cbicat}:=\Sh(\cbicat,J)$ for the sake of brevity, and let us drop the subscripts for $\Lambda$ and $\Gamma$:	
	\begin{enumerate}[(i)]
		\item 	This is just a restriction of the adjunction appearing in Theorem \ref{thm:topos_Giraud_colimite_pesato}: indeed, the equivalence of Hom-categories
		\[
		\Topos\co/\widetilde{\cbicat}(\Lambda(P), \Etopos)\simeq \Ind_\cbicat(P, \Gamma(\Etopos)) 
		\]
		appearing there restricts to a bijection of Hom-sets (up to equivalence of geometric morphisms)
		\[
		\Topos^s/_1\widetilde{\cbicat}(\Lambda(P), \Etopos)\simeq [\cbicat\op,\Set](P, \Gamma(\Etopos)) 
		\]
		because $\Lambda(P)$ is a 1-colimit of toposes by Corollary \ref{cor:topos_discfib_colimite_conico}.
		\item The fact that the image of $\Lambda$ factors through $\Topos\etale/\widetilde{\cbicat}$ is true by definition. The image of $\Gamma$ is contained in $\widetilde{\cbicat}$ as a consequence of Lemma \ref{lemma:immagine_gamma_in_Stack}, once recalled that a presheaf $P$ is a $J$-stack if and only if it is a $J$-sheaf (see Proposition \ref{prop:psh+stack=sheaf}).
		\item  Consider an étale geometric morphism $[\prod_Q:\widetilde{\cbicat}/Q\rightarrow \widetilde{\cbicat}]$: then
		\begin{align*}
			\Gamma([\textstyle\prod_Q]):= \Topos^s/_1\widetilde{\cbicat}(\widetilde{\cbicat}/\ell_J(-), \widetilde{\cbicat}/Q)\simeq \widetilde{\cbicat}(\ell_J(-),Q)\simeq Q
		\end{align*}
		This implies that $\Lambda\Gamma([\prod_Q])\simeq \Lambda(Q)\simeq [\prod_Q]$, and hence $[\prod_Q]$ is a fixed point for $\Lambda\dashv \Gamma$; conversely, if $[F:\Ftopos\rightarrow\widetilde{\cbicat}]$ is a fixed point then it is isomorphic to $[\prod_{\sheafify_J(\Gamma([F]))}]$ and hence it is étale. 
		The identity above also implies that $\Gamma\Lambda(P)\simeq \Gamma([\prod_{\sheafify_J(P)}])\simeq \sheafify_J(P)$, and in particular that a presheaf $P$ is a fixed point for $\Lambda\dashv \Gamma$ if and only if it is a $J$-sheaf.
		\item It follows from restricting the adjunction $\Lambda\dashv \Gamma$ to its fixed points.
	\end{enumerate}
\end{proof}

\section{The adjunction between presheaves and bundles over a topological space}\label{sec:adj_topologica}
In the present section we will recall the motivating example for the previous work, namely the adjunction $\Lambda\dashv \Gamma$ between the topos of presheaves over a topological space $X$ and the category of bundles over $X$. We also take a chance to analyse some aspects of the adjunction that will provide the motivation for the later developments in the case of preorder sites. Since this is the `original' fundamental adjunction, we will not decorate the two adjoints $\Lambda$ and $\Gamma$ with any subscripts.

Given a topological space $X$, denote by $\Ocal(X)$ its poset of open subsets: the topos $\Psh(X):=[\Ocal(X)\op,\Set]$\index{$\Psh(X)$} is called the category of presheaves (of sets) over $X$: the topos of sheaves for the canonical open cover topology $J\can_{\Ocal(X)}$ on $\Ocal(X)$ is denoted by $\Sh(X)$\index{$\Sh(X)$}. On the other hand, if $\Top$ is the 1-category of topological spaces, the 1-categorial slice $\Top/X$ is called the category of \emph{bundles} over $X$. 

Since $\Ocal(X)$ is posetal, there is at most one arrow $i:V\hookrightarrow U$ for any two $V$ and $U$ in $\Ocal(X)$: therefore, for any presheaf $P$ over $X$ and any $s\in P(U)$ its image $P(j)(s)$ is usually denoted just with $s_{|V}$. We recall that for a presheaf $P$ over $X$ the \emph{stalk of $P$ at $x\in X$}\index{stalk}\index{$P_x$} is defined as the colimit of sets $P_x:=\colim_{x\in U\in \Ocal(X)} P(U)$. For any $s\in P(U)$ its equivalence class in $P_x$ is called \emph{germ of $s$ at $x$}\index{germ}\index{$s_x$} and it is indicated by $s_x$: then $s\in P(U)$ and $t\in P(V)$ have the same germ at $x\in U\cap V$ if and only if there exists $W\subseteq U\cap V$ such that $x\in W$ and $s_{|W}=t_{|W}$. Of course, any arrow $h:P\rightarrow Q$ of presheaves induces an arrow $h_x:P_x\rightarrow Q_x$ by the universal property of colimits.

There is an adjunction\index{presheaf-bundle adjunction!for topological spaces} (cfr. \cite[Sections II.4, II.5, II.6]{maclanemoerdijk})
\[
\begin{tikzcd}
	{\Psh(X)} \ar[r, bend left,  start anchor={north east}, end anchor={north west}, "\Lambda"] \ar[r,phantom, "\dashv"{rotate=-90}] & \Top/X \ar[l, bend left, start anchor={south west}, end anchor={south east}, "\Gamma"]
\end{tikzcd}.\]

The functor $\Lambda$ maps a presheaf $P$ to its \emph{bundle of germs}, \ie the canonical projection $\pi_{P}:E_P=\coprod_{x\in X}P_{x}\to X$; on arrows $\Lambda$ acts by mapping $h:P\rightarrow Q$ to $\Lambda_h:=\coprod_{x\in X} h_x:\coprod_{x\in X}P_x\rightarrow \coprod_{x\in X} Q_x$. For any $s\in P(U)$ we can define a map $\dot{s}:U \to E_{P}$ sending a point $x\in U$ to the germ $s_{x}$: since
\[
\dot{s}(U) \cap \dot{t}(V) = \bigcup_{W\subseteq U \cap V, r=s|_{W}=t|_{W}} \dot{r}(W),
\]
for any $s\in P(U)$ and $t\in P(V)$, the sets of the form $\dot{s}(U)$ are the basis for a topology over $E_P$. With this topology one can show that $\pi_P$ and all the functions of kind $\Lambda_h$ or $\dot{s}$ are continuous. 

The functor $\Gamma$ is instead the \emph{local sections} functor, assigning to a bundle $p:E\to X$ the presheaf $\Gamma_{p}$ which sends each open set $U$ of $X$ to the set $\Gamma_{p}(U)$ of continuous maps $s:U\to E$ such that $p\circ s=i_{U}:U\hookrightarrow X$: these are called the \emph{sections} of $p$ defined on $U$. More compactly,  $\Gamma_p:=\Top/X([i_{(-)}],[p]):\Ocal(X)\op\rightarrow\Set$.
The functor $\Lambda$ takes values in the full subcategory $\Etale(X)$\index{$\Etale(X)$} of étale bundles on $X$ (that is, the local homeomorphisms to $X$), while the functor $\Gamma$ takes values in the full subcategory $\Sh(X)$ of $\Psh(X)$; in fact, the adjunction $\Lambda\dashv\Gamma$ restricts to an equivalence between $\Etale(X)$ and $\Sh(X)$. 
\begin{remark}
	One often unmentioned consequence of the topological adjunction $\Lambda\dashv \Gamma$ is that the space $E_P$ is a colimit of topological spaces over $X$. Indeed, by the very definition of adjunction we have that, for any bundle $q:Y\rightarrow X$, 
	\begin{align*}
		\Top/X([\pi_P], [q])&\simeq \Psh(X)(P,\Top/X([i_{(-)}],[q])\\
		&\simeq[(\fib P)\op,\Set](1, \Top/X([i_{\pi_P(-)}], [q])),
	\end{align*}
	where the last equivalence is a consequence of Proposition \ref{prop:eqv_categorie_laxoplaxps}. In other words, $E_P$ can be seen as the colimit of the diagram $D:\fib P\rightarrow \Top/X$ mapping every $(U,s)$ to the inclusion $U\hookrightarrow X$. 
\end{remark}

We will now provide some pointfree results related to the adjunction $\Lambda\dashv \Gamma$. The first result  states that set-theoretic sections for a bundle of germs are topological sections if and only if they are locally so: it will become useful later.
\begin{prop}\label{prop:sezione_sse_sezionilocali}
	Let $X$ be a topological space, $U$ an open set of $X$, $P$ a presheaf on $X$ and $s:U\to \coprod_{x\in X}P_{x}$ a map such that $p_{P}\circ s=i_{U}$. Then the following conditions are equivalent:
	\begin{enumerate}[(i)]
		\item $s$ is continuous;
		\item there is an open covering $\{U_{i}\hookrightarrow U \mid i\in I\}$ of $U$ such that for each $i\in I$ there is $t_{i}\in P(U_{i})$, $s|_{U_{i}}=\dot{t_{i}}$.
	\end{enumerate}
\end{prop}
\begin{proof}	
	We will use the fact that the collection of subsets of the form $\dot{t}(V)$ for $V\in\Ocal(X)$ and $t\in P(V)$ are an open covering and a basis of $\coprod_{x\in X}P_{x}$. We remark preliminarily that if we consider $z\in s\inv(\dot{t}(V))\subseteq U$, then there must be some $v\in V$ such that $s(z)=t_v$: this implies that $z=\pi_P(s(z))=\pi_P(t_v)=v$, so that $s\inv(\dot{t}(V))\subseteq V$, and $s(z)=t_z$.
	
	First suppose that $s$ is continuous: then the collection of subsets of the form $Z:=s\inv(\dot{t}(V))$ is an open covering of $U$, and we have shown above that $s|_{Z}=\dot{t|_{Z}}$: thus condition (ii) is satisfied.
	
	Conversely, assume that $(ii)$ holds. Since the subsets $\dot{t}(V)$ form a basis, showing the continuity of $s$ is equivalent to verifying that the inverse image under $s$ of any of these open sets is open. We will prove this by showing that $s\inv(\dot{t}(V))$ contains an open neighbourhood of each of its points. Consider thence $y\in s\inv(\dot{t}(V))$: we already know that $s(y)=t_y$, but by our hypothesis there is also $i\in I$ such that $y\in U_{i}$ and $s(y)=(t_i)_y$. Since $(t_{i})_{y}=t_{y}$ there is an open neighbourhood $W_{i}$ of $y$ such that $W_{i}\subseteq U_{i}\cap V$ and $t_{i}|_{W_{i}}=t|_{W_{i}}$. So $W_{i}$ is an open neighbourhood of $y$ contained in $s^{-1}(\dot{t}(V))$, and thus $s\inv(\dot{t}(V))$ is open.    
\end{proof}
\begin{remark}
	Notice that condition (ii) of Proposition \ref{prop:sezione_sse_sezionilocali} tells us that the topological condition of continuity of $s$ is equivalent to a purely set-theoretic condition. This is a first hint of the fact that the adjunction $\Lambda\dashv \Gamma$ can de bescribed without explicitly recurring to the category of topological spaces: we will provide stronger evidence of this in the following results. 
\end{remark}

The following technical lemma will allow us to show that sections or local homeomorphisms are always open maps: 
\begin{lemma}
	Let $s:U\to E$ be a section of a local homeomorphism $p:E\to X$ on an open set $U$ of $X$. Then, any open set $V$ of $U$ can be covered by a family $\{V_{i} \mid i\in I\}$ of open subsets $V_{i}\subseteq V$ such that for each $i\in I$, $s(V_{i})$ is contained in an open subset $A_{i}$ of $E$ such that $p|_{A_{i}}:A_{i}\to p(A_{i})$ is an homeomorphism onto an open set $p(A_{i})$ of $X$.  	
\end{lemma}
\begin{proof}
	Since $p$ is a local homeomorphism, there is an open covering $\{A_{i}\subseteq E \mid i\in I\}$ of $E$ such that for each $i$, $p|_{A_{i}}\to p(A_{i})$ is an homeomorphism onto an open set $p(A_{i})$ of $X$. By taking $V_{i}=s^{-1}(A_{i})\cap V$, we thus obtain an open covering of $V$ satisfying the required condition.
\end{proof}
\begin{cor}\label{cor:sezioni_bundleetale_aperte}
	Consider a topological space $X$ and $P$ in $\Psh(X)$: then any section of the étale bundle $E_P\rightarrow X$ is an open map.
\end{cor}
\begin{proof}
	Consider an open $V$ of $X$ and a section $s:V\rightarrow E_P$: given a covering $\{V_{i} \mid i\in I\}$ of $V$ as in the previous lemma, $s(V)=\bigcup_{i\in I}s(V_{i})$, and each $s(V_{i})$ is open as it is the inverse image of the open set $V_{i}=p(s(V_{i}))\subseteq p(A_{i})$ along the local homeomorphism $p|_{A_{i}}:A_{i}\to p(A_{i})$. 
\end{proof}

Finally, we recall that open maps are precisely those that they induce and adjunction at the level of topologies:
\begin{lemma}\label{lemma:open_map_induce_adjunction_topologies}
	Consider a continuous map $e:E\rightarrow X$: it is open if and only if $e\inv: \Ocal(X)\rightarrow \Ocal(E)$ admits a left adjoint $e_!$, which acts by mapping $U\subseteq E$ to $e(U)\subseteq X$.
\end{lemma}

We will now provide two results which pave the way for our future considerations: we defer all remarks about them to the end of the present section.
The first result shows that the bundle of germs for a presheaf $P$ is a topological site of presentation for the classifying topos of the fibration $\fib P$ (see Definition \ref{def:Giraud_topology_classifying_topos}). This gives a topos-theoretic motivation to the relevance of the étale bundle of a presheaf, and it also justifies why in the discrete adjunction of Section \ref{sec:discrete_adj} the 2-category $\Cl\Fib_\cbicat$ has become the rightful substitute for $\Top/X$.
\begin{prop}\label{prop:fasci__etalebundle_eqv_fasci_grfibration}
	Let $X$ be a topological space and $P$ a presheaf $P$ on it. Consider the category $\fib P$ and endow it with the Grothendieck topology $J_P$ defined in Proposition \ref{prop:fib_discreta_slice_topos}: in particular, $\{(U_i, s_{|U_i})\hookrightarrow (U, s)\ |\ i\in I \}$ is $J_P$-covering if and only if $\{U_i\ |\ i\in I \}$ is an open covering of $U$. Consider the bundles of germs $E_P$: then the functor $f_P:{\fib P}\to {\cal O}(E_{P})$ defined by $f_P(U, s):=\dot{s}(U)$ is a morphism and comorphism of sites which induces an equivalence of toposes
	\[
	C_{f_P}:\Sh(\fib P, J_{P})\simeq \Sh(E_{P}).
	\]  
	Moreover, said equivalence is compatible with the canonical geometric morphisms to $\Sh(X)$, and it is natural in $P$.
\end{prop}

\begin{proof}
	To prove that $f_P$ is a morphism of sites we use the characterization of \cite[Definition 3.2]{denseness}, which was divided in four conditions:
	\begin{enumerate}[(i)]
		\item $f_P$ sends covering families to covering families: indeed, if $\{(U_i, s_{|U_i})\}$ is an open cover of $(U,s)$ then $U=\bigcup_{i\in I} U_i$ and thus $\dot{s}(U)=\bigcup_{i\in I}\dot{s_{|U_i}}(U_i)$.
		\item for any $W$ in $\Ocal(E_P)$ there is a covering family $W_i\subseteq W$ such that for each $i$ there exists a $(U_i,s_i)$ in $\fib P$ such that $W_i\subseteq \dot{s_i}(U_i)$: this is trivial, because since the sets $\dot{s}(U)$ form a basis for $E_P$ there is a family of $(U_i, s_i)$ such that $W=\bigcup_{i \in I} \dot{s_i}(U_i)$.
		\item for any pair $(U,s)$, $(V,t)$ in $\fib P$ and every $W\subseteq \dot{s}(U)\cap\dot{t}(V)$ there exist an open covering $\{W_i\ |\ i\in I\}$ of $W$ and open subsets $U_i\subseteq U\cap V$ such that $s_{| U_i}=t_{| U_i}=r_i$ and $W_i\subseteq \dot{r_i}(U_i)$. Since $\dot{s}(U)\cap \dot{t}(V)=\bigcup_{Z\in I} \dot{s_{|Z}}(Z)$, where $I=\{Z\in \Ocal(X)\ |\ Z\subseteq U\cap V,\ s_{|Z}=t_{|Z} \}$, we can take the $U_i$'s above as the opens $Z$, set $r_Z:=s_{|Z}$ and $W_Z:= W\cap \dot{r_Z}(Z)$.
		\item The fourth condition concerns parallel pairs of arrows and is trivial in this case.
	\end{enumerate}
	
	To check that $f_P$ is a comorphism of sites, consider $(U,s)$ in $\fib P$ and an open covering of $\dot{s}(U)$: without loss of generality we may assume it to be of the form $\{\dot{t_i}(V_i)\ |\ i\in I \}$ for $(V_i, t_i)$ in $\fib P$. We want to show that there is a $J_P$-covering family $\{(U_j, s_j) \}$ of $(U,s)$ such that its image through $f_P$ is contained in the sieve generated by $\{\dot{t_i}(V_i)\ |\ i\in I \}$: that is, that is, every $\dot{s_j}(U_j)$ is contained in some $\dot{t_i}(V_i)$. First of all, notice that the equality $\bigcup_{i\in I}\dot{t_i}(V_i)=\dot{s}(U)$ implies immediately that $\bigcup_{i\in I}V_i=U$. Now consider any $v\in V_i$: since $(t_i)_v=(s_{|V_i})_{v}$ there must be an open neighbourhood $W_i(v)\subseteq V_i$ of $v$ such that $(t_i)_{|W_i(v)}=s_{|W_i(v)}$. Since each $V_i$ is covered by the $W_i(v)$, it follows that they also cover $U$: but then the objects $(W_i(v), (t_i)_{|W_i(v)})$ are a $J_P$-covering of $(U,s)$ satisfying the requirement above.

	To prove that the induced geometric morphism $C_{f_P}$ is an equivalence, we will exploit \cite[Proposition 7.18]{denseness}, that states that if $f_P$ is a morphism and comorphism of sites $f_P:(\fib P, J_{P})\to ({\cal O}(E_{P}), J\can_{{\cal O}(E_{P})})$ then it induces an equivalence of toposes 
	$\Sh(E_{P})\simeq \Sh(\fib P, J_{P})$ if and only if it is $J_P$-full and $J\can_{\Ocal(E_P)}$-dense. The definition of both appears in \cite[Definition 3.2]{denseness}: the $J_P$-fullness is trivial, while $J\can_{\Ocal(E_P)}$-density boils down to the $\dot{s}(U)$ being a basis for $E_P$.

	Finally, we recall that the canonical geometric morphism $\Sh(E_P)\rightarrow \Sh(X)$ is induced by the morphism of sites

	$\pi_P\inv:(\Ocal(X), J\can_{\Ocal(X)})\rightarrow (\Ocal(E_P), J\can_{\Ocal(E_P)})$.

	But notice that, since $\pi_P:E_P\rightarrow X$ is a local homeomorphism, it is an open map: this implies that $\pi_P\inv$ admits a left adjoint $\pi_P(-)$, which by \cite[Proposition 3.14]{denseness} is a comorphism  of sites such that $\Sh(\pi_P\inv)\cong C_{\pi_P(-)}$.

	If instead we consider an arrow of presheaves $h:P\rightarrow Q$ then $\Lambda_h$ is a local homeomorphism, since $\pi_Q$ and $\pi_P=\pi_Q\Lambda_h$ are, and hence it is also open: this implies again that the canonical geometric morphism $\Sh(\Lambda_h\inv):\Sh(E_P)\rightarrow \Sh(E_Q)$ is induced by the comorphism of sites $\Lambda_h(-):\Ocal(E_P)\rightarrow \Ocal(E_Q)$. Now, it is immediate to see that the diagrams of comorphisms on the left commute, and thus they induce the commutative diagrams of geometric morphisms on the right:

	\[
	\begin{tikzcd}
		\fib P \ar[dr] \ar[r, "f_P"] & \Ocal(E_P) \ar[d, "\pi_P(-)"]\\
		& \Ocal(X)
	\end{tikzcd}\hspace{1cm}
	\begin{tikzcd}
		\Sh(\fib P, J_P) \ar[dr] \ar[r, "\simeq"] & \Sh(E_P) \ar[d, "Sh(\pi_P)"] \\
		& \Sh(X)
	\end{tikzcd}
	\]
	\[
	\begin{tikzcd}
		\fib P \ar[d, "\fib h"'] \ar[r, "f_P"] & E_P \ar[d, "\Lambda_h"]\\
		\fib Q \ar[r, "f_Q"] & E_Q
	\end{tikzcd}\hspace{1cm}
	\begin{tikzcd}
		\Sh(\fib P, J_P) \ar[d, "{\fib h}"'] \ar[r, "\simeq"] & \Sh(E_P) \ar[d, "C_{\Lambda_h}"] \\
		\Sh( \fib Q, J_Q) \ar[r, "\simeq"] & \Sh(E_Q)
	\end{tikzcd}
	\]
\end{proof}

\begin{cor}
	The two functors
	\[
	[\Ocal(X)\op,\Set]\xrightarrow{\Lambda}\Top/X \xrightarrow{\Sh(-)}\Topos/\Sh(X) \]
	and
	\[
	[\Ocal(X)\op,\Set] \xrightarrow{\fib(-)}\Cl\Fib_{\Ocal(X)} \xrightarrow{F} \Topos/\Sh(X)
	\]
	are naturally isomorphic, where $F:=C_{(-)}\Giraud$ is the functor mapping any fibration $p:\dbicat\rightarrow\Ocal(X)$ to the geometric morphism $C_p:\Sh(\dbicat, M^p_{J\can_{\Ocal(X)}})\rightarrow \Sh(X)$ (see also Definition \ref{def:Giraud_topology_classifying_topos}).
\end{cor}

Our next result shows on the other hand that the sheafification functor $\Psh(X)\rightarrow \Sh(X)$, which we know to be equivalent to the composite $\Gamma\Lambda$, is essentially localic in spirit. For a quick recap on the basics of the theory of locales see Section \ref{sec:preorder}.
\begin{prop}\label{prop:fascificazione_topologica_con_framehom}
	Let $X$ be a topological space, $U$ an open set of $X$, $P$ a presheaf on $X$. Then the continuous sections $s$ on $U$ for the bundle $\Lambda(P)$ are in natural bijective correspondence with the frame homomorphisms $f:{\cal O}(E_P)\to \Ocal(U)$ such that $f\circ \pi_P\inv=i_U\inv=(-)\cap U$. In other words, there is a natural isomorphism of sheaves
	\[ \Top/_1X (-,E_P)\simeq \Locale/_1\Ocal(X)(\Ocal(-), \Ocal(E_P)):\Ocal(X)\op\rightarrow\Set,
	\]
	Moreover, the natural isomorphism is also natural on $P$.
\end{prop}
\begin{proof}
	Given a continuous map $s:U\to E_P$ such that $\pi_{P}\circ s=i_{U}$, then the frame homomorphism $s\inv: {\cal O}(E_P)\to \Ocal(U)$ satisfies $s\inv\circ \pi^{-1}=i_U\inv=(-)\cap U$. Let us show conversely that any frame homomorphism $f:{\cal O}(E_P)\to \Ocal(U)$ such that $f\circ \pi^{-1}=(-)\cap U$ is actually of the form $s^{-1}$ for a unique continuous section $s$. For this, we notice that the value of a continuous section $s$ at a point $x$ is uniquely determined by the basic open sets $\dot{t}(V)$ which contain $s(x)$; indeed, if $s(x)\in \dot{t}(V)$ then $s(x)=t_{x}$; in other words, the value of $s$ at $x$ is uniquely determined by the basic open sets $\dot{t}(V)$ such that $x\in s\inv(\dot{t}(V))$. Replacing $s\inv$ by $f$ we get the recipe for defining the continuous section $s_{f}$ associated with a frame homomorphism $f$ as above: we set $s_{f}(x)$ equal to $t_{x}$ for any $t\in P(V)$ such that $x\in f(\dot{t}(V))$. Let us show that this definition is well-posed. First, we observe that, given $x\in U$, the collection of pairs $(t, V)$ such that $x\in f(\dot{t}(V))$ is non-empty since, as ${\cal O}(E_{P})$ is covered by the subsets of the form $\dot{t}(V)$ and $f$ is a frame homomorphism to $\Ocal(U) $, $U$ is covered by the subsets of the form $f(\dot{t}(V))$. Next, let us notice that for any $(V, t)$, $f(\dot{t}(V))\subseteq V$, whence $t_{x}$ is well-defined; indeed, $f(\dot{t}(V))\subseteq f(\pi_P^{-1}(V))=V\cap U$. Lastly, we have to check that for any $(V, t)$ and $(V', t')$ such that $x\in f(\dot{t}(V))\cap f(\dot{t'}(V'))$, $t_{x}=t'_{x}$. We have that $\dot{t}(V) \cap \dot{t'}(V') = \bigcup_{W\subseteq V \cap V', r=t|_{W}=t'|_{W}} \dot{r}(W)$, so, since $f$ is a frame homomorphism, 
	\[
	f(\dot{t}(V)) \cap f(\dot{t'}(V')) = \bigcup_{W\subseteq V \cap V', r=t|_{W}=t'|_{W}} f(\dot{r}(W)).
	\]
	Therefore, there is an open set $W\subseteq V\cap V'$ and $r\in P(W)$ such that $x\in f(\dot{r}(W))\subseteq W$ and $r=t|_{W}=t'|_{W}$; so $t_{x}=r_{x}=t'_{x}$, as required. The continuity of the map $s$ thus defined follows from Proposition \ref{prop:sezione_sse_sezionilocali} since, by definition, it is locally represented by section of $P$, while its uniqueness follows from the above remark about its values being determined by the open sets containing them. 
	Finally, it is immediate to check that for $V\subseteq U$ and any $h:P\rightarrow Q$ the squares
	\[
	\begin{tikzcd}
		\Gamma(\Lambda_P)(U) \ar[r, "\sim"] \ar[d] & \Locale/_1\Ocal(X)(\Ocal(U), \Ocal(E_P))\ar[d, "-\circ(-\cap V)"]\\
		\Gamma(\Lambda_P)(V) \ar[r, "\sim"] & \Locale/_1\Ocal(X)(\Ocal(V), \Ocal(E_P))
	\end{tikzcd},
	\]\[\begin{tikzcd}
		\Gamma \Lambda_P \ar[d, "\Gamma \Lambda_h"'] \ar[r, "\sim"]& \Locale/_1\Ocal(X)(\Ocal(-), \Ocal(E_P))\ar[d, "-\circ \Lambda_h\inv"]\\
		\Gamma \Lambda_Q \ar[r, "\sim"]& \Locale/_1\Ocal(X)(\Ocal(-), \Ocal(E_Q))
	\end{tikzcd}
	\]
	are commutative: the first expresses the naturality of the isomorphism which was built above, and the second the naturality in $P$.  
\end{proof}
\begin{cor}
	The two functors
	\[
	\Psh(X)\xrightarrow{\Lambda} \Top/X \xrightarrow{\Gamma} \Psh(X)
	\]
	and
	\[
	\Psh(X) \hspace{-0.3ex}\xrightarrow{\Lambda} \hspace{-0.3ex}\Top/X\hspace{-0.3ex} \xrightarrow{\Ocal(-)}\hspace{-0.3ex} \Locale/_1\Ocal(X) \hspace{-0.3ex}\xrightarrow{[f]\mapsto \Locale/_1\Ocal(X)(\Ocal(-), [f])}\hspace{-0.3ex} \Psh(X)
	\]
	are naturally isomorphic.
\end{cor}

\begin{remarks}
	\begin{enumerate}[(i)]
		\item The condition $f\circ \pi^{-1}=(-)\cap U$, which is the dual of  $\pi_{P}\circ s=i_{U}$, means, concretely, that for any open set $Z$ of $X$,
		\[
		\bigvee_{V\subseteq Z, t\in P(V)}f(\dot{t}(V))=Z\cap U.
		\]
		
		\item It is true more generally that, for any topological space $X$, the open sets functor ${\cal O}:\Top\to \Locale$ induces an equivalence of categories between $\textbf{Etale}\slash X$ and the category $\textbf{LH}\slash X$, where $\textbf{LH}$ is the category of local homeomorphisms of locales to $X$ and local homeomorphisms between them, which is a full subcategory of $\textbf{Loc}\slash X$ (cf. \cite[Lemma C1.3.2(iii)-(v)]{elephant}). 
		
		\item Proposition \ref{prop:fascificazione_topologica_con_framehom} hints to the independence of the sheafification of a presheaf $P$ in $\Psh(X)$ from the points of $X$, since it describes sheafification as a hom functor of locales. The same localic nature of the sheafification process is shown by Proposition \ref{prop:fasci__etalebundle_eqv_fasci_grfibration}: indeed, the functor $f_P:\fib P\rightarrow E_P$ extends to an isomorphism of locales $\Id_{J_P}(\fib P) \isorightarrow \Ocal(E_P)$	 (see Remark \ref{rmk:connessione_localeadj_tpladj}), which is exactly the isomorphism induced by the equivalence of localic toposes $\Sh(E_P)\simeq \Sh(\fib P, J_P)$. These considerations reveal that the sheafification process is really point-less in spirit, even when performed for sheaves on a topological space.
		
		\item Consider $P$ in $\Psh(X)$ and the diagram 
		\[D:\fib P\rightarrow \Cosite/(\Ocal(X), J\can_{\Ocal(X)})\] which maps every object $(U,s)$ in $\fib P$ to the comorphism of sites $i_U(-):(\Ocal(U), J\can_{\Ocal(U)})=(\Ocal(X)/U, (J\can_{\Ocal(X)})_U)\rightarrow (\Ocal(X), J\can_{\Ocal(X)})$; the arrows $\dot{s}(-):(\Ocal(U), J\can_{\Ocal(U)})\rightarrow (\Ocal(E_P), J\can_{\Ocal(E_P)})$ provide a cocone under $D$. We have proven in Section \ref{sec:adjoints_to_Grothendieck} that $(\fib P, J_P)$ is the colimit of $D$: therefore we have a unique comorphism of sites $(\fib P, J_P)\rightarrow F(E_P)$ induced by the universal property of colimits, which is exactly the functor $f_P$ of Proposition \ref{prop:fasci__etalebundle_eqv_fasci_grfibration}.
		Finally, we have also shown in Section \ref{sec:adj_for_toposes} that when moving to toposes of sheaves over $\Sh(X)$, the colimit cocone for $(\fib P, J_P)$ is preserved: but since $\Sh(\fib P, J_P)\simeq \Sh(E_P)$, we can conclude that $\Sh(E_P)$ is the colimit in the category of toposes over $\Sh(X)$ for the diagram $\fib P\rightarrow \Topos/\Sh(X)$, $(U,s)\mapsto [\Sh(i_U):\Sh(U)\rightarrow \Sh(X)]$.
		
	\end{enumerate}
\end{remarks}

\subsection{Conditions for a generic map to be étale}
In the present section we will make a short digression: considering a topological space $X$ and a function
\[
\pi:E\rightarrow X,
\] 
we will analyse some set-theoretic conditions under which $\pi$ is an étale map. We have already mentioned in Section \ref{sec:adj_topologica} that étale spaces come as colimits of their local sections: and indeed, it is the choice of those sections that we want to be continuous that forces all the relevant structure onto $\pi$. This is made very explicit from the following result:
\begin{prop}{\cite[Lemma 2.4.7]{borceux3}} \label{prop:tpl_etale_lemma_borceux}
	Consider a local homeomorphism $\pi:E\rightarrow X$: then the topology on $E$ is the final topology making all its local sections continuous.
\end{prop}
Explicitly put, if we consider the family of continuous local sections of $\pi$, the topology on $E$ is the \emph{finest} topology making them continuous. As we can see, the topology on $E$ is canonically defined once we know which sections are the continuous ones. Moreover, for a local homeomorphism $\pi:E\rightarrow X$ the two following properties always hold:
\begin{enumerate}[(i)]
	\item $E$ is covered by its local sections, \ie the family of continuous sections of $\pi$ is jointly epic;
	\item the continuous sections of $\pi$ are open.
\end{enumerate}
We will see in a moment that assuming these two properties for a certain class of sections of $\pi$ is enough to build a topology $\tau^\sbicat_\pi$ on $E$ such that $\pi$ is a local homeomorphism with respect to it. After that, we will confront $\tau^\sbicat_\pi$ with another topology $\sigma^\sbicat_\pi$, which provides us with a subbase for the étale topology; finally, we will consider necessary and sufficient conditions for a map to be étale.

In general, for an open subset $U$ of $X$ we shall call \emph{local section of $\pi$ at $U$} any map $s:U\rightarrow E$ making the diagram 
\[
\begin{tikzcd}
	U \ar[dr, hook, "i_U"'] \ar[r, "s"] & E\ar[d, "\pi"]\\
	&X
\end{tikzcd}\]
commutative. We will denote by $Sec(\pi)$ the set of pairs $(U,s)$, where $U\in \Ocal(X)$ and $s:U\rightarrow E$ is a local section of $\pi$.
Notice in particular that every section $s$ of $\pi$ is injective, since $\pi s=i_U$ is injective.

First of all, let us consider the topology for $E$ which appeared in Proposition \ref{prop:tpl_etale_lemma_borceux}:
\begin{prop}\label{lemma:tpl_per_etale}\index{$\tau^\sbicat_\pi$}
	Consider a map $\pi:E\rightarrow X$ and suppose $X$ is a topological space. Consider a collection $\sbicat\subseteq Sec(\pi)$ of local sections of $\pi$: then the collection 
	\[
	\tau^\sbicat_\pi:=\{W\subseteq E\ |\ \forall (U,s)\in \sbicat,\ s\inv(W)\in\Ocal(U)  \}.
	\]
	is a topology on $E$, and it is the finest topology making all the local sections in $\sbicat$ continuous. Moreover, $\tau^\sbicat_\pi$ makes $\pi$ continuous.
\end{prop}
\begin{proof}
	Consider one local section $s:U\rightarrow E$ of $\pi$: the set
	\[
	\tau_{(U,s)}:=\{W\subseteq E\ |\ s\inv(W)\in \Ocal(U) \}
	\]
	provides a topology over $E$, which is the finest topology making $s$ continuous. It follows that the finest topology making all the sections $s\in \sbicat$ continuous will be the intersection of the topologies $\tau_{(U,s)}$ over all possible choices of $(U,s)\in \sbicat$, which is
	\[
	\bigcap_{(U,s)\in \sbicat}\tau_{(U,s)}=\{W\subseteq E\ |\ \forall\ (U,s)\in \sbicat,\ s\inv(W)\in\Ocal(U)) \},\]
	\ie the topology $\tau^\sbicat_\pi$ above. Finally, to prove that $\pi$ is continuous, consider an open $V\in\Ocal(X)$: then for every $(U,s)\in \sbicat$, we have that $s\inv \pi\inv(V)=i_U\inv(V):=V\cap U$, which is open, and thus $\pi\inv(V)\in \tau^\sbicat_\pi$.
\end{proof}
Let us now consider the relevant properties for the sections of an étale map which we listed at the beginnin of the section. First of all, joint surjectivity of $\sbicat$ is connected to the openness of $\pi$:
\begin{lemma}\label{lemma:locsec_jepic_pi_aperta}
	Consider $\pi:E\rightarrow X$ and $\sbicat\subseteq Sec(\pi)$ as in Lemma \ref{lemma:tpl_per_etale}: if the local sections in $\sbicat$ are jointly surjective over $E$ then $\pi$ is open.
\end{lemma}
\begin{proof}
	Consider any $W\in \tau^\sbicat_\pi$ and any $(U,s)\in \sbicat$. Notice that if $x\in s\inv(W)$, then $s(x)\in W$ and thus $x=\pi s(x)\in \pi(W)$: this implies that
	\[
	\bigcup_{(U,s)\in \sbicat}s\inv (W)\subseteq \pi(W)
	\] 
	for every $W\in \tau^\sbicat_\pi$. Conversely, take $x\in \pi(W)$, \ie $x=\pi(w)$ for some $w\in W$: since the local sections in $\sbicat$ are jointly surjective, there is some $(U,s)\in \sbicat$ and $u\in U$ such that $s(u)=w$, which implies $x=\pi s(u)=u\in s\inv(W)$. This proves the opposite inclusion, and since all the $s\inv (W)$ are open also $\pi(W)$ is open.
\end{proof}
Let us now turn our attention to openness. Since all the sections in $\sbicat$ should be open with respect to $\tau^\sbicat_\pi$, we should in particular require for every $(U,s)\in\sbicat$ that $s(U)$ be open in $E$. Spelling out this request explicitly we obtain the following:
\begin{itemize}
	\item[{$[\dagger]$}] for every $(U,s)$, $(V,t)\in \sbicat$ the set $t\inv(s(U))=\{x\in V\cap U\ |\ s(x)=t(x) \}$ is open in $X$.
\end{itemize}
Notice that $[\dagger]$ is in fact the generalization of a well known property of local homeomorphisms (see for instance \cite[Lemma 2.4.6]{borceux3}), the \emph{local equality condition} for sections: given two different local sections $(U,s)$ and $(U,t)$ of $\pi$ and an element $x\in U$ such that $s(x)=t(x)$, there exists an open neighbourhood $W\subseteq U$ of $x$ such that $s_{|W}=t_{|W}$.   
\begin{lemma}
	Consider $\pi:E\rightarrow X$ and $\sbicat\subseteq Sec(\pi)$ as in Lemma \ref{lemma:tpl_per_etale}: if $\sbicat$ satisfies $[\dagger]$ then its sections satisfy the local equality condition; the converse holds if $\sbicat$ is closed under subsections, \ie if $(U,s)\in \sbicat$ and $W\subseteq U$ is open then $(W,s_{|W})\in\sbicat$.
\end{lemma}
\begin{proof}
	Suppose that all sections in $\sbicat$ are open and consider two sections $(U,s)$ and $(U,t)$ and $x\in U$ such that $s(x)=t(x)$: then $x\in t\inv (s(U))$, which is an open subset of $U$ such that $s_{|t\inv(s(U))}=t_{|t\inv(s(U))}$, and hence the local equality condition is satisfied.
	
	Conversely, consider two sections $(U,s)$ and $(V,t)$ and an element $x\in t\inv (s(U))$: by the local equality condition, applied to the sections $(V\cap U, s_{|V\cap U})$ and $(V\cap U, t_{| V\cap U})$ of $\sbicat$, there exists some open neighbourhood $Z\subseteq V\cap U$ of $x$ such that $t_{|Z}=s_{|Z}$. This implies that $Z\subseteq t\inv (s(U))$ and hence $t\inv (s(W))$ is open in $X$. Since this holds for all choices of $(U,s)$ and $(V,t)$ we conclude that the sections of $\sbicat$ are open.  
\end{proof}
We are now ready to show that when $\sbicat$ satisfies $[\dagger]$ and is jointly surjective $\pi$ can be made into a local homeomorphism:
\begin{prop}\label{prop:tpl_etale_date_le_sezioni}
	Consider a map $\pi:E\rightarrow X$, where $X$ is a topological space, and a family of local sections $\sbicat\subseteq Sec(\pi)$. Suppose that the local sections of $\sbicat$ 
	\begin{enumerate}[(i)]
		\item are jointly surjective: for every $w\in E$ there exist $(U,s)\in\sbicat$ and $x\in U$ such that $w=s(x)$;
		\item $\sbicat$ satisfies $[\dagger]$: for every $(U,s)$ and $(V,t)$ in $\sbicat$ the set $t\inv(s(U))=s\inv(t(V))=\{x\in U\cap V\ |\ s(x)=t(x) \}$ is open in $X$.
	\end{enumerate} 
	If we endow $E$ with the topology $\tau^\sbicat_\pi$ of Lemma \ref{lemma:tpl_per_etale}, $\pi$ is a local homeomorphism, and the continuous sections of $\pi$ are precisely those that are gluings of subsections of sections in $\sbicat$.
\end{prop}
\begin{proof}
	Consider an element $w\in E$: we want an open neighbourhood $W\in \tau^\sbicat_\pi$ of $w$ such that $\pi(W)$ is open and $\pi_{|W}$ is a homeomorphism.
	
	Since the local sections of $\sbicat$ are jointly surjective, there exists $(U,s)\in \sbicat$ such that $w\in \Im(s)$. Since the sections of $\sbicat$ are open, $s(U)$ is an open neighbourhood of $w$. Finally, consider the arrow \[
	\pi_{|s(U)}:s(U)\rightarrow \pi s(U)=U:\] it is continuous and open since it restricts $\pi$, which is continuous and open (by Lemma \ref{lemma:locsec_jepic_pi_aperta}); it is also bijective, since it is the inverse of $s:U\isorightarrow s(U)$. Thus $\pi_{|s(U)}$ is a homeomorphism. 
	
	Finally, consider any continuous section $(V,t)\in Sec(\pi)$ of $\pi$, \ie such that $t$ is continuous; since $\pi$ is a local homeomorphism, $t$ is also open by Corollary \ref{cor:sezioni_bundleetale_aperte}. If we consider the open subset $t(V)\subseteq E$, by the joint surjectivity of the sections in $\sbicat$ we have that there exist $(U_i, s_i)\in \sbicat$ for $i\in I$ such that $t(V)\subseteq \bigcup_i s_i(U_i)$. Consider the opens $Z_i:=s_i\inv (t(V))=\{ x\in U_i\cap V\ |\ t(x)=s_i(x)\}\subseteq U_i\cap V$: then we have that $(Z_i, s_{i|Z_i})=(Z_i, t_{|Z_i})$ is a subsection both of $(U_i, s_i)$ and $(V,t)$. Now, for every $v\in V$ there exists $i\in I$ and $x\in U_i$ such that $t(v)=s_i(x)$: applying $\pi$ this implies that $v=x$, thus $v\in U_i$ and $t(v)=s_i(v)$, so that $v\in Z_i$. This tells us that $V$ is covered by the opens $Z_i$, and thus the section $t:V\rightarrow E$ is indeed the gluing of sections $t_{|Z_i}:Z_i\rightarrow E$, which are subsections of the sections $(U_i, s_i)$ of $\sbicat$.
\end{proof}
\begin{remark}
	The family $\sbicat$ may be completely arbitrary, but there is no loss in generality if we consider it closed under subsections and gluings, since such a closure does not affect the induced topology $\tau^\sbicat_\pi$. Indeed, suppose that $(U,s)\in \sbicat$ and call $\sbicat':=\sbicat\cup\{(W,s_{|W}) \}$ for a subsection $(W,s_{|W})$ of $(U,s)$: then evidently $\tau_\pi^{\sbicat'}\subseteq \tau^\sbicat_\pi$. Conversely if we take $V\in \tau^\sbicat_\pi$, we have that $(s_{|W})\inv(V)=s\inv(V)\cap W$, which is open in $W$, and hence $V\in \tau^{\sbicat'}_\pi$, so that $\tau^\sbicat_\pi=\tau^{\sbicat'}_\pi$. Now consider $(U_i,s_i)\in\sbicat$ such that for every $i$, $j$ we have $s_{i|U_i\cap U_j}=s_{j|U_i\cap U_j}$, call $U=\cup_{i}U_i$ and $s:U\rightarrow E$ the map obtained by gluing the maps $s_i$. Call $\sbicat'=\sbicat\cup\{(U,s)\}$. Of course, $\tau^{\sbicat'}_\pi\subseteq \tau^\sbicat_\pi$. Conversely, consider $W\in\tau^\sbicat_\pi$: in particular, for every $i$ we have $s_i\inv(W)$ open in $X$, and hence $s\inv(W)=\cup_is_i\inv(W)$ is open in $X$: thus $W\in \tau^\sbicat_\pi$, and again $\tau^{\sbicat'}_\pi=\tau^\sbicat_\pi$.
\end{remark}
Consider a topology $\Ocal(E)$ on $E$ making $\pi$ continuous, and denote by $Cont(\pi, \Ocal(E))\subseteq Sec(\pi)$ the class of local sections of $\pi$ that are continuous when $E$ is endowed with $\Ocal(E)$: then the last line of the theorem tells us that $Cont(\pi, \tau^\sbicat_\pi)$ is the closure of $\sbicat$ under subsections and gluings. By the previous remark, we have that
\[
\tau^\sbicat_\pi=\tau^{Cont(\pi, \tau^\sbicat_\pi)}_\pi.
\]
Now, suppose that the topology $\Ocal(E)$ over $E$ already makes $\pi$ into a local homeomorphism: then Proposition \ref{prop:tpl_etale_lemma_borceux} states that  
\[
\Ocal(E)=\tau^{Cont(\pi,\Ocal(E))}_\pi.\]
If we are given a smaller family of sections $\sbicat\subseteq Cont(\pi, \Ocal(E))$, we have for sure that 
\[
\Ocal(E)=\tau^{Cont(\pi,\Ocal(E))}_\pi\subseteq \tau^\sbicat_\pi:
\]
therefore, $\tau^\sbicat_\pi$ is the \emph{biggest} topology on $E$ making $\pi$ into a local homeomorphism with all the sections in $\sbicat$ continuous.

Since $\tau^\sbicat_\pi$ is an upper bound for the topologies making $\pi$ a local homeomorphism, we can wonder which topology could be the the lower bound. Such a topology $\sigma$ must in particular contain as opens all sets $s(U)$ for $(U,s)$. This leads us to consider on $E$ the topology $\sigma^\sbicat_\pi$\index{$\sigma^\sbicat_\pi$}, generated by the subbase 
\[
\bbicat_\sbicat=\{s(U)\subseteq E\ |\ (U,s)\in\sbicat\}.
\]\index{$\bbicat_\sbicat$}
An open set $W$ of $\sigma^\sbicat_\pi$ is either $E$ or a union of finite intersections of elements of $\bbicat_\sbicat$, \ie something of the form
\[
W=\bigcup_{i\in I} \left(\bigcap_{j=1}^{n_i}s_{i,j}(U_{i,j})\right)
\]
for $(U_{i,j}, s_{i,j})\in \sbicat$. This topology is closer in definition with those usually considered when studying structural sheaves and spectra for algebraic theories (we shall recall some examples in the following). 

The first thing that we have to remark is that the continuity of $\pi$ is not granted, when $E$ is endowed with $\sigma^\sbicat_\pi$: for instance, consider the set $X=\{0,1\}$ with the discrete topology, $E=\{0,1\}$ and take $\pi:E\rightarrow X$ to be the identity map; then take $\sbicat$ to consist only of the section $\{0\}\hookrightarrow\{0,1\}$. It is immediate to see that $\sigma^\sbicat_\pi=\{\emptyset,\{0\},\{0,1\}\}$, and hence $\pi$ is not continuous. On the other hand, continuity of the sections in $\sbicat$ is directly related with the openness condition $[\dagger]$ and the inclusion of $\sigma^\sbicat_\pi$ into $\tau^\sbicat_\pi$:
\begin{lemma}
	The following are equivalent:
	\begin{enumerate}[(i)]
		\item the sections in $\sbicat$ are continuous when $E$ is endowed with $\sigma^\sbicat_\pi$;
		\item $\sbicat$ satisfies $[\dagger]$, \ie the sections in $\sbicat$ are open with respect to $\tau^\sbicat_\pi$;
		\item $\sigma^\sbicat_\pi\subseteq \tau^\sbicat_\pi$.
	\end{enumerate}
\end{lemma} 
\begin{proof}
	(i)$\Rightarrow$(ii). Suppose that every section in $\sbicat$ is continuous, and consider two sections $(U,s)$ and $(V,t)$: since $s(U)\in \bbicat_\sbicat$, we have that $t\inv(s(U))$ is open in $X$. As this holds for all choices of $(U,s)$ and $(V,t)$ in $\sbicat$, we conclude that $\sbicat$ satisfies $[\dagger]$. 
	
	(ii)$\Rightarrow$(iii). Consider a section $(U,s)\in \sbicat$ and the corresponding basic open $s(U)\in \bbicat_\sbicat$. Since $[\dagger]$ holds, for every section $(V,t)\in\sbicat$ the subset $t\inv(s(U))$ is open in $X$, and thus $s(U)\in \tau^\sbicat_\pi$ by definition of $\tau^\sbicat_\pi$. We have $\bbicat_\sbicat\subseteq \tau^\sbicat_\pi$, which implies $\sigma^\sbicat_\pi\subseteq \tau^\sbicat_\pi$.
	
	(iii)$\Rightarrow$ (i). Consider a section $(U,s)\in\sbicat$ and any open $W\in \sigma^\sbicat_\pi$: since $\sigma^\sbicat_\pi\subseteq \tau^\sbicat_\pi$, the subset $s\inv(W)$ is open in $X$ by definition of $\tau^\sbicat_\pi$. Since this holds for every choice of $W$, the section $(U,s)$ is continuous.
\end{proof}

The opposite inclusion of topologies, on the other hand, does not seem to enjoy a similar nice characterization, but it is necessary for $\sbicat$ to be jointly epic:
\begin{lemma}
	The sections in $\sbicat$ are jointly epic if and only if the opens in $\bbicat_\sbicat$ cover $E$. If this holds, $\tau^\sbicat_\pi\subseteq \sigma^\sbicat_\pi$.
\end{lemma}
\begin{proof}
	The first consideration is tautological. 
	Consider $W\in \tau^\sbicat_\pi$: by definition, for every section $(U,s)\in\sbicat$ we have $s\inv(W)$ open in $X$. Now, suppose that the sections in $\sbicat$ are jointly epic: then for every $w\in W$ there exists $(U,s)\in \sbicat$ and $x\in U$ such that $s(x)=w$. Since $x\in s(s\inv(W))\subseteq W$ we have that $W$ is covered by opens of the form $s(s\inv (W))$, which are basic opens in $\bbicat_\sbicat$, we can conclude that $W\in\sigma^\sbicat_\pi$.
\end{proof}
The previous considerations imply the following:
\begin{cor}\label{cor:tpl_etale_tau=sigma}
	Consider a map $\pi:E\rightarrow X$, with $X$ a topological space, and $\sbicat\subseteq Sec(\pi)$ jointly epic and satisfying $[\dagger]$: then $\bbicat_\sbicat$ is a subbase for the topology $\tau^\sbicat_\pi$.
\end{cor}
\begin{proof}
	If $\sbicat$ satisfies $[\dagger]$ then $\sigma^\sbicat_\pi\subseteq \tau^\sbicat_\pi$, while if the sections in $\sbicat$ are jointly epic the opposite inclusion holds.
\end{proof}

Notice that if $\sbicat$ is jointly surjective and satisfies $[\dagger]$ there is exactly one possible topology on $E$ making $\pi$ into a local homeomorphism such that the sections in $\sbicat$ are continuous, since any other such topology must be squeezed between $\tau^\sbicat_\pi$ and $\sigma^\sbicat_\pi$ (which coincide). The equality of the two topologies $\tau^\sbicat_\pi$ and $\sigma^\sbicat_\pi$ alone is in general not enough to deduce that $\pi$ is a local homeomorphism; it does however imply that its restriction to the joint image of $\sbicat$ is:
\begin{prop}
	If $\sigma^\sbicat_\pi=\tau^\sbicat_\pi$, then the composite
	\[\bar{\pi}:\bar{E}=\bigcup_{(U,s)\in\sbicat}s(U)\hookrightarrow E\xrightarrow{\pi} X\]
	is a local homeomorphism when $\bar{E}$ is endowed with the topology $\sigma^\sbicat_{\bar{\pi}}$.
\end{prop}
\begin{proof}
	We can consider $\sbicat$ as a family of sections whose domain is $\bar{E}$ instead of $E$: of course then the sections in $\sbicat$ are jointly surjective over $\bar{E}$. On the other hand, the inclusion $\sigma^\sbicat_\pi\subseteq \tau^\sbicat_\pi$ implies that $\sbicat$ satisfies $[\dagger]$, and thus $\bar{\pi}$ is a local homeomorphism.	
\end{proof}
Let us conclude this section with a couple of further considerations. The first is that there are some canonical bases for the topology $\tau^\sbicat_\pi$, which include $\bbicat_\sbicat$ whenever $\sbicat$ is closed under subsections:
\begin{prop}\label{prop:basi_tpl_etale}
	Consider $\pi:E\rightarrow X$ and $\sbicat\subseteq Sec(\pi)$ as above, and suppose that $\sbicat$ is jointly surjective and satisfies $[\dagger]$:\begin{enumerate}[(i)]
		\item If $\abicat$ is a base for $X$, then 
		\[
		\bbicat^\sbicat_\abicat:=\{
		s(V)\ |\ (U,s)\in \sbicat,\ V\in\abicat,\ V\subseteq U
		\}\]\index{$\bbicat^\sbicat_\abicat$}
		is a basis for $\tau^\sbicat_\pi$;
		\item The collection
		\[
		\{
		s(V)\ |\ (U,s)\in \sbicat, V\subseteq U,\ V\mbox{ open}\}
		\]
		is a basis for $\tau^\sbicat_\pi$;
		\item If $\sbicat$ is closed under subsections, $\bbicat_\sbicat$ is a basis for $\tau^\sbicat_\pi$.
	\end{enumerate}
\end{prop}
\begin{proof}
	\begin{enumerate}[(i)]
		\item First of all, let us show that $\bbicat^\sbicat_\abicat$ is indeed a basis. If $(U,s)\in \sbicat$ and $w\in s(U)$ then $w=s(x)$ for some $x\in U$: since $U$ is open in $X$, there must exist some open $A\in \abicat$ such that $x\in A\subseteq U$, and thus $w\in s(A)$. This implies immediately that the opens in $\bbicat^\sbicat_\abicat$ cover $E$ whenever $\sbicat$ is jointly surjective. Consider now two sections $(U,s)$ and $(V,t)$, two basic opens $U'$, $V'\in \abicat$ such that $U'\subseteq U$ and $V'\subseteq V$ and an element $w\in s(U')\cap t(V')$: there must exist some $x\in t\inv(s(U))=\{x\in U\cap V\ |\ s(x)=t(x)\}$, which is open in $X$ by the condition $[\dagger]$, such that $w=s(x)=t(x)$. Moreover $x\in U'\cap V'$, and hence we can take a small enough open $A\in \abicat$ such that $x\in A\subseteq U'\cap V'\cap t\inv(s(U))$. This immediately implies that $s_{|A}=t_{|A}$ and $w\in s(A)\subseteq s(U')\cap t(V')$, and thus $\bbicat^\sbicat_\abicat$ is a basis.
		
		Every open $\bbicat^\sbicat_\abicat$ is open in $\tau^\sbicat_\pi$: indeed, if we take $(U,s)\in \sbicat$ and $U'\in\abicat$ such that $U'\subseteq U$, then for every other section $(V,t)\in\sbicat$ it holds that $t\inv(s(U'))=t\inv(s(U))\cap U'$, which is open in $X$. Therefore, $\bbicat^\sbicat_\abicat\subseteq \tau^\sbicat_\pi$. Finally, let us show that every open in $\tau^\sbicat_\pi$ can be covered by opens in the basis $\bbicat^\sbicat_\abicat$. 
		Consider $w\in W\in \tau^\sbicat_\pi$: since $\sbicat$ is jointly surjective there must exist $(U,s)\in\sbicat$ such that $w\in s(U)$, \ie $w=s(x)$ for some $x\in U$. By definition of $\tau^\sbicat_\pi$, the set $s\inv(W)\subseteq U$ is open in $X$, and thus there must exist a basic open $A\in \abicat$ such that $x\in A\subseteq s\inv(W)$, which implies $w\in s(A)\subseteq s(U)$.
		\item It follows from the previous item, when we take $\abicat=\Ocal(X)$.
		\item If $\sbicat$ is closed under subsections it is immediate to see that $\bbicat_\sbicat=\bbicat^\sbicat_{\Ocal(X)}$.
	\end{enumerate}
\end{proof}
Finally, we show that the request for the existence of a family of sections $\sbicat$ can be formulated in terms of a `local bijection' between the two sets $E$ and $X$:
\begin{prop}
	Consider a map $p:E\rightarrow X$, with $X$ a topological space. The following are equivalent:
	\begin{enumerate}[(i)]
		\item There exists a topology on $E$ such that $\pi$ is a local homeomorphism;
		\item There exist a family of sets $\{W_i\subseteq E\ |\ i\in I\}$ such that 
		\begin{enumerate}[(a)]
			\item $\bigcup_i W_i=E$;
			\item for every $i$, $\pi_{|W_i}:W_i\rightarrow \pi(W_i)$ is bijective and $\pi(W_i)\subseteq X$ is open;
			\item for every $i,j\in I$, $\pi(W_i\cap W_j)\subseteq X$ is open.
		\end{enumerate}
	\end{enumerate}
\end{prop}
\begin{proof}
	(i)$\Rightarrow$(ii) is obvious: if $\pi$ is a local homeomorphism then $E$ is endowed with $\Ocal(E)=\tau^{Cont(\pi, \Ocal(E))}_\pi$. Every $w\in E$ belongs to $s(U)$ for some continuous section $(U,s)\in Cont(\pi, \Ocal(E))$, and $\pi_{|s(U)}:s(U)\rightarrow U$ is a homeomorphism. Moreover, one easily verifies that
	\[
	\pi(s(U)\cap t(V))=t\inv(s(U)),
	\]
	which is open in $X$.
	
	(ii)$\Rightarrow$(i). We can set $U_i:=\pi(W_i)\subseteq X$ and 
	\[
	s_i:U_i\xrightarrow{\pi_{|W_i}\inv}W_i\hookrightarrow E:\]
	then one immediately sees that $\sbicat=\{(U_i,s_i)\ |\ i\in I \}$ is a family of sections for $\pi$ which is jointly surjective. Finally, a quick calculation shows that
	\[
	\pi(W_i\cap W_j)=s_i\inv(s_j(U_j)),\]
	which is open by hypothesis: this means that $\sbicat$ also satisfies $[\dagger]$, and thus the topology $\tau^\sbicat_\pi$ over $E$ makes $\pi$ into a local homeomorphism.
\end{proof}

\subsection{Algebraic spectra and their structural sheaves}
As an application of the previous results, we will consider two famous cases of spectra of algebraic theories,and show that the usual topologies on their structural sheaves can be recovered using canonical classes of local sections.

Consider a commutative ring $R$: we recall that the spectrum $\Spec(R)$ is defined as the set of its prime ideals. It can be endowed with the Zariski topology, whose base is that of opens of the kind 
\[
D(a):=\{\pfrak\in \Spec(R)\ |\ a\notin \pfrak \}
\] 
for every $a\in R$. Now, consider the set
\[
E:=\coprod_{\pfrak\in \Spec(R)} R_\pfrak\xrightarrow{\pi} \Spec(R),
\]
where $\pi$ is the canonical map sending each component $R_\pfrak$ to $\pfrak$. It is well known (see for instance \cite[pag. 70]{hartshorne}) that the set can be topologized so that the projection $\pi$ is a local homeomorphism: the corresponding sheaf $\Spec(R)\op\rightarrow\Set$ is the \emph{structural sheaf} of the ring $R$. The topology $\Ocal(E)$ on $E$ is induced by the base
\[
\bbicat=\big\{ B(a,g):=\{ [g]_\pfrak\in R_\pfrak\ |\ \pfrak\in D(a) \}\ \big|\ a\in R,\ g\in R[a\inv] \big\},
\]
where $[g]_\pfrak$ denotes the class of $g$ seen as an element of the localization $R_\pfrak$. What we can remark now is that $\Ocal(E)$ can now be understood as induced by the canonical class of sections
\[
\sbicat=\{s_g:D(a)\rightarrow E\ |\ a\in R,\ g\in R[a\inv],\ s_g(\pfrak):=[g]_\pfrak\}\subseteq Sec(\pi):
\]
we notice immediately that $B(a,g):=s_g(D(a))$, and therefore $\bbicat=\bbicat_\sbicat$.
The sections in $\sbicat$ are obviously jointly epic: every $h\in R_\pfrak$ is a fraction of the kind $x/y$ where $y\notin \pfrak$, and thus $h=[x/y]_\pfrak=s_{x/y}(\pfrak)$ for $s_{x/y}:D(y)\rightarrow E$. We can also show that $\sbicat$ satisfies $[\dagger]$. To do this consider any two sections $s_{x/a^n}:D(a)\rightarrow E$ and $s_{z/b^m}:D(b)\rightarrow E$, and take 
\[
\pfrak\in s_{z/b^m}\inv(s_{x/a^n}(D(a)))=\{\pfrak\in D(a)\cap D(b)\ |\ [x/a^n]_\pfrak=[z/b^m]_\pfrak \}: 
\]
since $[x/a^n]_\pfrak=[z/b^m]_\pfrak$ in $R_\pfrak$, there exists $w\notin \pfrak$ such that $wxb^m=wza^n$. Now, notice that $\pfrak$ also belongs to the open $D(wab)=D(w)\cap D(a)\cap D(b)$; conversely, if $\qfrak\in D(wab)$ then 
\[
[x/a^n]_\qfrak=[wxb^m/wa^nb^m]_\qfrak=[wza^n/wa^nb^m]_\qfrak=[z/b^m]_\qfrak,
\]
which means that $D(wab)\subseteq s_{z/b^m}\inv(s_{x/a^n}(D(a)))$: therefore, we have that the subset $s_{z/b^m}\inv(s_{x/a^n}(D(a)))$ is open in $\Spec(R)$. Applying Proposition \ref{prop:tpl_etale_date_le_sezioni}, we conclude that $\tau^\sbicat_\pi$ makes $\pi$ into an étale space over $\Spec(R)$; on the other hand, since $\sbicat$ is obviously closed under subsections we can apply Proposition \ref{prop:basi_tpl_etale} to conclude that $\bbicat$, as defined above, is a basis for $E$. In this way, taking the sheaf of cross-sections of this bundle, we recover the structure sheaf of $A$, as described in \cite{hartshorne}. 

We now turn our attention to spectra of MV-algebras: we recall that an \emph{MV-algebra} (see for instance \cite{dubuc_poveda}) is an abelian monoid $(A, \oplus, 0)$ endowed with a further 1-ary operation $\neg$ satisfying the axioms $\neg\neg x=x$, $x\oplus(\neg 0)=\neg 0$ and $\neg(\neg x\oplus y)\oplus y=\neg(\neg y\oplus x)\oplus x$. In particular we set $1:=\neg 0$. Any MV-algebras admits an order relation $\leq$ defined by $x\leq y$ if and only if there exists $z\in A$ such that $x\oplus z=y$. We will also make use of the derived operation $x\ominus y:=\neg(\neg x\oplus y)$: in particular, $x\leq y$ if and only if $x\ominus y=0$. Finally, there is a binary \emph{distance operation} $d$, defined by $d(x,y):=(x\ominus y)\oplus(y\ominus x)$.

An \emph{ideal} of $A$ is a submonoid $I\subseteq A$ such that if $y\in I$ and $x\leq y$ then $x\in I$; $I$ is said to be \emph{prime} if it is strictly contained in $A$, and moreover for every $x$ and $y$ in $A$ either $x\ominus y\in I$ or $y\ominus x\in I$. Every ideal $I$ of $A$ provides a congruence $\sim_I$ on the elements of $A$, defined by $x\sim_Iy$ if and only if $d(x,y)\in I$: the quotient by said congruence is denoted by $A/I$.

We define the \emph{spectrum of $A$} to be the set $\Spec(A)$ of prime ideals of $A$. For each $a\in A$, we define $W(a):=\{I\in \Spec(A)\ |\ a\in I \}$: one immediately checks that $W(0)=\Spec(A)$, $W(1)=\emptyset$ and $W(a)\cap W(b)=W(a\oplus b)$, and thus the sets $W(a)$ are a basis for a topology over $\Spec(A)$. We now introduce the set
\[E:=\coprod_{I\in Spec(A)} A/I\xrightarrow{\pi}\Spec(A):\]
it can be endowed with a topology making it into an étale space over $\Spec(A)$, induced by the class of global sections
\[
\sbicat=\{
\bar{a}:\Spec(A)\rightarrow E\ |\ a\in A,\ \bar{a}(I):=[a]_I
\}.
\]
Said sections are of course jointly epic over $E$. As per the condition $[\dagger]$, one can check that for a pair of sections $\bar{a}$ and $\bar{b}$ one has
\[
\bar{b}\inv(\bar{a}(\Spec(A)))=W(d(a,b)),
\]
which is open in $\Spec(A)$: thus the topology $\tau^\sbicat_\pi$ makes $\pi$ into a local homeomorphism. Applying Proposition \ref{prop:basi_tpl_etale}, we have that the topology $\tau^\sbicat_\pi$ is generated by the basis
\[
\bbicat^\sbicat_{\{W(a)\ |\ a\in A\} }=\{ \bar{b}(W(a)) | a\in A,\ b\in A  \}.
\]

We thus recover, by taking the sheaf of cross-sections of this bundle, the structure sheaf of $A$, as constructed in \cite{dubuc_poveda}.

\section{The restriction of the fundamental adjunction to preorders and locales}\label{sec:preorder}
We dedicate this section to further specializing the adjunctions
\[
\begin{tikzcd}
	{[\cbicat\op,\Set]} \ar[r, bend left, "\Lambda_{\Cat/_1\cbicat}", start anchor={north east}, end anchor={north west}] \ar[r, phantom, "\dashv"{rotate=270}] & \Cat/_1\cbicat \ar[l,bend left, "{\Gamma_{\Cat/_1\cbicat}}", end anchor={south east}, start anchor={south west}]
\end{tikzcd},\ 
\begin{tikzcd}
	{[\cbicat\op,\Set]} \arrow[r, "\Lambda_{\Topos^s/_1\Sh(\cbicat,J)}", bend left, end anchor={north west}, start anchor={north east}] &     \Topos^s/_1\Sh(\cbicat,J) \arrow[l, "\Gamma_{\Topos^s/_1\Sh(\cbicat,J)}", bend left, start anchor={south west}, end anchor={south east}] \ar[l, "\dashv"{rotate=270}, phantom]
\end{tikzcd}
\]
of Section \ref{sec:discrete_adj} to the case when the category $\cbicat$ is a preorder, and in particular when $\cbicat$ is a locale. The possibility of restricting our fundamental adjunctions to this context was already hinted by the multiple point-free results that we encountered in Section \ref{sec:adj_topologica}. We recall that a preorder can always be interpreted as a \emph{preorder category}\index{category!preorder} by setting, for every pair of objects $X$ and $Y$ of $\cbicat$, that $\Hom(Y,X)=\{*\}$ if and only if $Y\leq X$; moreover, functors between preorders are exactly the order preserving maps. This implies that the category of small preorders is a full sub-2-category $\Preord$\index{$\Preord$} of $\Cat$, whose 0-cells are small preorder categories.

The first remark is that the functor $\Lambda_{\Cat/_1\cbicat}:[\cbicat\op,\Set]\rightarrow \Cat/_1\cbicat$, which acts as the Grothendieck construction, factors through $\Preord/\cbicat$:
\begin{lemma}
	Consider a preorder $\cbicat$ and a presheaf $P:\cbicat\op\rightarrow \Set$: then $\fib P$ is a preorder.
\end{lemma}
\begin{proof}
	Consider $(X,s)$ and $(Y,t)$ in $\fib P$: a morphism $(Y,t)\rightarrow (X,s)$ is indexed by an arrow $y:Y\rightarrow X$ such that $P(y)(s)=t$. Hence if $\cbicat$ is a preorder there is at most one such arrow, when $Y\leq X$ and $t=s_{|Y}$, and thus $\fib P$ is a preorder.
\end{proof}
In particular, any slice $\cbicat/X$ is a preorder, since it is equivalent to $\fib \yo(X)$: as a matter of fact, $\cbicat/X$ is precisely the down-set $X\downarrow=\{Y\in \cbicat\ |\ Y\leq X\}$\index{down-set}\index{$(-)\downarrow$}, with the induced ordering.
The following is now an immediate corollary:
\begin{prop}\label{prop:adj_preord}\index{$\Lambda_{\Preord/_1\cbicat}\dashv\Gamma_{\Preord/_1\cbicat}$}
	Consider a small preorder category $\cbicat$: then there is a an adjunction
	\[
	\begin{tikzcd}
		{[\cbicat\op,\Set]} \ar[r, bend left, "\Lambda_{\Preord/_1\cbicat}", start anchor={north east}, end anchor={north west}] \ar[r, phantom, "\dashv"{rotate=270}] & \Preord/_1\cbicat \ar[l,bend left, "\Gamma_{\Preord/_1\cbicat}", end anchor={south east}, start anchor={south west}]
	\end{tikzcd}
	\]
	where:
	\begin{itemize}
		\item the notation $\Preord/_1\cbicat$ is that of Definition \ref{def:slice};
		\item both functors $\Lambda_{\Preord/_1\cbicat}$ map a presheaf over $\cbicat$ to the order-preserving map $\fib P\rightarrow \cbicat$;
		\item $\Gamma_{\Preord/_1\cbicat}$ is a contravariant hom-functor, defined for $p:\dbicat\rightarrow \cbicat$ as $\Gamma([p]):=\Preord/_1\cbicat(\cbicat/-, [p])$.
	\end{itemize}
\end{prop}
\begin{proof}
	We can consider the natural equivalence
	\[
	[\cbicat\op,\Set](P, \Cat/_1\cbicat(\cbicat/-, [p]))\simeq \Cat/_1 \cbicat( \fib P, [p]):
	\]
	where $P:\cbicat\op\rightarrow \Set$ and we choose $p:\dbicat\rightarrow \cbicat$ to be a functor between preorders: since $\fib P$ and every category $\cbicat/X$ are preorders,  
	$$\Cat/_1\cbicat(\cbicat/-,[p])=\Preord/_1\cbicat(\cbicat/-,[p])$$ and $$\Cat/_1\cbicat(\fib P, [p])=\Preord/_1\cbicat(\fib P, [p]),$$ the natural isomorphism above becomes
	\[
	[\cbicat\op,\Set](P, \Preord/_1\cbicat(\cbicat/-, [p]))\simeq \Preord/_1 \cbicat( \fib P, [p]).
	\]
\end{proof}
If we want to study the fixed points of $\Lambda_{\Preord/_1 \cbicat}\dashv \Gamma_{\Preord/_1 \cbicat}$ we need to isolate among all morphisms of preorders $Q\rightarrow \cbicat$ those that correspond to discrete fibrations, and those that correspond to discrete stacks. To do so we exploit the following technical result:
\begin{lemma}
	A functor $F:\abicat\rightarrow\bbicat$ is a cloven Grothendieck fibration if and only if for every $A$ in $\abicat$ the functor
	\[
	F_A:\abicat/A\rightarrow \bbicat/F(A)
	\]
	induced by restriction of $F$ has a right adjoint $F^A$ satisfying the identity $F_AF^A=\id_{\bbicat/F(A)}$. Moreover, $F$ is a discrete fibration if and only if the identity $F^AF_A=\id_{\abicat/A}$ is also verified.
\end{lemma}
\begin{proof}
	The first consequence is well known and can be found for instance in \cite[Theorem 2.10]{gray.fibredcofibredcategories}. In particular, $F^A$ acts by mapping any object $[h:X\rightarrow F(A)]$ of $\bbicat/F(A)$ to its cartesian lift $[\bar{h}:\bar{X}\rightarrow A]$ in $\abicat/A$. In particular, $F$ is discrete if and only if for every $h:X\rightarrow F(A)$ there is exactly one arrow $\bar{h}:\bar{X}\rightarrow A$ such that $F(\bar{h})=h$. This implies that for any arrow $y:Y\rightarrow A$ it must hold that $F^AF_A([y])=\overline{F(y)}=y$, since their images via $F_A$ are both $[F(y)]$. This means that $F^A$ is also a left inverse for $F_A$. The converse is obvious. 
\end{proof}
\begin{cor}\label{cor:etale_preordini}
	A monotone map $f:P\rightarrow \cbicat$ of preorders is a discrete fibration if and only if for every $a$ in $P$ the monotone map
	\[
	f_a:a\downarrow\rightarrow f(a)\downarrow\]
	defined by restriction of $f$ admits a pseudoinverse $f^a:f(a)\downarrow \rightarrow a\downarrow$.
\end{cor}
Finally, let us suppose that $\cbicat$ is endowed with a Grothendieck topology $J$, and consider a monotone map of preorders $f:P\rightarrow \cbicat$ satisfying the equivalent conditions of Corollary \ref{cor:etale_preordini}: it corresponds to the presheaf
$\bar{P}:\cbicat\op\rightarrow\Set$ defined by mapping each $X$ in $\cbicat$ to the set 
\[
\bar{P}(X):=\{ a\in P\ |\ f(a)=X\};
\]
moreover, for $Y\leq X$ the induced map $\bar{P}(X)\rightarrow \bar{P}(Y)$ maps every $a\in \bar{P}(X)$ to the element $f^a(Y)$ of $\bar{P}(Y)$: this because $f^a(Y)$ is (the domain of) the cartesian lift of $Y\leq X$ with codomain $a$. We can express what it means for $f$ to be a discrete stack by expliciting the condition for $\bar{P}$ to be a $J$-sheaf, in terms of matching families and amalgamations. This leads to the following definition:
\begin{defn}
	Consider a preorder site $(\cbicat,J)$. We shall call a monotone map of preorders $f:P\rightarrow\cbicat$ \emph{étale}\index{map of preorders!étale} if it satisfies the equivalent conditions of Corollary \ref{cor:etale_preordini}: more explicitly, if for every $a\in P$ there is a monotone map $f^a:f(a)\downarrow \rightarrow a\downarrow$ such that for every $b\leq a$ and $x\leq f(a)$ one has $f^a(f(b))\simeq b$ and $f(f^a(x))\simeq x$.
	
	We shall call $f$ \emph{$J$-étale}\index{map of preorders!$J$-étale} if it a discrete $J$-stack: explicitly, if for every sieve $S\in J(X)$ and every set of elements
	\[
	\{a_Y\in P\ |\ Y\in S \}
	\]
	satisfying $f(a_Y)=Y$ and such that whenever $Z\leq Y$ then $a_Z=f^{a_Y}(Z)$, there is a unique $a\in P$ such that $f(a)=X$ and $a_Y=f^a(Y)$.
	
	We shall denote by $\Etale(\cbicat)$\index{$\Etale(\cbicat)$} the full sub-2-category of $\Preord/\cbicat$ of étale posets over $\cbicat$, and by $\Etale(\cbicat,J)$\index{$\Etale(\cbicat,J)$} the full subcategory of $\Etale(\cbicat)$ of $J$-étale preorders over $\cbicat$.
\end{defn}
\begin{remark}
	Our definition of étale map $f:P\rightarrow\cbicat$ of preorders is a mild generalization of Definition 5 in \cite{jens.preorder}, where étale maps are defined for posets. Notice that in \textit{op. cit.} $f$ is seen as a map $P\op\rightarrow\cbicat\op$, and thus isomorphisms of the upper segments $x\uparrow$ and $f(x)\uparrow$ are required. 
\end{remark}
The following result is now completely tautological:
\begin{prop}
	The adjunction of Proposition \ref{prop:adj_preord} restricts to the equivalences
	\[
	[\cbicat\op,\Set]\simeq \Etale(\cbicat),\ \Sh(\cbicat,J)\simeq \Etale(\cbicat,J).\]
\end{prop}
We now move to specializing the fundamental adjunction 
\[
\Lambda_{\Topos^s/_1\Sh(\cbicat,J)} \dashv \Gamma_{\Topos^s/_1\Sh(\cbicat,J)}
\]
to the case when $(\cbicat,J)$ is a preorder site: we will observe that all the relevant data of this adjunction live at the level of sites, and this simplifies in particular the description of $\sheafify_J$. To see this we will need to exploit locales, so let us first recap some basic locale theory.

A \emph{frame}\index{frame} is a distributive lattice with all joins, and such that meets distribute over infinite joins, \ie
\[
a\wedge \left(\bigvee_{i\in I} b_i\right)= \bigvee_{i\in I} (a\wedge b_i)
\]
A morphism of frames is a map between them preserving the operations: frames form a category $\Frame$\index{$\Frame$}, which can be easily be made into a 2-category. The category $\Locale$\index{$\Locale$} of \emph{locales}\index{locale} is defined as $\Frame\op$: in particular, for a locale $L$ it is custom to denote its corresponding frame by $\Ocal(L)$, while if $g:L\rightarrow M$ is an arrow of locales its corresponding morphism of frames is denoted by $g\inv:\Ocal(M)\rightarrow \Ocal(L)$.

If we see a frame $F$ as a preorder site, its canonical topology corresponds with the join-cover topology\index{topology!join-cover}: $\{a_i\ |\ i\in I\}$ is a $J\can_F$-covering family for $a\in F$ if $\bigvee_{i\in I} a_i=a$. Given a locale $L$, the topos of sheaves of $L$ is the topos $\Sh(L):=\Sh(\Ocal(L), J\can_{\Ocal(L)})$. A \emph{localic topos}\index{topos!localic} is any topos of sheaves for a locale: their 2-category is denoted by $\LocTopos$\index{$\LocTopos$}. With canonical topologies, all morphisms of frames become automatically morphisms of sites: so for an arrow of locales $g:L\rightarrow M$ we denote by $\Sh(g):\Sh(L)\rightarrow \Sh(M)$ the induced geometric morphism $\Sh(g\inv):\Sh(\Ocal(L), J\can_{\Ocal(L)})\rightarrow \Sh(\Ocal(M), J\can_{\Ocal(M)})$. The following result will be of fundamental importance:
\begin{prop}[{\cite[Chapter IX, Proposition 6.2]{maclanemoerdijk}}]\label{locales_ff_nei_topos}
	The 2-functor 
	\[
	\Sh(-):\Locale\rightarrow\Topos
	\]
	is 2-full and faithful: \ie, there is a pseudonatural equivalence
	\[
	\Locale(L,M)\simeq \Topos(\Sh(L), \Sh(M)).
	\]
	In particular, 
	\[
	\Locale\simeq\LocTopos.
	\]
\end{prop}
Another useful tool that we will exploit is that of the $J$-ideals of a site $(\cbicat,J)$. A \emph{$J$-ideal}\index{$J$-ideal} of $\cbicat$ is a collection $\ibicat$ of objects of $\cbicat$ such that the following hold: firstly, if $X$ belongs to $\ibicat$ and there is an arrow $Y\rightarrow X$ then $Y$ also belongs to $\ibicat$; secondly, for any $X$ in $\cbicat$, if there exists $S\in J(X)$ such that for every arrow in $f$ the object $\dom(f)$ belongs to $\ibicat$, then $X$ belongs to $\ibicat$. In particular, every object $X$ of $\cbicat$ generates a principal $J$-ideal $\langle X\rangle_J$\index{$\langle X\rangle_J$}, which is the smallest $J$-ideal containing $X$: an element $Y$ of $\cbicat$ belongs to $\langle X\rangle_J$ if and only if it admits a $J$-covering $\{W_i\rightarrow Y \}$ such that every $W_i$ has an arrow to $X$.

The set of $J$-ideals of a small site is denoted by $\Id_J(\cbicat)$\index{$\Id_J(\cbicat)$}, and it is an alternative presentation site for $\Sh(\cbicat,J)$, when $\cbicat$ is a preorder:
\begin{lemma}[{\cite[Theorem 3.1]{caramello2011topostheoretic}}]
	Consider a preorder site $(\cbicat,J)$: then $\Id_J(\cbicat)$ is a locale when ordered by inclusion, and
	\[
	\Sh(\cbicat,J)\simeq \Sh(\Id_J(\cbicat)).\]
	In particular, every topos of sheaves for a preorder site is localic.
\end{lemma}
The passage from preorders to their locales of ideals also transforms continuous comorphisms into morphisms of locales: more explicitly, if $A:(\dbicat,K)\rightarrow (\cbicat,J)$ is a continuous comorphism of preorder sites, the inverse image map $A\inv(-)$ restricts to a morphism of frames $g_A\inv:\Id_J(\cbicat)\rightarrow \Id_K(\dbicat)$ such that the geometric morphisms $C_A$ and $Sh(g_A)$ are equivalent.

One can easily see that for a locale $L$ endowed with the canonical topology, a principal $J\can_J$-ideal $\langle a\rangle_J$ is just the {down-set} $a\downarrow$; moreover, all ideals are principal, since one can check easily that $\ibicat= (\bigvee \ibicat)\downarrow$: this has the consequence that there is an isomorphism of locales 
\[
\Id_{J\can_L}(L)\simeq L.
\]
Now consider the homomorphism of frames $-\wedge a: L\rightarrow a\downarrow$. In this case one speaks about the \emph{open sublocale}\index{open sublocale} $i_a:a\downarrow\hookrightarrow L$: more generally, an open sublocale of $L$ is any 
arrow to $L$ isomorphic to one of the kind. In the following lemma we shall show that when $(\cbicat,J)$ is a preorder, for every $X$ in $\cbicat$ the morphism of locales $\Id_{J_X}(\cbicat/X)\rightarrow \Id_J(\cbicat)$ is an open sublocale of $\Id_J(\cbicat)$:
\begin{lemma}\label{lemma:ideali_slice_preorder}
	Consider a preorder site $(\cbicat,J)$ and an object $X$ in $\cbicat$: then the morphism of locales  $\Id_{J_X}(\cbicat/X)\rightarrow \Id_{J}(\cbicat)$ is equivalent to the open sublocale $\Sub(\langle X\rangle_J):=(\langle X\rangle_J)\downarrow\hookrightarrow\Id_J(\cbicat)$ of subideals of the principal $J$-ideal $\langle X\rangle_J$ generated by $X$. 
\end{lemma}
\begin{proof}
	The functor $\ell_J:\cbicat\rightarrow \Sh(\cbicat,J)$ maps each $X$ to a subterminal $\ell_J(X)$: under the equivalence $\Sh(\cbicat,J)\simeq \Sh(\Id_J(\cbicat))$, it corresponds to the subterminal $\ell(\langle X\rangle _J)$. Therefore
	\[
	\Sh(\Id_{J_X}(\cbicat/X))\simeq \Sh(\cbicat/X, J_X)\simeq \Sh(\cbicat,J)/\ell_J(X)\simeq \Sh(\Id_J(\cbicat))/\ell(\langle X\rangle_J).
	\]
	Notice then that the latter topos, which is still localic, is the topos of sheaves for the locale of its subterminal objects: but subterminal objects of $\Sh(\Id_J(\cbicat))/\ell(\langle X\rangle_J)$ correspond to the subobjects of $\ell(\langle X\rangle_J)$, \ie to the sublocales of $\langle X\rangle_J$: thus we have
	\[
	\Sh(\Id_{J_X}(\cbicat/X))\simeq \Sh( \Sub(\langle X\rangle_J)),\]
	implying $\Id_{J_X}(\cbicat/X)\simeq  \Sub(\langle X\rangle_J)$.
\end{proof}
\begin{remark}
	Let us denote by $p:\cbicat/X\rightarrow\cbicat$ the usual fibration: then the isomorphism of the previous lemma can be explicited as follows:
	\[
	R:\Sub(\langle X\rangle_J)\rightarrow \Id_{J_X}(\cbicat/X),\ 
	R(\kbicat):=p\inv (\kbicat)=\{[Y\leq X]\ |\ Y\in\kbicat\};
	\]
	\[
	R\inv: \Id_{J_X}(\cbicat/X)\rightarrow \Sub(\langle X\rangle_J),\ R\inv(\ibicat):=
	\bigcup_{\substack{\jbicat\in \Id_J(\cbicat),\\\jbicat\subseteq \langle X\rangle_J\\p\inv(\jbicat)\subseteq \ibicat}}\jbicat.
	\]
\end{remark}

Finally, let us recall that we have a notion of \emph{étale morphism of locales}\index{locale!étale morphism of -}: a morphism $f:L\rightarrow M$ of locales is said to be étale if there exists elements $x_i$ in $L$ such that $\bigvee_i x_i=\top_L$, and elements $y_i\in M$ such that $f$ restricts to isomorphisms of open sublocales $x_i\downarrow\isorightarrow y_i\downarrow$ (see \cite[Definition 1.7.1]{borceux3}). When denoting by $\Locale\etale$\index{$\Locale\etale$} the category of locales and their étale morphisms, the equivalence of Proposition \ref{locales_ff_nei_topos} induces an equivalence
\[
\Locale\etale/F(L,M)\simeq \Topos\etale/_1\Sh(F)(\Sh(L),\Sh(M))
\]
for any three locales $F$, $L$ and $M$. 

With all these ingredients, we are now ready to specialize the fundamental adjunction (in its discrete form) to preorder sites:
\begin{prop}\label{prop:fundadj_preordini}\index{$\Lambda_{\Locale/_1\Id_J(\cbicat)}\dashv \Gamma_{\Locale/_1\Id_J(\cbicat)}$}
	Consider a small preorder site $(\cbicat,J)$: there is an adjunction
	\[
	\begin{tikzcd}
		{[\cbicat\op,\Set]} \arrow[r, "\Lambda_{\Locale/_1\Id_J(\cbicat)}", bend left, end anchor={north west}, start anchor={north east}] &     \Locale/_1\Id_J(\cbicat) \arrow[l, "\Gamma_{\Locale/_1\Id_J(\cbicat)}", bend left, start anchor={south west}, end anchor={south east}] \ar[l, "\dashv"{rotate=270}, phantom]
	\end{tikzcd}
	\]
	where:\begin{itemize}
		\item $\Lambda_{\Locale/_1\Id_J(\cbicat)}$ maps a presheaf $P:\cbicat\op\rightarrow \Set$ to the morphism of locales $g_{\pi_P}:\Id_{J_P}(\fib P)\rightarrow \Id_{J}(\cbicat)$ whose corresponding morphism of frames is $\pi_P\inv$ (induced by the $(J_P, J)$-continuous comorphism of sites $\pi_P:\fib P\rightarrow \cbicat$);
		\item $\Gamma_{\Locale/_1\Id_J(\cbicat)}$ 
		acts as a contravariant hom-functor, mapping a morphism of locales $g:L\rightarrow \Id_J(\cbicat)$ to $$\Locale/_1\Id_J(\cbicat)(\Sub(\langle -\rangle_J), L):\cbicat\op\rightarrow\Set:$$
		explicitly, the value of $\Gamma(L)$ at a certain $X\in \cbicat$ is the set of sections of $g$ over the open sublocale $\Sub(\langle X\rangle_J)\hookrightarrow\Id_J(\cbicat)$, \ie the locale morphisms $\Sub(\langle X\rangle_J)\rightarrow L$ making the following diagram commutative:
		\[\begin{tikzcd}[column sep=0.5]
			\Sub(\langle X\rangle_J) \ar[rr] \ar[dr] & &L \ar[dl, "g"]\\
			& \Id_J(\cbicat)&
		\end{tikzcd}.\]
	\end{itemize}
	Moreover, the composite $\Gamma_{\Locale/_1\Id_J(\cbicat)}\Lambda_{\Locale/_1\Id_J(\cbicat)}(P)$ corresponds to the sheafification $\sheafify_J(P)$: more explicitly, for every $X$ in $\cbicat$ we have
	\[\sheafify_J(P)(X):=\Locale/_1\Id_J(\cbicat)(\Sub(\langle X\rangle_J), \Id_{J_P}(\fib P)).\]
	The adjunction restricts to an equivalence
	\[
	\Sh(\cbicat,J)\simeq \Locale\etale/\Id_J(\cbicat).
	\]
\end{prop} 
\begin{proof}
	We can start by remarking that the functor
	\[
	\Lambda_{\Topos^s/_1\Sh(\cbicat,J)}:[\cbicat\op,\Set]\rightarrow \Topos^s/_1\Sh(\cbicat,J)
	\]
	takes values inside the smaller slice category $\LocTopos/_1\Sh(\cbicat,J)$: this is true since every site $(\fib P, J_P)$ is a preorder site, and hence the topos $\Sh(\fib P, J_P)$ is localic. On the other hand, given a localic topos $\Etopos\hspace{-0.2ex}\rightarrow\hspace{-0.2ex}\Sh(\cbicat,J)$, the legs of every cocone in $\Topos^s/_1\Sh(\cbicat,J) (\Sh(\cbicat/-, J_{(-)}), \Etopos)$, and the geometric morphism $\Sh(\fib P, J_P)\rightarrow\Etopos$ induced by the universal property of colimits, will be geometric morphisms of localic toposes too. This means that the pseudonatural equivalence
	\begin{align*}
		[\cbicat\op,\Set](P, \Topos^s/_1\Sh(\cbicat,J) (\Sh(\cbicat/-, J_{(-)}), \Etopos) )\simeq\\
		\simeq \Topos^s/_1\Sh(\cbicat,J)(\Sh(\fib P, J_P), \Etopos)
	\end{align*}
	restricts to an equivalence
	\begin{align*}
		[\cbicat\op,\Set](P, \LocTopos/_1\Sh(\cbicat,J) (\Sh(\cbicat/-, J_{(-)}), \Etopos) )\simeq\\
		\simeq \LocTopos/_1\Sh(\cbicat,J)(\Sh(\fib P, J_P), \Etopos)	
	\end{align*}
	and thus we have an adjunction
	\[
	\begin{tikzcd}
		{[\cbicat\op,\Set]} \arrow[r, "\Lambda_{\LocTopos/_1\Sh(\cbicat,J)}", bend left, end anchor={north west}, start anchor={north east}] &     \LocTopos/_1\Sh(\cbicat,J) \arrow[l, "\Gamma_{\LocTopos/_1\Sh(\cbicat,J)}", bend left, start anchor={south west}, end anchor={south east}] \ar[l, "\dashv"{rotate=270}, phantom]
	\end{tikzcd}.
	\]
	Since $(\cbicat,J)$ is a preorder site we have
	\[
	\LocTopos/_1\Sh(\cbicat,J)\simeq \LocTopos/_1\Sh(\Id_J(\cbicat))\simeq \Locale/_1\Id_J(\cbicat)
	\] 
	and we can conclude that there is an equivalence
	\begin{align*}
		[\cbicat\op,\Set](P, \Locale/_1\Id_J(\cbicat) (\Id_{J_{(-)}}(\cbicat/-), L) )\simeq \\\simeq \Locale/_1\Id_J(\cbicat)(\Id_{J_{P}}(\fib P), L)
	\end{align*}
	proving that $\Lambda_{\Locale/_1\Id_J(\cbicat)}\dashv \Gamma_{\Locale/_1\Id_J(\cbicat)}$. 
	Lemma \ref{lemma:ideali_slice_preorder} provides us the isomorphism $\Id_{J_{(-)}}(\cbicat/-)\simeq \Sub(\langle -\rangle_J)$, justifying the description of $\Gamma_{\Locale/_1\Id_J(\cbicat)}$ in the claim of the theorem. 
	
	Finally, the chain of known equivalences
	\[
	\Sh(\cbicat,J)\simeq \Topos\etale/_1\Sh(\cbicat,J)\simeq \Locale\etale/\Id_J(\cbicat)
	\]
	provides the restriction of the adjunction to sheaves.
\end{proof}
To specialize to the localic case, \ie when $(\cbicat,J)=(L, J\can_L)$ we can exploit the isomorphism $\Id_{J\can_L}(L)\cong L$:
\begin{cor}\label{cor:adj_locales}\index{$\Lambda_{\Locale/_1L}\dashv \Gamma_{\Locale/_1L}$} 
	Consider a locale $L$: there is an adjunction
	\[
	\begin{tikzcd}
		{[\Ocal(L)\op,\Set]} \arrow[r, "\Lambda_{\Locale/_1L}", bend left, end anchor={north west}, start anchor={north east}] &     \Locale/_1L \arrow[l, "\Gamma_{\Locale/_1L}", bend left, start anchor={south west}, end anchor={south east}] \ar[l, "\dashv"{rotate=270}, phantom]
	\end{tikzcd},
	\]
	where:\begin{itemize}
		\item $\Lambda_{\Locale/_1L}$ maps a presheaf $P:\Ocal(L)\op\rightarrow \Set$ to the morphism of locales $g_{\pi_P}:\Id_{J_P}(\fib P)\rightarrow L$, where $\pi_P:\fib P\rightarrow \cbicat$ and $J_P$ is Giraud's topology induced by the canonical topology $J\can_L$: $g_{\pi_P}$ corresponds to the homomorphism of frames $\pi_P\inv:L\rightarrow \Id_{J_P}(\fib P)$;
		\item $\Gamma_{\Locale/_1L}$ acts as a contravariant hom-functor, mapping a morphism of locales $g:M\rightarrow L$ to 
		\[
		\Locale/_1L((-)\downarrow, M):\Ocal(L)\op\rightarrow\Set:
		\]
		explicitly, the value of $\Gamma(M)$ at a certain $a\in L$ is the set of sections of $g:M\rightarrow L$ over the open sublocale $a\downarrow \hookrightarrow L$.
	\end{itemize}
	Moreover, the composite $\Gamma_{\Locale/_1L}\Lambda_{\Locale/_1L}(P)$ corresponds to the sheafification $\sheafify_{J\can_L}(P)$, \ie for any $a$ in $L$ we have
	\[
	\sheafify_{J\can_L}(P)(a):=\Locale/_1 L( a\downarrow, \Id_{J_P}(\fib P)).\]
\end{cor} 

\begin{remarks}\label{rmk:connessione_localeadj_tpladj}\noindent
	\begin{enumerate}[(i)]
		\item We will condense the previous considerations in Subsection \ref{sec:restrictibility}, showing a general context in which the fundamental adjunction restricts from the topos-theoretic environment to sites and their morphisms or comorphisms.
		\item The existence of the adjunction of Corollary \ref{cor:adj_locales} is implicit in the results of \cite[Section 2.6]{borceux3} and in the exercises from 9 to 12 of \cite[Chapter IX]{maclanemoerdijk}, though in both cases the focus is on the equivalence between sheaves over the locale and étale mappings to it. We also remark that, relating the equivalence $\Sh(L)\simeq \Etale/L$ to $\Sh(\cbicat,J)\simeq \Etale(\cbicat,J)$, which holds for any preorder site $(\cbicat,J)$, we can conclude that a monotone map of preorders $f:P\rightarrow \cbicat$ is $J$-étale if and only if the corresponding homomorphism of frames $f\inv:\Id_J(\cbicat)\rightarrow \Id_{J_P}(P)$ is an étale arrow of locales.
		\item We can now see how exactly the topological framework connects to the localic adjunction of the previous corollary. In Proposition \ref{prop:fascificazione_topologica_con_framehom} we have seen that for a topological space $X$ and a presheaf $P\in \Psh(X)$ it holds that, for any open $U$ of $X$,
		\[
		\sheafify_J(P)(U)\simeq \Locale/\Ocal(X) (\Ocal(U), \Ocal(E_P)),\]
		where $E_P$ is the étale bundle associated to $P$. By the last result however we have that
		\[
		\sheafify_J(P)(U)\simeq \Locale/\Ocal(X) (\Ocal(U), \Id_{J_P}(\fib P)),\]
		where $(\fib P, J_P)$ is the Giraud site associated to $P$. The apparent difference is immediately explained by noticing that the map
		\[\Id_{J_P}(\fib P)\rightarrow \Ocal(E_P),\ \ibicat\mapsto \bigcup_{(U,s)\in \ibicat}\dot{s}(U)
		\]
		provides an isomorphism between the two locales $\Ocal(E_P)$ and $\Id_{J_P}(\fib P)$. This is in fact an extension of the map $f_P$ we defined in Proposition \ref{prop:fasci__etalebundle_eqv_fasci_grfibration}, and indeed that result entails the isomorphism $\Id_{J_P}(\fib P)\simeq \Ocal(E_P)$, since the equivalence
		\[
		\Sh(E_P)\simeq\Sh(\fib P, J_P)\]
		of localic toposes restrict to an isomorphism of their locales of subterminal objects. Thus, even though classically one defines the sheafification of $P\in \Psh(X)$ as the local sections of the étale bundle $E_P\rightarrow X$, these considerations show explicitly that there is no need to work in the topological context, since all the relevant information lives at the localic level. Finally, we remark that this last consideration is essentially the content of \cite[Proposition 2.5.4]{borceux3}, where nonetheless there is no explicit reference to the locale $\Id_{J_P}(\fib P)$ but instead to the locale of closed subpresheaves of $P$.
	\end{enumerate}
\end{remarks}
We conclude this section by remarking that our description of the sheafification in the localic case is naturally site-theoretic, and that in the case of topological spaces sheafification can also be described using comorphisms of sites. First of all, the following holds:
\begin{lemma}\label{lemma:framehom_sse_morfsiti}
	Consider two frames $L$ and $M$: frame homomorphisms $L\rightarrow M$ correspond to morphisms of sites $(L, J\can_L)\rightarrow (M, J\can_M)$. In other words, the 2-functor
	\[
	\Frame \xhookrightarrow{L\mapsto (L, J\can_L) } \Site
	\]
	is 2-fully faithful.
\end{lemma}
\begin{proof}
	Since a frame is a category with finite limits, a functor $A:L\rightarrow M$ is a morphism of sites if and only if it preserves finite limits and is cover-preserving: this means that it preserves finite meets and arbitrary joins, \ie it is a homomorphism of frames.
\end{proof}
This implies that  for a locale $L$, an element $a$ of $L$ and a presheaf $P\in \Psh(L)$ we have
\begin{align*}
	\sheafify(P)(a)&=\Locale/L(a\downarrow, \Id_{J_P}(\fib P))\\
	&=(L, J\can_L)/\Site( (\Id_{J_P}(\fib P), J\can_{\Id_{J_P}(\fib P)}), (a\downarrow, J\can_{a\downarrow}) ).
\end{align*}
Suppose now that $L=\Ocal(X)$ for a topological space $X$. Proposition \ref{cor:sezioni_bundleetale_aperte} showed that that sections of the étale bundle $\pi:E_P\rightarrow X$ are open maps, and using Lemma \ref{lemma:open_map_induce_adjunction_topologies} we have that a section $s:U\rightarrow E_P$ (triangle on the left) induces \emph{two} commutative triangles:
\[
\begin{tikzcd}
	U \ar[dr, "i_U"'] \ar[r,"s"] & E_P\ar[d, "\pi"]\\
	& X
\end{tikzcd}\rightsquigarrow
\begin{tikzcd}
	\Ocal(U) \ar[dr, leftarrow, "i_U\inv"'] \ar[r,leftarrow,"s\inv"] & \Ocal(E_P)\ar[d, "\pi\inv", leftarrow]\\
	& \Ocal(X)
\end{tikzcd},\ 
\begin{tikzcd}
	\Ocal(U) \ar[dr, "(i_U)_!"'] \ar[r,"s_!"] & \Ocal(E_P)\ar[d, "\pi_!"]\\
	& \Ocal(X)
\end{tikzcd}
\] 
We already know that the middle triangle is a triangle of morphisms of sites, and by Applying \cite[Proposition 3.14]{denseness}, we have that the triangle on the right is a diagram of comorphisms of sites. In fact, more can be said:
\begin{prop}
	For any section $s:U\rightarrow E_P$, the comorphism \[s_!\hspace{-0.3ex}:\hspace{-0.3ex}(\Ocal(U),J\can_{\Ocal(U)})\hspace{-0.3ex}\rightarrow\hspace{-0.3ex} (\Ocal(E_P),J\can_{\Ocal(E_P)})\] is $(J\can_{\Ocal(U)}, J\can_{\Ocal(E_P)})$-continuous.
\end{prop}
\begin{proof}
	Consider a sheaf $W\in \Sh(E_P)$: we want to show that $W\circ s_!\op$ is a sheaf over $U$. To do so, consider a family of opens $\{V_i\subseteq U\ |\ i\in I \}$ and take for every $i$ an element $x_i\in W\circ s_!\op(V_i)=W(s(V_i))$ so that for every open $Z\subseteq V_i\cap V_j$ one has $x_{i|s(Z)}=x_{j|s(Z)}$. We want to prove the esistence of an amalgamation of the $x_i$, \ie of an element $x\in W(s(\cup V_i))$ such that for every $i$ one has $x_i=x_{|s(V_i)}$. To do so, it is sufficient to check that the elements $x_i$ are also a matching family for $W$ and the family of opens $\{s(V_i)\ |\ i\in I \}$, for this will provide an amalgamation $x\in W(\cup_i s(V_i))=W(s(\cup_i V_i))$ satisfying precisely the condition above. T see that the $x_i$'s are a matching family for $W$ we can restrict to considering the basic opens $\dot{r}(Z)$ of $E_P$. First of all, if $\dot{r}(Z)\subseteq s(V_i)$ and $z\in Z$, there exists $y\in V_i$ such that $r_z=s(y)$, and thus $z=\pi (r_z)=\pi s(y)=y$: therefore not only $Z\subseteq V_i$, but actually $\dot{r}=s_{|Z}$. If $\dot{r}(Z)\subseteq s(V_i)\cap s(V_j)$, we have that $x_{i|\dot{r}(Z)}=x_{i|s(Z)}=x_{j|s(Z)}=x_{j|\dot{r}(Z)}$, where the second equality holds by the matching condition for the $x_i$'s. This proves that the family of the $x_i$ is a matching family for the sheaf $W$ and the covering $\{s(V_i)\ |\ i\in I \}$, and thus it admits the amalgamation $x\in W(\cup s(V_i))$ we needed.
\end{proof}
Lemma \ref{prop:fascificazione_topologica_con_framehom} showed that \emph{any} frame homomorphism $f:\Ocal(E_P)\rightarrow\Ocal(U)$ satisfying $i_U\inv=f\pi\inv$ is of the form $s\inv$ for some section $s:U\rightarrow E_P$: this implies that it admits a left adjoint $s_!$ which by the last result is a $(J\can_{\Ocal(U)},J\can_{\Ocal(E_P)})$-continuous comorphism of sites. Conversely, consider a $(J\can_{\Ocal(U)},J\can_{\Ocal(E_P)})$-continuous comorphism of sites $B:\Ocal(U)\rightarrow \Ocal(E_P)$: then by \cite[Proposition 4.11(iii)]{denseness} it is cover-preserving, \ie preserves arbitrary colimits, and thus it admits a right adjoint which is a morphism of sites $f:\Ocal(E_P)\rightarrow \Ocal(U)$. We can thus conclude the following:
\begin{prop}
	Consider a topological space $X$, an open subset $U\subseteq X$ and a presheaf $P\in \Psh(X)$: then there are natural isomorphisms
	\begin{align*}
		\sheafify_J(P)(U)&:=\Top/_1X(U, E_P)\\
		&\simeq\Locale/\Ocal(X)(\Ocal(U), \Ocal(E_P))\\
		&\simeq (\Ocal(X), J\can_{\Ocal(X)})/_1\Site((\Ocal(E_P),J\can_{\Ocal(E_P)}), (\Ocal(U), J\can_{\Ocal(U)}))\\
		&\simeq \Com\cont/_1(\Ocal(X), J\can_{\Ocal(X)})((\Ocal(U),J\can_{\Ocal(U)}), (\Ocal(E_P), J\can_{\Ocal(E_P)})).
	\end{align*}
\end{prop}
\begin{remark}
	The last set is a slight abuse of notation: even though we are considering the continuous comorphisms of sites $\Ocal(U)\rightarrow \Ocal(E_P)$ over $\Ocal(X)$, the functor $\pi_!:\Ocal(E_P)\rightarrow \Ocal(X)$ does not belong to $\Com\cont$ since it is not continuous in general.
	On the other hand, the comorphism $(i_U)_!$ is continuous, since it is precisely the discrete fibration $p_U:\Ocal(X)/U\rightarrow \Ocal(X)$. 
\end{remark}

\subsection{The fundamental adjunction in the language of internal locales}

A further formulation of the fundamental adjunction in the localic setting can be provided when we recur to internal locales in a topos. We will see that in this context the base locale is `absorbed' by the topos we are working in, and thus the right adjoint $\Gamma$ presents itself simply as a contravariant hom-functor.

First of all, we remark that the definition of a frame can be interpreted in any topos. The part of the theory regarding the finitary structure is easily interpreted. an internal bounded meet-semilattice in a topos $\Etopos$ will be an object $L$ of $\Etopos$ provided with three arrows 
\[
1\xrightarrow{\top}L,\ 1\xrightarrow{\bot}L,\ L\times L\xrightarrow{\wedge} L,
\]
the interpretations of respectively the top element, the bottom element and the meet operation, which make the obvious diagrams commutative. As per the infinitary operation of arbitrary joins, it can be interpreted as an arrow 
\[\textstyle\bigvee:\powerset(L)\rightarrow L,
\]
where $\powerset(L)=\Omega^L$ is the power object of $L$ (see \cite[pag. 69]{elephant}): in the internal logic of $\Etopos$, we think of the arrow $\bigvee$ as mapping any $S\subseteq L$ to its join $\bigvee S\in L$. By writing suitable commutative diagrams expressing the usual axioms, we obtain the notion of an \emph{internal frame}\index{frame!internal} $L$ in $\Etopos$. A \emph{homomorphism of internal frames} $f:L\rightarrow M$ is an arrow commuting with all the operations. We shall denote by $\Frame(\Etopos)$ the category of internal frames of a topos $\Etopos$, and by $\Loc(\Etopos)=\Frame(\Etopos)\op$\index{$\Loc(\Etopos)$} its \emph{category of internal locales}\index{locale!internal}.

We now recall Proposition 2 from \cite[Chapter VI, §3]{joyal_tierney_galois} (see also \cite[Section C1.6]{elephant}), which shows that when $\Etopos$ is localic, there is a correspondence between internal and external locales:
\begin{prop}
	Consider a locale $L$: there is an equivalence of categories
	\[
	H:\Locale/_1L\isorightarrow \Loc(\Sh(L)).
	\]
	In particular, to a morphism of frames $f:L\rightarrow M$ is associated the internal locale
	\[
	H(f):L\op\rightarrow \Set,\ H(f)(a)=\{m\in M\ |\ m\leq f(a)\}.\]
\end{prop}
Using the equivalence $H$ we can reformulate the adjunction of Proposition \ref{prop:adj_preord}. To do so, take a preorder site $(\cbicat,J)$ and consider the composite
\[
\bar{\Lambda}:[\cbicat\op,\Set]\xrightarrow{\Lambda_{\Locale/1\Id_J(\cbicat)}}\Locale/_1\Id_J(\cbicat)\xrightarrow{H}\Loc\left(\Sh(\cbicat,J)\right):
\]
the first functor maps a presheaf $P:\cbicat\op\rightarrow\Set$ to the morphism of locales $g_{p_P}:\Id_{J_P}(\fib P)\rightarrow \Id_J(\cbicat)$, \ie to the homomorphism of frames $p_P\inv:\Id_J(\cbicat)\rightarrow \Id_{J_P}(\fib P)$; the second functor maps $g_{p_P}$ to a presheaf $\bar{\Lambda}(P):\cbicat\op\rightarrow \Set$ which is an internal locale of $\Sh(\cbicat,J)$. By the definition of the functor $H$, we have that for $Z$ in $\cbicat$
\[
\bar{\Lambda}(P)(Z):=\{ \ibicat\in \Id_{J_P}(\fib P)\ |\ p_P(\ibicat)\subseteq \langle Z\rangle_J \},
\]
where we use the fact that under the equivalence $\Sh(\cbicat,J)\simeq \Sh(\Id_J(\cbicat))$ the element $Z$ corresponds to the principal $J$-ideal $\langle Z\rangle_J$. In particular, when $P=\yo(X)$ we can exploit Lemma \ref{lemma:ideali_slice_preorder} to obtain the following chain of natural isomorphisms:
\begin{align*}
	\bar{\Lambda}(P)(Z)&:=\{\ibicat\in\Id_{J_X}(\cbicat/X)\ |\ p_X(\ibicat)\subseteq \langle Z\rangle_J \}\\
	&\simeq \{\ibicat \in \Id_J(\cbicat)\ |\ \ibicat\subseteq \langle X\rangle_J\cap \langle Z\rangle_J  \}\\
	&= \Sub(\langle X\rangle_J\cap \langle Z\rangle_J).
\end{align*}
We end up with the following result:
\begin{prop}
	Consider a preorder site $(\cbicat,J)$: there is an adjunction
	\[
	\begin{tikzcd}
		{[\cbicat\op,\Set]} \arrow[r, "\Lambda_{\Loc(\Sh(\cbicat,J))}", bend left, end anchor={north west}, start anchor={north east}] &     {\Loc(\Sh(\cbicat,J))} \arrow[l, "\Gamma_{\Loc(\Sh(\cbicat,J))}", bend left, start anchor={south west}, end anchor={south east}] \ar[l, "\dashv"{rotate=270}, phantom]
	\end{tikzcd}
	\]
	which acts as follows:
	\begin{itemize}
		\item The functor $\Lambda_{\Loc(\Sh(\cbicat,J))}$ maps a presheaf $P:\cbicat\op\rightarrow \Set$ to the internal locale acting for any $X$ in $\cbicat$ as
		\[
		\Lambda_{\Loc(\Sh(\cbicat,J))}(P)(X)=\bar{\Lambda}(P)(X)=\{\ibicat\in \Id_{J_P}(\fib P)\ |\ \pi_P(\ibicat)\subseteq \langle X\rangle_J \}
		\]
		\item The functor $\Gamma_{\Loc(\Sh(\cbicat,J))}$ acts by mapping $L\in \Loc(\Sh(\cbicat,J))$ to the presheaf
		\[
		\Gamma_{\Loc(\Sh(\cbicat,J))}(L):\cbicat\op\rightarrow\Set\]
		which acts as a contravariant hom-functor of internal locales:
		\[ \Gamma_{\Loc(\Sh(\cbicat,J))}:X\mapsto \Loc(\Sh(\cbicat,J))(\Sub(\langle -\rangle_J\cap\langle X\rangle_J),L).
		\]
	\end{itemize}
\end{prop}

\subsection{A general restrictibility condition for the fundamental adjunction}\label{sec:restrictibility}
Let us go back to the fundamental adjunction that we built for locales, this time stated in terms of a frame $F$:
\[
\begin{tikzcd}[row sep=40]
	{[F\op,\Set]} \arrow[r, "\Lambda_{\Topos^s/_1\Sh(F, J\can_F)}", bend left, end anchor={north west}, start anchor={north east}] \arrow[d, "\Lambda_{(F/\Frame)\op}"{yshift=-2ex}, bend left] &     \Topos^s/_1\Sh(F, J\can_F) \arrow[l, "\Gamma_{\Topos^s/_1\Sh(F, J\can_F)}"{xshift=3ex}, bend left, start anchor={south west}, end anchor={south east}] \ar[l, "\dashv"{rotate=270}, phantom]\\
	(F/_1\Frame)\op \ar[ur, hook, bend right, "\Sh"'] \arrow[u, "\Gamma_{(F/_1\Frame)\op}", bend left] \ar[u, "\dashv"{rotate=180}, phantom]&
\end{tikzcd}.
\]
Basically, what made it possible to forget the topos-theoretic data of the fundamental adjunction was the fact that the left adjoint $\Lambda$ factors through the functor $\Sh:(F/_1\Frame)\op\rightarrow \Topos^s/_1\Sh(F)$, which moreover is full and faithful. In fact, the same property of site-describability of the sheafification can be formulated more in general for special classes of sites and their co-/morphisms. To do this, we will exploit the following general result:
\begin{lemma}
	Consider the diagram of 2-categories
	\[
	\begin{tikzcd}
		\abicat \ar[r,bend left, "\Lambda"] \ar[rd, "\bar{\Lambda}"', bend right] \ar[r, phantom, "\dashv"{rotate=270}] & \bbicat \ar[l, bend left, "\Gamma"]\\
		& \cbicat \ar[u, hook, "i"']
	\end{tikzcd}
	,\]
	where $i$ is 2-fully faithful. If $i\bar{\Lambda}\simeq \Lambda$, then $\bar{\Gamma}:=\Gamma i$ is a right adjoint to $\bar{\Lambda}$, and $\Gamma\Lambda\simeq \bar{\Gamma}\bar{\Lambda}$. The same holds for 1-categories and $i$ fully faithful.
\end{lemma}
\begin{proof}
	It is immediate by the following chain of pseudonatural equivalences (resp. natural isomorphisms):
	\[
	\abicat(X, \Gamma i(Y))\simeq \bbicat( \Lambda(X), i(Y))\simeq \bbicat (i\bar{\Lambda}(X), i(Y)) \simeq \cbicat(\bar{\Lambda}(X), Y)
	.\]
\end{proof}
We obtain the following two corollaries, stating a site-describability condition for the fundamental adjunction:
\begin{cor}
	Consider a site $(\cbicat,J)$ and suppose that there is a functor $\abicat\rightarrow(\cbicat,J)/\Site$ such that the following hold: 
	\begin{itemize}
		\item The composite functor $\abicat \op \rightarrow ((\cbicat,J)/\Site)\op \rightarrow \Topos/\Sh(\cbicat,J)$ factors through $\Topos^s/_1\Sh(\cbicat,J)$ so that $i:\abicat\op\rightarrow \Topos^s/_1\Sh(\cbicat,J)$ is full and faithful;
		\item there is a functor $\bar{\Lambda}: [\cbicat\op,\Set]\rightarrow \abicat$ such that $\Lambda_{\Topos^s/_1\Sh(\cbicat,J)}\simeq i \bar{\Lambda}$:
	\end{itemize}
	then $\bar{\Lambda}$ admits a right adjoint $\bar{\Gamma}$, and $\sheafify_J(P)\simeq \bar{\Gamma}\bar{\Lambda}(P)$.
	\[
	\begin{tikzcd}
		((\cbicat,J)/\Site)\op \ar[r, "\Sh(-)"] & \Topos/_1\Sh(\cbicat,J)\\
		\abicat\op \ar[u] \ar[r, hook, "i"', dashed] & \Topos^s/_1\Sh(\cbicat,J) \ar[u, hook]	
	\end{tikzcd}\]\[
	\begin{tikzcd}
		{[\cbicat\op,\Set]} \ar[r, "\Lambda_{\Topos^s/_1\Sh(\cbicat,J)}"{yshift=1ex}] \ar[dr, "\bar{\Lambda}"', dashed] & {\Topos^s/_1\Sh(\cbicat,J)}
		\\
		& \abicat\op \ar[u, hook, "i"']	
	\end{tikzcd}
	\]
\end{cor}
\begin{cor}
	Consider a site $(\cbicat,J)$ and suppose that there is a functor $\abicat\rightarrow\Com/(\cbicat,J)$ such that the following hold: 
	\begin{itemize}
		\item The composite $\abicat \rightarrow \Cosite/(\cbicat,J) \rightarrow \Topos/\Sh(\cbicat,J)$ factors through $\Topos^s/_1\Sh(\cbicat,J)$ so that $i:\abicat\rightarrow \Topos^s/_1\Sh(\cbicat,J)$ is full and faithful;
		\item there is a functor $\bar{\Lambda}: [\cbicat\op,\Set]\rightarrow \abicat$ such that $\Lambda_{\Topos^s/_1\Sh(\cbicat,J)}\simeq i \bar{\Lambda}$:
	\end{itemize}
	then $\bar{\Lambda}$ admits a right adjoint $\bar{\Gamma}$, and $\sheafify_J(P)\simeq \bar{\Gamma}\bar{\Lambda}(P)$.
	\[
	\begin{tikzcd}
		\Cosite/(\cbicat,J) \ar[r, "C_{(-)}"] & \Topos/_1\Sh(\cbicat,J)\\
		\abicat \ar[u] \ar[r, hook, "i"', dashed] & \Topos^s/_1\Sh(\cbicat,J) \ar[u, hook]	
	\end{tikzcd}\]\[
	\begin{tikzcd}
		{[\cbicat\op,\Set]} \ar[r, "\Lambda_{\Topos^s/_1\Sh(\cbicat,J)}"{yshift=1ex}] \ar[dr, "\bar{\Lambda}"', dashed] & {\Topos^s/_1\Sh(\cbicat,J)}
		\\
		& \abicat\ar[u, hook, "i"']	
	\end{tikzcd}
	\]
\end{cor}

\begin{remark}
	Consider a preorder site $(\cbicat,J)$: we can take as the functor $\abicat\rightarrow (\cbicat,J)/\Site$ the functor
	\[
	\Id_J(\cbicat)/\Frame\rightarrow (\Id_J(\cbicat), J\can_{\Id_J(\cbicat)})/\Site
	\]
	mapping $f:\Id_J(\cbicat)\rightarrow L$ to the same arrow seen as a morphism of sites 
	$f: (\Id_J(\cbicat), J\can_{\Id_J(\cbicat)})\rightarrow (L,J\can_L)$ (cfr. Lemma \ref{lemma:framehom_sse_morfsiti}), we obtain again the adjunctions of Proposition \ref{prop:adj_preord} and Corollary \ref{cor:adj_locales}.
\end{remark}

\section{Four site-theoretic points of view on the associated sheaf functor}\label{sec:point_of_view_sheafify}

As we have seen in Section \ref{sec:discrete_adj}, for any essentially small site $(\cbicat,J)$ the sheafification functor
$\sheafify_J:[\cbicat\op,\Set]\rightarrow \Sh(\cbicat,J)$
can be described as the composite
\[
[\cbicat\op,\Set] \xrightarrow{\Lambda_{\Topos^s/_1\Sh(\cbicat,J)}}\Topos^s/_1\Sh(\cbicat,J)\xrightarrow{\Gamma_{\Topos^s/_1\Sh(\cbicat,J)}}[\cbicat\op,\Set]:
\]
\ie, for any presheaf $P:\cbicat\op\rightarrow\Set$ and $X$ in $\cbicat$ we have
\[
\sheafify_J(P)(X)\simeq \Topos^s/_1\Sh(\cbicat,J) (\Sh(\cbicat/X, J_X), \Sh(\fib P, J_P)).
\]
The aim of the following sections is to provide various general descriptions of the geometric morphisms in $\Topos^s/_1\Sh(\cbicat,J) (\Sh(\cbicat/X, J_X), \Sh(\fib P, J_P))$ using sites.

\subsection{The algebraic point of view: morphisms of sites}
We will begin by describing elements of $\sheafify_J(P)(X)$ exploiting the following 1-categorial variation of Theorem \ref{thm:classthm_morfismi_siti}:
\begin{prop}
	Consider a comorphism of sites $p:(\dbicat,K)\rightarrow (\cbicat,J)$ and a geometric morphism $E:\Etopos\rightarrow \Sh(\cbicat,J)$: denote by $A_E$ the morphism of sites $(\cbicat,J)\rightarrow (\Etopos, J\can_\Etopos)$ corresponding to $E$. Geometric morphisms in $\Topos/_1\Sh(\cbicat,J)([E], [C_p])$, \ie equivalence classes of geometric morphisms $F$ making the diagram
	\[\begin{tikzcd}[column sep=5]
		\Etopos \ar[d, "E"'] \ar[r, "F"]& \Sh(\dbicat,K) \ar[dl, "C_p"]\\
		\Sh(\cbicat,J)&
	\end{tikzcd}\]
	commutative up to isomorphism, are in bijective correspondence with equivalence classes of morphisms of sites 
	\[A_F:(\dbicat,K)\rightarrow (\Etopos, J\can_\Etopos)\]
	such that there is a natural transformation $\bar{\phi}:A_F\Rightarrow A_Ep$ whose induced natural transformation $\tilde{\phi}:F^*\sheafify_J\Rightarrow E^*\sheafify_J\lan_{p\op}$ is such that the composite
	\[
	F^*\sheafify_Kp^*\xRightarrow{\tilde{\phi}\circ p^*}E^*\sheafify_J\lan_{p\op}p^*\xRightarrow{E^*\sheafify_J\circ \epsilon}E^*\sheafify_J
	\]
	is invertible (where $\epsilon$ is the counit of $\lan_{p\op}\dashv p^*$).
\end{prop}
\begin{proof}
	Theorem \ref{thm:classthm_morfismi_siti} showed the equivalence
	$$\Topos\sslash \Sh(\cbicat,J)([E], [C_p])\simeq \Site((\dbicat,K),(\Etopos, J\can_\Etopos))/E^*\ell_Jp:$$
	in particular, a geometric morphism $F:\Etopos\rightarrow \Sh(\dbicat,K)$ endowed with a natural transformation $\phi:F^*C_p^*\Rightarrow E^*$ corresponds to a morphism of sites $A_F:(\dbicat,K)\rightarrow (\Etopos, J\can_\Etopos)$ endowed with a natural transformation $\bar{\phi}:A_F\Rightarrow A_E p$. Using Remark \ref{rmk:classthm_morfismi_relazione_2celle}, we see that $\phi$ is invertible if and only if $\bar{\phi}:A_F\Rightarrow A_E p$ satisfies the condition in the claim. 
	
	Finally, let us consider objects of 
	\[
	\Topos/_1\Sh(\cbicat,J)([E], [C_p]):\]  
	first of all, we notice that two equivalent geometric morphisms $F\cong F':\Etopos\rightarrow \Sh(\dbicat,K)$ correspond to equivalent morphisms of sites $A_F\cong A_{F'}:(\dbicat,K)\rightarrow (\Etopos, J\can_\Etopos)$. Moreover, suppose that up to equivalence $C_pF=E$: this means that there exists an invertible 2-cell $\phi: F^*C_p^*\Isorightarrow E^*$, and thus a natural trasformation $\bar{\phi}$ as we claimed.
\end{proof}
\begin{remark}
	Using again Remark \ref{rmk:classthm_morfismi_relazione_2celle}, we have that $\bar{\phi}$ satisfies the claim of the result if and only if for every presheaf $H:\cbicat\op\rightarrow\Set$, if we consider for every $X$ in $\cbicat$ the collection of elements $x\in H(X)$ (with $\name{x}:\yo(X)\rightarrow H$ the corresponding arrow of $[\cbicat\op,\Set]$) and the collection of arrow $y:p(D)\rightarrow X$, the arrows $\alpha_{x,y}$ defined as
	\[
	A_F(D) \xrightarrow{\bar{\phi}(D)}A_E(p(D))\xrightarrow{A_E(y)} A_E(X)=E^*\ell_J(X)\xrightarrow{E^*\sheafify_J(\name{x})}E^*\sheafify_J(H)
	\]
	form a colimit cocone in $\Etopos$. 
\end{remark}
The previous considerations allow us to express the sheafification of a functor as a set of equivalence classes of morphisms of sites:
\begin{prop}\label{prop:fascificazione_morfsiti}
	Consider a site $(\cbicat,J)$ and a presheaf $P:\cbicat\op\rightarrow\Set$ with corresponding Grothendieck fibration $p:\fib P\rightarrow \cbicat$. For an object $X$ in $\cbicat$, denote by $B_X:\cbicat\rightarrow \Sh(\cbicat/X, J_X)$ the flat $J$-continuous functor associated to $C_{p_X}:\Sh(\cbicat/X, J_X)\rightarrow \Sh(\cbicat,J)$: it acts by mapping any $Y$ in $\cbicat$ to the sheaf $\ell_J(Y)\circ p_X\op$, and accordingly on arrows. The set $\sheafify_J(P)(X)$ is isomorphic to the set of equivalence classes of morphisms of sites 
	\[
	A:(\fib P, J_P)\rightarrow (\Sh(\cbicat/X, J_X),J\can_{\Sh(\cbicat/X, J_X)})
	\]
	such that there is a natural transformation
	\[
	\textstyle\phi:A\Rightarrow B_Xp:\fib P\rightarrow \Sh(\cbicat/X,J_X)
	\]
	satisfying the following: for every presheaf $H:\cbicat\op\rightarrow\Set$, if we consider for every $Y$ in $\cbicat$ the collection of arrows $y:\yo(Y)\rightarrow H$ of $[\cbicat\op,\Set]$ and for every $(Z,s)$ in $\fib P$ the collection of arrows $z:Z\rightarrow Y$, the composites 
	\[
	\alpha_{y,z}: A(Z,s) \xrightarrow{{\phi}(Z,s)}B_X(Z)\xrightarrow{B_X(z)} B_X(Y)\xrightarrow{B_X(y)}B_X(H)
	\]
	form a colimit cocone in $\Sh(\cbicat/X, J_X)$. 
\end{prop}

\subsection{The geometric approach: comorphisms of sites that are indexed by $J$-covering families}\label{sec:sheafification_locmfam_comorfismi}
In the present section we will focus instead on presenting a geometric morphism
\[
H\in \Topos/_1\Sh(\cbicat,J)(\Sh(\cbicat/A, J_A), \Sh(\fib P, J_P))\simeq\sheafify_J(P)(A)\]
`locally' by morphisms of fibrations: \ie, we will see that one can consider a $J$-covering family $\{g_u:D_u\rightarrow A\ |\ u\in U \}$ such that restricting $H$ to each of the toposes $\Sh(\cbicat/D_u, J_{D_u})$ results in a geometric morphism induced by a morphism of fibrations $\cbicat/D_u\rightarrow \fib P$. This point of view connects with the definition of the elements of $\sheafify_J(P)(A)$ as locally matching families of elements of $P$, which can be found for instance in {\cite[Proposition 2.19]{denseness}}. The argument will be rather technical and we will just provide a sketch of it, but the strategy goes as follows:
\begin{itemize}
	\item we will express the topos $\Sh(\fib P, J_P)$ as a colimit of local homeomorphisms of the kind $\Sh(\cbicat,J)/\ell_J(p(X,s))$: by pulling the colimit cocone back along $H$ we will obtain an expression of $\Sh(\cbicat/A, J_A)$ as a colimit of local homeomorphisms of the kind $\Sh(\cbicat,J)/B(X,s)$;
	\item each of the toposes $\Sh(\cbicat,J)/B(X,s)$ will in turn be expressed as a colimit of local homeomorphisms of the kind $\Sh(\cbicat,J)/\ell_J(Y_\alpha)$;
	\item finally, we show that each topos $\Sh(\cbicat,J)/\ell_J(Y_\alpha)$ is covered by toposes $\Sh(\cbicat,J)/\ell_J(D_u)$ such that the composite geometric morphisms $$\Sh(\cbicat,J)/\ell_J(D_u)\rightarrow \Sh(\cbicat,J)/\ell_J(A)$$ and $$\Sh(\cbicat,J)/\ell_J(D_u)\rightarrow\Sh(\cbicat,J)/\ell_J(p(X,s))$$ are induced by arrows $g_u:D_u\rightarrow A$ and $g'_u:D_u\rightarrow p(X,s)$, and thus by comorphisms of sites between the corresponding fibrations.
\end{itemize} 
Basically, the gist of this technique is to exploit the stability of colimits of étale toposes with respect to pullbacks, and the fact that by refining the cocone enough we can present at the level of sites. 

First of all, combining Proposition \ref{prop:fib_discreta_slice_topos} and Corollary \ref{cor:topos_discfib_colimite_conico} we obtain the following:
\begin{lemma}\label{lemma:colimite_per_homeoloc_topos}
	Consider a site $(\cbicat,J)$ and a presheaf $P:\cbicat\op\rightarrow\Set$: then the cocone
	\[
	\begin{tikzcd}
		{\Sh(\cbicat,J)/\ell_J(p(Y,P(y)(s)))} \ar[dr, "\prod_{\sheafify_J(\name{s}\yo(y))}"'] \ar[r, "\prod_{\ell_J(y)}"] & {\Sh(\cbicat,J)/\ell_J(p(X,s))} \ar[d, "\prod_{\sheafify_J(\name{s})}"]\\
		& {\Sh(\cbicat,J)/\sheafify_J(P)}
	\end{tikzcd}\]
	is a colimit cocone, where each $(X,s)$ is an object of $\fib P$ and $\name{s}:\yo(X)\rightarrow P$ is the arrow corresponding to $s\in P(X)$ by Yoneda's lemma.
\end{lemma}
Now, if we take a geometric morphism $H$ making the triangle
\[\begin{tikzcd}[column sep=0.3]
	{\Sh(\cbicat, J)/\ell_J(A)} \ar[rr, "H"] \ar[dr, "\prod_{\ell_J(A)}"'] && {\Sh(\cbicat, J)/\sheafify_J(P)} \ar[dl, "\prod_{\sheafify_J(P)}"]\\
	& {\Sh(\cbicat,J)}&
\end{tikzcd}\]
commutative, by Lemma \ref{lemma:morfgeom_tra_topos_slice_su_base} we have that $H=\prod_h$ for some $h:\ell_J(A)\rightarrow \sheafify_J(P)$. the inverse image $H^*$ corresponds to pulling back along $h$: in particular, we shall consider the pullbacks
\[
\begin{tikzcd}
	B(X,s)
	\ar[d, "b{(X,s)}"'] \ar[r, "c{(X,s)}"] & \ell_J(p(X,s)) \ar[d, "\sheafify_J(\name{s})"]\\
	\ell_J(A) \ar[r, "h"] & \sheafify_J(P) \ar[ul, phantom, "\lrcorner"near end]
\end{tikzcd}.\]
Pulling back the colimit in Lemma \ref{lemma:colimite_per_homeoloc_topos} along $H$, we obtain a colimit cocone
\[
\begin{tikzcd}
	{\Sh(\cbicat,J)/B(Y,P(y)(s))} \ar[dr, "\prod_{b(Y, P(y)(s))}"'] \ar[r] & {\Sh(\cbicat,J)/B(X,s)} \ar[d, "\prod_{b(X,s)}"]\\
	& {\Sh(\cbicat,J)/\ell_J(A)}
\end{tikzcd}:\]
this follows from the considerations at the end of Section \ref{sec:colimits_of_toposes}. Now, applying again Lemma \ref{lemma:colimite_per_homeoloc_topos} we have that each  of the toposes $\Sh(\cbicat,J)/B(X,s)$ can be expressed as a colimit whose legs are of the kind
\[
\Sh(\cbicat,J)/\ell_J(Y_\alpha)\xrightarrow{\prod_{\alpha}} \Sh(\cbicat,J)/B(X,s)
\]
for all arrows $\alpha:\ell_J(Y_\alpha)\rightarrow B(X,s)$; notice that by composing colimit cocones we obtain a jointly epic family of arrows
\[
\ell_J(Y)\xrightarrow{\alpha}B(X,s)\xrightarrow{b(X,s)} \ell_J(A)
.\]
Finally, let us recall the following result:
\begin{lemma}[{\cite[Proposition 2.5]{denseness}}]
	Consider a site $(\cbicat,J)$ and an arrow $\alpha:\ell_J(X)\rightarrow \ell_J(Y)$: there exist two families of arrows $\{f_u:D_u\rightarrow X \ |\ u\in U\}$ and $\{g_u:D_u\rightarrow Y\ |\ u\in U\}$ of $\cbicat$ such that $\{f_u\ |\ u\in U\}$ is $J$-covering, $\ell_J(g_u)=\alpha\ell_J(f_u)$ and for every span $h:W\rightarrow D_u$, $k:W\rightarrow D_v$ in $\cbicat$ such that $f_uh=f_vk$ then $\ell_J(g_uh)=\ell_J(g_vk)$.
\end{lemma}
This lemma states that every arrow between representables is `locally induced' at the level of the presentation site.
We can apply this lemma to the composite arrows
\[
\ell_J(Y)\xrightarrow{\alpha}B(X,s)\xrightarrow{b(X,s)} \ell_J(A),\ \ell_J(Y)\xrightarrow{\alpha}B(X,s)\xrightarrow{c(X,s)} \ell_J(p(X,s)):
\]
by refining enough the $J$-covering family $\{f_u:D_u\rightarrow Y_\alpha\}$ we conclude that there exist two families $\{g_u:D_u\rightarrow A\}$ and $\{g'_u:D_u\rightarrow X\}$, all indexed by a set $U$, such that $c(X,s)\alpha\ell_J(f_u)=\ell_J(g'_u)$ and $b(X,s)\alpha\ell_J(f_u)=\ell_J(g_u)$. Now one can check that the diagram
\[
\begin{tikzcd}
	{\Sh(\cbicat,J)/\ell_J(D_u)}\ar[d, "\prod_{\ell_J(f_u)}"'] \ar[ddr, "\prod_{\ell_J(g'_u)}"]
	\arrow[
	rounded corners, 
	to path={ 
		-- node[above]{\scriptsize $\prod_{\sheafify_J(\name{s}\yo(g'_u))}\simeq C_{\fib(\name{s}\yo(g'_u))}$} ([xshift=30ex]\tikztostart.east) 
		|- (\tikztotarget.east)}, anchor=center]{dddr}
	\arrow[
	rounded corners,
	to path={
		-- ([xshift=-2ex]\tikztostart.west)
		|- (\tikztotarget.west)}, anchor=center]{ddd}
	&
	\\
	{\Sh(\cbicat,J)/\ell_J(Y_\alpha)} \ar[d, "\prod_\alpha"', "\prod_{\ell_J(g_u)}"{xshift=-20ex}]&\\
	{\Sh(\cbicat,J)/B(X,s)} \ar[d, "\prod_{b(X,s)}"'] \ar[r, "\prod_{c(X,s)}"]& {\Sh(\cbicat,J)/\ell_J(p(X,s))} \ar[d, "\prod_{\sheafify_J(\name{s})}"]\\
	{\Sh(\cbicat,J)/\ell_J(A)} \ar[r, "H"']& {\Sh(\cbicat,J)/\sheafify_J(P)}
\end{tikzcd}
\] 
is commutative: hence for every arrow $g_u$ 
the composite $H\prod_{\ell_J(g_u)}$ is equivalent to the functor $C_{\fib (\name{s}\yo(g'_u)) }$, \ie it is presented by a comorphism of sites. Moreover, The family of all arrows
$\ell_J(g_u)$ is jointly epic (since the families of the arrows $b(X,s)$, $\alpha$ and $\ell_J(f_u)$ all are), and hence the family $\{g_u\}$ is $J$-covering.  We can thus sum up our conclusions as follows:
\begin{prop}
	Consider a site $(\cbicat,J)$, an object $A$ of $\cbicat$, a presheaf $P:\cbicat\op\rightarrow \Set$ and a geometric morphism $H$ making the diagram
	\[\begin{tikzcd}[column sep=0.3]
		{\Sh(\cbicat/A, J_A)} \ar[rr, "H"] \ar[dr, "C_{p_A}"'] && {\Sh(\fib P, J_P)} \ar[dl, "C_{p_P}"]\\
		& {\Sh(\cbicat,J)}&
	\end{tikzcd}\]
	commutative (up to equivalence): then there is a $J$-covering family $\{g_u:D_u\rightarrow A\ |\ u\in U \}$ and a family of morphisms of fibrations $\{G_u: \cbicat/D_u\rightarrow \fib P\ |\ u\in U\}$ such that the composite geometric morphism
	\[
	\textstyle \Sh(\cbicat/D_u, J_{D_u}) \xrightarrow{C_{\fib g_u}} \Sh(\cbicat/A, J_A)\xrightarrow{H} \Sh(\fib P, J_P)
	\]
	is equivalent to the geometric morphism
	\[
	\textstyle C_{G_u}:\Sh(\cbicat/D_u, J_{D_u})\rightarrow \Sh(\fib P, J_P).\]
\end{prop}
Of course, a geometric morphism $H$ admits multiple presentations in this way, since one can choose different families $\{f_u\}$, $\{g_u\}$ and $\{g'_u\}$ to present it at the level of sites. Notice that starting instead from the comorphisms $G_u$, one can recover $H$ quite easily: indeed, the geometric morphisms $C_{G_u}:\Sh(\cbicat/D_u, J_{D_u})\rightarrow \Sh(\fib P, J_P)$ form a cocone under the diagram indexed by the elements of the $J$-covering family $\{g_u\}$, which induces a geometric morphism from their colimit (which is $\Sh(\cbicat/A, J_A)$) to $\Sh(\fib P, J_P)$.

\subsection{Using comorphisms of sites to $[(\fib P)\op, \Set]$}
There is a third way of describing geometric morphisms  $H$ in the hom-category $\Topos^s/_1\Sh(\cbicat,J)(\Sh(\cbicat/X, J_X), \Sh(\fib P, J_P))$ that does not relie on locally matching conditions, but instead on $J$-equivalence. To do so, we first need to recall that any geometric morphism $H:\Sh(\cbicat/X, J_X)\rightarrow \Sh(\fib P, J_P)$ is a local homeomorphism and therefore is essential. Since both its domain and codomain are toposes over $\Sh(\cbicat,J)$ obtained from continuous comorphisms of sites, we may apply Proposition \ref{prop:classthm_essenziali_pshlifting} to obtain the following:
\begin{cor}\label{cor:sheafififcation_comorfismi_verso_prefasci}
	Consider a site $(\cbicat,J)$ and a presheaf $P:\cbicat\op\rightarrow \Set$. Denote by $p:\fib P\rightarrow \cbicat$ and $ p_X:\cbicat/X\rightarrow \cbicat$ the relevant fibrations. Then giving a geometric morphism $H$ such that the triangle
	\[
	\begin{tikzcd}[column sep=0.3]
		{\Sh(\cbicat/X, J_X)} \ar[rr, "H"] \ar[dr, "C_X"'] && {\Sh(\textstyle \fib P, J_P)} \ar[dl, "C_{p}"]\\
		& {\Sh(\cbicat,J)}&
	\end{tikzcd}
	\]
	is commutative (up to equivalence) is the same as giving a continuous comorphism of sites $B:(\cbicat/X, J_X)\rightarrow([(\fib P)\op, \Set], \widehat{J_P})$, up to $J_P$-equivalence, endowed with a natural transformation
	$\tau:B\Rightarrow p^*\yo_\cbicat p_X$ such that the composite
	\[\bar{\tau}:\lan_{p\op} B\xRightarrow{\lan_{p\op}
		\circ \tau}\lan_{p\op}p^*\yo_\cbicat p_X \xRightarrow{\epsilon'\circ \yo_\cbicat p_X}\yo_\cbicat p_X\]
	is sent by $\sheafify_J$ to an isomorphism, where $\epsilon'$ is the counit of $\lan_{p\op}\dashv p^*$. In particular, for any $X$ in $\cbicat$ it follows that 
	$\sheafify_J(P)(X)$ is isomorphic to the set of pairs $(B, \tau)$ as above, where $B$ is chosen up to $J$-equivalence.
\end{cor} 
\begin{remark}
	Recall that for $W:(\fib P)\op\rightarrow\Set$, $(Y,U)$ in $\fib P$ and $Z$ in $\cbicat$ then $\lan_{q\op}(W)(Z):=\colim_{\phi:Z\rightarrow q(Y,U)} W(Y,U)$. We can describe the natural transformation $\bar{\tau}:=(\epsilon'\circ \yo_\cbicat p_X)(\lan_{q\op}\circ \tau)$ componentwise as follows: first, set $[w:W\rightarrow X]$ in $\cbicat/X$, so that the component $\bar{\tau}([w])$ is an arrow
	$\lan_{q\op}B([w])\rightarrow \yo_\cbicat(W)$ of $[\cbicat\op,\Set]$; then notice that its component at $Z$ in $\cbicat$ is the unique arrow
	\[
	\colim_{\phi:Z\rightarrow q(Y,U)} B([w])(Y,U) \rightarrow \colim_{\phi:Z\rightarrow q(Y,U)} \cbicat(Y,W) \rightarrow \cbicat(Z,W)
	\]
	induced by colimit property from the cocone whose leg indexed by $\phi:Z\rightarrow q(Y,U)$ is the arrow
	\[
	B([w])(Y,U)\xrightarrow{\tau([w])(Y,U)}\cbicat(Y,W)\xrightarrow{-\circ \phi}\cbicat(Z,W).\]
\end{remark}

\subsection{The various forms of sheafification for a presheaf over a topological space}
Consider a topological space $X$ and a presheaf $P:\Ocal(X)\op\rightarrow\Set$. The content of the previous sections provides us with various possible ways of describing the sheafification $\sheafify(P)$: indeed, for $U$ in $\Ocal(X)$, we have
\begin{align}
	\sheafify(P)(U)&=\Top/X(U, E_P)\\
	&\simeq \Locale/\Ocal(X) (\Ocal(U), \Ocal(E_P))\\
	&\simeq(\Ocal(X), J\can_{\Ocal(X)})/\Site((\Ocal(E_P), J\can_{\Ocal(E_P)}), (\Ocal(U), J\can_{\Ocal(U)}))\\
	&\simeq\Cosite\cont/(\Ocal(X), J\can_{\Ocal(X)})((\Ocal(U), J\can_{\Ocal(U)}),(\Ocal(E_P), J\can_{\Ocal(E_P)}))\\
	&=\Topos^s/\Sh(X) (\Sh(U), \Sh(E_P))\\
	&\simeq \Topos^s/\Sh(X)(\Sh(U), \Sh(\fib P, J_P))\tag{8.5*}\\
	&\simeq\Site^\dagger ((\fib P, J_P), (\Ocal(U), J\can_{\Ocal(U)}))\\
	&\simeq\Cosite\cont\JPeqv ( (\Ocal(U), J\can_{\Ocal(U)}), ([(\fib P)\op,\Set], \widehat{J_P} ))/^\bullet p^*\yo_{\Ocal(X)} p_U\\
	&\simeq \mbox{Locally matching families of } \Fib_{\Ocal(X)}(\Ocal(-), \fib P)
\end{align} 
where the symbol $^\dagger$ denotes the morphisms of sites $$A:(\fib P, J_P)\rightarrow (\Ocal(U), J\can_{\Ocal(U)})$$ satisfying the condition
\begin{itemize}
	\item[($^\dagger$)] there is a natural transformation $A\Rightarrow i_U\inv p$, where $i_U\inv:=(-\cap U):\Ocal(X)\rightarrow \Ocal(U)$ and $p:\fib P\rightarrow \Ocal(X)$, whose induced natural transformation ${\phi}:\Sh(A)^*\sheafify_{J_P}\Rightarrow \Sh(i_U)^*\sheafify_{J\can_{\Ocal(X)}}\lan_{p\op}$ at the level of geometric morphism is such that the composite
	\[\begin{tikzcd}
		{\Sh(A)^*\sheafify_{J_P}p^*}
		\ar[d, Rightarrow, "{\phi}\circ p^*"]\\
		{\Sh(i_U)^*\sheafify_{J\can_{\Ocal(X)}}\lan_{p\op}p^*}
		\ar[d, Rightarrow, "{\Sh(i_U)^*\sheafify_{J\can_{\Ocal(X)}}\circ \epsilon}"]\\
		\Sh(i_U)^*\sheafify_{J\can_{\Ocal(X)}}{}
	\end{tikzcd}\]
	is invertible (where $\epsilon$ is the counit of $\lan_{p\op}\dashv p^*$) (this results from Proposition \ref{prop:fascificazione_morfsiti}).
\end{itemize} 
On the other hand, the symbol $^\bullet$ states that we consider $J_P$-equivalence classes of continuous comorphisms of sites $B:\Ocal(U)\rightarrow [(\fib P)\op,\Set]$ with a natural transformation $\tau:B\Rightarrow p^*\yo_{\Ocal(X)}p_U$ satisfying the requirement
\begin{itemize}
	\item[($^\bullet$)]: the composite
	\[\lan_{p\op} B\xRightarrow{\lan_{p\op}
		\circ \tau}\lan_{p\op}p^*\yo_\cbicat p_X \xRightarrow{\epsilon'\circ \yo_\cbicat p_X}\yo_\cbicat p_X\]
	is sent by $\sheafify_J$ to an isomorphism, where $\epsilon'$ is the counit of $\lan_{p\op}\dashv p^*$ (see Corollary \ref{cor:sheafififcation_comorfismi_verso_prefasci}).
\end{itemize}
Let us recapitulate how we can go from one form to the other. 

The connection between the first four items is quite easy. We already know how to go from (1) to (2) and (3), by mapping a section $s$ to $s\inv$; the viceversa is the content of Proposition \ref{prop:fascificazione_topologica_con_framehom}; instead, mapping $s$ to $s_!=s(-):\Ocal(U)\rightarrow \Ocal(E_P)$ provides a way to go from (1) to (4). Starting instead from a continuous comorphism of sites in (4), its right adjoint is a morphism of sites, and so we go back to (3).

By applying either $\Sh(-)$ or $C_{(-)}$ we can go from (1), (2), (3) and (4) to (5): conversely, one can easily go back from (5) to (2) by restricting a geometric morphism $\Sh(U)\rightarrow \Sh(E_P)$ to the locales of subterminal objects. 

One can go from (5) to (5$^*$) and viceversa since $\Sh(\fib P, J_P)\simeq \Sh(E_P)$ (see Proposition \ref{prop:fasci__etalebundle_eqv_fasci_grfibration}).

One can go from (5$^*$) to (6) simply by restricting the inverse image $H^*:\Sh(\fib P, J_P)\rightarrow \Sh(U)$ to representables, and viceversa a morphism of sites $\fib P\rightarrow \Sh(U)$ induces a geometric morphism $\Sh(U)\rightarrow \Sh(\fib P, J_P)$. 

The connection between (5$^*$) and (7) is the content of Corollary \ref{cor:sheafififcation_comorfismi_verso_prefasci}, while the connection between (5$^*$) and (8) was sketched in Section \ref{sec:sheafification_locmfam_comorfismi}.

To conclude, let us show the explicit connection between the first items and the last items, without going through the categories of geometric morphisms.

If we start from (1), \ie from a section $s:U\rightarrow E_P$, it is easy to verify that the corresponding morphism of sites
\[
A_s:\fib P \rightarrow \Ocal(U)\]
is defined by mapping any object $(V,t)$ of $\fib P$ to the open $s\inv (\dot{t}(V))\subseteq U$. The inclusions
\[
s\inv (\dot{t}(V))=\{x\in U\cap V\ |\ s(x)=t_x  \}\subseteq U\cap V
\]
provide the components of a natural transformation $A_s\Rightarrow i_U\inv p$ such that satisfying the condition ($^\dagger$): thus we have obtained an element of (6). Viceversa, consider such a morphism of sites
\[
A:(\fib P, J_P)\rightarrow (\Ocal(U), J\can_{\Ocal(U)})
\] 
satisfying the condition ($^\dagger$): it is a matter of computation to see that we can define a homomorphism of frames $f_A:\Ocal(E_P)\rightarrow \Ocal(U)$ by setting the image of the basic opens as
\[
f_A(\dot{t}(V)):=\bigcup\{W\subseteq U\ |\ W\subseteq A(V,t) \}
\]
and requiring that $f_A$ preserve arbitrary unions. This allows us to go from  (6) to (2).

Finally, let us connect (1) and (7). Starting from a section $s:U\rightarrow E_P$, first of all we consider the continuous comorphism of sites $s(-):\Ocal(U)\rightarrow \Ocal(E_P)\simeq \Id_{J_P}(\fib P)$ of item (4): by considering the corresponding geometric morphism $\Sh(U)\rightarrow \Sh(\Id_{J_P}(\fib P))$ and applying to it Proposition \ref{prop:eqv_essgeomorf_comcont_pshlifting} we obtain that $s$ corresponds to the continuous comorphism of sites 
\[
\Ocal(U) \xrightarrow{R} \Sh(\fib P, J_P)\hookrightarrow[(\fib P)\op,\Set],
\]
where in particular $R$ maps an open $V\subseteq U$ to the union of subterminals
\[
\bigcup_{\substack{W\subseteq V,\\
		r\in P(W)\\ \mbox{\scriptsize\st}\, s_W=\dot{r} }}\ell_{J_P}(W,r).
\]

Conversely, start from a continuous comorphism of sites $B:\Ocal(U)\rightarrow [(\fib P)\op,\Set]$ as in item (7): it induces a geometric morphism $F:\Sh(U)\rightarrow \Sh(\fib P, J_P)$, whose inverse image acts by mapping a $J_P$-sheaf $H:(\fib P)\op\rightarrow \Set$ to the $J\can_{\Ocal(U)}$-sheaf
\[
F^*(H):\Ocal(U)\op\rightarrow\Set,\ F^*(H)(V):=[(\fib P)\op, \Set](B(V),H).
\]
By restricting to subterminals we obtain a frame homomorphism $\Id_{J_P}(\fib P)\simeq \Ocal(E_P)\rightarrow \Ocal(U)$, \ie an element of (2) and (3).
\chapter{The categorical \emph{petit/gros} topos construction}

One of the key ingredients of a categorical approach to mathematics is the idea that the study of an object $X$ in an environment $\cbicat$ is essentially the study of all the possible points of view that the environment allows for that object: that is, of the representable presheaf $\cbicat(-,X)$. This `Yoneda paradigm' is justified by the fact that embedding $\cbicat$ into its topos of presheaves allows us to perform many categorical constructions that the base category may lack, such as limits or colimits. Of course, particular needs may require toposes other than the topos of presheaves over the base category. When working with a topological space $X$, the concept of \emph{gros topos}\index{topos!gros} $TOP(X)$ fulfills exactly that need: the category $\Top/X$ of bundles over $X$ embeds into it, and moreover all all the relevant topos-theoretic invariants of $X$ (namely its cohomology) are preserved. This latter fact is a consequence of the fact that the \emph{petit topos}\index{topos!petit} of $X$, \ie the topos of sheaves $\Sh(X)$, is in fact a retract of the big topos $TOP(X)$, as shown in Section 4.10 of \cite[Exposé IV]{SGA4_I}. Section 2.5 \textit{ibid.}, which is dedicated to the introduction of the \textit{gros site} and the \textit{gros topos}, lists them among the many contributions of J. Giraud to topos theory, and then goes forth:
\begin{quote}
	\emph{L’avantage du gros topos de $X$ sur le petit, c’est que le site qui le définit contient}
	$\Top/X$
	\emph{comme sous-catégorie pleine [...] Par
	suite, un espace $X'$ sur $X$ est connu à $X$-isomorphisme près quand on connaît le faisceau (}$\in TOP(X)$\emph{) qu’il définit; donc la notion de faisceau sur (le gros site de) $X$ peut être considéré comme une généralisation de celle d’espace topologique au-dessus de $X$, à l’aide de laquelle toutes les constructions de la théorie des faisceaux prennent un sens
	pour les espaces topologiques sur $X$.}\footnote{\textit{The advantage of the big topos of $X$ over the small topos is that the site defining it contains $\Top/X$ as a full subcategory [...] It follows that a space $X'$ over $X$ is known up to $X$-isomorphism whenever we know the sheaf ($\in TOP(X)$) which it defines; thus the notion of sheaf on (the big site of) $X$ can be considered as a generalization of that of topological space over $X$, thanks to which all the constructions in the theory of sheaves acquire meaning for topological spaces over $X$}.}
\end{quote}
In the following sections, we shall recall the basic results about the dichotomy \emph{petit-gros topos}, both in topology and algebra; after that we shall prove that \emph{every} Grothendieck topos can be naturally regarded as a small topos embedded in an associated (very) big topos, and that this embedding allows one to view any object of the original topos as an étale morphism to the terminal object in the associated big topos. This seems to realize Grothendieck's aspiration of viewing any object of a topos geometrically as an étale space over the terminal object, as expressed in his 1973 Buffalo lectures and brought to the public attention by Colin McLarty in his recent talks (for instance in \cite{mclarty.seminar}):
\begin{quote}
	\emph{The intuition is the following: viewing objects of a topos as being something like étalé spaces over the final object of the topos, and the induced topos over an object as just the object itself. That is I think the way one should handle the situation.}
	
	\emph{It's a funny situation because in strict terms, you see, the language which I want to push through doesn't make sense. But of course there are a number of mathematical statements which substantiate it.}	
\end{quote}
Indeed, in Grothendieck's work, the toposes that are most naturally thought as generalized spaces arise as `petit' toposes, connected to `gros' toposes by essential morphisms and local morphisms defining a retraction.

\section{The petit-gros topos dichotomy in topology and algebra}

We dedicate this section to recalling how the the \ac small' and \ac big' toposes of a topological space and of a ring are obtained, as guide examples. 

The literature on the subject is both sterminate and sparse, therefore only the very fundamental results will be recalled, without any claim of originality or exhaustiveness.

Let us start from the case of topological spaces: we refer to \cite[Exposé IV, Sections 2.5 and 4.10]{SGA4_I}, and \cite[pagg. 415-416]{maclanemoerdijk}. We have already said that the \emph{small topos} of $X$ is the topos
\[
\Sh(X):=\Sh(\Ocal(X), J\can_{\Ocal(X)}),\]
\ie the topos of sheaves over the frame $\Ocal(X)$ endowed with its canonical topology:
for an open subset $U$ of $X$, a covering family of $U$ is a family of opens $\{U_i\ |\ i\in I\}$ such that $\bigcup_{i\in I}U_i=U$ (see also Section \ref{sec:adj_topologica}). 

We can define a similar Grothendieck topology on any suitably small category of topological spaces. Given a universe $\ubicat$, consider the category $\Top_\ubicat$ of $\ubicat$-small topological spaces (see Appendix \ref{app:universes}). For every $\ubicat$-small topological space $X$, we consider covering any $\ubicat$-small family of open embeddings $\{U_i\hookrightarrow X\ |\ i\in I\}$ which is jointly surjective, and this provides a topology $J$ over $\Top_\ubicat$: this is called the \emph{open cover topology}\index{topology!open cover} (see for instance \cite[Chapter III, §2]{maclanemoerdijk}). We can then consider the topology $J_X$ induced by $J$ on $\Top_\ubicat/X$ along the fibration $\Top_\ubicat/X\rightarrow \Top_\ubicat$, in the sense of Proposition \ref{prop:top_minima_comorfismo}: the site $(\Top_\ubicat/X, J_X)$ is evidently neither $\ubicat$-small nor a $\ubicat$-site in principle, and this is why it is called the \emph{big site of $X$}. If we want to compute its topos of sheaves, we can avoid size issues either by choosing a suitably small subcategory of $\Top_\ubicat/X$ or by assuming the existence of a bigger universe $\vbicat\ni\ubicat$ such that the big site of $X$ is a $\vbicat$-site (we will exploit this latter point of view). We end up with two different toposes related to $X$:
\[
\Sh(X),\ \Sh_\vbicat(\Top_\ubicat/X, J_X).
\]  
The latter is the \emph{big topos of $X$}. We remark that the topology $J_X$ is subcanonical, and therefore the category of topological spaces $\Top_\ubicat/X$ embeds fully faithfully into $\Sh(\Top_\ubicat/X, J_X)$: this is the technical motivation behind the claim in the introduction to this chapter that the sheaves over the big site of $X$ generalize bundles over $X$. Moreover, the construction of big toposes for topological spaces is functorial: indeed, a map $f:X\rightarrow Y$ of $\ubicat$-topological spaces induces a functor $(f\circ-):\Top_\ubicat/X\rightarrow \Top_\ubicat/Y$, which is a comorphism of sites and thus induces a geometric morphism $\Sh_\vbicat(\Top_\ubicat/X, J_X)\rightarrow \Sh_\vbicat(\Top_\ubicat/Y, J_Y)$ (see \cite[Exposé IV, § 4.1.3]{SGA4_I}). Finally, if we consider the fully faithful functor
\[
i_X:\Ocal(X)\hookrightarrow \Top_\ubicat/X
\]
which maps every $U\subseteq X$ to the open immersion $U\hookrightarrow X$, one can easily verify that it is both a morphism and a comorphism of sites, when $\Ocal(X)$ is endowed with the canonical topology: therefore, it induces a pair of geometric morphisms of $\vbicat$-toposes
\[
Sh(X)\xrightarrow{C_{i_X}}\Sh_\vbicat(\Top_\ubicat/X, J_X),\ \Sh_\vbicat(\Top_\ubicat/X, J_X)\xrightarrow{\Sh(i_X)}\Sh(X),
\]
which satisfy $C_{i_X}^*=\Sh(i_X)_*$. The fact that $i_X$ is fully faithful implies that $C_{i_X}$ is an embedding, and hence $\Sh(X)$ is a retract of $\Sh_\vbicat(\Top_\ubicat/X, J_X)$, with $C_{i_X}$ being the section and $\Sh(i_X)$ the retraction. The sheaves in the image of $(C_{i_X})_*$ are usually called the \emph{big étale sheaves over $X$}, since they correspond to the étale bundles over $X$.

Another classical example of a {petit-gros} topos situation is provided by the Zariski topos of a ring. Starting with a commutative ring $A$,
we consider its spectrum
\(
\Spec(A)=\{\pfrak\subseteq A\ |\ \pfrak\mbox{ prime ideal} \},
\)
which is a topological space when endowed with the \emph{Zariski topology}\index{topology!Zariski} defined by the base of opens $D(a):=\{\pfrak\in \Spec(A)\ |\ a\notin \pfrak\}$ where $a$ varies in $A$. The \emph{small Zariski topos of $A$}\index{topos!small Zariski} is then defined as the topos of sheaves $\Sh(\Spec(A))$ over the spectrum of the ring.
To build the big site, we move to the category $\Ring\fp$ of finitely presented commutative rings, which will be the algebraic counterpart of $\Top_\ubicat$: to be more precise, we need to work with $\Ring\fp\op$, since it is the category holding the geometric information\footnote{This is one of the many instances of the common mantra that `geometry is algebra$\op$'.}. The category $\Ring\fp$ can easily be endowed with a Grothendieck co-topology defined by the co-basis of co-covering families of the form
\[
\{B\rightarrow B[s\inv]\ |\ s\in S\},
\]
where $S\subseteq B$ is any collection of elements of $B$ such that the generated ideal $\langle S\rangle$ coincides with $B$: we obtain therefore a topology on $\Ring\fp\op$. By considering the slice $A/\Ring\fp$, which is nothing but the category of finitely presented $A$-algebras, we can induce a co-topology on it, and hence a topology on the category on $(A/\Ring\fp)\op\simeq \Ring\fp\op/A$. Said topology is denoted by $\zbicat_A$ and called the \emph{Zariski topology on the big site of $A$}. The \emph{big Zariski topos of the ring $A$}\index{topos!big Zariski} is defined as the topos of sheaves
\[
\Sh(\Ring\op\fp/A,\zbicat_A).
\]
Finally, we compare the two toposes. Again, we consider a functor
\[
i_A:\Ocal(\Spec(A))\rightarrow (\Ring\fp/A)\op 
\]
defined on basic opens by mapping $D(a)$ to $A\rightarrow A[a\inv]$: one shows that said functor is a full and faithful morphism and comorphism of sites, and thus it induces a section
\[C_{i_A}:\Sh(\Spec(A))\rightarrow \Sh((\Ring\fp/A)\op,\zbicat_A).\]
Thus the ring-theoretic scenario essentially mimics the topological one. We mention in passing by that \cite{anel.spectra} provides a very general framework for the construction of small and big toposes stemming from factorization systems, which applies to most algebraic examples of the big-small topos dichotomy (including rings). It is also interesting to remark that the Zariski topology can be retrieved by the open cover topology in a canonical way, as the following results show. 
\begin{lemma}
	Consider a homomorphism of rings $f:A\rightarrow B$: it induces a frame homomorphism
	\[
	\bar{f}:=\Ocal(f\inv):\Ocal(\Spec(A))\rightarrow \Ocal(\Spec(B))
	\] 
	between topologies, defined as $\bar{f}(U):=\{\qfrak\ |\ f\inv(\qfrak)\in U \}$. Suppose that there exists $a\in A$ such that $\bar{f}$ factors via $D(a)\cap-:\Ocal(\Spec(A))\rightarrow D(a)\downarrow=\Ocal(\Spec(A[a\inv]))$, as in the diagram
	\[
	\begin{tikzcd}
		A \ar[d, "f"'] \ar[r] & A[a\inv] \ar[dl, "\beta", dashed] \\B&
	\end{tikzcd}\mapsto
	\begin{tikzcd}
		\Ocal(\Spec(A))\ar[d, "\bar{f}"'] \ar[r, "D(a)\cap-"] & D(a)\downarrow \ar[dl, "\alpha"]\\ \Ocal(\Spec(B))&
	\end{tikzcd}:
	\]
	then there exists a unique $\beta$ lifting the arrow $\alpha$.	
\end{lemma}
\begin{proof}
	A quick computation shows that 
	\[
	\bar{f}(D(a))=D(f(a))
	\]
	and thus 
	\[
	D(f(a)) =\alpha(D(a)) =\alpha(\top_{D(a)\downarrow}) =\top_{\Ocal(\Spec(B))} =\Spec(B).
	\]
	This implies that $f(a)$ is invertible in $B$, and thus there exists a unique homomorphism $\beta$ as above.
\end{proof}

\begin{prop}
	Consider the functor $\Spec:\Ring\op\rightarrow \Top$ and endow $\Top$ with the open cover topology, denoted $T$: then the Zariski topology $Z$ on $\Ring\op$ coincides with the smallest topology $M^{\Spec}_T$ making $\Spec$ into a comorphism of sites.
\end{prop}
\begin{proof}
	Preliminarily, consider a $T$-covering family $\{U_i\hookrightarrow\Spec(A) \}$: without loss of generality every $U_i$ can be set to be of the form $D(a_i)$ and hence the covering family is the image of the family $\{A\rightarrow A[a_i\inv]\ |\ i\in I\}$, which is $Z$-covering since $\Spec(A)=\cup D(a_i)$ implies that $\langle a_i\rangle=A$. This shows that 
	\[
	\Spec:(\Ring\op, Z)\rightarrow (\Top, T)
	\] 
	is a comorphism of sites and thus $M^{\Spec}_T\subseteq Z$.
	
	Conversely, if we start from a $Z$-covering family $\{A\rightarrow A[a_i\inv]\ | \ i\in I,\ \langle a_i\rangle =A\}$, we want to show that the sieve $S$ it generates belongs to $\Spec^+T$. Notice that its image via $\Spec$ provides the open-covering family $\{D(a_i)\hookrightarrow\Spec(A)\ |\ i\in I\}$. We can then consider the sieve
	\[
	R=\{f:A\rightarrow B\ |\  f\inv:\Spec(B)\rightarrow \Spec(A)\mbox{ factors via some } D(a_i) \}:
	\]
	it is $M^{\Spec}_T$-covering, and it is in fact one of the sieves in the coverage that defines $M^{\Spec}_T$ (see Proposition \ref{prop:top_minima_comorfismo}). Now, if $f\inv:\Spec(B)\rightarrow \Spec(A)$ factors via some $D(a_i)$ then the arrow $\Ocal(f\inv)=\bar{f}$ factors via $D(a)\downarrow$: by the previous lemma, $f$ factors through $A\rightarrow A[a_i\inv]$ and thus it belongs to $S$. We have $R\subseteq S$ and therefore $S$ belongs to ${\Spec}^+T$. 
\end{proof}
Both the topological and the ring-theoretic petit-gros topos situation are encompassed by Exercise 4.10.6 in \cite[Exposé IV]{SGA4_I}, which consider any subcanonical site $(\cbicat, J)$ with pullbacks with a class $\mbicat\subseteq Mor(\cbicat)$ of morphisms satisfying the following properties:
\begin{enumerate}
	\item the property `$f\in\mbicat$' is stable under pullback;
	\item $\mbicat$ is closed for identities and compositions;
	\item given $y:Y\rightarrow X$ in $\cbicat$, if there exists a $J$-covering family $\{X_i\rightarrow X \}$ such that every pullback $Y\times_X X_i\rightarrow X_i$ belongs to $\mbicat$, then also $y\in\mbicat$;
	\item Every $J$-sieve $S$ is refined by a $J$-covering family of arrows in $\mbicat$: that is, there exist a $J$-covering family $\{y_i:Y_i\rightarrow X\ |\ i\in I\}\subseteq S$  such that each $y_i$ belongs to $\mbicat$.
\end{enumerate}
For every object $X$ of $\cbicat$, let us denote by $\mbicat/X$ the full subcategory of $\cbicat/X$ whose objects are morphisms in $\mbicat$: then we consider the Giraud topology $J_X$ on $\cbicat/X$, which induces a biggest topology $T_X$ over $\mbicat/X$ making the inclusion
\[
i_X:\mbicat/X\hookrightarrow \cbicat/X
\]
$(T_X,J_X)$-continuous. One can then show that $i_X$ is a morphism and comorphism of sites, and thus that it induces a section of toposes \[
\Sh(\mbicat/X, T_X)\hookrightarrow \Sh(\cbicat/X, J_X)
\]
The essential point in this framework is the existence of a class of `distinguished morphisms' $\mbicat$: these morphisms are the open embeddings in the topological case, and the localizations in the ring-theoretic case. One can be even more general by dropping the hypothesis that we work in a full and faithful subcategory of $\cbicat/X$:
\begin{prop}[{\cite[Theorem 7.20]{denseness}}] \label{prop:petitgros_monografia}
	Consider a site $(\cbicat,J)$ and an object $X$ of $\cbicat$. Suppose that there exist a site $(\dbicat_X,T_X)$ and a morphism and comorphism of sites
	\[
	i_X:(\dbicat_X, T_X)\rightarrow (\cbicat/X, J_X)
	\]
	which is $K$-full and $K$-faithful (see \cite[Definition 5.14]{denseness}): then it induces a section of geometric toposes
	\[
	C_{i_X}:\Sh(\dbicat_X, T_X)\rightarrow \Sh(\cbicat/X, J_X)
	\]
	with retraction the geometric morphism $\Sh(i_X)$.
\end{prop}

\section{Every Grothendieck topos is a \ac small topos'}

We define a Grothendieck topology $J^{\textup{ét}}$\index{$J^{\textup{ét}}$} on $\Topos$, which we call the \emph{étale cover topology}\index{topology!étale cover}, by postulating that a sieve on a topos $\Etopos$ is $J^{\textup{ét}}$-covering if and only if it contains a family $\{\Etopos\slash A_{i} \to \Etopos \mid i\in I\}$ of canonical local homeomorphisms such that the family of arrows $\{!_{A_{i}}:A_{i}\to 1_\Etopos \mid i\in I\}$ is epimorphic in $\Etopos$. We should thus have a `big' topos $\Sh(\Topos, J^{\textup{ét}})$ with a canonical functor $\ell:\Topos\to \Sh(\Topos, J^{\textup{ét}})$, and for any Grothendieck topos $\Etopos$ we can consider the slice topos
\[
\Sh(\Topos, J^{\textup{ét}})\slash \ell(\Etopos)\simeq \Sh(\Topos\slash \Etopos, J^{\textup{ét}}_{\Etopos}),
\]  
where $J^{\textup{ét}}_{\Etopos}$\index{$J^{\textup{ét}}_{\Etopos}$} is the Grothendieck topology whose covering sieves are those which are sent by the forgetful functor $\Topos\slash {\Etopos}\to \Topos$ to $J^{\textup{ét}}$-covering families. We call this topos the \emph{big topos} associated with $\Etopos$. Note that the étale cover topology represents a natural topos-theoretic counterpart of the open-cover topology on the category of topological spaces. 

Now consider in particular a site $(\cbicat,J)$ of definition for $\Etopos$, and denote by $L$ the composite of the two canonical functors
\[
\cbicat\xrightarrow{\ell_J} \Sh(\cbicat,J) \xrightarrow{\ell} \Topos/\Sh(\cbicat,J): 
\]
then $L$ is fully faithful, and when $\Topos/\Sh(\cbicat,J)$ is endowed with the topology $J^{\textup{ét}}_{\Sh(\cbicat,J)}$ then $L$ is both a morphism and a comorphism of sites. By Proposition \ref{prop:petitgros_monografia}, it follows that the `petit' topos $\Sh({\cal C}, J)$ is a coadjoint retract of the `big' topos $\Sh(\Topos\slash \Sh({\cal C}, J), J^{\textup{ét}}_{\Sh({\cal C}, J)})\simeq \Sh(\Topos, J^{\textup{ét}})\slash \ell(\Sh({\cal C}, J))$ via the geometric morphisms 
\[
C_{L}:\Sh({\cal C}, J) \to \Sh(\Topos\slash \Sh({\cal C}, J), J^{\textup{ét}}_{\Sh({\cal C}, J)})
\] 
induced by $L$ as a comorphism of sites and 
\[
\Sh(L): \Sh(\Topos\slash \Sh({\cal C}, J), J^{\textup{ét}}_{\Sh({\cal C}, J)}) \to \Sh({\cal C}, J)
\] 
induced by $L$ as a morphism of sites; moreover, $\Sh(L)$ is local and $C_{L}$ is an essential inclusion.  

There are some size issues concerning the (very) `big' site $$(\Topos\slash \Sh({\cal C}, J), J^{\textup{ét}}_{\Sh({\cal C}, J)}):$$ indeed, this is site is \emph{not} small-generated. To fix this problem, we can resort to two possibilities:

\begin{itemize}
	\item We replace the site $(\Topos\slash \Sh({\cal C}, J), J^{\textup{ét}}_{\Sh({\cal C}, J)})$ with a suitable small-generated site containing as objects the étale morphisms. 
	\item We change the Grothendieck universe with respect to which considering sheaves. This is a sensible choice as it does not affect the properties of the geometric morphisms under consideration to be local or essential (see Appendix \ref{app:universes}).
\end{itemize}  
\chapter{Relative toposes, relative sites}\label{chap:relative_sites}

In this chapter we introduce the notion of \emph{relative site} (with respect to a small-generated site $({\cal C}, J)$), and the corresponding notion of \emph{relative topos}, that is, of topos of sheaves on a relative site  (endowed with a structure geometric morphism towards the base topos).

We first define the relative analogue of presheaf toposes, as the Giraud toposes associated with an arbitrary indexed category. Then we define relative sheaf toposes as the subtoposes of such toposes, which can be represented as toposes of sheaves on a relative site.
In this context, we introduce the notion of a relative topology generated by horizontal and vertical data (which we call an \emph{orthogonally generated} topology), and give a number of examples of such topologies naturally arising in the mathematical practice.

Then we introduce the notion of \emph{relative site of a geometric morphism}, and more generally of relative site of a morphism of sites. This allows us to prove that every geometric morphism towards a topos ${\cal E}$ is equivalent to the structure morphism of a relative topos over $\cal E$.  

\section{Relative presheaf toposes}

As observed in Remark \ref{remrelativepresheaves}(ii), the equivalence
\[
\Gir_J(\dcat) \simeq \Ind_\cbicat(\dcat\Vop, \canst_{(\cbicat,J)})
\]
of Corollary \ref{correlativepresheaves} shows that, given a small-generated site $({\cal C}, J)$ and a $\cal C$-indexed category ${\mathbb D}$, the relative topos $\Gir_J(\dcat)$ yields the appropriate notion of \ac topos of $\Sh({\cal C}, J)$-valued presheaves on $\mathbb D$'. This motivates the following definition.

\begin{defn}
Let $({\cal C}, J)$ be a small-generated site. A \emph{relative presheaf topos}\index{topos!relative presheaf -} over $\Sh({\cal C}, J)$ is a topos of the form $C_{p_{\mathbb D}}:\Gir_J(\dcat)=\Sh({\cal G}({\mathbb D}), J_{\mathbb D})\to \Sh({\cal C}, J)$, where $\mathbb D$ is a $\cal C$-indexed category.
\end{defn}

\begin{remark}
In the interest of maximal generality, we do not require $\mathbb D$ to be a stack on $({\cal C}, J)$, nor to be the $\cal C$-indexing of an internal category in $\Sh({\cal C}, J)$. In fact, indexed categories simultaneously generalize stacks and internal categories (on the other hand, not every internal category is a stack).
\end{remark}

\section{Relative sheaf toposes}

As any Grothendieck topos is a subtopos of a presheaf topos, so we define relative toposes as subtoposes of relative presheaf toposes. 
Recall that, for any small-generated site $({\cal C}, J)$, the subtoposes of the topos $\Sh({\cal C}, J)$ correspond precisely to the Grothendieck topologies $J'$ on $\cal C$ such that $J \subseteq  J'$. This leads us to the notion of \emph{relative site}: 

\begin{defn}
Let $({\cal C}, J)$ be a small-generated site. A \emph{relative site}\index{site!relative -} over $({\cal C}, J)$ is a site of the form $({\cal G}({\mathbb D}), J')$, where ${\mathbb D}$ is a $\cal C$-indexed category and $J'$ is a Grothendieck topology on ${\cal G}({\mathbb D})$ containing the Giraud topology $J_{\mathbb D}$. 

Any relative site $({\cal G}({\mathbb D}), J')$ is endowed with the structure comorphism of sites $p_{\mathbb D}:({\cal G}({\mathbb D}), J') \to ({\cal C}, J)$.
\end{defn}

Trivial relative sites are those such that the Grothendieck topology $J'$ coincides with Giraud's topology $J_{\mathbb D}$; as in the classical setting, they yield relative presheaf toposes. Accordingly, arbitrary relative sites yield arbitrary subtoposes of relative presheaf toposes:

\begin{defn}
Let $({\cal C}, J)$ be a small-generated site. A \emph{relative topos}\index{topos! relative sheaf -} over $\Sh({\cal C}, J)$ is a Grothendieck topos $\cal E$, together with a geometric morphism $p:{\cal E}\to \Sh({\cal C}, J)$.
\end{defn}

The following result shows that we could have alternatively defined relative toposes as the toposes of sheaves on a relative site:

\begin{thm}
Let $({\cal C}, J)$ be a small-generated site. Then any relative site over $({\cal C}, J)$ yields a relative topos over $\Sh({\cal C}, J)$; more precisely, any relative site $$p_{{\mathbb D}}:({\cal G}({\mathbb D}), J')\to ({\cal C}, J)$$ induces the relative topos
\[
C_{p_{\mathbb D}}:\Sh({\cal G}({\mathbb D}), J')\to \Sh({\cal C}, J).
\]
Conversely, any relative topos $p:{\cal E}\to \Sh({\cal C}, J)$ is of the form $C_{p_{\mathbb D}}$ for some relative site $p_{{\mathbb D}}:({\cal G}({\mathbb D}), J')\to ({\cal C}, J)$ (for instance, one can take $p_{{\mathbb D}}$ to be the canonical relative site of $p$, in the sense of Definition \ref{defrelativesiteofamorphism} - see Theorem \ref{thm:relativesitegeometricmorphism} below).
\end{thm}\qed

\begin{remark}
There are some size issues involved in the notion of relative site. We do not require smallness hypotheses in our definition as, for technical reasons, it is convenient to be able to work also with large presentation sites. One can resolve such issues either by working with respect to a bigger Grothendieck universe, or by showing that the relevant sites under consideration are in fact small-generated (for instance, one can show that the canonical site of a geometric morphism, in the sense of Definition \ref{defrelativesiteofamorphism}, is small-generated). 
\end{remark}

\subsection{Orthogonally generated topologies}

Every Grothendieck topology on a category of the form ${\cal G}({\mathbb D})$ has horizontal and vertical data associated with it, that is, the collection of sieves on the base category whose cartesian liftings are covering for the topology, and the collection of covering sieves for the topology which are generated by families of arrows entirely lying in some fibre of $\mathbb D$. 

It is natural to wonder under which conditions and to which extent a Grothendieck topology on ${\cal G}({\mathbb D})$ can be recovered from such horizontal and vertical data. This motivates the following definition:

\begin{defn}
Let $({\cal C}, J)$ be a small-generated site and $\mathbb D$ be a $\cal C$-indexed category. A Grothendieck topology on  $\gbicat(\dcat)$ is said to be \emph{orthogonally generated}\index{topology!orthogonally generated -} if there is a family of horizontal data (i.e. cartesian liftings of $J$-covering sieves) or vertical data (i.e. sieves entirely lying in some fibre ${\mathbb D}(X)$) generating it.
\end{defn}

As it can be naturally expected, and as is shown by the following example, not all topologies on $\gbicat(\dcat)$ are orthogonally generated:

\begin{ex}
	Let $\cal C$ be the category $\twocat$, with two objects $0$ and $1$ and one morphism $t$, and the indexed category $\dcat:\twocat\op\rightarrow\CAT$ mapping $0$ to the category ${\mathbb D}(0)$ (isomorphic to $\twocat$) having two objects $x$ and $x'$ and just one non-identical morphism $\alpha:x\to x'$, $1$ to the terminal category $\onecat=\{*\}$, and $t$ to the functor ${\mathbb D}(1)\to {\mathbb D}(0)$ sending $*$ to the object $x'$. The category  $\gbicat(\dcat)$ has three objects $(0, x)$, $(0, x')$ and $(1, *)$ related by the following morphisms:
	\[\begin{tikzcd}
		{(0, x)} \\
		{(0, x')} && {(1, *)}
		\arrow["{(1, \alpha)}"', from=1-1, to=2-1]
		\arrow["{(t, 1)}", from=2-1, to=2-3]
		\arrow["{(t, \alpha)}", from=1-1, to=2-3]
	\end{tikzcd}\]

    Now, the assignment sending the objects $(0, x)$ and $(0, x')$ to the collections $\{M_{(0, x)}\}$ and $\{M_{(0, x')}\}$ of maximal sieves on them and the object $(1, *)$ to the collection of sieves $\{M_{(1, *)}, \{(t, \alpha)\}, \{(t, \alpha), (t, 1)\}\}$ is clearly a Grothendieck topology which is not of the form $L(H, V)$, as the horizontal and vertical data associated with it \ac forget' the diagonal covering morphism $(t, \alpha)$.   
\end{ex}	

On the other hand, several important Grothendieck topologies on fibrations are generated by horizontal and vertical data:

\begin{enumerate}[(i)]
    \item Any Giraud topology is orthogonally generated (in fact, generated by horizontal data). 
    
    \item The total topology of a fibred site (in the sense of \cite{SGA4_II}) is orthogonally generated (in fact, generated by vertical data).
    
    \item Given an internal locale in $\Sh({\cal C}, J)$ the Grothendieck topology on the site externalizing it is orthogonally generated.
    
    \item The topology on the site presenting the over-topos at a model, introduced in \cite{CaramelloOsmond}, is orthogonally generated.
\end{enumerate}

We shall see in the next section that the canonical relative sites of a geometric morphisms provide a class of Grothendieck topologies which are not in general orthogonally generated; still, it can be shown that the relative topology of a locally connected morphism is always orthogonally generated.

\subsection{The canonical relative site of a geometric morphism}\label{sec:canonicalrelativesite}

Let $f:\Ftopos\rightarrow\Etopos$ be a geometric morphism. As observed in Example \ref{ex:indexedcategories}(ii), there is an $\Etopos$-indexed category $\icat_{f}$ associated with $f$, defined on objects by mapping $E$ of $\Etopos$ to the slice topos $\Ftopos/f^*(E)$, with its transition morphisms being the obvious pullback functors. We can perform this more in general. Consider two sites $(\cbicat,J)$ and $(\dbicat,K)$ and a $(J,K)$-continuous functors $A:\cbicat\rightarrow \dbicat$: then we can consider the $\cbicat$-indexed category $\icat_A$\index{$\icat_A$} defined by mapping any $X$ in $\cbicat$ to the slice $\Sh(\dbicat,K)/\ell_K(A(X))$, and whose transition morphisms are pullback functors. Notice that the fibration $\icat_f$ defined from the geometric morphism $f$ is now a particular instance of this, where we take as continuous functors the morphism of sites $f^*:(\Etopos, J\can_\Etopos)\rightarrow (\Ftopos, J\can_\Ftopos)$. Every fibration of this form is in fact a stack:
\begin{prop}
	Let $A:(\cbicat,J)\rightarrow (\dbicat,K)$ be a $(J,K)$-continuous functor: then the fibration $\icat_A$ defined above is a $J$-stack. In particular, for every geometric morphism $f:\Ftopos\rightarrow\Etopos$ the $\Etopos$-indexed category $\icat_{f}$ of Example \ref{ex:indexedcategories}(ii) is a $J\can_\Etopos$-stack.
\end{prop}
\begin{proof}
	Notice that $\icat_A$ corresponds to the composite pseudofunctor
	\[
	\cbicat\op\xrightarrow{A\op} \dbicat\op\xrightarrow{\canst_{(\dbicat,K)}}\CAT,
	\]
	where $\canst_{(\dbicat,K)}$ denotes the canonical stack for the site $(\dbicat,K)$ (see Definition \ref{def:stack_canonica_su_sito} and Theorem \ref{thm:canst_stack}).Then we can exploit the notion of direct image of fibrations, introduced in Section \ref{sec:dir_imm_fibrations}, and the fact that the direct image along a continuous functor maps stacks to stacks (Proposition \ref{prop:ftcont_preservano_stack}), to conclude that $\icat_A$ is a $J$-stack.
\end{proof}
For a geometric morphism $f:\Ftopos\rightarrow \Etopos$, the fibration associated to $\icat_f$ is made as follows: objects over $E$ in $\Etopos$ are arrows $[u:U\rightarrow f^*(E)]$ of $\Ftopos$, and morphisms $(e,a):[v:V\rightarrow f^*(E')]\rightarrow [u:U\rightarrow f^*(E)]$ are indexed by two arrows $e:E'\rightarrow E$ and $a:V\rightarrow U$ making the diagram
\[
\begin{tikzcd}
	V\ar[d, "v"'] \ar[r, "a"] & U\ar[d,"u"]\\
	f^*(E') \ar[r, "f^*(e)"'] & f^*(E)
\end{tikzcd}\]
commutative. Cartesian arrows of $\gbicat(\icat_f)$ are characterized as those such that the square above is a pullback square.

Notice that the fibration $\gbicat(\icat_\Ftopos)$ corresponds in fact with the comma category $\comma{1_\Ftopos}{f^*}$: we have already met this kind of category in Theorem \ref{thm:morfgeom_presentato_commacategory}, where we showed that any geometric morphism induced by a morphism of sites $A:(\cbicat,J)\rightarrow (\dbicat,K)$ can be described as the geometric morphism induced by the fibration $\comma{1_\dbicat}{A}\rightarrow \cbicat$ upon endowing the domain with a suitable Grothendieck topology $\bar{K}$. Applying that in our specific case, we obtain the following result:

\begin{thm}\label{thm:relativesitegeometricmorphism}
	Let $f:\Ftopos\rightarrow \Etopos$ be a geometric morphism: then there exists a Grothendieck topology $J_f$\index{$J_f$} over $\gbicat(\icat_f)$ such that the two toposes $\Ftopos$ and $\Sh(\gbicat(\icat_f), J_F)$ are equivalent as $\Etopos$-toposes: more specifically, a family $\{(e_i, a_i):[v_i:V_i\rightarrow F^*(E_i)]\to[u:U\rightarrow F^*(E)]\ |\ i\in I\}$ is $J_f$-covering if and only if the family $\{e_i:E_i\rightarrow E\ |\ i\in I\}$ is epimorphic in $\Etopos$.
	
	Thus the geometric morphism $f:\Ftopos\rightarrow\Etopos$ presents $\Ftopos$ as a \emph{topos of relative sheaves over the stack $\icat_f$}:
	\[
	\Ftopos\simeq \Sh(\gbicat(\icat_f), J_f)=:\Sh_{\Etopos}(\icat_f, J_f).
	\]
	We call the Grothendieck topology $J_f$ the \emph{relative topology} of $f$\index{topology! relative - of a geometric morphism}.
\end{thm}
\begin{proof}
	We apply Theorem \ref{thm:morfgeom_presentato_commacategory} to the morphism of sites \[f^*:(\Etopos, J\can_\Etopos)\rightarrow (\Ftopos,J\can_\Ftopos):\] the category $\comma{1_\Ftopos}{f^*}$ can be endowed with a topology $\overline{J\can_\Ftopos}$ such that $\pi_\Etopos:\comma{1_\Ftopos}{f^*}\rightarrow \Etopos$ is a comorphism of sites and $\pi_\Ftopos:\comma{1_\Ftopos}{f^*}\rightarrow \Ftopos$ induces an equivalence of toposes making the diagram
	\[
	\begin{tikzcd}
		\Ftopos \ar[d, "f"']  \ar[r, "\sim", no head] &{\Sh(\Ftopos,J\can_\Ftopos)} \ar[d, "\Sh(f^*)"']\ar[r, no head, "\sim"]& {\Sh(\comma{1_\Ftopos}{f^*}, \overline{J\can_\Ftopos})}\ar[dl, "C_{\pi_\Etopos}"]\\
		\Etopos\ar[r, "\sim", no head] &
		{\Sh(\Etopos,J\can_\Etopos)}&
	\end{tikzcd}
	\] 
	is commutative. Setting $J_f:=\overline{J\can_\Ftopos}$ concludes the proof.
\end{proof}

For any small-generated site $({\cal C}, J)$, we shall denote by $J_{\canst_{(\cbicat,J)}}$ the relative topology on $\canst_{(\cbicat,J)}$ of the identical geometric morphism on $\Sh({\cal C}, J)$. As shown by the following result, the canonical relative site $(\canst_{(\cbicat,J)}, J_{\canst_{(\cbicat,J)}})$ is an alternative site of presentation for $\Sh(\cbicat,J)$:

\begin{cor}\label{cor:canonicalstackrelativesite}
	Consider an essentially small site $(\cbicat,J)$: the canonical stack $\pi_{(\cbicat,J)}:\canst_{(\cbicat,J)}\rightarrow \cbicat$ induces an equivalence of toposes $$C_{\pi_{\canst_{(\cbicat,J)}}}:\Sh(\canst_{(\cbicat,J)}, J_{\canst_{(\cbicat,J)}})\isorightarrow \Sh(\cbicat,J).$$
\end{cor}\qed

\begin{proof}
We recall that $\pi_{\canst_{(\cbicat,J)}}:\canst_{(\cbicat,J)}\rightarrow \cbicat$ was defined as the fibration $\pi:\comma{1_{\Sh(\cbicat,J)}}{\ell_J}\rightarrow \cbicat$, with $\pi$ being the canonical projection onto $\cbicat$ (cf. Definition \ref{def:stack_canonica_su_sito}). Our thesis thus follows from Theorem \ref{thm:relativesitegeometricmorphism}, applied to the identical morphism on $\Sh({\cal C}, J)$.  
\end{proof}

We can consider, more specifically, geometric morphisms induced by arbitrary morphisms of sites: 

\begin{defn}\label{defrelativesiteofamorphism}
Let $A:(\cbicat,J)\rightarrow (\dbicat,K)$ be a morphism of small-generated sites. The \emph{relative site of $A$}\index{site! relative - of a morphism of sites} is the site $({\mathbb I}_{A}, J^{K}_{A})$ (where $J^{K}_{A}$\index{$J^{K}_{A}$} is equal to the topology $\overline{K}$ of Theorem \ref{thm:morfgeom_presentato_commacategory}), together with the canonical projection functor $\pi_{A}:{\mathbb I}_{A} \to {\cal C}$, which is a comorphism of sites $({\mathbb I}_{A}, J^{K}_{A}) \to ({\cal C}, J)$.

Given a geometric morphism $f:{\cal F}\to {\cal E}$, regarded as a morphism of sites $f^{\ast}:({\cal E}, J^{\textup{can}}_{\cal E})\to ({\cal F}, J^{\textup{can}}_{\cal F})$, we call the site $({\mathbb I}_{f^{\ast}}, J^{J^{\textup{can}}_{\cal F}}_{f^{\ast}})$ the \emph{relative site of $f$}\index{site!relative - of a geometric morphism}. 
\end{defn}

Then, by Theorem \ref{thm:morfgeom_presentato_commacategory}, we have the following generalization of Theorem \ref{thm:relativesitegeometricmorphism}:

\begin{thm}
Let $A:(\cbicat,J)\rightarrow (\dbicat,K)$ be a morphism of small-generated sites. Then the geometric morphism $\Sh(A)$ induced by $A$ coincides with the structure geometric morphism $C_{\pi_{A}}$ associated with the relative site $\pi_{A}:({\mathbb I}_{A}, J^{K}_{A})\to ({\cal C}, J)$. 
\end{thm}\qed


\bibliography{biblio}{}

\begin{thebibliography}{10}

\bibitem{SGA4_I}
{\em Th\'{e}orie des topos et cohomologie \'{e}tale des sch\'{e}mas. {T}ome 1}.
\newblock Lecture Notes in Mathematics, Vol. 269. Springer-Verlag, Berlin-New
  York, 1972.
\newblock S\'{e}minaire de G\'{e}om\'{e}trie Alg\'{e}brique du Bois-Marie
  1963--1964 (SGA 4), dirig\'{e} par M. Artin, A. Grothendieck et J. L.
  Verdier, avec la collaboration de N. Bourbaki, P. Deligne et B. Saint-Donat.

\bibitem{SGA4_II}
{\em Th\'{e}orie des topos et cohomologie \'{e}tale des sch\'{e}mas. {T}ome 2}.
\newblock Lecture Notes in Mathematics, Vol. 270. Springer-Verlag, Berlin-New
  York, 1972.
\newblock S\'{e}minaire de G\'{e}om\'{e}trie Alg\'{e}brique du Bois-Marie
  1963--1964 (SGA 4), dirig\'{e} par M. Artin, A. Grothendieck et J. L.
  Verdier, avec la collaboration de N. Bourbaki, P. Deligne et B. Saint-Donat.

\bibitem{anel.spectra}
M.~Anel.
\newblock Grothendieck topologies from unique factorisation systems, 2009.

\bibitem{borceux3}
F.~Borceux.
\newblock {\em Handbook of categorical algebra. 3 - Categories of sheaves},
  volume~52 of {\em Encyclopedia of Mathematics and its Applications}.
\newblock Cambridge University Press, Cambridge, 1994.

\bibitem{bunge.stacksandinternalcat}
M.~Bunge.
\newblock Stack completions and {M}orita equivalence for categories in a topos.
\newblock {\em Cah. Topol. Géom. Différ.}, 20:401--436, 1979.

\bibitem{BungePare}
M.~Bunge and R.~Paré.
\newblock Stacks and equivalence of indexed categories.
\newblock {\em Cah. Topol. Géom. Différ.}, 20(4):373--399, 1979.

\bibitem{caramello2011topostheoretic}
O.~Caramello.
\newblock A topos-theoretic approach to {S}tone-type dualities.
\newblock \emph{arXiv:math.CT/1103.3493}, 2011.

\bibitem{caramello.libro}
O.~Caramello.
\newblock {\em Theories, {S}ites, {T}oposes: Relating and studying mathematical
  theories through topos-theoretic `bridges'}.
\newblock Oxford University Press, 2018.

\bibitem{denseness}
O.~Caramello.
\newblock Denseness conditions, morphisms and equivalences of toposes.
\newblock \emph{arXiv:math.CT/1906.08737}, 2020.

\bibitem{CaramelloOsmond}
O.~Caramello and A.~Osmond.
\newblock The over-topos at a model.
\newblock \emph{arXiv:math.CT/2104.05650}, 2021.

\bibitem{dependent}
O.~Caramello and R.~Zanfa.
\newblock On the dependent product in toposes.
\newblock \emph{arXiv:math.CT/1908.08488} (2019), revised version to appear in
  \emph{Mathematical Logic Quarterly}.

\bibitem{diaconescu.pullback}
R.~{D}iaconescu.
\newblock {C}hange of base for toposes with generators.
\newblock {\em {J}. {P}ure {A}ppl. {A}lgebra}, 6(3):191--218, 1975.

\bibitem{dubuc_poveda}
E.~J. Dubuc and Y.~A. Poveda.
\newblock Representation theory of {M}{V}-algebras.
\newblock {\em Annals of Pure and Applied Logic}, 161(8):1024--1046, 2010.

\bibitem{giraud.cohomologie}
J.~Giraud.
\newblock {\em Cohomologie non ab{\'e}lienne}.
\newblock Grundlehren der mathematischen Wissenschaften. Springer, 1971.

\bibitem{giraud.classifying}
J.~Giraud.
\newblock Classifying topos.
\newblock In F.~W. Lawvere, editor, {\em Toposes, Algebraic Geometry and
  Logic}, pages 43--56. Springer Berlin Heidelberg, 1972.

\bibitem{gray.fibredcofibredcategories}
J.~W. Gray.
\newblock Fibred and cofibred categories.
\newblock In S.~Eilenberg, D.~K. Harrison, S.~MacLane, and H.~R{\"o}hrl,
  editors, {\em Proceedings of the Conference on Categorical Algebra}, pages
  21--83. Springer Berlin Heidelberg, 1966.

\bibitem{hartshorne}
R.~Hartshorne.
\newblock {\em Algebraic geometry}.
\newblock Graduate Texts in Mathematics, No. 52, Springer-Verlag, 1977.

\bibitem{jens.preorder}
J.~Hemelaer.
\newblock A {T}opological {G}roupoid {R}epresenting the {T}opos of {P}resheaves
  on a {M}onoid.
\newblock {\em Appl. Categ. Structures}, 28(5):749--772, 2020.

\bibitem{hollander}
S.~Hollander.
\newblock Diagrams indexed by {G}rothendieck constructions.
\newblock {\em Homology Homotopy Appl.}, 10(3):193--221, 2008.

\bibitem{2dimcategories}
N.~Johnson and D.~Yau.
\newblock 2-dimensional categories.
\newblock \emph{arXiv:math.CT/2002.06055}, 2020.

\bibitem{elephant}
P.~T. Johnstone.
\newblock {\em Sketches of an elephant: a topos theory compendium. {V}ol. 1 and
  2}, volume 43-44 of {\em Oxford Logic Guides}.
\newblock Oxford University Press, Oxford, 2002.

\bibitem{joyal_tierney_galois}
A.~Joyal and M.~Tierney.
\newblock An extension of the {G}alois theory of {G}rothendieck.
\newblock {\em Mem. Amer. Math. Soc.}, 51(309):vii+71, 1984.

\bibitem{kelly2005}
G.~M. Kelly.
\newblock Basic concepts of enriched category theory.
\newblock {\em Repr. Theory Appl. Categ. 10}, pages vi+137, 2005.
\newblock Reprint of the 1982 original, Cambridge University Press.

\bibitem{LaumonBailly}
G.~Laumon and L.~Moret-Bailly.
\newblock {\em Champs algebriques}.
\newblock Springer, 1999.

\bibitem{low.universes}
Z.~L. Low.
\newblock Universes for category theory.
\newblock \emph{arXiv:math.CT/1304.5227}, 2014.

\bibitem{lurie}
J.~Lurie.
\newblock {\em Higher Topos Theory}.
\newblock Princeton University Press, 2009.

\bibitem{maclanemoerdijk}
S.~MacLane and I.~Moerdijk.
\newblock {\em Sheaves in Geometry and Logic: An Introduction to Topos Theory}.
\newblock Springer, 1994.

\bibitem{mclarty.seminar}
C.~McLarty.
\newblock \textit{Grothendieck's 1973 topos lectures}, 2018 seminar.
\newblock See \url{https://www.youtube.com/watch?v=5AR55ZsHmKI}.

\bibitem{nlab:2-limit}
{nLab authors}.
\newblock 2-limit.
\newblock \url{http://ncatlab.org/nlab/show/2-limit}, May 2020.
\newblock \href{http://ncatlab.org/nlab/revision/2-limit/52}{Revision 52}.

\bibitem{nlab:grothendieck_construction}
{nLab authors}.
\newblock {{G}}rothendieck construction.
\newblock \url{http://ncatlab.org/nlab/show/Grothendieck%20construction}, Aug.
  2020.
\newblock
  \href{http://ncatlab.org/nlab/revision/Grothendieck%20construction/62}{Revision
  62}.

\bibitem{nlab:slice_2-category}
{nLab authors}.
\newblock {S}lice 2-category.
\newblock \url{http://ncatlab.org/nlab/show/slice%202-category}, Feb. 2020.
\newblock \href{http://ncatlab.org/nlab/revision/slice%202-category/4}{Revision
  4}.

\bibitem{nlab:street_fibration}
{nLab authors}.
\newblock {{S}}treet fibration.
\newblock \url{http://ncatlab.org/nlab/show/Street%20fibration}, Aug. 2020.
\newblock
  \href{http://ncatlab.org/nlab/revision/Street%20fibration/14}{Revision 14}.

\bibitem{riehlcontext}
E.~Riehl.
\newblock {\em Category Theory in Context}.
\newblock Dover Publications, 2017.

\bibitem{Shulman2}
M.~Shulman.
\newblock Large categories and quantifiers in topos theory, 2021 slides.
\newblock See
  \url{http://home.sandiego.edu/~shulman/papers/cambridge-stacksem.pdf}.

\bibitem{shulman.exact}
M.~Shulman.
\newblock Exact completions and small sheaves.
\newblock {\em Theory and Applications of Categories}, 27(7):97--173, 2012.

\bibitem{Shulman1}
M.~Shulman.
\newblock Comparing material and structural set theories.
\newblock {\em Annals of Pure and Applied Logic}, 170(4):465--504, 2019.

\bibitem{Street.2dimsheaftheory}
R.~Street.
\newblock Two-dimensional sheaf theory.
\newblock {\em Journal of Pure and Applied Algebra}, 23:251--270, 1982.

\bibitem{vistoli.stack}
A.~Vistoli.
\newblock Notes on {G}rothendieck topologies, fibered categories and descent
  theory.
\newblock \emph{arXiv:math.AG/0412512}, 2004.

\end{thebibliography}
\bibliographystyle{abbrv}

\vspace{1.5cm}

\textbf{Acknowledgements:} We are grateful to IHES for making it possible to arrange a visit of the second-named author to the first-named author, during which a part of this work has been written. 

\newpage

\textsc{Olivia Caramello} 

\vspace{0.2cm}
{\small \textsc{Dipartimento di Scienza e Alta Tecnologia, Universit\`a degli Studi dell'Insubria, via Valleggio 11, 22100 Como, Italy.}\\
	\emph{E-mail address:} \texttt{olivia.caramello@uninsubria.it}}

\vspace{0.2cm}

{\small \textsc{Institut des Hautes \'Etudes Scientifiques, 35 Route de Chartres,
		91440 Bures-sur-Yvette, France.}\\
	\emph{E-mail address:} \texttt{olivia@ihes.fr}}

\vspace{0.6cm}

\textsc{Riccardo Zanfa} 

\vspace{0.2cm}
{\small \textsc{Dipartimento di Scienza e Alta Tecnologia, Universit\`a degli Studi dell'Insubria, via Valleggio 11, 22100 Como, Italy.}\\
	\emph{E-mail address:} \texttt{rzanfa@uninsubria.it}}

\appendix

\chapter{Scheme of the adjunctions}

\section{Adjoints to the Grothendieck construction}
\[\begin{tikzcd}
	\Cl\Fib_\cbicat \arrow[r, "\Lambda_{\CAT/\cbicat}", bend left, ""{below, name=A}] &   \CAT/\cbicat \arrow[l, "\Gamma_{\CAT/\cbicat}", bend left, ""{above, name=B}] \ar[l, "\dashv"{rotate=270}, phantom]
\end{tikzcd},\]
\[\begin{tikzcd}
	\Ind_\cbicat \ar[r, bend left, "\Lambda_{\Cosite/(\cbicat,J)}", start anchor={north east}, end anchor={north west}] \ar[r, phantom, "\dashv"{rotate=270}]& \Com/(\cbicat,J) \ar[l, bend left, "{\Gamma_{\Cosite/(\cbicat,J)}}", start anchor={south west}, end anchor={south east}]
\end{tikzcd}\]
\[ \begin{tikzcd}
	\Ind_\cbicat \ar[r, bend left, "\Lambda_{\Cosite\cont/(\cbicat,J)}", start anchor={north east}, end anchor={north west}] \ar[r, phantom, "\dashv"{rotate=270}]& \Com\cont/(\cbicat,J) \ar[l, bend left, "{\Gamma_{\Cosite\cont/(\cbicat,J)}}", start anchor={south west}, end anchor={south east}]
\end{tikzcd}\]
\[\begin{tikzcd}
	{[\cbicat\op,\Set]}  \ar[r, bend left, "\Lambda_{\Cat/_1\cbicat}", start anchor={north east}, end anchor={north west}] \ar[r, phantom, "\vdash"{rotate=90}]& {\Cat/_1\cbicat} \ar[l, bend left, "{\Gamma_{\Cat/_1\cbicat}}", start anchor={south west}, end anchor={south east}]	
\end{tikzcd}\]

\section{Fundamental adjunctions}
\[\begin{tikzcd}
	\Cl\Fib^J_\cbicat \arrow[r, "\Lambda_{\Topos/\Sh(\cbicat,J)\co}", bend left, start anchor={north east}, end anchor={north west}] \ar[r,phantom, "\vdash"{rotate=90}]&   \Topos/\Sh(\cbicat,J)\co\phantom{^V} \arrow[l, "\Gamma_{\Topos/\Sh(\cbicat,J)\co}", bend left,  start anchor={south west}, end anchor={south east}]
\end{tikzcd},\ 
\begin{tikzcd}
	\Cl\Fib^J_\cbicat \arrow[r, "\Lambda_{\EssTopos/\Sh(\cbicat,J)\co}", bend left, start anchor={north east}, end anchor={north west}] \ar[r,phantom, "\vdash"{rotate=90}]&   \EssTopos/\Sh(\cbicat,J)\co\phantom{^V} \arrow[l, "\Gamma_{\EssTopos/\Sh(\cbicat,J)\co}", bend left,  start anchor={south west}, end anchor={south east}]
\end{tikzcd}\]

\section{Discrete adjunctions}
\[
\begin{tikzcd}
	{[\cbicat\op,\Set]} \arrow[r, "\Lambda_{\Topos^s/_1\Sh(\cbicat,J)}", bend left, start anchor={north east}, end anchor={north west}] &     \Topos^s/_1\Sh(\cbicat,J) \arrow[l, "\Gamma_{\Topos^s/_1\Sh(\cbicat,J)}", bend left, start anchor={south west}, end anchor={south east}] \ar[l, "\dashv"{rotate=270}, phantom]
\end{tikzcd}.\]
\[\begin{tikzcd}
	{[\cbicat\op,\Set]} \ar[r, bend left, "\Lambda_{\Preord/_1\cbicat}", start anchor={north east}, end anchor={north west}] \ar[r, phantom, "\dashv"{rotate=270}] & \Preord/_1\cbicat \ar[l,bend left, "\Gamma_{\Preord/_1\cbicat}", end anchor={south east}, start anchor={south west}]
\end{tikzcd}\]
\[\begin{tikzcd}
	{[\cbicat\op,\Set]} \arrow[r, "\Lambda_{\Locale/_1\Id_J(\cbicat)}", bend left, end anchor={north west}, start anchor={north east}] &     \Locale/_1\Id_J(\cbicat) \arrow[l, "\Gamma_{\Locale/_1\Id_J(\cbicat)}", bend left, start anchor={south west}, end anchor={south east}] \ar[l, "\dashv"{rotate=270}, phantom]
\end{tikzcd}\]\[
\begin{tikzcd}
	{[\Ocal(L)\op,\Set]} \arrow[r, "\Lambda_{\Locale/_1L}", bend left, end anchor={north west}, start anchor={north east}] &     \Locale/_1L \arrow[l, "\Gamma_{\Locale/_1L}", bend left, start anchor={south west}, end anchor={south east}] \ar[l, "\dashv"{rotate=270}, phantom]
\end{tikzcd}\]
\[\begin{tikzcd}
	\Lambda:{\Psh(X)} \ar[r, bend left,  start anchor={[yshift=1ex]east}, end anchor={[yshift=1ex]west}] \ar[r,phantom, "\dashv"{rotate=-90}] & \Top/X \ar[l, bend left, start anchor={[yshift=-1ex]west}, end anchor={[yshift=-1ex]east}]:\Gamma
\end{tikzcd}.\]

\chapter{Some results on Grothendieck universes}\label{app:universes}
In this appendix we will recap some known results about Grothendieck universes, along with some original results. Our main interest is in the study of geometric morphisms between toposes of sheaves valued in different universes. Universes were first introduced in the appendix of Exposé I in \cite{SGA4_I}; for the following we also refer to \cite{low.universes}.

First of all, let us recall that a \emph{Grothendieck universe}\index{universe} is a set $\ubicat$ satisfying the following four axioms:
\begin{enumerate}[(i)]
	\item if $x\in \ubicat$ and $y\in x$ then $y\in \ubicat$;
	\item if $x,y\in\ubicat$ then $\{x,y\}\in\ubicat$;
	\item if $x\in \ubicat$ then $\powerset(x)\in\ubicat$;
	\item if $I\in \ubicat$, for any map $f:I\rightarrow \ubicat$ then $\bigcup_{i\in I} f(i)\in\ubicat$.
\end{enumerate}
Note that this is called \emph{pre-universe} in \cite{low.universes}, and a universe is a pre-universe that contains the set $\omega$ of von Neumann finite ordinals. This latter condition is equivalent to asking that a univers is not empty, as remarked after Proposition 7 in \cite[Appendice]{SGA4_I}. 
In short, a universe is a set closed under the usual set theoretic operations that can be performed on its elements. Universes can be used to avoid the dichotomy set/class, when dealing with size issues in category theory: instead of small sets, one can speak about \emph{$\ubicat$-small} sets, \ie sets that are isomorphic to an element of $\ubicat$. If one is given a set which is not $\ubicat$-small, \ie which is \emph{$\ubicat$-large}, one can suppose that there is a wider universe $\vbicat$ containing both said set and $\ubicat$, and thus  `widen the horizon'. We recall though that the assumption that every set be contained in a universe is a powerful set-theoretic axiom, which implies the existence of a strongly inaccessible cardinal containing every other chosen cardinal. 

A category $\cbicat$ is a \emph{$\ubicat$-category} if for every pair of objects $Y$ and $X$ of $\cbicat$ the hom-set $\cbicat(Y,X)$ is $\ubicat$-small; the category $\cbicat$ is \emph{locally $\ubicat$-small category} in the terminology of \cite{low.universes}. In a similar fashion, one sais that $\cbicat$ is \emph{$\ubicat$-small} if its set of morphisms is $\ubicat$-small.

Denote by $\Set_\ubicat$ the topos of $\ubicat$-small sets: then it is a $\ubicat$-category. Given a further universe $\vbicat$ such that $\ubicat\in \vbicat$ (by axiom (ii) we have also $\ubicat\subseteq \vbicat$) we have an obvious full and faithful inclusion
\[\Set_\ubicat\hookrightarrow\Set_\vbicat.\]
For every category $\cbicat$ we can consider its $\ubicat$-presheaves, \ie the contravariant functors $\cbicat\op\rightarrow\Set_\ubicat$; if $\cbicat$ is $\ubicat$-small then $[\cbicat\op,\Set_\ubicat]$ is locally $\ubicat$-small. If $\ubicat\in \vbicat$, a $\ubicat$-category $\cbicat$ is also a $\vbicat$-category and we have again a full and faithful inclusion
\[
[\cbicat\op,\Set_\ubicat]\hookrightarrow[\cbicat\op,\Set_\vbicat].
\]
Now consider a site $(\cbicat,J)$: it is called a \emph{$\ubicat$-site} if it admits a $J$-dense full subcategory that is $\ubicat$-small (see \cite[Definitions 3.0.1 and 3.0.2]{SGA4_I}); alternatively one could say that $(\cbicat,J)$ is \emph{$\ubicat$-small generated}. In particular, the site is $\ubicat$-small if $\cbicat$ is a $\ubicat$-small category. One can speak of $\ubicat$-sheaves on a $\ubicat$-site, thus obtaining the topos $\Sh_\ubicat(\cbicat,J)$: then a \emph{$\ubicat$-topos} is a $\ubicat$-category equivalent to a category of $\ubicat$-sheaves for a $\ubicat$-site. Moreover, in analogy with the presheaf case, given an inclusion of universes $\ubicat\subseteq\vbicat$ one can consider the $\ubicat$-site $(\cbicat,J)$ as a $\vbicat$-site, thus producing a full and faithful inclusion of sheaf toposes
\[\Sh_\ubicat(\cbicat,J)\hookrightarrow\Sh_\vbicat(\cbicat,J).\]
Finally, we recall that all usual properties of toposes are true when one works with universes: for instance, a topos of $\ubicat$-sheaves is closed under $\ubicat$-small limits and colimits. This implies that when one considers the sheafification functor $[\cbicat\op,\Set]\rightarrow\Sh_\ubicat(\cbicat,J)$, which is defined by colimits, it commutes with the expansion of universes: that is, we are left with the two (essentially) commutative squares
\[
\begin{tikzcd}
	{\Sh_\vbicat(\cbicat,J)} \ar[r, hook]
	& {[\cbicat\op,\Set_\vbicat]}
	\\
	{\Sh_\ubicat(\cbicat,J)} \ar[u, hook]\ar[r,hook] & {[\cbicat\op,\Set_\ubicat]} \ar[u, hook]
\end{tikzcd}
\begin{tikzcd}
	{\Sh_\vbicat(\cbicat,J)}	& {[\cbicat\op,\Set_\vbicat]}\ar[l, "\sheafify_J"']
	\\
	{\Sh_\ubicat(\cbicat,J)} \ar[u,hook] & {[\cbicat\op,\Set_\ubicat]} \ar[l, "\sheafify_J"']\ar[u, hook]
\end{tikzcd}
\]
This is essentially the content of \cite[Exposé II, Proposition 3.6]{SGA4_I}.

As we have anticipated above, our main interest is to study the behaviour of geometric morphisms with respect to the change of universe. In the following we will always assume $\ubicat\in\vbicat$ to be two universes and $(\cbicat,J)$ and $(\dbicat,K)$ to be $\ubicat$-sites. Consider two geometric morphisms $F:\Sh_\ubicat(\dbicat,K)\rightarrow \Sh_\ubicat(\cbicat,J)$ and $G:\Sh_\vbicat(\dbicat,K)\rightarrow \Sh_\vbicat(\cbicat,J)$: we will say that $G$ is an extension of $F$, or that $F$ is a restriction of $G$, if the two squares
\[
\begin{tikzcd}
	{\Sh_\vbicat(\dbicat,K)} \ar[r, "G_*"] & {\Sh_\vbicat(\cbicat,J)}\\
	{\Sh_\ubicat(\dbicat,K)} \ar[u,hook] \ar[r, "F_*"] & {\Sh_\ubicat(\cbicat,J)}\ar[u, hook]
\end{tikzcd}
\begin{tikzcd}
	{\Sh_\vbicat(\dbicat,K)} \ar[r, leftarrow, "G^*"] & {\Sh_\vbicat(\cbicat,J)}\\
	{\Sh_\ubicat(\dbicat,K)} \ar[u,hook] \ar[r, "F^*", leftarrow] & {\Sh_\ubicat(\cbicat,J)}\ar[u, hook]
\end{tikzcd}
\]
are commutative up to natural isomorphism. For instance, we can restate our previous considerations about $\sheafify_J$ by saying that $\Sh_\vbicat(\cbicat,J)\hookrightarrow[\cbicat\op,\Set_\vbicat]$ restricts to $\Sh_\ubicat(\cbicat,J)\hookrightarrow[\cbicat\op,\Set_\ubicat]$. 
This holds more in general for any geometric morphism induced by a morphism or a comorphism of sites:
\begin{prop}\label{prop:universes}
	Let ${\cal U}\subseteq {\cal V}$ be universes and $({\cal C}, J)$ and $({\cal D}, K)$ be $\cal U$-sites. 
	\begin{enumerate}[(i)]
		\item Let $F:({\cal C}, J)\to ({\cal D}, K)$ be a morphism of $\cal U$-sites. Then the $F$ is also a morphism of $\cal V$-sites $({\cal C}, J)\to ({\cal D}, K)$ and the geometric morphism $\Sh_{\cal U}(F):\Sh_{\cal U}({\cal D}, K) \to \Sh_{\cal U}({\cal C}, J)$ is (up to isomorphism) the restriction of $\Sh_{\cal V}(F):\Sh_{\cal V}({\cal D}, K) \to \Sh_{\cal V}({\cal C}, J)$.
		
		\item Let $G:({\cal D}, K)\to ({\cal C}, J)$ be a comorphism of $\cal U$-sites. Then $G$ is also a comorphism of $\cal V$-sites $({\cal D}, K)\to ({\cal C}, J)$, and geometric morphism $(C_{G})_{\cal U}:\Sh_{\cal U}({\cal D}, K)\to \Sh_{\cal U}({\cal C}, J)$ is given by the restriction of $(C_{G})_{\cal V}:\Sh_{\cal V}({\cal D}, K)\to \Sh_{\cal V}({\cal C}, J)$. 
	\end{enumerate}	
\end{prop}	
\begin{proof}
	\begin{enumerate}[(i)]
		\item Since both $\Sh_\ubicat(F)_*$ and $\Sh_\vbicat(F)_*$ act as the precomposition $-\circ F\op$, evidently the former is the restriction of the latter. Moreover, \cite[Exposé II, Proposition 1.5]{SGA4_I} shows that $F$ is $(J,K)$-continuous as a $\ubicat$-functor if and only if it is $(J,K)$-continuous as a $\vbicat$-functor, and that $\Sh_\ubicat(F)^*$ is the restriction of $\Sh_\vbicat(F)^*$. So far we have the two (essentially) commutative squares
		\[
		\begin{tikzcd}
			{\Sh_\vbicat(\dbicat,K)} \ar[r, "Sh_\vbicat(F)_*"] & {\Sh_\vbicat(\cbicat,J)}\\
			{\Sh_\ubicat(\dbicat,K)} \ar[u,hook] \ar[r, "Sh_\ubicat(F)_*"] & {\Sh_\ubicat(\cbicat,J)}\ar[u, hook]
		\end{tikzcd}
		\begin{tikzcd}
			{\Sh_\vbicat(\dbicat,K)} \ar[r, leftarrow, "Sh_\vbicat(F)^*"] & {\Sh_\vbicat(\cbicat,J)}\\
			{\Sh_\ubicat(\dbicat,K)} \ar[u,hook] \ar[r, "Sh_\ubicat(F)^*", leftarrow] & {\Sh_\ubicat(\cbicat,J)}\ar[u, hook]
		\end{tikzcd}.
		\] 
		We are left with considerations on the flatness of $F$: that is, we want to know whether $\Sh_\ubicat(F)^*$ preserves finite limites if and only if $\Sh_\vbicat(F)^*$ does. Since $\Sh_\ubicat(F)^*$ is a restriction of $\Sh_\vbicat(F)^*$ one implication is obvious. For the converse we can resort to the definition of morphism of sites provided in \cite[Definition 3.2]{denseness}: $F$ is a morphism of sites as a $\ubicat$-functor if and only if it satisfies the four requirements of said definition, which are site-theoretic and thus are independent from the universe of choice. 
		\item Again, since $(C_G)_\ubicat^*$ and $(C_G)_\vbicat^*$ both act as the precomposition $-\circ G\op$, the former is a restriction of the latter. Finally, \cite[Exposé II, Proposition 2.3(4)]{SGA4_I} shows that $(C_G)_{\ubicat*}$ is the restriction of $(C_G)_{\vbicat*}$. 
	\end{enumerate}
\end{proof}
The previous proposition is based on the consideration that the properties `being a morphism of sites' and `being a comorphism of sites' can be formulated entirely at the level of the sites. We cans state this as a general principle:
\begin{metathm}
	Consider a morphism (resp. comorphism) of $\ubicat$-sites $F:(\cbicat,J)\rightarrow (\dbicat,K)$: then any property $P$ of $\Sh_\ubicat(F)$ (resp. $(C_F)_\ubicat$) that is site-theoretic is stable under extension. \end{metathm}
\begin{proof}
	Suppose that $\Sh_\ubicat(F)$ satisfies a property $P$ if and only if the morphism of sites $F$ satisfies a site-theoretic property $Q$: that is, $Q$ can be expressed in terms of $\ubicat$-small families of arrows and objects of the $\ubicat$-sites $(\cbicat,J)$ and $(\dbicat,K)$. Now consider any wider universe $\vbicat\ni \ubicat$: since $\ubicat$-small sets are $\vbicat$-small, if $F$ satisfies $Q$ as a morphism of $\ubicat$-sites it obviously satisfies $Q$ as a morphism of $\vbicat$-sites, and thus $\Sh_\vbicat(F)$ satisfies $P$.
\end{proof}
Notice that the converse may not hold: if the sites are not $\ubicat$-small, we may have that $F$ satisfies a property for $\vbicat$-sites (for instance, involving a sufficiently large set of morphisms) without it satisfying the same property for $\ubicat$-sites.

On the matter of restrictibility of properties, we can start by remarking that some constructions on sheaves are preserved and reflected by the inclusion $\Sh_\ubicat(\cbicat,J)\hookrightarrow\Sh_\vbicat(\cbicat,J)$: for instance, it is known that a $\ubicat$-small diagram $D$ of $\ubicat$-sheaves in $\Sh_\vbicat(\cbicat,J)$, has as co-/limit a $\ubicat$-sheaf, and it coincides with the co-/limit calculated in $\Sh_\ubicat(\cbicat,J)$. In general, we will call \emph{operation} any process through which we can associate to a diagram in a category a further object, such as the calculation of limits and colimits. We will say that an operation is \emph{stable under restriction} if, whenever $\ubicat\in \vbicat$, the choice of $\ubicat$-small diagrams of $\ubicat$-sheaves yields a $\ubicat$-sheaf as output. Now suppose that a a geometric morphism $G:\Sh_\vbicat(\dbicat,K)\rightarrow\Sh_\vbicat(\cbicat,J)$ admits a restriction $F:\Sh_\ubicat(\dbicat,K)\rightarrow \Sh_\ubicat(\cbicat,J)$, and suppose that $G$ satisfies some property $P$ which can be characterized by operations that are stable under restriction. Then of course $F$ satisfies the same property: this because, the action of $F$ on $\ubicat$-sheaves corresponds (up to isomorphism) to that of $G$, and the property $P$, when applied to $\ubicat$-small diagrams of $\ubicat$-sheaves, yields again $\ubicat$-sheaves. We end up with the following metatheorem:
\begin{metathm}
	Consider a geometric morphism $G:\Sh_\vbicat(\dbicat,K)\rightarrow\Sh_\vbicat(\cbicat,J)$: any property of $G$ that can be formulated in terms of operations that stable under restriction is stable under restriction.
\end{metathm}
An immediate consequence of these two metatheorems are the two following results: 
\begin{prop}
	Consider two universes $\ubicat\in\vbicat$, two $\ubicat$-sites $(\cbicat,J)$ and $(\dbicat,K)$, and two geometric morphisms $F:\Sh_\ubicat(\dbicat,K)\rightarrow \Sh_\ubicat(\cbicat,J)$ and $G:\Sh_\vbicat(\dbicat,K)\rightarrow \Sh_\vbicat(\cbicat,J)$ such that $F$ is a restriction of $G$. Then the following holds:
	\begin{enumerate}[(i)]
		\item if $G$ is an embedding, $F$ is an embedding;
		\item if $G$ is a surjection, $F$ is a surjection;
		\item if $G$ is essential, $F$ is essential;
		\item if $G$ is local (\ie $G_*$ has a fully faithful right adjoint $G^!$), $F$ is local; moreover, $F^!$ is the restriction of $G^!$ to $\ubicat$-sheaves.
	\end{enumerate}
\end{prop}
\begin{proof}
	\begin{enumerate}[(i)]
		\item The geometric morphism $G$ is an inclusion if and only if $G_*$ is fully faithful: since the two functors $\Sh_\ubicat(\cbicat,J)\hookrightarrow \Sh_\vbicat(\cbicat,J)$ and $\Sh_\ubicat(\dbicat,K)\hookrightarrow \Sh_\vbicat(\dbicat,K)$ are also fully faithful, this immediately forces $F_*$ to be fully faithful and thus $F$ is an embedding.
		\item $G$ is a surjection if and only if $G^*$ is faithful: by applying an argument similar to that of the previous item, we immediately conclude that $F^*$ is faithful and hence $F$ is a surjection.
		\item $G$ is essential if and only if $G^*$ has a further left adjoint, and this happens if and only if $G^*$ preserves $\vbicat$-small colimits: but this implies that $F^*$ preserves $\ubicat$-small colimits, and thus $F$ is essential. 
		\item $G_*$ has a right adjoint if and only if $G_*$ preserves arbitrary limits: we then have the same argument as in the previous item. Notice now that the right adjoint $G^!$ of $G_*$ an be defined as follows: for $P\in \Sh_\vbicat(\cbicat,J)$, $G^!(P):\dbicat\op\rightarrow\Set$, $G^!(P)(D):=\Sh_\vbicat(\cbicat,J)(G_*(\ell_J(D)),P)$. If $P\in\Sh_\ubicat(\cbicat,J)$, a quick computation shows that $$G^!(P)(D)\simeq\Sh_\ubicat(\cbicat,J)(F_*(\ell_J(D)),P)=F^!(P)(D),$$ and thus $F^!$ restricts $G^!$.
		Finally, the considerations on the full faithfulness of $F^!$ are the same as those in the previous items.
	\end{enumerate}
\end{proof}
\begin{prop}
	Consider a morphism of $\ubicat$-sites $F:(\cbicat,J)\rightarrow (\dbicat,K)$:\begin{enumerate}[(i)]
		\item $\Sh_\ubicat(F)$ is a surjection if and only if $\Sh_\vbicat(F)$ is a surjection.
		\item $\Sh_\ubicat(F)$ is an embedding if and only if $\Sh_\vbicat(F)$ is an embedding.
	\end{enumerate}
	Consider a comorphism of $\ubicat$-sites $F:(\dbicat,J)\rightarrow (\cbicat,J)$:
	\begin{enumerate}[(i)]
		\setcounter{enumi}{2}
		\item $(C_F)_\ubicat$ is a surjection if and only if $(C_F)_\vbicat$ is a surjection.
		\item $(C_F)_\ubicat$ is an embedding if and only if $(C_F)_\vbicat$ is an embedding.
	\end{enumerate}
\end{prop}
\begin{proof}
	All the restrictions were proven in the previous result. Conversely:
	\begin{enumerate}[(i)]
		\item $\Sh_\ubicat(F)$ is a surjection if and only if $F$ is cover-reflecting, by \cite[Theorem 6.3(i)]{denseness}: since said property is site-theoretic, $\Sh_\vbicat(F)$ is also a surjection.
		\item By \cite[Theorem 6.3(iii)]{denseness}, $\Sh_\ubicat(F)$ is an embedding if and only if $F:(\cbicat,J_F)\rightarrow (\dbicat,K)$ is a weakly dense morphism of sites, where $J_F$ is the topology over $\cbicat$ of those sieves that are sent by $F$ to $K$-covering families. In turn, \cite[Proposition 5.5]{denseness} formulates the weak denseness condition purely in site-theoretic terms, and thus $\Sh_\vbicat(F)$ is also an embedding.
		\item \cite[Proposition 7.1]{denseness} provides a purely site-theoretic description for $(C_F)_\ubicat$ to be surjective, and thus it is stable under extension.
		\item \cite[Proposition 7.6]{denseness} characterizes the property of $(C_F)_\ubicat$ to be an inclusion quantifying over arrows of the site $(\dbicat,K)$ and the presheaf topos $[\dbicat\op,\Set_\ubicat]$: therefore, said property is still satisfied when $F$ is considered as a comorphism of $\vbicat$-sites.
	\end{enumerate}
\end{proof}

\printindex
\end{document}